\documentclass[oneside, a4paper,11pt,reqno]{amsart}
\usepackage{yfonts}

\RequirePackage[colorlinks,citecolor=blue,urlcolor=blue]{hyperref}

\makeatletter
\def\timenow{\@tempcnta\time
  \@tempcntb\@tempcnta
  \divide\@tempcntb60
  \ifnum10>\@tempcntb0\fi\number\@tempcntb
  \multiply\@tempcntb60
  \advance\@tempcnta-\@tempcntb
  :\ifnum10>\@tempcnta0\fi\number\@tempcnta}
\makeatother

\usepackage[ddmmyyyy,hhmmss]{datetime}
\usepackage{color}

\usepackage{euscript}

\usepackage[applemac]{inputenc}
\usepackage{amssymb,amsmath,amsthm,dsfont}
\usepackage{graphics,graphicx}
\usepackage[T1]{fontenc}
\usepackage{color}
\usepackage{epsfig,psfrag}

\newtheorem{theo}{Theorem}[section]
\newtheorem{prop}[theo]{Proposition}
\newtheorem{lemme}[theo]{Lemma}

\newtheorem{remarque}[theo]{Remark}
\newtheorem{fact}[theo]{Fact}

\newcommand{\egloi}{\stackrel{\textgoth{L}}{=}}

\def\tA2{\tilde{A}({\tilde L_2})}
\def \tts2{\tilde \tau^+_2(  h_t/2 ))}

\title{Path decomposition of a spectrally negative L\'evy process,
and local time of a diffusion in this environment}

\author{Gr\'egoire V\'echambre}
\address{NYU Shanghai, Office 1133, 1555 Century avenue, Pudong New Area, Shanghai 200122, China}

\email{gregoire.vechambre@ens-rennes.fr}


\date{\today\ \`a \currenttime}
\begin{document} 

\maketitle

\begin{abstract}
We study the convergence in distribution of the supremum of the local time and of the favorite site for a transient diffusion in a spectrally negative L\'evy potential. To do so, we study the $h$-valleys of a spectrally negative L\'evy process, and we prove in particular that the renormalized sequence of the $h$-minima converges to the jumping times sequence of a standard Poisson process. 
\end{abstract}

\pagestyle{myheadings}
\markboth{Right}{Path decompostion of a spectrally negative L\'evy process,
and local time of a diffusion in this environment}

%\tableofcontents

\section{Introduction}

{Let $(V(x), \ x \in \mathbb{R})$ be a spectrally negative L\'evy process on $\mathbb{R}$ which is not the opposite of a subordinator (in particular, $V$ can not be a compound Poisson process), drifts to $-\infty$ at $+\infty$, and such that $V(0)=0$. }We denote the Laplace exponent of $V$ by $\Psi_V$: 
\[ \forall t, \lambda \geq 0, \ \mathbb{E} \left [ e^{\lambda V(t)} \right ] = e^{t \Psi_V(\lambda)}. \]
It is well-known, for such $V$, that $\Psi_V$ is {convex} and admits a non trivial zero that we denote here by $\kappa$, $\kappa := \inf \{ \lambda > 0, \ \Psi_V(\lambda) = 0 \} > 0$. By convexity we have $\Psi_V'(\kappa) > 0$. 
%We also make two extra assumptions on $V$: 
%\begin{itemize}
%\item $V$ has unbounded variation, 
%\item $V(1) \in L^{p}$ for some $p>1$. 
%\end{itemize}

%Note that the hypothesis that $V$ has unbounded variation implies that $V$ is not the opposite of a subordinator, which impose $V$ to fluctuate so we can have a valleys decomposition for the paths of $V$. Moreover, the fact that $V$ has actually unbounded variation and is not, only, not the opposite of a subordinator is necessary to have continuity at local minima and apply excursion theory to the reflected processes $V - \underline{V}$ and $\hat{V} - \underline{\hat{V}}$ where $\hat{V}:= -V$ is the dual process. This will allow us to identify the laws in the valleys decomposition. 
%
%The technical hypothesis of the existence of a moment of order greater than $1$ for $V$ allows us to apply a result from \cite{Bertoinyor} which states that the tail distribution of some exponential functional of $V$ is of order $x^{-\kappa}$. In \cite{AndDev} and \cite{advech}, they use a valleys decomposition for the drifted Brownian motion $W_{\kappa} := W(.) - \frac{\kappa}{2} .$ and a similar estimate on an exponential functional of $W_{\kappa}$ to study the diffusion in the random environment $W_{\kappa}$ when $\kappa \in ]0, 1[$. 
%
We are here interested in a diffusion in this potential $V$. Let us recall that such a diffusion $(X(t),\ t\geq 0)$ in a random c\`ad-l\`ag potential $V$ is defined informally by $X(0)=0$ and
\[ \text{d}X(t)=d\beta(t)-\frac{1}{2}V'(X(t))\text{d}t, \]
where $\beta$ is a Brownian motion independent from $V$. Rigorously, $X$ is defined by its conditional generator given $V$, 
\begin{align*} & \frac{1}{2}e^{V(x)}\frac{\text{d}}{\text{d} x}\left(e^{-V(x)}\frac{\text{d}}{\text{d} x}\right). 
\end{align*}

The fact that $V$ drifts to $- \infty$ puts us in the case where the diffusion $X$ is a.s. transient to the right (see Subsection \ref{factsandnotations} for more details). In \cite{Singh}, Singh makes the study of the asymptotic behavior of $X$. When $0 < \kappa < 1$, he proves in particular that
\begin{eqnarray}
X(t)/t^{\kappa} \overset{\mathcal{L}}{\underset{t \rightarrow + \infty}{\longrightarrow}} C \left ( 1/\mathcal{S}_{\kappa} \right )^{\kappa}, \label{cvsingh}
\end{eqnarray}
where $C$ is some explicit positive constant depending on $V$ and $\mathcal{S}_{\kappa}$ follows a completely asymmetric $\kappa$-stable distribution. Putting this in relation with the results of Kawazu and Tanaka \cite{KawazuTanaka}, we see that $\kappa$ plays a similar role as the drift of the Brownian environment, at least for the asymptotic behavior of the diffusion. In this paper, we prove that the same is true for the behavior of the local time of $X$. We denote by $(\mathcal{L}_X(t, x), t \geq 0, x \in \mathbb{R})$ the version of the local time that is continuous in time and c\`ad-l\`ag in space, and we define respectively the supremum of the local time and the favorite site until instant $t$ as
\[ \mathcal{L}_X^*(t) = \sup_{x \in \mathbb{R}} \mathcal{L}_X(t, x) \ \ \ \text{and} \ \ \ F^*(t) := \inf \{ x \in \mathbb{R}, \ \mathcal{L}_X(t, x) \vee \mathcal{L}_X(t, x-) = \mathcal{L}_X^*(t) \}. \]
We study the convergence in distribution of $\mathcal{L}_X^*(t)$ and $F^*(t)$ when $\kappa > 1$, and of $\mathcal{L}_X^*(t)$ when $0 < \kappa < 1$. When $V$ is a drifted Brownian motion, the case where $0 < \kappa < 1$ has already been studied by Andreoletti, Devulder and V\'echambre in \cite{advech}. In this case, there is a useful renewal structure obtained from a valleys decomposition of the potential and the Markov property for the diffusion. 
%This, and the knowledge of the right tails of the contributions to the local time and to the time spent by the diffusion in each valley, allow to prove the convergence of the supremum of local time $\mathcal{L}_X^*(t)$. 
The limit distribution they obtain involves a $\kappa$-stable subordinator and an exponential functional of the environment conditioned to stay positive. When $0 < \kappa < 1$, we extend their result to the diffusion in $V$ and obtain a limit distribution in terms of the exponential functionals of $V$ and its dual conditioned to stay positive: 
\[ I(V^{\uparrow}) := \int_0^{+ \infty} e^{- V^{\uparrow} (t)}dt \ \ \ \text{and} \ \ \ I(\hat V^{\uparrow}) := \int_0^{+ \infty} e^{- \hat V^{\uparrow} (t)}dt. \]
{In the above definitions $\hat V$ is the dual of $V$, it is equal in law to $-V$, and $V^{\uparrow}$, $\hat V^{\uparrow}$ denote respectively $V$ and $\hat V$ conditioned to stay positive (their definitions are recalled in Subsection \ref{extrema}).} It is proved in Theorems 1.1 and 1.13 of V\'echambre \cite{foncexpovech} that these functionals are indeed finite and admit some finite exponential moments. The existence of finite exponential moments for $I(V^{\uparrow})$ and $I(\hat V^{\uparrow})$ is of fundamental interest for our generalization of the results of \cite{advech}. 

The almost sure asymptotic behavior of the local time is studied, in the discrete transient case, by Gantert and Shi \cite{GanShi}. We believe that the present work will allow us, in the future, to link the asymptotic almost sure behavior of $\mathcal{L}_X^*(t)$ with the left tail of $I(V^{\uparrow})$ {when $0 < \kappa < 1$, this is a work in preparation by the author \cite{psvech}}. This tail is investigated in \cite{foncexpovech} and can be very different, when $V$ is a general spectrally negative L\'evy process, than when $V$ is a drifted Brownian motion. As a consequence we can expect, in the L\'evy case {with $0 < \kappa < 1$}, many possible behaviors for the almost sure asymptotic of $\mathcal{L}_X^*(t)$, and this is our main motivation to generalize, here, the study of \cite{advech} to the L\'evy case. 
%In particular, when $V$ has bounded variation we expect the $limsup$ of $\mathcal{L}_X^*(t)$ to behave similarly as in the discrete transient case that is studied by Gantert and Shi \cite{GanShi}. 
%+ COMME CA ON A LES COMPORTEMENTS ANALOGUES AU DISCRET The L\'evy case features many possible asymptotic behaviors for the almost sure asymptotic of $\mathcal{L}_X^*(t)$, some of them being similar to the behaviors observed in the discrete transient case by Gantert and Shi \cite{GanShi}. Since our ultimate goal is to study the almost sure behavior...
%When $0 < \kappa < 1$, we generalize the proof given in \cite{advech} and also consider the contribution to distance, which allow us to give a new proof of the result in SNGH. Most of the effort comes from giving a definition of valleys which is more adapted in to our context, and the generalization of the estimates that are well known, and often explicit, in the drifted Brownian case, whereas not in the L\'evy case. We note that this approach can be pushed further in order to get the almost sure behavior of the supremum of local time, this is a work in preparation by the author. 

When $\kappa > 1$, we adopt the point of view of \cite{Singh} and link the local time to a generalized Ornstein-Uhlenbeck process. This approach also provides the convergence of the favorite site and can certainly be used to study the almost sure behavior of the supremum of the local time. 
%in \cite{psvech}. 

\subsection{Main results} \label{results}

%EVENTUELLEMENT DEFINIR LES VALLEES ICI. SI ON FAIT CA, METTRE LA DECOMP DE WILLIAMS ICI. 

%The decomposition of drifting Brownian motion $W_{\kappa}$ into valleys and the knowledge of the law of these valleys has proved very useful in \cite{AndDev} and in \cite{advech} to study the diffusion in the environment $W_{\kappa}$. In particular, in \cite{advech}, they prove the convergence in law for the supremum of local time $\mathcal{L}_X^*$ of such a diffusion, and their limit is expressed in term of an exponential functional of $W_{\kappa}^{\uparrow}$, that is $W_{\kappa}$ conditioned to stay positive. 
%
%Unfortunately, they use many explicit estimates for $W_{\kappa}$, $W_{\kappa}^{\uparrow}$ and exponential functional of these processes that are not true or can not be made explicit for general spectrally negative L\'evy processes. In order to generalize their result, we prove some estimates on $V^{\uparrow}$ and $\hat V^{\uparrow}$ in section \ref{estimates}. Our crucial estimates deal with the hitting times of $V^{\uparrow}$ and $\hat V^{\uparrow}$. 
%%(LES CITER ICI ?) 
Our main results are convergences in distribution for the supremum of the local time. {When $\kappa > 1$,} let us define the constants $K$ and $m$ similarly as in \cite{Singh}: 
\begin{eqnarray}
K := \mathbb{E} \left [ \left ( \int_0^{+\infty} e^{V(t)} dt \right )^{\kappa -1} \right ] \ \ \ \text{and} \ \ \ m := \frac{-2}{\Psi_V(1)} > 0. \label{defdescte}
\end{eqnarray}
{ Note that $K = \mathcal{M}_V(\kappa)$, where $\mathcal{M}_V$ is the Mellin transform of the random variable $\int_0^{+\infty} e^{V(t)} dt$. According to Theorem 2.7 of \cite{bersteingamma}, $\mathcal{M}_V$ can be expressed in term of Berstein-gamma functions. This allows to write $K$ as an infinite product.} For any $\alpha, s>0$, let $\mathcal{F}(\alpha, s)$ denote the Fr\'echet distribution with parameters $\alpha$ and $s$, that is, the distribution with distribution function 
\[ \mathcal{F}(\alpha, s) \left ( [0, t] \right ) = e^{-(s/t)^{\alpha}}, \ \forall t > 0. \]
When $\kappa > 1$, the limit distribution of the supremum of the local time can be expressed as follows: 
\begin{theo} \label{kappa>1}
If $\kappa>1$, 
%\text{If} \ \kappa>1, \ 
\[ \mathcal{L}^*_X(t) / t^{1/\kappa} \overset{\mathcal{L}}{\underset{t \rightarrow + \infty}{\longrightarrow}} \mathcal{F}(\kappa, 2 (\Gamma(\kappa) \kappa^2 K/m)^{1/\kappa}). \]
%where $\mathcal{F}(a,b)$ denotes a Frechet distribution with parameter $a$ and $b$. 
\end{theo}

\textbf{Examples:} In some cases the parameters of the limit distribution are more explicit. 
\begin{itemize}
\item Let $W_{\kappa}$ be the $\kappa$-drifted Brownian motion: $W_{\kappa}(t) := W(t) - \frac{\kappa}{2} t$. If we choose $V = W_{\kappa}$ (for $\kappa > 1$), then $K = 2^{\kappa - 1}/\Gamma(\kappa)$ (see Example 1.1 in \cite{Singh}) and $m = 4/(\kappa - 1)$. The limit distribution of the supremum of the local time is therefore $\mathcal{F}(\kappa, 4 (\kappa^2 (\kappa - 1)/8)^{1/\kappa})$. This is precisely the second point of Theorem 1.6 of Devulder \cite{Devmaxloc}. 

\item If $V$ is such that $\kappa = 2$ then 
\[ K = \mathbb{E} \left [ \int_0^{+\infty} e^{V(t)} dt \right ] = \int_0^{+\infty} \mathbb{E} \left [ e^{V(t)} \right ] dt = \int_0^{+\infty} e^{t \Psi_V(1)} dt = \frac{-1}{\Psi_V(1)} = \frac{m}{2}, \]
and the limit distribution of $\mathcal{L}^*_X(t) / t^{1/2}$ is therefore $\mathcal{F}(2, 2 \sqrt{2})$. 
\end{itemize}

We also prove that the distribution of the favorite site is asymptotically uniform: 
\begin{theo} \label{kappa>1favsite}
If $\kappa>1$, 
%\text{If} \ \kappa>1, \ 
\[ m F^*(t) / t \overset{\mathcal{L}}{\underset{t \rightarrow + \infty}{\longrightarrow}} \mathcal{U}, \]
where $\mathcal{U}$ denotes the uniform distribution on $[0, 1]$. 
\end{theo}

When $0 < \kappa < 1$, we have to introduce some notations in order to express the limit distribution. Let $G_1$ and $G_2$ be two independent random variables with $G_1 \egloi I(V^{\uparrow})$ and $G_2 \egloi I(\hat V^{\uparrow})$. We define $\mathcal{R} := G_1 + G_2$. $\mathcal{R}$ is the analogous of $\mathcal{R}_{\kappa}$ defined in $(1.2)$ of \cite{advech} (if $V$ if the $\kappa$-drifted Brownian motion, then $\mathcal{R} \egloi \mathcal{R}_{\kappa}$, indeed, it is known that $W_{\kappa}^{\uparrow} \egloi \hat W_{\kappa}^{\uparrow}$, so $\mathcal{R}$ is, as $\mathcal{R}_{\kappa}$, the sum of two independent copies of $I(W_{\kappa}^{\uparrow})$). Let also {$\mathcal{C} := K / \Psi_V'(\kappa)$, where $K$ has the same meaning as in \eqref{defdescte} and where we recall that $\Psi_V'(\kappa) > 0$. $\mathcal{C}$ is linked to the right tail of the exponential functional of $V$ (see Lemma \ref{foncexpov} in Section \ref{estimates}). We put $\mathcal{C}' := 2^{\kappa} \Gamma(\kappa + 1) \mathcal{C} = 2^{\kappa} \Gamma(\kappa + 1) K/ \Psi_V'(\kappa)$.} Now, let $\mathcal{Y}_1$ be the $\kappa$-stable subordinator with Laplace exponent $\mathcal{C}' \Gamma(1-\kappa) \lambda^{\kappa}$: 
\[ \forall t, \lambda \geq 0, \ \mathbb{E} \left [ e^{- \lambda \mathcal{Y}_1(t)} \right ] =  e^{-t \mathcal{C}' \Gamma(1-\kappa) \lambda^{\kappa}}. \]
%We denote by $(D([0,+ \infty),\mathbb{R}^2),J_1)$ the space of c\`ad-l\`ag functions with $J_1$-Skorokhod topology. 
%%and denote by  ... the convergence in law for this topology.  
%On this space we define a 2-dimensional pure jump 
We now consider the pure jump L\'evy process $(\mathcal{Y}_1, \mathcal{Y}_2)$ where the component $\mathcal{Y}_2$ is defined multiplying each jump of $\mathcal{Y}_1$ by an independent copy of $\mathcal{R}$. We can also define $(\mathcal{Y}_1, \mathcal{Y}_2)$
%with value in $\mathbb{R}^+\times \mathbb{R}^+$ 
from its $\kappa$-stable L\'evy measure  $\nu$ supported on $]0, +\infty[ \times ]0, +\infty[$ and defined by 
\[ \forall x > 0, y > 0, \ \nu \left ( [x, + \infty[ \times [y, + \infty[ \right ) = \frac{\mathcal{C}'}{y^{\kappa}} \mathbb{E} \left [ \mathcal{R}^{\kappa} \mathds{1}_{\mathcal{R} \leq \frac{y}{x}} \right ] + \frac{\mathcal{C}'}{x^{\kappa}} \mathbb{P} \left ( \mathcal{R} > \frac{y}{x} \right ). \]
%A more intuitive way to construct $(\mathcal{Y}_1, \mathcal{Y}_2)$ is to take $\mathcal{Y}_1$ as the $\kappa$-stable subordinator with L\'evy measure $\frac{\mathcal{C}}{1 + \kappa} \frac{1}{x^{1 + \kappa}} dx$ and then to construct $\mathcal{Y}_2$ to be the process whose jumps are those of $\mathcal{Y}_1$ multiplied by independent copies of the variable $I(V^{\uparrow})$. 
It is easy to see that the first definition implies the second and also that the one-dimensional distributions of the L\'evy process $(\mathcal{Y}_1, \mathcal{Y}_2)$ have Laplace transform
\[ \forall t, \alpha, \beta \geq 0, \ \mathbb{E} \left [ e^{-\alpha \mathcal{Y}_1(t) - \beta \mathcal{Y}_2(t) } \right ] = e^{- t \mathcal{C}' \Gamma(1-\kappa) \mathbb{E} \left [ (\alpha + \beta \mathcal{R})^{\kappa} \right ] }. \]
As in Theorem 1.3 of \cite{advech}, our limit distribution for the supremum of the local time is a function of $(\mathcal{Y}_1, \mathcal{Y}_2)$. 
For $Z$ an increasing c\`ad-l\`ag process and $s \geq 0$, we put respectively $Z(s-)$, $Z^{\natural}(s)$ and $Z^{-1}(s)$ for respectively the left-limit of $Z$ at $s$, the largest jump of $Z$ before $s$ and the generalized inverse of $Z$ at $s$: 
\[ Z(s-) = \lim_{r \underset{<}{\rightarrow} s} Z(r), \ Z^{\natural}(s):= \sup_{0\leq r \leq s}(Z(r)-Z(r-)),\ Z^{-1}(s):= \inf\{u \geq 0,\ Z(u) > s\}, \]
where $\inf \emptyset = +\infty$ by convention. We now define the couple of random variables $( \mathcal{I}_1, \mathcal{I}_2)$: 
\begin{align*}
& \mathcal{I}_1:= \mathcal{Y}_1^{\natural}(\mathcal{Y}_2^{-1}(1) -), \ \mathcal{I}_2:= {\left [1 - \mathcal{Y}_2(\mathcal{Y}_2^{-1}(1) -)\right ]} \times \frac{\mathcal{Y}_1(\mathcal{Y}_2^{-1}(1)) - \mathcal{Y}_1(\mathcal{Y}_2^{-1}(1) -)}{\mathcal{Y}_2(\mathcal{Y}_2^{-1}(1)) - \mathcal{Y}_2(\mathcal{Y}_2^{-1}(1) -)}. 
\end{align*}

We have

\begin{theo} \label{cvdutl} 
If $0 < \kappa < 1$, $V$ has unbounded variation and $V(1) \in L^p$ for some $p>1$ then
\[ \mathcal{L}^*_X(t)/t \overset{\mathcal{L}}{\underset{t \rightarrow + \infty}{\longrightarrow}} \mathcal{I} := \max \{ \mathcal{I}_1,\mathcal{I}_2 \}. \]

\end{theo}

This result is a generalization, for more general environments, of Theorem 1.3 of \cite{advech}. We now give some heuristics about the proof and the expression of the limit distribution. 

{The proof of this result relies on the decomposition of the potential into valleys (defined a little after), that is, holes of potential of a certain height. These valleys are independent and, with a large probability, visited successively by the diffusion with no return to a previous valley. Each valley has a bottom where the diffusion remains trapped for some time before escaping the valley. We prove that the contributions to the local time and to the time spent by the diffusion between the bottoms of two consecutive valleys are negligible, so the local maxima of the local time are localized at the bottom of the valleys, where most of the time is spent. To every visited valley we can associate 1) a peak of local time at its bottom, 2) the amount of time spent by the diffusion in this valley before escaping it. These couples of random variable associated to the successive valleys form, roughly speaking, an \textit{iid} sequence in $\mathbb{R}^2$. We show that the renormalized sum of the terms of this sequence converges to the bivariate $\kappa$-stable subordinator $(\mathcal{Y}_1, \mathcal{Y}_2)$ when $t$ goes to infinity. For large $t$, each jump of $(\mathcal{Y}_1, \mathcal{Y}_2)$ represents the contribution of a valley: the first component of the jump represents the renormalized peak of local time at the bottom of this valley and the second component represents the renormalized time taken to escape the valley. Note that the two components of the pure jump subordinator $(\mathcal{Y}_1, \mathcal{Y}_2)$ are correlated. Then, we approximate the supremum of the local time at instant $t$ by a continuous functional of the above mentioned \textit{iid} sequence. Thanks to the \textit{continuous mapping theorem} we deduce the convergence of the supremum of the local time to this functional evaluated at $(\mathcal{Y}_1, \mathcal{Y}_2)$. 

The shape of the limit distribution in Theorem \ref{cvdutl} can be explained in the following way: $\mathcal{I}_1$ is, by definition, the largest jump of $\mathcal{Y}_1$ strictly before the hitting time of $1$ by $\mathcal{Y}_2$, it represents the largest peak of local time before the penetration of the diffusion in the last valleys (the valley that contains the diffusion at time $t$). Via $\mathcal{I}_2$, the limit law also takes into consideration the contribution of the last valley. The peak of local time in this last valley is represented by $\mathcal{Y}_1(\mathcal{Y}_2^{-1}(1)) - \mathcal{Y}_1(\mathcal{Y}_2^{-1}(1) -)$ (the first component of the jump of $(\mathcal{Y}_1, \mathcal{Y}_2)$ at the hitting time of $1$ by $\mathcal{Y}_2$). We only keep a certain proportion of this peak of local time: the part that is made before instant $t$. Among the time spent in the last valley by the diffusion, the proportion of the time spent in this valley before instant $t$ is represented by ${[1 - \mathcal{Y}_2(\mathcal{Y}_2^{-1}(1) -)] /[\mathcal{Y}_2(\mathcal{Y}_2^{-1}(1)) - \mathcal{Y}_2(\mathcal{Y}_2^{-1}(1) -)]}$. Multiplying this proportion by $\mathcal{Y}_1(\mathcal{Y}_2^{-1}(1)) - \mathcal{Y}_1(\mathcal{Y}_2^{-1}(1) -)$, which represent the peak of local time in the last valley, we obtain $\mathcal{I}_2$ which represent the contribution of the last valley to the local time before instant $t$. It is natural that the limit distribution of the supremum of the local time at time $t$ is the maximum of the random variables $\mathcal{I}_1$ and $\mathcal{I}_2$. }
%
%Since, in our context, $V$ can have negative jumps, a fundamental point of our study is to understand the influence of negative jumps on the local time. Let $v$ be a fixed c\`ad-l\`ag environment to which we add a negative jump of size $-s < 0$ at some point $a$: 
%\begin{eqnarray}
%\tilde v(x) := v(x) \text{ if } x < a, \ v(x) - s \text{ if } x \geq a, \label{ajoutdessauts0}
%\end{eqnarray}
%
%We compare the local times of the diffusions in $v$ and $\tilde v$ in the following proposition: 
%
%\begin{prop} \label{ajoutdessauts}
%
%Let $0<a<b$ and $v$ an environment to which we add a jump at $a$ as in definition \eqref{ajoutdessauts0}. We suppose that the diffusion in environment $v$ starting from $0$ reaches $b$ almost surely, then, there exists a coupling $(Y, \tilde Y)$, where $Y$ has the same law as the diffusion in potential $v$ and $\tilde Y$ has the same law as the diffusion in potential $\tilde v$, such that almost surely
%\[ \forall x \in [0, b], \ \mathcal{L}_{\tilde Y}(\tau(\tilde Y, b), x) \leq \mathcal{L}_{Y}(\tau(Y, b), x). \]
%
%\end{prop}
%
%In other words, the local time decreases (stochastically) if we add a negative jump. Thanks to this property, we are able to prove the negligibility of the local time between the bottoms of the consecutive valleys. The first step is to cut some jumps of $V$ and to generalize, in Lemma \ref{analog3.3sautscoupes}, the negligibility proved in Lemma 3.3 of \cite{advech}. The second step is to use Proposition \ref{ajoutdessauts} to prove that the remaining jumps can be added without changing the negligibility. 

Our generalization of the results known in the Brownian case also yields some other results such as the convergence of the supremum of the local time before and after the last valley (the first two points of Theorem 1.5 in \cite{advech}).

{As we said before, when $0 < \kappa < 1$, our study relies deeply on the decomposition of the potential into valleys and on the localization of the contributions, to the local time and to the time spent by the diffusion, near the bottoms of these valleys. We now introduce some definitions. The notion of $h$-extrema was first introduced by Neveu et al. \cite{NevPit}, and studied in the case of drifted Brownian motion by Faggionato \cite{Faggionato}. For $h>0$, we say that $x\in\mathbb{R}$ is an $h$-minimum for $V$ if there exist $u<x<v$ such that $V(y) \wedge V(y-) \geq V(x) \wedge V(x-)$ for all $y\in[u,v]$, $V(u)\geq (V(x) \wedge V(x-))+h$ and $V(v-)\geq (V(x) \wedge V(x-))+h$. Moreover, $x$ is an $h$-maximum for $V$ if $x$ is an $h$-minimum for $-V$, and $x$ is an $h$-extremum for $V$ if it is an $h$-maximum or an $h$-minimum for $V$. 

Since $V$ is not a compound Poisson process, it is known (see Proposition VI.4, in \cite{Bertoin}) that it takes pairwise distinct values in its local extrema. Combining this with the fact that $V$ has almost surely c\`ad-l\`ag paths and drifts to $-\infty$ without being the opposite of a subordinator, we can check that the set of $h$-extrema is discrete, forms a sequence indexed by $\mathbb{Z}$, unbounded from below and above, and that the $h$-minima and $h$-maxima alternate. Typically, an $h$-minima is the bottom of what we will call an $h$-valley (see Section \ref{etudevallees}).

We denote respectively by $(m_i,\ i \in \mathbb{Z})$ and $(M_i,\ i \in \mathbb{Z})$ the increasing sequences of $h$-minima and of $h$-maxima of $V$, such that $m_{0} \leq 0<m_1$ and $m_i<M_i<m_{i+1}$ for every  $i\in\mathbb{Z}$. 

In fact, we have to make $h$ going to infinity when $t$ goes to infinity, this is why we briefly study the asymptotic 
%of the $h$-valleys. The definition of the $h$-extrema and $h$-valleys are given in the next section but we already note the following result about the asymptotic 
of $(m_1, m_2, ...)$, the sequence of $h$-minima, in the following theorem. }

\begin{theo} \label{seqmincvverspoiss}

%If $\kappa > 0$, then w
When $h$ goes to infinity, the renormalized random sequence $e^{-\kappa h}(m_1, m_2, ...)$ converges in distribution to the jumping times sequence of a standard Poisson process with parameter $q$ (which depends explicitly on the law of $V$). If $V = W_{\kappa}$, the $\kappa$-drifted Brownian motion, then $q = \kappa^2/2$. 

\end{theo}

%If $V = W_{\kappa}$, the $\kappa$-drifted Brownian motion, the calculations can be made more explicit. In particular, using $(2.7)$ in FAGGIONATO, we can prove that for $W_{\kappa}$, the parameter $q$ in Proposition \ref{asympmin} (which equals the one in Theorem \ref{seqmincvverspoiss}) equals $\kappa^2/2$. 

Some of the estimates used to prove this theorem will also be useful to establish the negligibility of the local time between the bottoms of two consecutive $h$-valleys when $0 < \kappa < 1$. However, the main interest of this theorem is that it informs us on the typical distance between two consecutive $h$-minima, and this provides {a} heuristic explanation for why the method based on valleys fails when $\kappa > 1$. 

As we can see in Theorem \ref{seqmincvverspoiss}, the distance between two consecutive $h$-minima is of order $e^{\kappa h}$. In the case $0 < \kappa < 1$, what happens between the bottoms of the $h$-valleys can be neglected, so the main contributions to the local time and to the time spent by the diffusion are localized at the $h$-minima and are highly correlated, this explains the form of the limit distribution for the favorite site given in Theorem 1.5 of \cite{advech}. When $\kappa > 1$, this distance between two consecutive $h$-minima is so large that the amount of time spent by the diffusion between the $h$-minima is no longer negligible compared to the amount of time spent in the bottoms of the $h$-valleys, and there are extreme values taken by the local time between the $h$-valleys. In fact, it is impossible to use the valleys to localize the large values of the local time as we do when $0 < \kappa < 1$, this explains the asymptotic uniform distribution for the favorite site. Moreover, the case $\kappa > 1$ is the case where $X$ has positive speed. The local time at time $t$ is then close to the local time at the hitting time of $t/m$ ($m$ being defined in \eqref{defdescte}) and the latter can be expressed thanks to a generalized Ornstein-Uhlenbeck process. The supremum of the local time is therefore similar to the supremum of the heights of \textit{iid} excursions of a Markov process. This explains the appearance of a Fr\'echet distribution in the limit law of the supremum of the local time. 

{The assumption that $V$ is spectrally negative is crucial in our study and there are several reasons for this: 
\begin{itemize}
\item The Laplace exponent $\Psi_V$ and its non trivial zero $\kappa$ are defined when the asymptotic tail at $+\infty$ of the L\'evy measure of $V$ is thin compared to all negative exponentials. This is the case in particular when $V$ is spectrally negative. The fact that $\kappa$ is defined is essential since the value of $\kappa$ determines many important properties of the diffusion, see for example \eqref{cvsingh} and the above results. As we said, $\kappa$ plays a similar role as the drift of the environment when the latter is Brownian. Also, $\kappa$ determines the right tail of the exponential functional $\int_0^{+\infty} e^{V(t)} dt$, and this tail is determinant in the proof of the convergence in distribution of $\mathcal{L}^*_X(t)$ when $0 < \kappa < 1$. In all cases, we see from Theorems \ref{kappa>1} and  \ref{cvdutl} that $\kappa$ appears in the limit distribution of $\mathcal{L}^*_X(t)$. 
\item Since $V$ is spectrally negative, $V^{\uparrow}$ is defined and has infinite life-time even though $V$ drifts to $-\infty$ at $+\infty$. Moreover, since $V$ is spectrally negative, $V^{\uparrow}$ (killed at some hitting time) is exactly equal in law to the ascending part of the bottom of the valleys of $V$, see Subsection \ref{loidesval}. 
\item As we see in Theorem \ref{cvdutl}, the random variable $\mathcal{R}$, whose law is the convolution of the laws of $I(V^{\uparrow})$ and $I(\hat V^{\uparrow})$, appears in the limit distribution of $\mathcal{L}^*_X(t)$ when $0 < \kappa < 1$. We have hints that the asymptotic almost sure behavior of $\mathcal{L}_X^*(t)$ is also linked to the left tail of $\mathcal{R}$ (see the work in preparation \cite{psvech}). This tail can be known very precisely when $V$ is spectrally negative, thanks to the results of \cite{foncexpovech}. In the latter paper the methods rely deeply on the fact that $V$ is spectrally negative and it is shown that the absence of positive jumps is determinant for the shape of the left tail of $I(V^{\uparrow})$ (and therefore of $\mathcal{R}$), that is, this tail would be totally different if $V$ had positive jumps. 
\item The largest part of this work consists in proving precise estimates for the environment $V$. Some of these estimates require the fact that $V$ is spectrally negative and most of them are simpler to obtain under this assumption. 
\end{itemize}
}

The rest of the paper is organized as follows. In Section \ref{kappagd} we study the case where $\kappa > 1$ and prove Theorems \ref{kappa>1} and \ref{kappa>1favsite}. In Section \ref{etudevallees} we recall the definition of the $h$-valleys of $V$ and establish some usual properties such as the independence of the consecutive slopes and the law of the bottoms of the valleys. We end this Section by proving Theorem \ref{seqmincvverspoiss}. In Section \ref{genedesres} we assume $0 < \kappa < 1$ and prove Theorem \ref{cvdutl}. In Section \ref{estimates} we prove some fundamental estimates on $V$, $V^{\uparrow}$ and $\hat V^{\uparrow}$ and some technical results. 

%More precisely, $V$ can be split at the center of each valley into two independent parts, one having its law close to $V^{\uparrow}$ and the other having its law close to $\hat V^{\uparrow}$. We study the properties of $m_2 - m_1$
%

\subsection{Facts and notations} \label{factsandnotations}

We denote by $(Q, \gamma, \nu)$ the generating triplet of $V$ so $\Psi_V$ can be expressed as
\begin{eqnarray}
\Psi_V(\lambda) = \frac{Q}{2} \lambda^2 - \gamma \lambda + \int_{-\infty}^0 (e^{\lambda x} - 1 - \lambda x \mathds{1}_{|x| < 1}) \nu(dx). \label{levkhin}
\end{eqnarray}

If $V$ jumps at instant $u$ we denote $\Delta V(u) := V(u) - V(u-)$, the jump of $V$ at $u$. For $r > 0$ we define $V^{< -r}$ to be the sum of the jumps of $V$ that are less than $-r$: 
\[ \forall s \geq 0, \ V^{< -r}(s) := \sum_{0 \leq u \leq s} \Delta V(u) \mathds{1}_{ \left \{ \Delta V(u) < -r \right \} }. \]

$V$ is the sum of the processes $V - V^{< -r}$ and $V^{< -r}$ that can be seen to be independent spectrally negative L\'evy processes, thanks to the L\'evy-Khintchine formula \eqref{levkhin}. According to Corollary VII.2 of Bertoin \cite{Bertoin} we have $\mathbb{E} [ V(1) ] < 0$ so for $r$ chosen large enough we have $\mathbb{E} [ [V - V^{< -r}](1) ] < 0$ which, thanks to Corollary VII.2 of \cite{Bertoin}, implies that $V - V^{< -r}$ drifts to $-\infty$. In other words, removing the very large jumps of $V$ does not change its convergence to $-\infty$. 

For $Y$ a process and $S$ a borelian set, we denote
\[ \tau(Y, S) := \inf \left \{ t \geq 0, \ Y(t) \in S \right \}, \ \ \ \mathcal{K}(Y, S) := \sup \left \{ t \geq 0, \ Y(t) \in S \right \}. \]
We shall only write $\tau(Y, x)$ instead of $\tau(Y, \{x\})$ and $\tau(Y, x+)$ instead of $\tau(Y, [x, +\infty [)$. Since $V$ has no positive jumps we see that it reaches each positive level continuously: $\forall x > 0, \tau(V, x+) = \tau(V, x)$. Moreover, the law of the supremum of $V$ is known, it is an exponential distribution with parameter $\kappa$ (see Corollary VII.2 in \cite{Bertoin}). 

$\underline{Y}$ denotes the infimum process of $Y$: $\forall t \geq 0, \ \underline{Y}(t) := \inf_{[0, t]} Y$. The process $V - \underline{V}$ is known as the process $V$ reflected at its infimum. Note that, by Proposition VI.1 of \cite{Bertoin}, it is a c\`ad-l\`ag Markov process. The same holds for $\hat V - \underline{\hat V}$. 

If $Y$ is Markovian and $x \in \mathbb{R}$ we denote $Y_x$ for the process $Y$ starting from $x$. For $Y_0$ we shall only write $Y$. For any (possibly random) time $T > 0$, we write $Y^T$ for the process $Y$ shifted and centered at time $T$: $\forall s \geq 0, \ Y^T(s) := Y(T+s)-Y(T)$. 

Let $B$ be a Brownian motion starting at $0$ and independent from $V$. A diffusion in potential $V$ can be defined via the formula 
\begin{equation}
X(t) := A_V^{-1}(B(T_V^{-1}(t))), \label{exprediff1}
\end{equation}
where 
\[ A_V(x) := \int_0^x e^{V(u)} du \text{ and for } 0 \leq s \leq \tau \left ( B, A_V(+\infty) \right ), \ T_V(s) := \int_0^s e^{-2 V(A_V^{-1}(B(u)))} du. \]
%\[ A_V(x) := \int_0^x e^{V(u)} du \text{ and for } 0 \leq s \leq \tau \left ( B, \int_0^{+\infty} e^{V(u)} du \right ), \ T_V(s) := \int_0^s e^{-2 V(A_V^{-1}(B(u)))} du. \]

{$A_V$ is the scale function of the diffusion $X$ (at fixed environment $V$). Note that, by the time-reversal property for L\'evy processes, the process $(-V(-(x-)), \ x \geq 0)$ has the same law as $(V(x), \ x \geq 0)$ so $V$ converges almost surely to $+\infty$ at $-\infty$. In particular we have $A_V(-\infty) = -\infty$ while $A_V(+\infty) < +\infty$. This implies that $X$ is transient toward $+\infty$. }

It is known that the local time of $X$ at $x$ until instant $t$ has the following expression:   
\begin{equation}
\mathcal{L}_X(t,x)= e^{- V(x)}\mathcal{L}_B(T_V^{-1}(t),A_V(x)), \label{expretl1}
\end{equation}
%We sometimes study some events for the diffusion under t
where $\mathcal{L}_B(.,.)$ is the local time of $B$. For the hitting times of $r \in \mathbb{R}$ by the diffusion $X$ we shall use the frequent notation $H(r)$ (instead of $\tau(X, r)$). {Note that we have 
\begin{align}
\mathcal{L}_X(H(r),x) & = e^{- V(x)}\mathcal{L}_B{[\tau (B, A_V(r)),A_V(x)]}, \label{expretl2} \\
H(r) & = \int_{-\infty}^r \mathcal{L}_X(H(r),x) dx = \int_{-\infty}^r e^{- V(x)}\mathcal{L}_B{[\tau(B, A_V(r)), A_V(x)]} dx. \label{expretl2.1}
\end{align}}
We denote by $H_{+}(r)$ (respectively $H_{-}(r)$) the total amount of time spent by the diffusion in $[0, +\infty [$ (respectively $]-\infty, 0]$) before $H(r)$. We have obviously $H(r) = H_{+}(r) + H_{-}(r)$. 

The quenched probability measure $P^V$ is the probability measure conditionally on the potential $V$. When we deal we events relative to the diffusion $X$, $\mathbb{P}$ represents the annealed probability measure, it is defined as $\mathbb{P} := \int P^{\omega} (.) P(V \in d \omega)$. 
%ATTENTION, IL Y A AMBIGUITE ENTRE $\mathbb{P}$ POUR L'ANNEALED ET $\mathbb{P}$ NORMAL...
%we denote respectively the expectations...

We denote by $d_{VT}$ the total variation distance between two probability distributions on the same space. 

{By convention we put $\max_{1 \leq i \leq 0}... = 0$ and $\sum_{i=1}^{0}... = 0$. For example a maximum or a sum of quantities indexed by $\{1, ..., k-1\}$ is defined to equal $0$ in the case $k=1$.} 

%If $A$ is a process with finite life-time, we denote its life-time by $\zeta(A)$. 
%$P^V$, $\mathbb{P}$, 
%$\mathcal{L}_{Y}$, $\nu$, $V^{\sharp}$, $\hat V^{\uparrow}_U$, 

\section{Supremum of the local time when $\kappa > 1$} \label{kappagd}

We now treat the case $\kappa > 1$. Since some of the lemmas we state in this section are true in a general context we do not assume $\kappa > 1$ yet. We thus have $\kappa \in ]0, +\infty[$ unless mentioned otherwise. As we mentioned in {the end of Subsection \ref{results}}, the valleys are of no use in this section so we have to study directly the expression of the local time. It is given by \eqref{expretl1}. {According to the main result of \cite{Singh}, if $\kappa > 1$ and $t$ is large, then $t$ is close to $H(t/m)$ where $m:=-2/\Psi_V(1)$ as in \eqref{defdescte} (a version of this fact, adapted to our context, is contained in Lemma \ref{chgtvar} below).} It is then convenient to look at the local time until {a hitting} time. It has a simpler expression given by \eqref{expretl2}. 
%\begin{equation}
%\mathcal{L}_X(H(r),x)= e^{- V(x)}\mathcal{L}_B(\tau (B, A_V(r)),A_V(x)). \label{expretl2}
%\end{equation}
The supremum of the local time until instant $H(r)$ can thus be written
\begin{equation}
\mathcal{L}^*_X(H(r))= \max \{ \mathcal{M}_1(r), \mathcal{M}_2(r) \}, \label{expretl3}
\end{equation}
{where 
\[ \mathcal{M}_2(r) := \sup_{x \in [0,r]} \mathcal{L}_X(H(r),x) = \sup_{x \in [0,r]} e^{- V(x)}\mathcal{L}_B{[\tau (B, A_V(r)),A_V(x)]}, \]
and 
\begin{align}
\mathcal{M}_1(r) & := \sup_{x < 0} \mathcal{L}_X(H(r),x) = \sup_{x < 0} e^{- V(x)}\mathcal{L}_B{[\tau (B, A_V(r)),A_V(x)]} \nonumber \\
& \leq \left ( \sup_{x < 0} e^{- V(x)} \right ) \times \left (\sup_{y \in \mathbb{R}} \mathcal{L}_B{[\tau (B, A_V(+\infty)), y]} \right ) < +\infty. \label{tlneg}
\end{align}
Recall that $V$ converges almost surely to $+\infty$ at $-\infty$ (see Subsection \ref{factsandnotations}), so in particular the factor $\sup_{x < 0} e^{- V(x)}$ is finite. Also, since $A_V(+\infty) < +\infty$, as mentioned in Subsection \ref{factsandnotations}, the second factor is the supremum of the local time of a Brownian motion until a finite hitting time, it is thus also almost surely finite. This explains the finiteness of the upper bound in \eqref{tlneg}. }

We also define the favorite positive site until time $H(r)$ by
\[ F^*_+(H(r)) := \text{argmax}_{[0,r]} \mathcal{L}_X(H(r),.) := \inf \left \{ x \in [0,r], \ \mathcal{L}_X(H(r), x) \vee \mathcal{L}_X(H(r), x-) = \mathcal{M}_2(r) \right \}. \]
To lighten the notations, we often use the notation $argmax$ in this section. Now, note that $\mathcal{M}_2(r)$ is the same as $J_2(r)$ in Section 2 of \cite{Singh} with a supremum on $[0,r]$ in place of an integral on this interval. In particular, if similarly as in Section 2 of \cite{Singh} we define 
\[ Z(x) := e^{V(x)} R \left ( \int_0^x e^{-V(y)} dy \right ), \]
where $R$ is a two-dimensional squared Bessel process independent from $V$, we get
\begin{eqnarray}
\mathcal{M}_2(r) \overset{\mathcal{L}}{=} \sup_{x \in [0,r]} Z(x) \ \ \ \text{and} \ \ \ F^*_+(H(r)) \overset{\mathcal{L}}{=} r - \text{argmax}_{[0,r]} Z. \label{expretl4}
\end{eqnarray}

We can therefore prove our results using some known properties of $Z$. Let us recall from Section 5 of \cite{Singh} the definition of $L$, the local time of $Z$ for the position $1$, of $n$ the associated excursion measure, and of $L^{-1}$ the right continuous inverse of $L$. We denote by $\xi$ a generic excursion. 

\subsection{The local time at hitting times}

We now prove a lemma to justify rigorously that the local time of the diffusion at some instant can be approximated by the local time at {a hitting} time. {As an intermediary result, the lemma also proves \eqref{invtl}: that a deterministic time can be approximated by the inverse of the local time of $Z$ taken at some deterministic value. The interest is that the expressions involving $Z$ can then be studied via the excursions of $Z$, which we will systematically do. }

\begin{lemme} \label{chgtvar}

Let us denote by $Q$ the positive constant denoted by $n [\zeta]$ in Section 5 of \cite{Singh}. For $r$ large enough we have 
\begin{eqnarray}
\mathbb{P} \left ( L^{-1}(r/Q - r^{3/4}) \leq r \leq L^{-1}(r/Q + r^{3/4}) \right ) & \geq 1 - r^{-1/4}. \label{invtl}
\end{eqnarray}
Assume $\kappa > 1$. For any $\alpha \in ]\max \{ 3/4, 1/ \kappa \}, 1[$ their exists $\epsilon > 0$ such that for $r$ large enough we have
\begin{align}
%\mathbb{P} \left ( L^{-1}(r/Q - r^{3/4}) \leq r \leq L^{-1}(r/Q + r^{3/4}) \right ) & \geq 1 - r^{-\epsilon}, \label{invtl} \\
\mathbb{P} \left ( H(r/m - r^{\alpha}) \leq r \leq H(r/m + r^{\alpha}) \right ) & \geq 1 - r^{-\epsilon}, \label{tpsatt} \\
\mathbb{P} \left ( \mathcal{L}^*_X(H(r/m - r^{\alpha})) \leq \mathcal{L}^*_X(r) \leq \mathcal{L}^*_X(H(r/m + r^{\alpha}) ) \right ) & \geq 1 - r^{-\epsilon}. \label{enctl}
\end{align}

\end{lemme}

\begin{proof}
According to Lemma 5.1 in \cite{Singh}, $\mathbb{E} [L^{-1}(1)] = Q$ and $L^{-1}(1)$ admits some finite positive exponential moments, so, in particular, it has moments of the second order. Let us define the L\'evy process $U(t) := L^{-1}(t) - Q t$. $U$ has finite mean equal to $0$ and moments of the second order, this implies that $\mathbb{E} [ |U(t)|^2 ] = t \mathbb{E} [ |U(1)|^2 ]$. Using Markov's inequality  we get
%From the stationarity and independence of the increments we have $Var (L^{-1}(t)) = t Var (L^{-1}(1)) < +\infty$. 
\begin{eqnarray}
\forall t, s > 0, \ \mathbb{P} \left ( | U(t) | > s \right ) \leq s^{- 2} \mathbb{E} \left [ | U(t)|^2 \right ] = s^{- 2} t \mathbb{E} \left [ |U(1)|^2 \right ]. \label{chgtvarest1}
\end{eqnarray}
Then, 
\begin{align*}
\mathbb{P} \left ( r > L^{-1}(r/Q + r^{3/4}) \right ) & \leq \mathbb{P} \left ( | L^{-1}(r/Q + r^{3/4}) - (r + Q r^{3/4}) | > Q r^{3/4} \right ) \\
& = \mathbb{P} \left ( | U(r/Q + r^{3/4}) | > Q r^{3/4} \right ) \leq (r/Q + r^{3/4}) \mathbb{E} \left [ |U(1)|^2 \right ] / Q^2 r^{3/2}, 
\end{align*}
where we have used \eqref{chgtvarest1} with $t = r/Q + r^{3/4}$ and $s = Q r^{3/4}$. We get a similar estimate for 

\noindent $\mathbb{P} \left ( r < L^{-1}(r/Q - r^{3/4}) \right )$ so \eqref{invtl} follows. We now turn to \eqref{tpsatt}. Until the end of this proof we assume $\kappa > 1$ {and we fix $\alpha \in ]\max \{ 3/4, 1/ \kappa \}, 1[$}. {Recall the definitions of $H_-$ and $H_+$ in Subsection \ref{factsandnotations}}, we have
\begin{align}
\mathbb{P} \left ( H(r/m - r^{\alpha}) > r \right ) & \leq \mathbb{P} \left ( H_-(r/m - r^{\alpha}) + H_+(r/m - r^{\alpha}) - (r - m r^{\alpha}) > m r^{\alpha} \right ) \nonumber \\
& \leq \mathbb{P} \left ( H_-(+\infty) > m r^{\alpha}/2 \right ) \nonumber \\
& + \mathbb{P} \left ( \int_0^{r/m - r^{\alpha}} Z(x) dx - (r - m r^{\alpha}) > m r^{\alpha}/2 \right ). \label{chgtvarest1.1}
\end{align}
We have used the fact, from Section 2 of \cite{Singh}, that $\int_0^t Z(x) dx$ has the same law as $H_+(t)$ (denoted by $J_2(t)$ there). The second term in the right hand side of \eqref{chgtvarest1.1} is less than
\begin{align*}
& \mathbb{P} \left ( \int_0^{\tau(Z, 1)} Z(x) dx > m r^{\alpha}/4 \right ) + \mathbb{P} \left ( \int_{\tau(Z, 1)}^{r/m - r^{\alpha}} Z(x) dx - (r - m r^{\alpha}) > m r^{\alpha}/4 \right ) \\
\leq & \mathbb{P} \left ( \tau(Z, 1) > m r^{\alpha}/4 \right ) + \mathbb{P} \left ( \int_{0}^{r/m - r^{\alpha}} Z_1(x) dx - (r - m r^{\alpha}) > m r^{\alpha}/4 \right ), 
\end{align*}
where we put $Z_1 := Z(\tau(Z, 1) + .)$. It is well defined since $Z$ has no positive jumps and thus reaches $[1, +\infty[$ continuously. Since $Z$ is a Markov process, $Z_1$ has indeed the same law as $Z$ starting from $1$. According to Proposition 4.3 of \cite{Singh} we have 
\begin{eqnarray}
\mathbb{P} ( \tau(Z,1) > t) \leq e^{-c t}, \label{supdeZ2}
\end{eqnarray}
for some constant $c > 0$ when $t$ is large enough. Then, according to Lemma \ref{tpsdanslesnegl} we have $\mathbb{P} ( H_-(+\infty) > m r^{\alpha}/2 ) \leq c' r^{-\frac{\alpha \kappa}{2+\kappa}}$ for some positive constant $c'$ and $r$ large enough. Putting all this into \eqref{chgtvarest1.1} we obtain for $r$ large enough
\begin{eqnarray}
\mathbb{P} \left ( H(r/m - r^{\alpha}) > r \right ) \leq \mathbb{P} \left ( \left | \int_{0}^{r/m - r^{\alpha}} Z_1(x) dx - (r - m r^{\alpha}) \right | > m r^{\alpha}/4 \right ) + c_1 r^{-\frac{\alpha \kappa}{2+\kappa}}, \label{chgtvarest1.2}
\end{eqnarray}
for some positive constant $c_1$. Let us define $\tilde U(t) := \int_0^{L^{-1}(t)} Z_1(x) dx - m L^{-1}(t)$. Note that, from Section 7 of \cite{Singh}, we can see that $\tilde U$ is a L\'evy process with finite mean equal to $0$ and such that $\mathbb{P} ( \tilde U(1) > x ) \sim_{x \rightarrow +\infty} c x^{-\kappa}$ for some positive constant $c$. {Moreover $\tilde U$ has no negative jumps so $\mathbb{E} [(\tilde U(1) \wedge (-1))^2] < +\infty$ (see Theorem 25.3 in \cite{Sato}). As a consequence, $\mathbb{E} [ |\tilde U(1)|^{\gamma} ] < +\infty$ for any $\gamma \in ]1/\alpha, \kappa \wedge (4/3)[$ (and note that we have $1/\alpha < \kappa \wedge (4/3)$ by the definition of $\alpha$ so the interval $]1/\alpha, \kappa \wedge (4/3)[$ is indeed non-empty). We choose such a $\gamma$ and use successively Markov's inequality and Von Barh-Esseen's inequality (which can be applied with our choice of $\gamma$ since $]1/\alpha, \kappa \wedge (4/3)[ \subset ]1,2[$):} 
\begin{eqnarray}
\forall t, s > 0, \ \mathbb{P} \left ( | \tilde U(t) | > s \right ) \leq s^{- \gamma} \mathbb{E} \left [ | \tilde U(t)|^{\gamma} \right ] \leq 2 s^{- \gamma} \left ( \lfloor t \rfloor \mathbb{E} \left [ | \tilde U(1)|^{\gamma} \right ] + \sup_{v \in [0,1]} \mathbb{E} \left [ | \tilde U(v)|^{\gamma} \right ] \right ). \ \ \label{chgtvarest2}
\end{eqnarray}

We now use \eqref{invtl} and get that for $y$ large enough, $y \leq L^{-1}(y/Q + y^{3/4})$ with probability greater than $1 - y^{-1/4}$. Therefore, with probability greater than $1 - y^{-1/4}$ we have 
\begin{align*}
\int_0^{y} Z_1(x) dx - m y & \leq \int_0^{L^{-1}(y/Q + y^{3/4})} Z_1(x) dx - m y = \tilde U( y/Q + y^{3/4} ) + m U( y/Q + y^{3/4} ) + mQ y^{3/4}. 
\end{align*}
Let $C$ be an arbitrary positive constant. Applying \eqref{chgtvarest2} and \eqref{chgtvarest1} with $t = y/Q + y^{3/4}$ and $s = C y^{\alpha}$ we get for $y$ large enough
\[ \mathbb{P} \left ( \int_0^{y} Z_1(x) dx - m y > C y^{\alpha} \right ) \leq c_2 \left ( y^{1 - \alpha \gamma} + y^{1 - 2 \alpha} + y^{-1/4} \right ), \]
for some positive constant $c_2$, depending on $C$. We prove a similar inequality for $my - \int_0^{y} Z_1(x) dx$ so we have actually, for $c_3$ a positive constant (depending on $C$) and $y$ large enough, 
\begin{eqnarray}
\mathbb{P} \left ( \left | \int_0^{y} Z_1(x) dx - m y \right | > C y^{\alpha} \right ) \leq c_3 \left ( y^{1 - \alpha \gamma} + y^{1 - 2 \alpha} + y^{-1/4} \right ). \label{chgtvarest3}
\end{eqnarray}
Applying \eqref{chgtvarest3} with a good choice of $C$, $y = r/m - r^{\alpha}$ and putting into \eqref{chgtvarest1.2} we get
\begin{eqnarray}
\mathbb{P} \left ( H(r/m - r^{\alpha}) > r \right ) \leq r^{-\epsilon}/2, \label{chgtvarest4}
\end{eqnarray}
where $\epsilon > 0$ is chosen to be less than $\max \{ \alpha \gamma -1, 2 \alpha - 1, \alpha \kappa/(2 + \kappa), 1/4 \}$ and $r$ is large enough. Then, 
\begin{align*}
\mathbb{P} \left ( H(r/m + r^{\alpha}) < r \right ) & \leq \mathbb{P} \left ( H_+(r/m + r^{\alpha}) < r \right ) \leq \mathbb{P} \left ( \int_{\tau(Z, 1)}^{r/m + r^{\alpha}} Z(x) dx < r \right ) \\
& \leq \mathbb{P} \left ( \int_{0}^{r/m + r^{\alpha}/2} Z_1(x) dx < r \right ) + \mathbb{P} ( \tau(Z,1) > r^{\alpha}/2), 
\end{align*}
where we recall that $Z_1 := Z(\tau(Z, 1) + .)$. Using \eqref{supdeZ2} and \eqref{chgtvarest3}, we conclude the same way as for $\mathbb{P} ( H(r/m - r^{\alpha}) > r )$: we get $\mathbb{P} \left ( H(r/m + r^{\alpha}) < r \right ) \leq r^{-\epsilon}/2$ for $r$ large enough, and combining with \eqref{chgtvarest4} we get \eqref{tpsatt}. 

Finally, \eqref{enctl} is only a consequence of \eqref{tpsatt} and of the increases of $\mathcal{L}^*_X$. 

\end{proof}

We also need an almost sure version of \eqref{invtl}: 

\begin{lemme} \label{itlog}

Almost surely, for all $t$ large enough we have
\[ Q t - t^{3/4} \leq L^{-1}(t) \leq Q t + t^{3/4}. \]

\end{lemme}

\begin{proof}
Since $L^{-1}(1)$ admits moments of the second order (see the above proof) we have an iterated logarithm law. 
\end{proof}

We now study the supremum of the process $Z$ in term of its excursion measure $n$. 

\begin{lemme} \label{supdeZ}

There is $\epsilon > 0$ and $r_0 > 0$ such that for all $r \geq r_0$ and $h>1$ we have
\[ e^{-(r/Q + r^{7/8}) n(\sup \xi > h)} - r^{-\epsilon} \leq \mathbb{P} \left ( \sup_{x \in [0,r]} Z(x) \leq h \right ) \leq e^{-(r/Q - r^{7/8}) n(\sup \xi > h)} + r^{-\epsilon}. \]

\end{lemme}

\begin{proof}
Recall the notation $Z_1 := Z(\tau(Z, 1) + .)$ and that $Z_1$ has indeed, by the Markov property, the same law as $Z$ starting from $1$. For $h > 1$ we have
\begin{eqnarray}
\mathbb{P} \left ( \sup_{x \in [0,r]} Z(x) \leq h \right ) = \mathbb{P} \left ( \sup_{x \in [\tau(Z,1),r]} Z(x) \leq h \right ) \geq \mathbb{P} \left ( \sup_{x \in [0,r]} Z_1(x) \leq h \right ), \label{supdeZ1}
\end{eqnarray}
because the length of $[\tau(Z,1),r]$ is less than $r$. Let us choose $\eta \in ]0, 3/4[$. We have
\begin{eqnarray}
\mathbb{P} \left ( \sup_{x \in [\tau(Z,1),r]} Z(x) \leq h \right ) \leq \mathbb{P} \left ( \sup_{x \in [0,r-r^{\eta}]} Z_1(x) \leq h \right ) + \mathbb{P} \left ( \tau(Z,1) > r^{\eta} \right ). \label{supdeZ2.1}
\end{eqnarray}

From \eqref{supdeZ1}, \eqref{supdeZ2.1} and \eqref{supdeZ2}, we see that we only need to prove the lemma with $Z_1$ instead of $Z$ and $r^{3/4}$ instead of $r^{7/8}$. We only prove the lower bound, since the proof of the upper bound is similar. For $r$ large enough so that \eqref{invtl} is true we have
\[ \mathbb{P} \left ( \sup_{x \in [0,r]} Z_1(x) \leq h \right ) \geq \mathbb{P} \left ( \sup_{x \in [0,L^{-1}(r/Q + r^{3/4})]} Z_1(x) \leq h \right ) - r^{-1/4}. \]
From the point of view of excursions, the probability in the right hand side is only the probability that no excursion higher than $h$ occurs before $L$ exceeds $r/Q + r^{3/4}$. {Since the process, indexed by the local time $L$, of excursions of $Z_1$ is a Poisson point process} with intensity measure $n(.)$ (see for example Theorem IV.10 in \cite{Bertoin}), the number of excursions higher than $h$ before $L$ exceeds $r/Q + r^{3/4}$ follows a Poisson distribution with parameter $(r/Q + r^{3/4}) n(\sup \xi > h)$. The probability that no such excursion occurs is therefore $e^{-(r/Q + r^{3/4}) n(\sup \xi > h)}$. We thus get 
\[ \mathbb{P} \left ( \sup_{x \in [0,L^{-1}(r/Q + r^{3/4})]} Z_1(x) \leq h \right ) = e^{-(r/Q + r^{3/4}) n(\sup \xi > h)}, \]
and the result follows. 

\end{proof}

\subsection{Proof of Theorems \ref{kappa>1} and \ref{kappa>1favsite}}

In this subsection we assume that $\kappa > 1$. 

\begin{proof} of Theorem \ref{kappa>1}

According to \eqref{enctl}, we only need to prove that 
\[ \frac{\mathcal{L}^*_X(H(r))}{(r m)^{1/\kappa}} \overset{\mathcal{L}}{\underset{r \rightarrow + \infty}{\longrightarrow}} \mathcal{F}(\kappa, 2 (\Gamma(\kappa) \kappa^2 K/m)^{1/\kappa}) \]
or, equivalently, 
\begin{eqnarray}
\frac{\mathcal{L}^*_X(H(r))}{r^{1/\kappa}} \overset{\mathcal{L}}{\underset{r \rightarrow + \infty}{\longrightarrow}} \mathcal{F}(\kappa, 2 (\Gamma(\kappa) \kappa^2 K)^{1/\kappa}). \label{lacvloi0}
\end{eqnarray}
Now, recall \eqref{expretl3}. Combining it with \eqref{tlneg} and \eqref{expretl4}, we are left to prove that
\begin{eqnarray}
\forall a > 0, \ \mathbb{P} \left ( \sup_{x \in [0,r]} Z(x) \leq a r^{1/\kappa} \right ) \underset{r \rightarrow + \infty}{\longrightarrow} e^{-2^{\kappa} \Gamma(\kappa) \kappa^2 K /a^{\kappa}}. \label{lacvloi}
\end{eqnarray}

According to Proposition 5.1 of \cite{Singh} we have 
\begin{eqnarray}
n(\sup \xi > h) \underset{h \rightarrow +\infty}{\sim} Q 2^{\kappa} \Gamma(\kappa) \kappa^2 K/ h^{\kappa}. \label{singh}
\end{eqnarray}
Putting $h = a r^{1/\kappa}$ in Lemma \ref{supdeZ} and combining with \eqref{singh} (applied with $h = a r^{1/\kappa}$) we get \eqref{lacvloi} and the result follows. 
%convergence of $\mathcal{L}^*_X(t)/t^{1/\kappa}$ follows. 

\end{proof}

\begin{proof} of Theorem \ref{kappa>1favsite}

We now study the asymptotic of the favorite site. We first prove that 
\begin{eqnarray}
F^*(H(r)) / r & \overset{\mathcal{L}}{\underset{r \rightarrow + \infty}{\longrightarrow}} \mathcal{U}. \label{ptfentpsat}
\end{eqnarray}
{$\mathcal{L}^*_X(H(r))$ converges in probability to $+\infty$ because of \eqref{lacvloi0}. Therefore, according to \eqref{expretl3} and \eqref{tlneg}, $\mathbb{P}(\mathcal{L}^*_X(H(r))= \mathcal{M}_2(r))$ converges to $1$ when $r$ goes to infinity. Then, thanks to \eqref{expretl4}, we see that \eqref{ptfentpsat} will follow if we prove that}
\begin{eqnarray}
\text{argmax}_{[0,r]} Z / r & \overset{\mathcal{L}}{\underset{r \rightarrow + \infty}{\longrightarrow}} \mathcal{U}. \label{ptfentpsatbis}
\end{eqnarray}

{The idea is, in $\text{argmax}_{[0,r]} Z$, to replace $Z$ by $Z_1$ and $r$ by $L^{-1}(r/Q)$ (mainly thanks to \eqref{invtl}). As a consequence the $\text{argmax}$ can be linked to the position of the highest excursion of $Z_1$ which is uniform.} Let us choose $\gamma \in ]0, 1/8 \kappa[$. Similarly as in the proof of Lemma \ref{supdeZ}, we have
\[ \mathbb{P} \left ( \sup_{x \in [0,L^{-1}(r/Q - r^{7/8})]} Z_1(x) > r^{1/\kappa - \gamma} \right ) = 1-e^{-(r/Q - r^{7/8}) n(\sup \xi > r^{1/\kappa - \gamma})}, \]
and using \eqref{singh} with $h = r^{1/\kappa - \gamma}$, we get that this probability converges to $1$ when $r$ goes to infinity: 
\begin{eqnarray}
\mathbb{P} \left ( \sup_{x \in [0,L^{-1}(r/Q - r^{7/8})]} Z_1(x) > r^{1/\kappa - \gamma} \right ) \underset{r \rightarrow + \infty}{\longrightarrow} 1. \label{ptfentpsat1}
\end{eqnarray}

Since $L^{-1}(r/Q - r^{7/8})$ is a stopping time for $Z_1$ and is such that $Z_1(L^{-1}(r/Q - r^{7/8}))=1$, we have that
\[ \mathbb{P} \left ( \sup_{x \in [L^{-1}(r/Q - r^{7/8}),L^{-1}(r/Q + r^{7/8})]} Z_1(x) < r^{1/\kappa - \gamma} \right ) = \mathbb{P} \left ( \sup_{x \in [0,L^{-1}(2r^{7/8})]} Z_1(x) < r^{1/\kappa - \gamma} \right ), \]
and we prove, similarly as for \eqref{ptfentpsat1}, that the probability in the right hand side converges to 1. We thus have 
\begin{eqnarray}
\mathbb{P} \left ( \sup_{x \in [L^{-1}(r/Q - r^{7/8}),L^{-1}(r/Q + r^{7/8})]} Z_1(x) < r^{1/\kappa - \gamma} \right ) \underset{r \rightarrow + \infty}{\longrightarrow} 1. \label{ptfentpsat2}
\end{eqnarray}

Putting \eqref{ptfentpsat1} and \eqref{ptfentpsat2} together, we get that, with a probability converging to $1$, $\text{argmax}_{[0,.]} Z_1$ stays constant on $[L^{-1}(r/Q - r^{7/8}),L^{-1}(r/Q + r^{7/8})]$: 
\begin{align}
\mathbb{P} \left ( \forall x \in [L^{-1}(r/Q - r^{7/8}),L^{-1}(r/Q + r^{7/8})], \text{argmax}_{[0,x]} Z_1 = \text{argmax}_{[0,L^{-1}(r/Q)]} Z_1 \right ) \underset{r \rightarrow + \infty}{\longrightarrow} 1. \label{ptfentpsat3}
\end{align}

Recall the notation $Z_1 := Z(\tau(Z, 1) + .)$ and that $Z_1$ has indeed, by the Markov property, the same law as $Z$ starting from $1$. Note that, on $\{ \tau(Z, 1) < r \}$, we have $\text{argmax}_{[0,r]} Z = \text{argmax}_{[0,r-\tau(Z, 1)]} Z_1$. Then, choose $\eta \in ]0, 3/4[$. According to \eqref{supdeZ2}, $\tau(Z, 1) < r^{\eta} < r$ with a probability converging to $1$ and, applying \eqref{invtl} we get
%with, in the lower bound, some $\alpha' \in ]\max \{ 1/2, 1/ \kappa, \eta \}, \alpha[$, we get 
\[ L^{-1}(r/Q - r^{7/8}) \leq L^{-1}((r - r^{\eta})/Q - (r - r^{\eta})^{3/4}) \leq r - r^{\eta} \leq r \leq L^{-1}(r/Q + r^{7/8}), \]
with a probability converging to $1$. {As a consequence we have $L^{-1}(r/Q - r^{7/8}) \leq r-\tau(Z, 1) \leq L^{-1}(r/Q + r^{7/8})$ with a probability converging to $1$. Combining with \eqref{ptfentpsat3} and the fact that $\text{argmax}_{[0,r]} Z = \text{argmax}_{[0,r-\tau(Z, 1)]} Z_1$ (with a large probability),} we thus get
\begin{eqnarray}
\mathbb{P} \left ( \text{argmax}_{[0,r]} Z = \text{argmax}_{[0,L^{-1}(r/Q)]} Z_1 \right ) \underset{r \rightarrow + \infty}{\longrightarrow} 1. \label{ptfentpsat4}
\end{eqnarray}

Let us denote $\hat L := L(\text{argmax}_{[0,L^{-1}(r/Q)]} Z_1)$. {We consider the Poisson point process of excursions of $Z_1$ associated with the local time $L$ and the excursions corresponding to jumps occurring in $[0, r/Q]$. $\hat L$ then represents the instant in $[0, r/Q]$ when occurs the jump corresponding to the highest of these excursions . }It is a well-known property of Poisson point processes that it follows a uniform distribution on $[0, r/Q]$: 
\begin{eqnarray}
Q \hat L /r \overset{\mathcal{L}}{=} \mathcal{U}. \label{ptfentpsat5}
\end{eqnarray}
For the process $Z_1$, this excursion begins at time $L^{-1}(\hat L -)$ and ends at time $L^{-1}(\hat L)$ so 
\begin{align*}
L^{-1}(\hat L -) \leq \text{argmax}_{[0,L^{-1}(r/Q)]} Z_1 \leq L^{-1}(\hat L). 
%\label{ptfentpsat6}
\end{align*}

Then, note that arbitrary high excursions of $Z_1$ arise if we wait long enough. As a consequence, $L^{-1}(\hat L -)$ converges almost surely to $+\infty$ when $r$ goes to infinity. We can thus apply Lemma \ref{itlog} and get that almost surely, for $r$ large enough: 
\[ Q \hat L - \hat L^{3/4} \leq L^{-1}(\hat L -) \leq \text{argmax}_{[0,L^{-1}(r/Q)]} Z_1\leq L^{-1}(\hat L) \leq Q \hat L + \hat L^{3/4}. \]

Combining with \eqref{ptfentpsat5}, we deduce that
\[ \text{argmax}_{[0,L^{-1}(r/Q)]} Z_1 / r \overset{\mathcal{L}}{\underset{t \rightarrow + \infty}{\longrightarrow}} \mathcal{U}. \]
Putting this together with \eqref{ptfentpsat4}, we get \eqref{ptfentpsatbis} and \eqref{ptfentpsat} follows. We now have to prove the result for the favorite site until a deterministic time, instead of the hitting time $H(r)$. For this, let us choose $\alpha$ as in Lemma \ref{chgtvar} and $\gamma \in ]0, (1-\alpha)/ \kappa[$. We prove that, with high probability, the favorite site remains constant between times $H(r/m - r^{\alpha})$ and $H(r/m + r^{\alpha})$, that is, 
\begin{eqnarray}
\mathbb{P} \left ( \forall x \in [H(r/m - r^{\alpha}),H(r/m + r^{\alpha})], \ F^*(x) = F^*(H(r/m) \right ) \underset{r \rightarrow + \infty}{\longrightarrow} 1. \label{ptnormal1}
\end{eqnarray}

First, note that, as a consequence of \eqref{ptfentpsat}, 
\begin{eqnarray}
\mathbb{P} \left ( F^* (H(r/m - r^{\alpha})) \leq r/m - 2 r^{\alpha} \right ) \underset{r \rightarrow + \infty}{\longrightarrow} 1, \label{ptnormal2}
\end{eqnarray}
and, as a consequence of \eqref{lacvloi0} we have
%, for $\gamma \in ]0, (1-\alpha)/\kappa[$, 
\begin{eqnarray}
\mathbb{P} \left ( \mathcal{L}^*_X(H(r/m - r^{\alpha})) > r^{1/ \kappa - \gamma} \right ) \underset{r \rightarrow + \infty}{\longrightarrow} 1. \label{ptnormal3}
\end{eqnarray}
Then, since $H(r/m - r^{\alpha})$ is a stopping time for the diffusion $X$ we have
\begin{eqnarray}
\mathbb{P} \left ( \inf_{[H(r/m - r^{\alpha}), +\infty[} X > r/m - 2r^{\alpha} \right ) = \mathbb{P} \left ( \inf_{[0, +\infty[} X > -r^{\alpha} \right ) \underset{r \rightarrow + \infty}{\longrightarrow} 1. \label{ptnormal4}
\end{eqnarray}
Using similarly the Markov property together with \eqref{lacvloi0}, we get
\begin{align}
& \mathbb{P} \left ( \sup_{x \in [r/m - 2r^{\alpha}, r/m + r^{\alpha}]} \mathcal{L}_X(H(r/m + r^{\alpha}), x) \leq r^{1/ \kappa - \gamma} \right ) \nonumber \\
= & \mathbb{P} \left ( \sup_{x \in [0, 3r^{\alpha}]} \mathcal{L}_X(H(3r^{\alpha}), x) \leq r^{1/ \kappa - \gamma} \right ) \underset{r \rightarrow + \infty}{\longrightarrow} 1. \label{ptnormal5}
\end{align}

The four estimates \eqref{ptnormal2}, \eqref{ptnormal3}, \eqref{ptnormal4} and \eqref{ptnormal5} tell us that with high probability when $r$ is large: at time $H(r/m - r^{\alpha})$, the supremum of the local time has been reached {at a position located before the position $r/m - 2 r^{\alpha}$ and this supremum} is larger than $r^{1/ \kappa - \gamma}$. Moreover, the diffusion will never reach back $]-\infty, r/m - 2 r^{\alpha}]$, and at time $H(r/m + r^{\alpha})$, the supremum of the local time on $[r/m - 2 r^{\alpha}, +\infty [$ is less than $r^{1/ \kappa - \gamma}$. As a consequence, with a probability converging to $1$, the favorite site does not move between {times} $H(r/m - r^{\alpha})$ and $H(r/m + r^{\alpha})$, this proves \eqref{ptnormal1}. Then, \eqref{ptnormal1} together with \eqref{tpsatt} give
\[ \mathbb{P} \left ( F^*(r) = F^*(H(r/m)) \right ) \underset{r \rightarrow + \infty}{\longrightarrow} 1. \]
This together with \eqref{ptfentpsat} proves the sought convergence in distribution for the favorite site. 

\end{proof}

\section{Path decomposition of a spectrally negative L\'evy process} \label{etudevallees}

\subsection{$h$-valleys, scale function, and some processes conditioned to stay positive} \label{extrema}

{Recall from the discussion before Theorem \ref{seqmincvverspoiss} the notion of $h$-extrema. Recall also that $(m_i,\ i \in \mathbb{Z})$ and $(M_i,\ i \in \mathbb{Z})$ are respectively the increasing sequences of $h$-minima and of $h$-maxima of $V$, such that $m_{0} \leq 0<m_1$ and $m_i<M_i<m_{i+1}$ for every  $i\in\mathbb{Z}$. }

As in \cite{Faggionato}, we define the \textit{classical $h$-valleys} as the fragments of the trajectory of $V$ between two $h$-maxima, translated at the $h$-minima between them: the $i^{th}$ classical $h$-valley is the process \[ \left ( V^{(i)}(x), \ M_{i-1} \leq x \leq M_i \right ) \ \ \ \textit{where} \ \ \ V^{(i)} := V(x)-V(m_i), \ \ \forall x\in\mathbb{R}. \]

{Recall that $\hat V$ is the dual of $V$, it is equal in law to $-V$.} In order to state the law of these valleys, we need to recall some definitions about $V$ and $\hat{V}$ conditioned to stay positive but first, let us recall a useful fact. 
\begin{fact} \label{vnb}

Let $Y$ be a spectrally negative L\'evy process which is not the opposite of a subordinator, then it is regular for $]0, +\infty[$ and the regularity for $]-\infty, 0[$ is equivalent with $Y$ being of unbounded variation. 

\end{fact}

\begin{proof}

The regularity for $]0, +\infty[$ and the condition for the regularity for $]-\infty, 0[$ are stated respectively in Theorem VII.1 and in Corollary VII.5 of \cite{Bertoin}. 

\end{proof}

{
Let $W$ be the scale function of $V$, defined as in Section VII.2 of Bertoin \cite{Bertoin}. It satisfies 
\begin{eqnarray}
\forall 0 < x < y, \ \mathbb{P} \left ( \tau(V_x,y) < \tau(V_x,]-\infty, 0]) \right ) = W (x) / W (y). \label{scalefct}
\end{eqnarray}
According to Theorem VII.8 in \cite{Bertoin}, this function is continuous, increasing, and for any $\lambda > \kappa$, 
\[ \int_{0}^{+\infty} e^{-\lambda x} W(x) dx = \frac{1}{\Psi_V(\lambda)} < +\infty. \]

We now explain how the scale function $W$ allows to define $V^{\uparrow}$, that is, $V$ conditioned to stay positive (see \cite{Bertoin}, Section VII.3 for proofs and more details). The relation 
\[ \forall t \geq 0, \ x, y > 0, \ p^{\uparrow}_t(x, dy) := W(y) \times \mathbb{P} \left ( V_x(t) \in dy, \ \inf_{[0, t]} V_x > 0 \right ) / W(x), \]
defines a Markovian semigroup. 
%Soit $V$ un processus de L\'evy spectralement n\'egatif et non monotone, et soit $q_t(x, dy)$ le semi-groupe de transition du processus $V_x$ tu\'e en $0$: 
%\[ \forall t \geq 0, \ \forall x, y > 0, \ q_t(x, dy) := \mathbb{P} \left ( V_x(t) \in dy, \ \inf_{[0, t]} V_x > 0 \right ). \]
%On d\'efinit un semi-groupe Markovien $p^{\uparrow}_t$ par la relation 
%\[ \forall t \geq 0, x, y > 0, \ p^{\uparrow}_t(x, dy) := W(y) q_t(x, dy) / W(x). \]
For any $x > 0$, we denote by $V^{\uparrow}_x$ the Markovian process starting from $x$, taking values in $]0, +\infty[$ and whose transition semigroup is $p^{\uparrow}_t$. $V^{\uparrow}_x$ is commonly called $V$ \textit{conditioned to stay positive starting from $x$}. This is justified by the fact that for any $ y> x > 0$ the process $( V_x^{\uparrow}(t), \ 0 \leq t \leq \tau(V_x^{\uparrow}, y))$ is equal in distribution to the process $( V_x(t), \ 0 \leq t \leq \tau(V_x, y))$ conditionally on $\{ \tau(V_x, y) < \tau(V_x, ]- \infty, 0]) \}$. 
%En d'autres termes, pour les processus tu\'es en leur temps d'atteinte de $y$, $V_x^{\uparrow}$ a m\^eme loi que $V_x$ conditionn\'e (dans le sens usuel) \`a rester positif. 
It is not possible to condition in the usual {way} $V_x$ to remain positive for ever because this occurs with a probability equal to $0$ (since in our case $V$ drifts to $-\infty$). However, $V^{\uparrow}_x$ is well defined, has infinite life-time, and it is possible to show that $V^{\uparrow}_x$ converges almost surely to $+\infty$. 
%et v\'erifie 
%\[ \forall y > x > 0, \mathcal{L} \left ( V_x^{\uparrow}(t), \ 0 \leq t \leq \tau(V_x^{\uparrow}, y) \right ) = \mathcal{L} \left ( V_x(t), \ 0 \leq t \leq \tau(V_x, y) \big | \tau(V_x, y) < \tau(V_x, ]- \infty, 0]) \right ). \]

It is possible to show that there exists a process $V^{\uparrow}_0$, that we denote by $V^{\uparrow}$, starting from $0$ and which is the limit in distribution of the processes $V^{\uparrow}_x$ as $x$ goes to $0$. $V^{\uparrow}$ is called $V$ \textit{conditioned to stay positive}, it is a Feller process whose restriction to $]0, +\infty[$ of the transition laws is given by $p^{\uparrow}_t$. Moreover, note that for any positive $x$, we have from the Markov property and the absence of positive jumps that the process $V^{\uparrow}$, shifted at $\tau(V^{\uparrow}, x)$, is equal in law to $V^{\uparrow}_x$. 
%Note also that, as well as $V$, $V^{\uparrow}$ has no positive jumps and reaches every positive level continuously. 

We also recall the definition of $V$ conditioned to drift to $+\infty$ (see \cite{Bertoin}, Section VII.1 for proofs and more details). Note that, since $\kappa$ is a zero of $\Psi_V$, the process $e^{\kappa V(t)}$ is a martingale with expectation $1$. This allows to define the convolution semigroup of probability measures $(p^{\sharp}_t)_{t \geq 0}$ via the relation 
\begin{eqnarray}
p^{\sharp}_t (dy) := e^{\kappa y} \mathbb{P} \left ( V(t) \in dy \right ). \label{convsgsharp}
\end{eqnarray}
The convolution semigroup $(p^{\sharp}_t)_{t \geq 0}$ is associated to a spectrally negative L\'evy process that we denote by $V^{\sharp}$ and which is commonly called \textit{$V$ conditioned to drift to $+\infty$}. This is justified by the fact that 
\begin{itemize}
\item $V^{\sharp}$ drifts to $+\infty$, 
\item For any $y > 0$ the process $( V^{\sharp}(t), \ 0 \leq t \leq \tau(V^{\sharp}, y))$ is equal in distribution to the process $( V(t), \ 0 \leq t \leq \tau(V, y))$ conditionally on $\{ \tau(V, y) < +\infty \}$.  
\end{itemize}
It is not difficult to see from \eqref{convsgsharp} that the Laplace exponent $\Psi_{V^{\sharp}}$ of $V^{\sharp}$ is obtained by translation of $\Psi_V$: $\Psi_{V^{\sharp}} = \Psi_V(\kappa + .)$. 
%where $\kappa$ is the non-trivial zero of $\Psi_V$. 
As a consequence $\Psi'_{V^{\sharp}}(0) > 0$ which, according to Corollary VII.2 in \cite{Bertoin}, proves that $V^{\sharp}$ drifts to $+\infty$ as stated above. It is also proven that $V^{\uparrow} = (V^{\sharp})^{\uparrow}$. As a consequence, for any $x > 0$, $V^{\uparrow}_x$ is equal in law to $V^{\sharp}_x$ conditionally on $\{ \inf_{[0, +\infty[} V^{\sharp}_x > 0 \}$. We will sometimes use this fact to prove results about $V^{\uparrow}$ in Section \ref{estimates}. 
}

We now define $\hat V^{\uparrow}$, that is, $\hat V$ conditioned to stay positive. According to Fact \ref{vnb}, $V$ is regular for $]0, +\infty[$, so $\hat{V}$ is for $]-\infty, 0[$. Moreover, $\hat V$ drifts to $+\infty$. We can thus define the Markov family $( \hat V^{\uparrow}_x, \ x > 0 )$ as in Doney \cite{Doney}, Chapter 8. {It can be seen from there that for any $x > 0$ the process $\hat V^{\uparrow}_x$ is Markovian and has infinite life-time.} If moreover $V$ has unbounded variation then $\hat V$ is regular for $]0, +\infty[$, and from Theorem 25 of \cite{Doney}, we have that $\hat V^{\uparrow}_0$, that we shall denote by $\hat V^{\uparrow}$, is well defined. 
%and that $\hat V^{\uparrow}_0$, that we denote by $\hat V^{\uparrow}$, is indeed well defined, and that $\hat{V}^{\uparrow}_x$ converges in the Skorokhod space to $\hat{V}^{\uparrow}$ when $x$ goes to $0$. 

Here again, for any $x \geq 0$, the process $\hat V^{\uparrow}_x$ must be seen as $\hat V$ \textit{conditioned to stay positive and starting from $x$}. Note that, since $\hat V$ converges almost surely to infinity, for $x > 0$, $\hat V^{\uparrow}_x$ is equal in law to $\hat V_x$ conditioned in the usual {way} to remain positive.

\subsection{Law of the valleys} \label{loidesval}

We now prove some facts about the law of the consecutive valleys near their bottom. In order to delimit the bottom of the valleys we define
\begin{equation*}
    \tau_i^-(h) := \sup \{x < m_i,\  V^{(i)}(x) \geq h\},
\
    \tau_i(h):=  \inf \{x > m_i,\  V^{(i)}(x)=h\}. 
\end{equation*}
For any $i \in \mathbb{N}^*$, let
{$P_1^{(i)}$ } be the truncated process $(V^{(i)}(m_i-x),\ 0 \leq x \leq m_i-\tau_i^-(h))$, and
${P_2}^{(i)}$ the truncated process $(V^{(i)}(m_i+x),\  0\leq x \leq \tau_i(h)-m_i)$. We have: 

\begin{prop}\label{Fact_Williams} (valleys decomposition) \\
Assume that $V$ has unbounded variation. All the processes of the family $(P_j^{(i)}, \ i \geq 1, \ j \in \{ 1,2 \})$, are independent and :
\begin{itemize}
\item For all $i \geq 2$, the law of $P_1^{(i)}$ is absolutely continuous with respect to the law of the process $(\hat{V}^{\uparrow}(x))_{0 \leq x \leq \tau(\hat{V}^{\uparrow}, h+)}$ and has density $c_h /(1-e^{-\kappa \hat{V}^{\uparrow}(\tau(\hat{V}^{\uparrow}, h+))})$ with respect to this law, where $c_h$ is a constant depending on $h$ {and that converges to $1$ when $h$ goes to infinity. }
\item For all $i \geq 1$, $P_2^{(i)}$ is equal in law to $(V^{\uparrow}(x))_{0 \leq x \leq \tau(V^{\uparrow}, h)}$ (this statement is true even if $V$ has bounded variation). 
%, that is, $V^{\uparrow}$ killed when it reaches $h$. 
\end{itemize}
Moreover, the density $c_h /(1-e^{-\kappa \hat{V}^{\uparrow}(\tau(\hat{V}^{\uparrow}, h+))})$ is bounded by $2$ for $h$ deterministically large enough and it converges to $1$ when $h$ goes to infinity. 
\end{prop}

\begin{remarque}
$P_1^{(1)}$ may be a part of the so-called \textit{central slope}, so its law is different. 
\end{remarque}

\begin{proof}
We assume that $V$ has unbounded variation, recall from Fact \ref{vnb} that this implies a regularity condition. In Lemma 4 of Cheliotis \cite{Cheliotis2006715}, the assertion that $\underline{\lim}_{t \rightarrow + \infty} V(t) = -\infty, \overline{\lim}_{t \rightarrow + \infty} V(t) = +\infty$ can be dropped. Indeed, because of the regularity condition satisfied by $V$, the stopping times $\underline \tau_{k}$ and $\overline \tau_{k}$ in \cite{Cheliotis2006715} are almost surely finite for $V$. This lemma is therefore true in our context. As a consequence, the proof of Lemma 1 of \cite{Cheliotis2006715} also applies in our context. We thus get the fact that the slopes are independent and that (except the central slope), all descending (respectively ascending) slopes have the same law. 

For the law of the ascending slopes (not covering the origin) until their hitting time of $h$, we use the classical argument that can be found, for example, in the proof of Theorem 2 in \cite{Faggionato}: $P_2^{(i)}$ is equal in law to the first excursion higher than $h$ of $V-\underline{V}$, before the hitting time of $h$ of this excursion. Using Proposition VII.15 of \cite{Bertoin} (which does not require $V$ to have unbounded variation), we prove that the latter is equal in law to $(V^{\uparrow}(x))_{0 \leq x \leq \tau(V^{\uparrow}, h)}$, that is, $V^{\uparrow}$ killed when hitting $h$. 

Then, by the time-reversal property, the descending slopes (not covering the origin) of $V$ are equal in law to the ascending slopes of $\hat V$ so, here again, we get that $P_1^{(i)}$ is equal in law to the first excursion higher than $h$ of $\hat V-\underline{\hat V}$, before the hitting time of $[h, +\infty[$ of this excursion. Unfortunately, since $\hat V$ is not spectrally negative, we can no longer apply the previous argument to determine the law of the excursion. 
However, since $]- \infty, 0[$ and $]0, + \infty[$ are regular for $\hat{V}$ (because they are for $V$) and $\hat{V}$ does not drift to $-\infty$, we can apply Proposition 4.7 of Duquesne \cite{Duquesne2003339} but we first need to introduce some notations. 

{Let $(\hat{L}^{-1}, \hat{I})$ denote the descending ladder process of $\hat{V}$, $\hat{L}(t) := \underline{\hat{V}}(t)$ is a local time at $0$ of $\hat{V} - \underline{\hat{V}}$ (this comes from Theorem VII.1 of \cite{Bertoin} and duality, since $\hat{V}$ is spectrally positive) and $\hat{L}^{-1}$ is the inverse of this local time. For any $t \geq 0$, $\hat{I}(t) := -\underline{\hat{V}}(\hat{L}^{-1}(t))$. We denote by $\hat{\mathcal{N}}$ the excursion measure of $\hat{V} - \underline{\hat{V}}$ associated with $\hat{L}$, and for $\xi$ an excursion, let $\zeta(\xi) := \inf \{ s > 0, \ \xi(s) = 0 \}$ denote its life-time. 
%(DIRE QUEL TL)

%Since $\hat{V}$ is spectrally positive, it can be seen from \cite{Bertoin} (DIRE LE CHAPITRE) that for any $t \geq 0, \hat{U}(t)=t$. 
We denote by $\hat{\mathcal{I}}$ the potential measure of $\hat{I}$, $\hat{\mathcal{I}}([0, x]) := \mathbb{E} [ \int_0^{\hat{L}(+\infty)} \mathds{1}_{\hat{I}(t) \in [0, x]} dt ]$ for $x \geq 0$. By the expression of the potential measure of the ascending ladder height process of a spectrally negative L\'evy process, page 191 of \cite{Bertoin}, and by duality we have} $\hat{\mathcal{I}}([0, x]) = (1-e^{-\kappa x})/ \kappa$ for any $x \geq 0$. Proposition 4.7 of \cite{Duquesne2003339} yields that for any measurable function $G$ defined on the space of c\`ad-l\`ag functions from $[0, +\infty[$ to $\mathbb{R}$ with possibly finite life-time, we have
\begin{align*}
\mathbb{E} \left [ G \left ( (\hat{V}^{\uparrow}(s))_{0 \leq s \leq \tau(\hat{V}^{\uparrow}, h+)} \right ) \right ] & = \hat{\mathcal{N}} \left ( G \left ( (\xi(s))_{0 \leq s \leq \tau(\xi, h+)} \right ) \hat{\mathcal{I}}([0, \xi (\tau(\xi, h+))[) \big | \tau(\xi, h+) < \zeta(\xi) \right ) \\
& \times \hat{\mathcal{N}} \left ( \xi, \tau(\xi, h+) < \zeta(\xi) \right ) \\
& = \hat{\mathcal{N}} \left ( \frac{(1-e^{-\kappa \xi(\tau(\xi, h+)) })}{c_h} G \left ( (\xi(s))_{0 \leq s \leq \tau(\xi, h+)} \right ) \big | \tau(\xi, h+) < \zeta(\xi) \right )
\end{align*}
where we have set $c_h := \kappa / \hat{\mathcal{N}} \left ( \xi, \tau(\xi, h+) < \zeta(\xi) \right )$. As a consequence, for any measurable function $F$, we get that 
\[ \hat{\mathcal{N}} \left ( F \left ( (\xi(s))_{0 \leq s \leq \tau(\xi, h+)} \right ) | \tau(\xi, h+) < \zeta(\xi) \right ) = \mathbb{E} \left [ \frac{c_h}{1-e^{-\kappa \hat{V}^{\uparrow}(\tau(\hat{V}^{\uparrow}, h+))}} F \left ( (\hat{V}^{\uparrow}(s))_{0 \leq s \leq \tau(\hat{V}^{\uparrow}, h+)} \right ) \right ]. \]
That is, the law of the first excursion of $\hat{V} - \underline{\hat{V}}$ higher than $h$, killed when reaching $[h, +\infty[$, is absolutely continuous with respect to the law of the process $(\hat{V}^{\uparrow}(s))_{0 \leq s \leq \tau(\hat{V}^{\uparrow}, h+)}$ and has density $c_h /(1-e^{-\kappa \hat{V}^{\uparrow}(\tau(\hat{V}^{\uparrow}, h+))})$ with respect to this law. 
%
%We have already explained that the law of $P_1^{(i), h}, i \geq 2$ is the same as the law of the first excursion higher than $h$, and killed when reaching $[h, +\infty[$, of the process $\hat{V} - \underline{\hat{V}}$, so the result follows. 

{For the last assertion of the proposition, note that $c_h = \kappa/\hat{\mathcal{N}} ( \xi, \tau(\xi, h+) < \zeta(\xi) )$ increases, when $h$ goes to infinity, to $c_{\infty} :=  \kappa/\hat{\mathcal{N}} ( \xi, \zeta(\xi) = + \infty )$. $\hat{\mathcal{N}} ( \xi, \zeta(\xi) = + \infty )$ is the measure of the set of infinite excursions above $0$ for $\hat{V} - \underline{\hat{V}}$, it is strictly positive since $\hat{V}$ converges almost surely to infinity. Also, we have almost surely that $\hat{V}^{\uparrow}(\tau(\hat{V}^{\uparrow}, h+)) \geq h$. Hence, 
\[ c_h / \left ( 1-e^{-\kappa \hat{V}^{\uparrow}(\tau(\hat{V}^{\uparrow}, h+))} \right ) \leq c_{\infty} /(1-e^{-\kappa h}) \leq 2 c_{\infty}, \]
where the last inequality is true for $h$ deterministically large enough.} Then, 
\[ \mathbb{P} \left ( c_h/ \left ( 1-e^{-\kappa \hat{V}^{\uparrow}(\tau(\hat{V}^{\uparrow}, h+))} \right ) \underset{h \rightarrow +\infty}{\longrightarrow} c_{\infty} \right ) = 1. \]
{By the dominated convergence theorem and the fact that $c_h /(1-e^{-\kappa \hat{V}^{\uparrow}(\tau(\hat{V}^{\uparrow}, h+))})$ is normalized at $1$ (because it is a density), we get $c_{\infty} = 1$ and the last assertion of the proposition follows. }
\end{proof}

\subsection{Standard valleys} \label{coin}

Since $V$ drifts to $-\infty$, the descending phases of $V$ between two $h$-minima are quite important and have to be taken in consideration for the study of the diffusion in $V$. Therefore, as in \cite{advech}, we here define a new sequence $(\tilde m_i)_{i \geq 1}$ of $h$-minima that are separated by descending phases of $V$ and we then show that for a large number of indices, this sequence coincides with the sequence $(m_i)_{i \geq 1}$ with an overwhelming probability. Our definition of the standard valleys {is} similar to the one given in Section 2.2 of \cite{advech} but we have to improve it to get a definition more adapted to our context (where $V$ can jump). We first introduce some notations. 

$\delta \in ]0, 1/2[$ is defined once and for all in the paper and can be chosen as small as we want. In this subsection $h$ is a fixed positive number such that $e^{(1-\delta)\kappa h} \geq h$. 
%Let \[ h^+ := (1+ \kappa+ 2\delta) h. \] INUTILE
%We define  $\tilde \tau_0(h) := 0$ (POURQUOI ???), $\tilde L_0:=0$, $\tilde m_0^h:=0$,
%and recursively for $i \geq 1$ (see Figure), 
We define  $\tilde \tau_0(h) = \tilde L_0:=0$ and recursively for $i \geq 1$, 
\begin{eqnarray}
     \tilde L_i^{\sharp}
&:= &
%    \inf\{x>\tilde L_{i-1}^{h, +},\ V(x) \leq V(\tilde L_{i-1}^{h, +})-h^+   \},
    \inf\{x>\tilde L_{i-1},\ V(x)\leq V(\tilde L_{i-1})-e^{(1-\delta)\kappa h}   \},
\nonumber\\
    \tilde \tau_i(h)
& := &
    \inf \big\{x \geq   \tilde L_i^{\sharp},\ V(x)-\inf\nolimits_{[\tilde L_i^{\sharp},x]}V = h\big\},
\nonumber\\
    \tilde{m}_i
&:= &
    \inf\big\{x \geq \tilde L_i^{\sharp},\ V(x)=\inf\nolimits_{[\tilde L_i^{\sharp},\tilde \tau_i(h)]} V \big\},
\nonumber\\
%    \tilde L_i^*
%&:= &
%    \inf\{x>\tilde \tau_i(h),\ V(x)\leq V(\tilde \tau_i(h))-h/4 \},
%\nonumber\\
    \tilde L_i
&:= &
    \inf\{x>\tilde \tau_i(h),\ V(x) {- V(\tilde{m}_i)} \leq h/2 \}
\nonumber\\
&= &
    {\inf\{x>\tilde \tau_i(h),\ V(x) - V(\tilde \tau_i(h)) \leq -h/2 \},}
\nonumber\\
    \tilde \tau_i^-(a)
& := &
    \sup \{x < \tilde m_i,\  V(x)-V(\tilde{m}_i) \geq a\}, \ \forall a \in [0, h],
\nonumber\\
    \tilde \tau_i^+(a)
& := &
    \inf \{x > \tilde m_i,\  V(x)-V(\tilde{m}_i) = a\}, \ \forall a \in [0, h].
\nonumber
\end{eqnarray}
%EST-IL UTILE DE GARDER LA DESCENTE IMPOSEE ?
Note that all these random variables depend on $h$, even if this does not appear in the notations. We also introduce the equivalent of $V^{(i)}$ for the $\tilde m_i, i \in \mathbb{N}^*$ as follows: 
\[ \tilde V^{(i)}(x) := V(x)-V(\tilde m_i), \ \ \forall x\in\mathbb{R}. \]
We call $i^{th}$ \textit{standard valley} the re-centered truncated potential
$(\tilde V^{(i)} (x),\ \tilde L_{i-1} \leq x \leq \tilde L_{i} )$. 

\begin{remarque} \label{iid}
The random times $\tilde L_i^{\sharp}$, $\tilde \tau_i(h)$, and $\tilde L_{i}$ are stopping times. As a consequence, the sequence $(\tilde V^{(i)} (x + \tilde m_i),\ \tilde L_{i-1} - \tilde m_i \leq x \leq \tilde L_{i} - \tilde m_i)_{i \geq 1}$ is \textit{iid}. 
\end{remarque}

Our definitions take in consideration the absence of positive jumps for $V$, in particular $\tilde \tau_i(h) < + \infty$ and $V(\tilde \tau_i(h)-) = V(\tilde \tau_i(h))=V(\tilde m_i)+h$. We can see that the $\tilde m_i$, $i\in\mathbb{N}^*$, are $h$-minima. The next lemma, which is the analogous of Lemma 2.3 of Andreoletti, Devulder \cite{AndDev}, shows that, with high probability, the sequence $(\tilde m_i)_{i \geq 1}$ coincides with the sequence $(m_i)_{i \geq 1}$ for indices $i \leq n$ when $n$ does not grow too fast with $h$. 

%It remain to show that with high probability, the minima $(m_i)_{1 \leq i \leq n_t}$ coincides with the set of minima $(\tilde{m}_i)_{1 \leq i \leq n_t}$, that is, with high probability, the two kind of minima that we defined are the same for a large number of indices. 

\begin{lemme} \label{minimacoincide} 

There is $h_0 > 0$ such that for all $n \geq 1$ and $h \geq h_0$, 
\[ \mathbb{P} \left ( \mathcal{V}_{n, h} := \overset{n}{\underset{i=1}{\cap}} \left \{ m_i = \tilde{m}_i \right \} \right ) \geq 1 - n e^{- \delta \kappa h /3}. \]

\end{lemme}

\begin{proof}

As we said before, $\{\tilde m_i,\ i\in\mathbb{N}^*\}\subset \{m_i, \ i\in\mathbb{N}^*\}$.
Hence on the complementary $\mathcal{V}_{n, h}^c$ of $\mathcal{V}_{n, h}$,
considering the smallest $1\leq i\leq n$ such that $m_i \neq \tilde m_i$, we have
$m_{i-1} = \tilde m_{i-1} < m_i <\tilde m_{i}$. 

%If for this $i$, $\tilde L_{i}^{h, \sharp}<m_i^h<\tilde m_{i}^h$, since $m_i^h$ is an $h$-minimum, there would be a $v_i>m_i^h$ such that $V(m_i^h)=\inf_{[m_i^h,v_i]} V$ and $V(v_i-)\geq V(m_i^h)+h \geq \inf_{[\tilde L_{i}^{h, \sharp},v_i]}V+h$.
%So, $m_i^h<\tilde m_i^h\leq \tilde \tau_{i}(h)\leq v_i$, by definition of $\tilde m_i^h$ and $\tilde \tau_{i}(h)$.
%Then $V(m_i^h)=\inf_{[m_i^h,v_i]}V\leq V(\tilde m_{i}^h)$, which contradicts the definition of $\tilde m_{i}^h$.
For this $i$, we can not have $m_i \in [\tilde L_{i}^{\sharp}, \tilde m_{i}[$. Indeed, in this case, $m_i - \tilde L_{i}^{\sharp}$ would be the starting point of an excursion higher than $h$ for $V^{\tilde L_{i}^{\sharp}} - \underline V^{\tilde L_{i}^{\sharp}}$, but from the definition of $\tilde m_{i}$, $\tilde m_{i} - \tilde L_{i}^{\sharp}$ is the starting point of the first excursion higher than $h$ for $V^{\tilde L_{i}^{\sharp}} - \underline V^{\tilde L_{i}^{\sharp}}$, which contradicts $m_i<\tilde m_{i}$. 
Hence, $m_i \in ]\tilde m_{i-1}, \tilde L_{i}^{\sharp}[$. 

$\tilde m_{i-1}$ is an $h$-minimum and there cannot be any $h$-maximum belonging to $[\tilde m_{i-1}, \tilde \tau_{i-1}(h)[$ so, as the $h$ minima and the $h$ maxima alternate, necessarily we have $m_i > \tilde \tau_{i-1}(h)$. Hence $m_i \in ]\tilde \tau_{i-1}(h), \tilde L_{i}^{\sharp}[$. 

Since $m_i$ is an $h$-minimum, there is $v_i>m_i$ such that
$V(m_i)=\inf_{[m_i,v_i]} V$ and
$V(v_i) = V(m_i)+h$. Since by definition of $\tilde L_{i}^{\sharp}$ we have $V(\tilde L_{i}^{\sharp})=\inf_{[\tilde \tau_{i-1}(h), \tilde L_{i}^{\sharp}]} V$, we cannot have $\tilde L_{i}^{\sharp} \in ]m_i, v_i]$ so we must have 
\[ \tilde \tau_{i-1}(h) < m_i < v_i < \tilde L_{i}^{\sharp}. \]
For similar reasons we can neither have $\tilde L_{i-1} \in ]m_i, v_i]$ so we have 
\[ \text{either   } \tilde \tau_{i-1}(h) < m_i < v_i < \tilde L_{i-1} \ \text{   or   } \ \tilde L_{i-1} \leq m_i < v_i < \tilde L_{i}^{\sharp}. \]
We have 
%\begin{eqnarray}
%\mathcal{V}_{n_h}^c \subset \overset{n_h}{\underset{i=1}{\cup}} \left ( \left \{ \tilde \tau_{i-1}(h) \leq m_i^h < v_i < L_{i}^{h, +} \right \} \cup \left \{ L_{i}^{h, +} \leq m_i^h < v_i < L_{i}^{h, \sharp} \right \} \right ). \label{minimacoincide0}
%\end{eqnarray}
\[ \mathcal{V}_{n, h}^c = \left ( \overset{n}{\underset{i=1}{\cap}} \left \{ m_i = \tilde{m}_i \right \} \right )^c \subset \overset{n}{\underset{i=1}{\cup}} \left ( E_i^1 \cup E_i^2 \right ), \]
where
\begin{align*}
E_i^1 & := \left \{ \tau \left ( V^{\tilde \tau_{i-1}(h)} - \underline{V}^{\tilde \tau_{i-1}(h)},h \right ) < \tau \left ( V^{\tilde \tau_{i-1}(h)}, ]-\infty, -h/2] \right ) \right \}, \\
E_i^2 & := \left \{ \tau \left ( V^{\tilde L_{i-1}} - \underline{V}^{\tilde L_{i-1}},h \right ) < \tau \left ( V^{\tilde L_{i-1}}, ]-\infty, -e^{(1-\delta)\kappa h}] \right ) \right \}. 
\end{align*}
{Indeed, $\tau ( V^{\tilde \tau_{i-1}(h)} - \underline{V}^{\tilde \tau_{i-1}(h)},h ) + \tilde \tau_{i-1}(h)$ represents the instant when is achieved the first ascend of $V$ from its minimum after $\tilde \tau_{i-1}(h)$ and, from the definition of $\tilde L_{i-1}$, 
%and the fact that $V(\tilde \tau_{i-1}(h))=V(\tilde m_{i-1})+h$, 
 we have $\tilde L_{i-1} = \tau(V^{\tilde \tau_{i-1}(h)}, ]-\infty, -h /2]) +  \tilde \tau_{i-1}(h)$. Similarly, $\tau ( V^{\tilde L_{i-1}} - \underline{V}^{\tilde L_{i-1}},h ) + \tilde L_{i-1}$ represents the instant when is achieved the first ascend of $V$ from its minimum after $\tilde L_{i-1}$ and, from the definition of $\tilde L_{i}^{\sharp}$, we have $\tilde L_{i}^{\sharp} = \tau ( V^{\tilde L_{i-1}}, ]-\infty, -e^{(1-\delta)\kappa h}] ) + \tilde L_{i-1}$. }Since $\tilde \tau_{i-1}(h)$ and $\tilde L_{i-1}$ are stopping times, $V^{\tilde \tau_{i-1}(h)}$ and $V^{\tilde L_{i-1}}$ are equal in law to $V$. Since moreover $e^{(1-\delta)\kappa h} \geq h/2$, we get for $h$ large enough 
\begin{eqnarray}
1 - \mathbb{P} \left ( \overset{n}{\underset{i=1}{\cap}} \left \{ m_i = \tilde{m}_i \right \} \right ) \leq 2 n \mathbb{P} \left ( \tau \left (V - \underline{V},h \right ) < \tau \left (V, ]-\infty, -e^{(1-\delta)\kappa h}] \right ) \right ). \label{minimacoincide5}
\end{eqnarray}
{Then, according to Lemma \ref{monteavantdescente} applied with $a=h$, $b=e^{(1-\delta)\kappa h}$ and $\eta = \delta /2$ we have 
\[ \mathbb{P} \left ( \tau \left (V - \underline{V},h \right ) < \tau \left (V, ]-\infty, -e^{(1-\delta)\kappa h}] \right ) \right ) \leq \left ( 1 + 2 e^{(1-\delta)\kappa h}/\delta h \right ) e^{-\kappa (1-\delta /2) h}. \]
Let us define $h_0$ large enough so that the latter is less than $e^{- \delta \kappa h /3}/2$ whenever $h\geq h_0$. Combining with \eqref{minimacoincide5} we obtain 
\[ \forall h \geq h_0, \ 1 - \mathbb{P} \left ( \overset{n}{\underset{i=1}{\cap}} \left \{ m_i = \tilde{m}_i \right \} \right ) \leq n e^{- \delta \kappa h /3}, \]
which yields the results. }

\end{proof}

We now make use of the preceding lemma to precise the law of the bottoms of the standard valleys. First, for any $i \in \mathbb{N}^*$, let us define 
\begin{align*}
& \tilde P_1^{(i)} := (\tilde V^{(i)}(\tilde m_i-x),\ 0 \leq x \leq \tilde m_i- \tilde \tau_i^-(h)), \ \tilde P_2^{(i)} := (\tilde V^{(i)}(\tilde m_i+x),\  0\leq x \leq \tilde \tau_i(h)- \tilde m_i), \\
& \tilde P_3^{(i)} := (\tilde V^{(i)}(\tilde \tau_i(h)+x),\  0\leq x \leq \tilde L_i - \tilde \tau_i(h)). 
\end{align*}
Note that for any index $i$, $m_i = \tilde m_i$ implies $P_1^{(i)} = \tilde P_1^{(i)}$ and $P_2^{(i)} = \tilde P_2^{(i)}$. 

\begin{prop} \label{standardwilliams}

Assume that $V$ has unbounded variation. All the processes of the family $(\tilde P_j^{(i)}, \ i \geq 1, \ j \in \{ 1, 2, 3 \})$ are independent and for all $i \geq 1$, 
\begin{align*}
& d_{VT} \left ( \tilde P_1^{(i)}, P_1^{(2)} \right ) \leq 2 e^{- \delta \kappa h /3}, \ \ \ \ \ \ \ \ \ \ \ \ \ \tilde P_2^{(i)} \egloi (V^{\uparrow}(x))_{0 \leq x \leq \tau(V^{\uparrow}, h)}, \\
& \tilde P_3^{(i)} \egloi (h + V(x))_{0 \leq x \leq \tau(V, ]-\infty, -h/2])}. 
\end{align*}
The statements about the laws of $\tilde P_2^{(i)}$ and $\tilde P_3^{(i)}$ do not require the hypothesis of unbounded variation. The statements about the law of $\tilde P_1^{(i)}$ is true only for $h$ large enough. 

\end{prop}

\begin{proof}
By Remark \ref{iid} the sequence $(\tilde P_1^{(i)}, \tilde P_2^{(i)}, \tilde P_3^{(i)})_{i \geq 1}$ is \textit{iid} so, for the independence, we only need to prove that for any $i \geq 1$, $\tilde P_1^{(i)}$, $\tilde P_2^{(i)}$ and $\tilde P_3^{(i)}$ are independent. 
%Some ideas in the proof of Proposition \ref{Fact_Williams} still apply here: 
$\tilde P_2^{(i)}$ is, after {the stopping time $\tilde L_i^{\sharp}$}, the first excursion of $V - \underline{V}$ greater than $h$, considered up to its hitting time of $h$. It is therefore independent from the previous slopes and, using again Proposition VII.15 of \cite{Bertoin} (which does not require $V$ to have unbounded variation), we see that it has the same law as $(V^{\uparrow}(x))_{0 \leq x \leq \tau(V^{\uparrow}, h)}$. Also, the Markov property applied at the stopping time $\tilde \tau_i(h)$ gives the asserted law of $\tilde P_3^{(i)}$ and its independence from $(\tilde P_1^{(i)}, \tilde P_2^{(i)})$. 

For the assertion on $\tilde P_1^{(i)}$, {note that $d_{VT} ( \tilde P_1^{(2)}, P_1^{(2)} ) \leq \mathbb{P} (\tilde P_1^{(2)} \neq P_1^{(2)})$ and} for $h$ large enough $\mathbb{P} (\tilde P_1^{(2)} \neq P_1^{(2)}) \leq 2 e^{- \delta \kappa h /3}$ according to Lemma \ref{minimacoincide} (applied with $n := 2$). Then, recall that for any $i \geq 1$, $\tilde P_1^{(i)}$ is equal in law to $\tilde P_1^{(2)}$ according to Remark \ref{iid}. The assertion on $\tilde P_1^{(i)}$ follows. 

\end{proof}

{We now consider the first ascend of $h$ from the minimum, after $\tilde \tau_i(h)$, 
\begin{align}
\tilde \tau_{i+1}^*(h) & := \inf \left \{ u\geq \tilde \tau_i(h),\ V(u)-\inf_{[\tilde \tau_i(h),u]} V = h \right \}, \label{ascend} \\ 
\tilde{m}_{i+1}^* & := \inf \left \{ u \geq \tilde \tau_i(h),\ V(u)=\inf_{[\tilde \tau_i(h),\tilde \tau_{i+1}^*(h)]} V \right \}. \label{min}
\end{align}
Note that $\tilde \tau_{i+1}^*(h) - \tilde \tau_{i}(h) = \tau(V^{\tilde \tau_{i}(h)}-\underline{V}^{\tilde \tau_{i}(h)}, h)$, so $\tilde \tau_{i+1}^*(h) - \tilde \tau_{i}(h)$ belongs to the first excursion higher than $h$ of the process $V^{\tilde \tau_{i}(h)}-\underline{V}^{\tilde \tau_{i}(h)}$ and $\tilde{m}_{i+1}^* - \tilde \tau_{i}(h)$ is the starting point of this excursion. }We prove that $\tilde \tau_{i+1}^*(h)$ and $\tilde{m}_{i+1}^*$ coincide with $\tilde\tau_{i+1}(h)$ and $\tilde m_{i+1}$ with a high probability. 

\begin{lemme} \label{tpscoinc}

There is a positive constant $c$ such that for $h$ large enough, 
\[ \forall i \geq 1, \ \mathbb{P} \left ( \tilde \tau_{i}^*(h) = \tilde\tau_{i}(h), \ \tilde m_{i}^* = \tilde m_{i} \right ) \geq 1 - e^{-c h}. \]

\end{lemme}

\begin{proof}

Note that $\tilde m_{i}^* = \tilde m_{i}$ whenever $\tilde \tau_{i}^*(h) = \tilde\tau_{i}(h)$. We thus only prove the latter. Fix $i \geq 2$ and recall the definitions of $\tilde \tau_{i}^*(h)$ and $\tilde\tau_{i}(h)$: {$\tilde \tau_{i}^*(h)$ is the first ascend of $V$ from its minimum after $\tilde \tau_{i-1}(h)$ while $\tilde\tau_{i}(h)$ is the first ascend of $V$ from its minimum after $\tilde L_{i}^{\sharp}$. By definition we} have $V(\tilde L_{i}^{\sharp}) = \inf_{[\tilde \tau_{i-1}(h), \tilde L_{i}^{\sharp}]} V$ so, if $\tilde L_{i}^{\sharp} < \tilde \tau_{i}^*(h)$, we must have $\tilde \tau_{i}^*(h) = \tilde\tau_{i}(h)$. Therefore, 
\begin{eqnarray}
\left \{ \tilde \tau_{i}^*(h) \neq \tilde\tau_{i}(h) \right \} \subset \left \{ \tilde \tau_{i}^*(h) \leq \tilde L_{i}^{\sharp} \right \} = \left \{ \tilde \tau_{i}^*(h) \leq \tilde L_{i-1} \right \} \cup \left \{ \tilde L_{i-1} < \tilde \tau_{i}^*(h) \leq \tilde L_{i}^{\sharp} \right \}. \label{tpscoinc1}
\end{eqnarray}
{Recall that $\tilde L_{i-1} - \tilde \tau_{i-1}(h) = \tau(V^{\tilde \tau_{i-1}(h)}, ]-\infty, -h /2])$. Since $\tilde \tau_{i}^*(h) - \tilde \tau_{i-1}(h) = \tau(V^{\tilde \tau_{i-1}(h)}-\underline{V}^{\tilde \tau_{i-1}(h)}, h)$ we get 
\[ \mathbb{P} \left ( \tilde \tau_{i}^*(h) \leq \tilde L_{i-1} \right ) = \mathbb{P} \left ( \tau(V^{\tilde \tau_{i-1}(h)}-\underline{V}^{\tilde \tau_{i-1}(h)}, h) < \tau(V^{\tilde \tau_{i-1}(h)}, ]-\infty, -h /2]) \right ). \]
Then, from the Markov property, $V^{\tilde \tau_{i-1}(h)}$ is equal in law to $V$. Combining with Lemma \ref{monteavantdescente} applied with $a = h, b = h/2$ and $\eta = 1/4$ we get 
\begin{eqnarray}
\mathbb{P} \left ( \tilde \tau_{i}^*(h) \leq \tilde L_{i-1} \right ) = \mathbb{P} \left ( \tau(V-\underline{V}, h) < \tau(V, ]-\infty, -h /2]) \right ) \leq 3 e^{- 3 \kappa h /4}. \label{tpscoinc2}
\end{eqnarray}}
Then, $V(\tilde L_{i-1}) = \inf_{[\tilde \tau_{i-1}(h), \tilde L_{i-1}]} V$ so, if $\tilde \tau_{i}^*(h) > \tilde L_{i-1}$ we have $\tilde \tau_{i}^*(h) - \tilde L_{i-1} = \tau ( V^{\tilde L_{i-1}} - \underline{V}^{\tilde L_{i-1}}, h)$. {Recall that $\tilde L_{i}^{\sharp} - \tilde L_{i-1} = \tau ( V^{\tilde L_{i-1}}, ]-\infty, -e^{(1-\delta)\kappa h}] )$. We thus get
\[ \mathbb{P} \left ( \tilde L_{i-1} < \tilde \tau_{i}^*(h) \leq \tilde L_{i}^{\sharp} \right ) \leq \mathbb{P} \left ( \tau \left ( V^{\tilde L_{i-1}}-\underline{V}^{\tilde L_{i-1}}, h \right ) < \tau \left (V^{\tilde L_{i-1}}, ]-\infty, -e^{(1-\delta) \kappa h}] \right ) \right ). \]
Then, from the Markov property, $V^{\tilde L_{i-1}}$ is equal in law to $V$. Combining with Lemma \ref{monteavantdescente} applied with $a=h$, $b=e^{(1-\delta)\kappa h}$ and $\eta = \delta /2$ we get 
\begin{align}
\mathbb{P} \left ( \tilde L_{i-1} < \tilde \tau_{i}^*(h) \leq \tilde L_{i}^{\sharp} \right ) & \leq \mathbb{P} \left ( \tau \left ( V-\underline{V}, h \right ) < \tau \left (V, ]-\infty, -e^{(1-\delta) \kappa h}] \right ) \right ) \nonumber \\ 
& \leq \left ( 1 + 2 e^{(1-\delta)\kappa h}/\delta h \right ) e^{-\kappa (1-\delta /2) h}. \label{tpscoinc3}
\end{align}
For $h$ large enough the latter is less than $e^{- \delta \kappa h /3}$, so putting together \eqref{tpscoinc1}, \eqref{tpscoinc2} and \eqref{tpscoinc3} we get the result when $i \geq 2$. }If $i = 1$, \eqref{tpscoinc1} and \eqref{tpscoinc3} are true but recall that, by definition, $\tilde \tau_{0}(h) = \tilde L_0 =0$, so $\{ \tilde \tau_{1}^*(h) \leq \tilde L_{0} \} = \emptyset$. The result is therefore also true for $i = 1$. 

\end{proof}

\subsection{Exponential functionals of the bottom of a standard valley}

Let $J(h)$ be the exponential functional of the bottom of the first standard valley: 
\[ J(h) := \int_{\tilde \tau_1^-(h/2)}^{\tilde \tau_1^+(h/2)} e^{- \tilde V^{(1)}(u)} du, \]
{where $\tilde \tau_1^-(.)$ and $\tilde \tau_1^+(.)$ have the same meaning as in the beginning of Subsection \ref{coin}. 

This functional will appear in Section \ref{genedesres} for the study of the behavior of the diffusion near the bottom of the valleys. Indeed, we will see that the amount of time spent in the bottom of a valley is approximately the local time at the bottom of the valley multiplied by an independent factor having the same law as this functional. It is thus crucial to determine the behavior of the distribution of $J(h)$ when $h$ is large. Recall from the Introduction (just after Theorem \ref{kappa>1favsite}) that $\mathcal{R}$ is defined to be a random variable whose law is the convolution of the law of $I(V^{\uparrow})$ and of the law of $I(\hat V^{\uparrow})$. The results of Subsections \ref{loidesval} and \ref{coin} allow to think that $J(h)$ becomes close to $\mathcal{R}$ when $h$ is large enough. We now prove a very tight convergence of $J(h)$ to $\mathcal{R}$ using Propositions \ref{Fact_Williams}, \ref{standardwilliams} and the existence, proved in Theorems 1.1 and 1.13 of \cite{foncexpovech}, of some finite exponential moments for $I(V^{\uparrow})$ and $I(\hat V^{\uparrow})$. This result is crucial in Section \ref{genedesres} where we make $\mathcal{R}$ appear in the limit distribution of the supremum of the local time. }
%, and its proof is done by using a natural almost sure representation of the convergence. 

%ON DIR QU'ON POURRAIT L'AVOIR POUR LES VALLEES NORMALES MAIS QUE C'EST PLUS SIMPLE POUR LES VALLEES TILDES ET QUE C'EST SOUS CETTE FORME QU'ON EN A BESOIN. 

%\begin{prop} \label{cvr}
%
%There exist a probability space on which we can define a family of random variables $( \tilde{R}_1^t )_{t > 0}$ and a random variable $\tilde{\mathcal{R}}$ such that
%\[ \forall t > 0, \ \tilde{R}_1^t \overset{\mathcal{L}}{=} R_1^t, \ \ \tilde{\mathcal{R}} \overset{\mathcal{L}}{=} \mathcal{R}, \]
%and
%\[ \mathbb{P} \left ( \left | \tilde{R}_1^t - \tilde{\mathcal{R}} \right | > e^{-h_t/8} \right ) \leq ... \]
%Moreover the family $( \tilde{R}_1^t )_{t > 0}$ almost surely increases to $\tilde{\mathcal{R}}$. 
%
%\end{prop}

\begin{prop} \label{cvr}

%Assume $V$ has unbounded variation. There exists a probability space on which we can define a family of random variables $( \tilde J(h) )_{h > 0}$ and a random variable $\tilde{\mathcal{R}}$ such that
%\[ \forall h > 0, \ \tilde J(h) \overset{\mathcal{L}}{=} J(h), \ \ \tilde{\mathcal{R}} \overset{\mathcal{L}}{=} \mathcal{R}, \]
%and the family $( \tilde J(h) )_{h > 0}$ converges almost surely and in all $L^p$ spaces to $\tilde{\mathcal{R}}$. As a consequence, the moments of any positive order of $J(h)$ converge to those of $\mathcal{R}$ when $h$ goes to infinity. 
Assume that $V$ has unbounded variation. The family of random variables $( J(h) )_{h > 0}$ converges in distribution to $\mathcal{R}$ as $h$ goes to infinity, moreover there exists a positive $\lambda_0$ such that
\begin{eqnarray}
\forall \lambda < \lambda_0, \ \ \ \mathbb{E} \left [ e^{\lambda J(h)} \right ] \underset{h \rightarrow +\infty}{\longrightarrow}  \mathbb{E} \left [ e^{\lambda \mathcal{R}} \right ], \label{cvrlapl}
\end{eqnarray}
and the above quantities are all finite. As a consequence, the moments of any positive order of $J(h)$ converge to those of $\mathcal{R}$ when $h$ goes to infinity. 

\end{prop}

\begin{proof}

We consider a probability space on which are defined two independent processes $Z_1$ and $Z_2$ with $Z_1 \overset{\mathcal{L}}{=} \hat V^{\uparrow}$ and $Z_2 \overset{\mathcal{L}}{=} V^{\uparrow}$. We define 
\begin{align*}
\tilde{I}(h) & := \int_0^{\tau (Z_1, h/2 +)} e^{-Z_1(x)}dx + \int_0^{\tau (Z_2, h/2)} e^{- Z_2(x)}dx, \ \forall h > 0 \\
\tilde{\mathcal{R}} & := \int_0^{+\infty} e^{-Z_1(x)}dx + \int_0^{+\infty} e^{- Z_2(x)}dx. 
\end{align*}
We have trivially the equality in law $\tilde{\mathcal{R}} \overset{\mathcal{L}}{=} \mathcal{R}$ and the almost sure increase of the $\tilde{I}(h)$ to $\tilde{\mathcal{R}}$. Then, from the definitions of $\tilde P_1^{(1)}$ and $\tilde P_2^{(1)}$ in Subsection \ref{coin}, 
%
%Thanks to propositions \ref{vposlapl} and \ref{vneglapl}, $\tilde{\mathcal{R}}$ admits a Laplace transform on a neighborhood of $0$ so it belongs to all $L^p$ spaces. Then, for every $t > 0$, $\tilde{I}(h)$ is positive and less than $\tilde{\mathcal{R}}$. Therefore, the family $( \tilde{I}(h) )_{t > 0}$ is bounded in all $L^p$ spaces, and we deduce that the almost sure convergence of $( \tilde{I}(h) )_{t > 0}$ to $\tilde{\mathcal{R}}$ holds in all $L^p$ spaces. 
\begin{eqnarray}
J(h) = \int_{0}^{\tau (\tilde P_1^{(1)}, h/2 +)} e^{- \tilde P_1^{(1)}(u)} du + \int_{0}^{\tau(\tilde P_2^{(1)}, h/2)} e^{- \tilde P_2^{(1)}(u)} du. \label{exprjhpi}
\end{eqnarray}
%\[ J(h) = \int_{0}^{\tilde m_1 - \tilde \tau_1^-(h/2)} e^{- \tilde V^{(1)}(\tilde m_1 - u)} du + \int_{0}^{\tilde \tau_1(h/2) - \tilde m_1} e^{- \tilde V^{(1)}(\tilde m_1 + u)} du. \]
According to Proposition \ref{standardwilliams}, the two terms are independent and the second is equal in law to the second term of $\tilde{I}(h)$. It is easy to see that, if we have four real random variables $A$, $B$, $C_A$ and $C_B$ where $A$ and $C_A$ (respectively $B$ and $C_B$) are defined on the same probability space and independent, and such that $C_A$ and $C_B$ have the same law, then we have this inequality for the total variation distance: 
\[ d_{VT}(C_A + A, C_B + B) \leq d_{VT}(A, B). \]

In our case, we thus have 
%\begin{eqnarray}
%d_{VT} \left (J(h), \tilde{I}(h) \right ) \leq d_{VT} \left (\int_{0}^{\tilde m_1 - \tilde \tau_1^-(h/2)} e^{- \tilde V^{(1)}(\tilde m_1 - u)} du, \int_0^{\tau (Z_1, h/2 +)} e^{- Z_1(x)}dx \right ). \label{tvtef1}
%\end{eqnarray}
\begin{eqnarray}
d_{VT} \left (J(h), \tilde{I}(h) \right ) \leq d_{VT} \left (\int_{0}^{\tau (\tilde P_1^{(1)}, h/2 +)} e^{- \tilde P_1^{(1)}(u)} du, \int_0^{\tau (Z_1, h/2 +)} e^{- Z_1(x)}dx \right ). \label{tvtef1}
\end{eqnarray}
Then, we see that if $A$ and $B$ are random variables (valued on a metric space $E$), and $f$ is a mesurable mapping from $E$ to $\mathbb{R}$, we have
\[ d_{VT} \left (f(A), f(B) \right ) \leq d_{VT} \left (A, B \right ). \]

In our case, this yields
%\begin{align}
%& d_{VT} \left (\int_0^{\tau(\tilde P_1^{(1)}, h/2)} e^{- \tilde P_1^{(1)}(u)} du, \int_0^{\tau (Z_1, h/2 +)} e^{- Z_1(x)}dx \right ) \leq d_{VT} \left ( \tilde P_1^{(1)}, \left ( Z_1(x) \right )_{0 \leq x \leq \tau(Z_1, h)} \right ) \nonumber \\
%\leq & d_{VT} \left ( \tilde P_1^{(1)}, \left ( Z_1(x) \right )_{0 \leq x \leq \tau(Z_1, h)} \right ). \label{tvtef2}
%\end{align}
\begin{align}
d_{VT} \left (\int_0^{\tau(\tilde P_1^{(1)}, h/2+)} e^{- \tilde P_1^{(1)}(u)} du, \int_0^{\tau (Z_1, h/2 +)} e^{- Z_1(x)}dx \right ) \leq d_{VT} \left ( \tilde P_1^{(1)}, \left ( Z_1(x) \right )_{0 \leq x \leq \tau(Z_1, h+)} \right ). \label{tvtef2}
\end{align}
From the triangular inequality and Proposition \ref{standardwilliams}, we have for $h$ large enough, 
\begin{align}
d_{VT} \left ( \tilde P_1^{(1)}, \left ( Z_1(x) \right )_{0 \leq x \leq \tau(Z_1, h+)} \right ) & \leq d_{VT} \left ( \tilde P_1^{(1)}, P_1^{(2)} \right ) + d_{VT} \left ( P_1^{(2)}, \left ( Z_1(x) \right )_{0 \leq x \leq \tau(Z_1, h+)} \right ) \nonumber \\
& \leq 2 e^{- \delta \kappa h /3} + d_{VT} \left ( P_1^{(2)}, \left ( Z_1(x) \right )_{0 \leq x \leq \tau(Z_1, h+)} \right ). \label{tvtef2.1}
\end{align}

Now, according to Proposition \ref{Fact_Williams}, $P_1^{(2)}$ is absolutely continuous with respect to the law of the process $(\hat{V}^{\uparrow}(s))_{0 \leq s \leq \tau(\hat{V}^{\uparrow}, h+)} \overset{\mathcal{L}}{=} ( Z_1(x) )_{0 \leq x \leq \tau(Z_1, h+)}$ and has density $c_h/(1-e^{-\kappa \hat{V}^{\uparrow}(\tau(\hat{V}^{\uparrow}, h+))})$ with respect to this law. It is known that, if a random variable $B$ has a density $d_B$ with respect to a random variable $A$ (both valued in the same metric space), then, their total variation distance is expressed as follows :
\[ d_{VT} \left (A, B \right ) = \frac1{2} \int \left | 1 - d_B(x) \right | \times \mathbb{P} \left ( A \in dx \right ) . \]
That is, in our case, 
\begin{eqnarray}
d_{VT} \left ( P_1^{(2)}, \left ( Z_1(x) \right )_{0 \leq x \leq \tau(Z_1, h+)} \right ) = \frac1{2} \mathbb{E} \left [ \left | 1 - \frac{c_h}{1-e^{-\kappa \hat{V}^{\uparrow}(\tau(\hat{V}^{\uparrow}, h+))}} \right | \right ]. \label{tvtef3}
\end{eqnarray}

Now combining \eqref{tvtef1}, \eqref{tvtef2}, \eqref{tvtef2.1} and \eqref{tvtef3} we get for large $h$, 
\[ d_{VT} \left (J(h), \tilde{I}(h) \right ) \leq 2 e^{- \delta \kappa h /3} + \frac1{2} \mathbb{E} \left [ \left | 1 - \frac{c_h}{1-e^{-\kappa \hat{V}^{\uparrow}(\tau(\hat{V}^{\uparrow}, h+))}} \right | \right ]. \]
Because of the last assertion in Proposition \ref{Fact_Williams} and dominated convergence, the right hand side converges to $0$ when $h$ goes to infinity. 
From this and the convergence of $( \tilde{I}(h) )_{h > 0}$ to $\tilde{\mathcal{R}}$ we deduce that $(J(h))_{h>0}$ converges in distribution to $\tilde{\mathcal{R}}$ (and therefore to $\mathcal{R}$). 
%, so by the almost sure representation theorem, we can define a family of random variables $( \tilde J(h) )_{h > 0}$ such that 
%\[ \forall h > 0, \ \tilde J(h) \overset{\mathcal{L}}{=} J(h) \ \text{  and  } \ \mathbb{P} \left ( \left | \tilde J(h) - \tilde{\mathcal{R}} \right | \underset{h \rightarrow +\infty}{\longrightarrow} 0 \right ) = 1. \]
In order to justify \eqref{cvrlapl}, it only remain to prove a uniform integrability condition. In particular, \eqref{cvrlapl} will follow if we prove the existence of a positive $\lambda_0$ and a positive finite constant $C$ such that
\begin{eqnarray}
\forall h > 0, \ \ \ \mathbb{E} \left [ e^{\lambda_0 J(h)} \right ] \leq C. \label{tvtef4.00}
\end{eqnarray}
Thanks to Theorems 1.1 and 1.13 of \cite{foncexpovech}, the positive random variables $I(V^{\uparrow})$ and $I(\hat V^{\uparrow})$ admit some finite exponential moments. We can therefore choose a positive $\lambda_0$ such that 
\begin{eqnarray}
\mathbb{E} \left [ e^{\lambda_0 I(V^{\uparrow})} \right ] < +\infty \ \ \ \text{and} \ \ \ \mathbb{E} \left [ e^{2 \lambda_0 I(\hat V^{\uparrow})} \right ] < +\infty. \label{tvtef4.0}
\end{eqnarray}
{Let us fix such a positive $\lambda_0$ and $h > 0$. Recall \eqref{exprjhpi} and the fact that, according to Proposition \ref{standardwilliams}, $\tilde P_1^{(1)}$ and $\tilde P_2^{(1)}$ are independent. As a consequence }
\begin{align}
\mathbb{E} \left [ e^{\lambda_0 J(h)} \right ] & = \mathbb{E} \left [ \exp \left ( \lambda_0 \int_0^{\tau(\tilde P_1^{(1)}, h/2+)} e^{-\tilde P_1^{(1)}(u)} du \right ) \right ] \times \mathbb{E} \left [ \exp \left ( \lambda_0  \int_0^{\tau(\tilde P_2^{(1)}, h/2)} e^{-\tilde P_2^{(1)}(u)} du \right ) \right ]. \label{tvtef4}
\end{align}
Then, according to Proposition \ref{standardwilliams}, 
\begin{align}
\mathbb{E} \left [ \exp \left ( \lambda_0  \int_0^{\tau(\tilde P_2^{(1)}, h/2)} e^{-\tilde P_2^{(1)}(u)} du \right ) \right ] = \mathbb{E} \left [ \exp \left ( \lambda_0 \int_0^{\tau(V^{\uparrow}, h/2)} e^{-V^{\uparrow}(u)} du \right ) \right ] \leq \mathbb{E} \left [ e^{\lambda_0 I(V^{\uparrow})} \right ]. \label{tvtef5}
\end{align}

From Remark \ref{iid} and the fact that $(\tilde m_i)_{i \geq 1}$ is a subsequence of $(m_i)_{i \geq 1}$ we have, 
\begin{align*}
& \mathbb{E} \left [ \exp \left ( \lambda_0 \int_0^{\tau(\tilde P_1^{(1)}, h/2+)} e^{-\tilde P_1^{(1)}(u)} du \right ) \right ] = \mathbb{E} \left [ \exp \left ( \lambda_0 \int_0^{\tau(\tilde P_1^{(2)}, h/2+)} e^{-\tilde P_1^{(2)}(u)} du \right ) \right ] \\
 = & \mathbb{E} \left [ \exp \left ( \lambda_0 \int_0^{\tau(P_1^{(2)}, h/2+)} e^{-P_1^{(2)}(u)} du \right ) \mathds{1}_{\tilde P_1^{(2)} = P_1^{(2)}} \right ] + \sum_{i > 2} \mathbb{E} \left [ \exp \left ( \lambda_0 \int_0^{\tau(P_1^{(i)}, h/2+)} e^{-P_1^{(i)}(u)} du \right ) \mathds{1}_{\tilde P_1^{(2)} = P_1^{(i)}} \right ] \\
 \leq & \mathbb{E} \left [ \exp \left ( \lambda_0 \int_0^{\tau(P_1^{(2)}, h/2+)} e^{-P_1^{(2)}(u)} du \right ) \right ] + \sum_{i > 2} \sqrt{\mathbb{E} \left [ \exp \left ( 2 \lambda_0 \int_0^{\tau(P_1^{(i)}, h/2+)} e^{-P_1^{(i)}(u)} du \right ) \right ]} \sqrt{\mathbb{P} \left ( \tilde P_1^{(2)} = P_1^{(i)} \right )}. 
\end{align*}
Now, according to Proposition \ref{Fact_Williams}, {the law of $P_1^{(i)}$ is, for any $i \geq 2$, absolutely continuous with respect to the law of the process $(\hat{V}^{\uparrow}(x))_{0 \leq x \leq \tau(\hat{V}^{\uparrow}, h+)}$ and the density is bounded by $2$ when $h$ is large enough. As a consequence the two expectations in the above expression} are (for large $h$) less than respectively $2 \mathbb{E} [ e^{\lambda_0 I(\hat V^{\uparrow})} ]$ and  $2 \mathbb{E} [ e^{2 \lambda_0 I(\hat V^{\uparrow})} ]$. Using the arguments of the proof of Lemma \ref{minimacoincide} we can prove that for $h$ large enough $\mathbb{P} ( \tilde P_1^{(2)} = P_1^{(i)} ) \leq e^{- \delta \kappa (i - 2) h /3}$. This proves that
\begin{eqnarray}
\mathbb{E} \left [ \exp \left ( \lambda_0 \int_0^{\tau(\tilde P_1^{(1)}, h/2+)} e^{-\tilde P_1^{(1)}(u)} du \right ) \right ] \leq 3 \sqrt{\mathbb{E} \left [ e^{2 \lambda_0 I(\hat V^{\uparrow})} \right ]}, \label{tvtef6}
\end{eqnarray}
%\[ = 2^p \left ( \mathbb{E} \left [ \left | \int_0^{\tau(\hat V^{\uparrow}, h/2)} e^{- \hat V^{\uparrow}(u)} du \right |^p \frac{c_h}{1-e^{-\kappa \hat{V}^{\uparrow}(\tau(\hat{V}^{\uparrow}, h+))}} \right ] + \mathbb{E} \left [ \left | \int_0^{\tau(V^{\uparrow}, h/2)} e^{-V^{\uparrow}(u)} du \right |^p \right ] \right ), \]
%from Proposition \ref{standardwilliams}, MODIFIER
%\[ \leq 2^p \left ( 2 \mathbb{E} \left [ \left | I(\hat V^{\uparrow}) \right |^p \right ] + \mathbb{E} \left [ \left | I(V^{\uparrow}) \right |^p \right ] \right ), \]
%\begin{align*}
%& \leq 2^p \left ( \mathbb{E} \left [ \left | I(\hat V^{\uparrow}) \right |^p \frac{c_h}{1-e^{-\kappa \hat{V}^{\uparrow}(\tau(\hat{V}^{\uparrow}, h+))}} \right ] + \mathbb{E} \left [ \left | I(V^{\uparrow}) \right |^p \right ] \right ) \\
%& \leq 2^p \left ( 2 \mathbb{E} \left [ \left | I(\hat V^{\uparrow}) \right |^p \right ] + \mathbb{E} \left [ \left | I(V^{\uparrow}) \right |^p \right ] \right ), 
%\end{align*}
for $h$ large enough. Now, combining \eqref{tvtef4}, \eqref{tvtef5}, \eqref{tvtef6} and \eqref{tvtef4.0}, we get \eqref{tvtef4.00} so \eqref{cvrlapl} is proved as well. 
%that the family $( \tilde J(h) )_{h > 0}$ is bounded in all $L^p$ spaces, and we deduce that the convergence of this family to $\tilde{\mathcal{R}}$ holds in all $L^p$ spaces, which is the result. 

\end{proof}

\subsection{Asymptotic of the sequence of $h$-minima} \label{firstmin}

In this subsection, we are interested in the asymptotic distance between the $h$-minima when $h$ goes to infinity. This leads to estimates that are useful to study the local time of the diffusion in $V$ outside the bottoms of the valleys, and Theorem \ref{seqmincvverspoiss} is also proved in the end of this subsection. 
%\begin{align*}
%\tau_1^*(h) & := \inf \left \{ u \geq 0, \ V(u)- \inf_{[0, u]} V \geq h \right \}, \\
%m_1^*(h) & := \inf \left \{ u \geq 0, \ V(u) = \inf_{[0, \tau_1^*(h)]} V \right \}. 
%\end{align*}
%ON LAISSE TOMBER REGULIER ET RECCURENT ET L'EXPR DU POISSON
%We define $\mathcal{F}$ the space of c\`ad-l\`ag functions from $[0, +\infty[$ to $\mathbb{R}$, starting at zero and killed at the first positive instant when they reach $0$. Note that this instant can possibly be infinite. Since $V$ drifts to $-\infty$, $\left \{ 0 \right \}$ is recurrent for $V - \underline{V}$ and with the help of Fact \ref{vnb}, we see that it is also instantaneous, Proposition 10 of chapter IV of \cite{Bertoin} applies, providing that the excursions of this process above $0$ form a Poisson point process on $\mathcal{F}$. That is, if $L$ is the continuous local time at $0$ of $V - \underline{V}$, which, thanks to chapter IV of \cite{Bertoin}, is well defined, and $L^{-1}$ is it's right-continuous inverse, then the excursion process of $V - \underline{V}$ defined by 
%\begin{align*}
%e_t(s) := & (V - \underline{V}) (L^{-1}(t-) + s) & \text{if } 0 \leq s \leq L^{-1}(t) - L^{-1}(t-) \\
%& 0 & \text{if } s \geq L^{-1}(t) - L^{-1}(t-)
%\end{align*}
First, we define the first ascend of $h$ for $V - \underline{V}$: 
\[ \tau^*(h) := \inf \left \{ u \geq 0, \ (V- \underline{V})(u) = h \right \}, \ m^*(h) := \inf \left \{ u \geq 0, \ V(u) = \underline{V}(\tau^*(h)) \right \}. \]
{Note that, since $\tilde \tau_{0}(h) = 0$, $\tau^*(h)$ and $m^*(h)$ coincide almost surely with respectively $\tilde \tau_{1}^*(h)$ and $\tilde m_{1}^*$, defined in \eqref{ascend} and \eqref{min}. }We study $m^*(h)$ by the mean of excursion theory. Let $\mathcal{F}$ denote the space of excursions, that is,  c\`ad-l\`ag functions from $[0, +\infty[$ to $\mathbb{R}$, starting at zero and killed at the first positive instant when they reach $0$. Note that this instant can possibly be infinite. For $\xi \in \mathcal{F}$, recall the notation $\zeta(\xi) := \inf \{ s > 0, \ \xi(s) = 0 \}$ for the length of the excursion $\xi$. Also, let $\mathcal{F}_{h, -}$ and $\mathcal{F}_{h, +}$ denote respectively the set of excursions whose height is strictly less than $h$ and the set of excursions higher than $h$: 
\[ \mathcal{F}_{h, -} := \left \{ \xi \in \mathcal{F}, \ \sup_{[0, \zeta]} \xi < h \right \}, \ \ \ \ \ \mathcal{F}_{h, +} := \left \{ \xi \in \mathcal{F}, \ \sup_{[0, \zeta]} \xi \geq h \right \}. \]

Recall from Subsection \ref{factsandnotations} that $V - \underline{V}$ is a c\`ad-l\`ag Markov process. With the help of Fact \ref{vnb}, we see that $\{0\}$ is instantaneous for $V - \underline{V}$ (and it is regular if and only if $V$ has unbounded variation). Excursion theory above $0$ can thus be applied to $V - \underline{V}$ ({see Section IV of \cite{Bertoin}}). Let $L$ be a local time at $0$ for $V - \underline{V}$, $\mathcal{N}$ the associated excursion measure, and $L^{-1}$ the right continuous inverse of $L$. Then, the excursions, indexed by $L$, above $0$ of $V - \underline{V}$ form a Poisson point process on $\mathcal{F}$ with intensity measure $\mathcal{N}$. In the irregular case (when $V$ has bounded variation) the local time $L$ has to be defined artificially as in \cite{Bertoin}, Section IV.5. In this case, the excursion measure is proportional to the law of the first excursion and in particular the total mass of the excursion measure is finite. 

Let us define $S^{h, -}$ and $S^{h, +}$ to be two independent pure jump subordinators with L\'evy measure respectively $\zeta \mathcal{N}(\mathcal{F}_{h, -} \cap .)$ and $\zeta \mathcal{N}(\mathcal{F}_{h, +} \cap .)$, the image measures of respectively $\mathcal{N}(\mathcal{F}_{h, -} \cap .)$ and $\mathcal{N}(\mathcal{F}_{h, +} \cap .)$ by $\zeta$. Since $\zeta \mathcal{N}(\mathcal{F}_{h, -} \cap .) + \zeta \mathcal{N}(\mathcal{F}_{h, +} \cap .) = \zeta \mathcal{N}$, the L\'evy-Khintchine formula yields that $S := S^{h, -} + S^{h, +}$ is a pure jump subordinator with L\'evy measure $\zeta \mathcal{N}$, it is therefore equal in law to $L^{-1}$. We also define $T_h$ to be an exponential random variable with parameter $\mathcal{N}(\mathcal{F}_{h, +})$, independent from $S^{h, -}$, $S^{h, +}$ and $S$. We can now express $m^*(h)$ in term of these objects. 

\begin{lemme} \label{asympminsubexp}

\[ m^*(h) \overset{\mathcal{L}}{=} S^{h, -}(T_h), \ \ \ \text{and} \ \ \ \tau^*(h) - m^*(h) \overset{\mathcal{L}}{=} \tau(V^{\uparrow}, h). \]

\end{lemme}

\begin{proof} 
%of Lemma \ref{asympminsubexp}

Considering $(e_t)_{t \geq 0}$, the process of excursions indexed by $L$ of $V - \underline{V}$, we have that $L(\tau^*(h))$ is the instant when occurs the first jumps belonging to $\mathcal{F}_{h, +}$, and this jump corresponds to the excursions having $m^*(h)$ as starting point. We can thus write
\begin{eqnarray}
m^*(h) = \sum_{0 \leq t \leq L(\tau^*(h))} \zeta(e_t) \mathds{1}_{e_t \in \mathcal{F}_{h, -}}. \label{asympminsubexp1}
\end{eqnarray}
{The restriction $(e_t \mathds{1}_{e_t \in \mathcal{F}_{h, -}})_{t \geq 0}$ of the Poisson point process $(e_t)_{t \geq 0}$ to the subset $\mathcal{F}_{h, -}$ is a Poisson point process with intensity measure $\mathcal{N}(\mathcal{F}_{h, -} \cap .)$. Then, the process $(\zeta(e_t) \mathds{1}_{e_t \in \mathcal{F}_{h, -}})_{t \geq 0}$ of the images, by the measurable function $\zeta$, of the points of $(e_t \mathds{1}_{e_t \in \mathcal{F}_{h, -}})$ is a Poisson point process on $\mathbb{R}_+$ with intensity measure $\zeta \mathcal{N}(\mathcal{F}_{h, -} \cap .)$. As a consequence the process in the right hand side of \eqref{asympminsubexp1}, $\sum_{0 \leq t \leq .} \zeta(e_t) \mathds{1}_{e_t \in \mathcal{F}_{h, -}}$, is the sum of the jumps of a Poisson point process on $\mathbb{R}_+$ with intensity measure $\zeta \mathcal{N}(\mathcal{F}_{h, -} \cap .)$. From the L\'evy-Ito representation, it has the same law as the subordinator $S^{h, -}$. Also, since $L(\tau^*(h))$ is the instant of the first jump of $(e_t)_{t \geq 0}$ in $\mathcal{N}(\mathcal{F}_{h, +})$, it follows an exponential distribution with parameter $\mathcal{N}(\mathcal{F}_{h, +})$ and is independent from the process $(e_t \mathds{1}_{e_t \in \mathcal{F}_{h, -}})_{t \geq 0}$. As a consequence $m^*(h)$ has the same law as $S^{h, -}$ taken at $T_h$ (an independent exponential time with parameter $\mathcal{N}(\mathcal{F}_{h, +})$). This yields the result for $m^*(h)$. }

Then, $(V(x + m^*(h)) - V(m^*(h)), \ 0 \leq x \leq \tau^*(h) - m^*(h))$ is, considered up to its hitting time of $h$, the first excursion higher than $h$ of $V - \underline{V}$. As mentioned in the proofs of Propositions \ref{Fact_Williams} and \ref{standardwilliams}, the latter is equal in law to $(V^{\uparrow}(x))_{0 \leq x \leq \tau(V^{\uparrow}, h)}$. The result about $\tau^*(h) - m^*(h)$ follows. 

\end{proof}

We are now left to study $S$ and $\mathcal{N}(\mathcal{F}_{h, +})$, the parameter of $T_h$. For $S$, we have the following lemma :

\begin{lemme} \label{asympminesp}

The random variable $L^{-1}(1)$ (or equivalently $S(1)$) admits some finite exponential moments. 

%\[ \mathbb{E} \left [ L^{-1}(1) \right ] < +\infty. \]

\end{lemme}

\begin{proof}

{According to Theorem VII.4(ii) of \cite{Bertoin} we have 
\begin{eqnarray}
\forall \alpha \geq 0, \ \mathbb{E} \left [ e^{-\alpha L^{-1}(1)} \right ] = \exp \left ( c \alpha / \Phi_{V}(\alpha) \right ), \label{laplaceinvtl}
\end{eqnarray}
where $\Phi_{V} (\alpha) := \inf \{ x \geq \kappa, \Psi_{V}(x) = \alpha \}$ and $c$ is a positive constant. 
%and $c = \Phi_{V}'(0) = 1/\Psi_{V}'(\kappa) \in ]0, +\infty[$. The expression of $c$ comes from the fact that the Laplace transform of $L^{-1}(1)$ at $0$ equals $1$. Indeed, $L^{-1}$ is not a killed subordinator since $V$ drifts to $-\infty$. 

$V(1)$ has its Laplace transform, as well as its Laplace exponent $\Psi_{V}$, defined on $[0, +\infty[$. Then, since $\Psi_{V}(\kappa) = 0$ and $\Psi'_{V}(\kappa) > 0$, the holomorphic local inversion theorem yields that 
%$\Psi_{V}^{-1}$, that is, 
$\Phi_{V}$ extends on a neighborhood of $0$. As a consequence, the right hand side of \eqref{laplaceinvtl} extends on a neighborhood of $0$. This proves that $L^{-1}(1)$ has a Laplace transform defined on a neighborhood of $0$ so the result is proved. }

\end{proof}

The next lemma deals with the asymptotic behavior of $\mathcal{N}(\mathcal{F}_{h, +})$. 

\begin{lemme} \label{asympminparam}

%\[ e^{\kappa h} \mathcal{N}(\mathcal{F}_{h, +}) \underset{h \rightarrow +\infty}{\longrightarrow} \mathcal{N}(\mathcal{F}_{1, +}) \times \left (e^{\kappa} - \mathbb{E} \left [ e^{\kappa V_1(\tau(V_1, ]-\infty, 0]))} \right ] \right ). \]
\[ \mathcal{N}(\mathcal{F}_{h, +}) = e^{-\kappa h} \mathcal{N}(\mathcal{F}_{1, +}) \times \left (e^{\kappa} - \mathbb{E} \left [ e^{\kappa V_1(\tau(V_1, ]-\infty, 0]))} \right ] \right ) + \mathcal{O} (e^{-2\kappa h}). \]

\end{lemme}

\begin{proof}
We fix $h > 1$. From the Markov property applied at the hitting time of $1$ in the excursions belonging to $\mathcal{F}_{1, +}$ we get
\begin{eqnarray}
\mathcal{N}(\mathcal{F}_{h, +}) / \mathcal{N}(\mathcal{F}_{1, +}) = \mathbb{P} \left ( \tau \left (V_1, h \right ) < \tau \left (V_1, ]-\infty, 0] \right ) \right ) =: p_h. \label{asympminparam1}
\end{eqnarray}

{Recall from Subsection \ref{factsandnotations} that $\sup_{[0, +\infty[} V$ follows an exponential distribution with parameter $\kappa$.} Then, 
\begin{align*}
e^{-\kappa(h-1)} & = \mathbb{P} \left ( \sup_{[0, +\infty[} V_1 > h \right ) \\
& = p_h + \mathbb{P} \left ( \tau(V_1, ]-\infty, 0]) < \tau(V_1, h), \ \sup_{[0, +\infty[} V_1(\tau(V_1, ]-\infty, 0]) + .) > h \right ) \\
& {= p_h + \mathbb{P} \left ( \tau(V_1, ]-\infty, 0]) < \tau(V_1, h), \ \sup_{[0, +\infty[} V_1^{\tau(V_1, ]-\infty, 0])} > h - V_1(\tau(V_1, ]-\infty, 0])) \right )} \\
& = p_h + \mathbb{E} \left [ e^{-\kappa {[h - V_1(\tau(V_1, ]-\infty, 0]))]}} \mathds{1}_{\tau(V_1, ]-\infty, 0]) < \tau(V_1, h)} \right ], 
\end{align*}
where we have used the Markov property at the stopping time $\tau(V_1, ]-\infty, 0])$. We get
\begin{align*}
p_h & = e^{-\kappa h} \left ( e^{\kappa} - \mathbb{E} \left [ e^{\kappa V_1(\tau(V_1, ]-\infty, 0]))} \mathds{1}_{\tau(V_1, ]-\infty, 0]) < \tau(V_1, h)} \right ] \right ) \\
& = e^{-\kappa h} \left ( e^{\kappa} - \mathbb{E} \left [ e^{\kappa V_1(\tau(V_1, ]-\infty, 0]))} \right ] + \mathbb{E} \left [ e^{\kappa V_1(\tau(V_1, ]-\infty, 0]))} \mathds{1}_{\tau(V_1, ]-\infty, 0]) > \tau(V_1, h)} \right ] \right ). 
\end{align*}
{Combining with \eqref{asympminparam1} we get 
\begin{align*}
\mathcal{N}(\mathcal{F}_{h, +}) & = e^{-\kappa h} \mathcal{N}(\mathcal{F}_{1, +}) \times \left (e^{\kappa} - \mathbb{E} \left [ e^{\kappa V_1(\tau(V_1, ]-\infty, 0]))} \right ] \right ) \\ 
& + e^{-\kappa h} \mathcal{N}(\mathcal{F}_{1, +}) \times \mathbb{E} \left [ e^{\kappa V_1(\tau(V_1, ]-\infty, 0]))} \mathds{1}_{\tau(V_1, ]-\infty, 0]) > \tau(V_1, h)} \right ] 
\end{align*}
and it only remains to bound the last term.} Since $e^{\kappa V_1(\tau(V_1, ]-\infty, 0]))}$ is almost surely less than $1$ we have
\[ 0 \leq \mathbb{E} \left [ e^{\kappa V_1(\tau(V_1, ]-\infty, 0]))} \mathds{1}_{\tau(V_1, ]-\infty, 0]) > \tau(V_1, h)} \right ] \leq \mathbb{P} \left ( \tau(V_1, h) < +\infty \right ) = \mathbb{P} \left ( \sup_{[0, +\infty[} V_1 > h \right ) = e^{-\kappa (h-1)}, \]
and the result follows. 
%By dominated convergence, the above expectation converges to $\mathbb{E} \left [ e^{\kappa V_1(\tau(V_1, ]-\infty, 0]))} \right ]$ so 
%\[ e^{\kappa h} p_h \underset{h \rightarrow +\infty}{\longrightarrow} e^{\kappa} - \mathbb{E} \left [ e^{\kappa V_1(\tau(V_1, ]-\infty, 0]))} \right ]. \]
%Combining with \eqref{asympminparam1} we get the result. 

\end{proof}

{
\begin{remarque} \label{autrecte}
Recall the spectrally negative L\'evy process $V^{\sharp}$ (known as $V$ conditioned to drift to $+\infty$) defined in Subsection \ref{extrema}. Let $(L^{\sharp, -1}, I^{\sharp})$ denote the descending ladder process of $V^{\sharp}$: $L^{\sharp}$ is a local time at $0$ of $V^{\sharp} - \underline{V^{\sharp}}$ and $L^{\sharp, -1}$ is the inverse of this local time. For any $t \geq 0$, $I^{\sharp}(t) := -\underline{V^{\sharp}}(L^{\sharp, -1}(t))$. We denote by $\mathcal{I}^{\sharp}$ the potential measure of $I^{\sharp}$, $\mathcal{I}^{\sharp}([0, x]) := \mathbb{E} [ \int_0^{L^{\sharp}(+\infty)} \mathds{1}_{I^{\sharp}(t) \in [0, x]} dt ]$ for $x \geq 0$. Then, the constant $e^{\kappa} - \mathbb{E} [ e^{\kappa V_1(\tau(V_1, ]-\infty, 0]))} ]$ in Lemma \ref{asympminparam} can be re-expressed as $e^{\kappa} \mathcal{I}^{\sharp}([0,1]) / \mathcal{I}^{\sharp}([0,+\infty[)$ or $e^{\kappa} \mathbb{P} (\inf_{[0, +\infty[} V_1^{\sharp} > 0 )$. Indeed, from the proof of the lemma, we have that 
%\[ e^{\kappa} - \mathbb{E} \left [ e^{\kappa V_1(\tau(V_1, ]-\infty, 0]))} \right ] = \lim_{h \rightarrow +\infty} e^{\kappa h} p_h = \lim_{h \rightarrow +\infty} e^{\kappa h} \mathbb{P} \left ( \tau \left (V_1, h \right ) < \tau \left (V_1, ]-\infty, 0] \right ) \right ). \]
\begin{eqnarray}
e^{\kappa} - \mathbb{E} \left [ e^{\kappa V_1(\tau(V_1, ]-\infty, 0]))} \right ] = \lim_{h \rightarrow +\infty} e^{\kappa h} p_h, \label{autrecte1}
\end{eqnarray}
where $p_h = \mathbb{P} ( \tau (V_1, h ) < \tau (V_1, ]-\infty, 0] ) )$. According to \eqref{scalefct} we have $p_h = W(1)/W(h)$, where $W$ is the scale function of $V$ defined in the beginning of Subsection \ref{extrema}. According to (9.4.6) of \cite{Doney} there is a positive constant $c$ such that for any $x > 0$ we have $W(x) = c e^{\kappa x} \mathcal{I}^{\sharp}([0,x])$. As a consequence 
\[ \lim_{h \rightarrow +\infty} e^{\kappa h} p_h = \lim_{h \rightarrow +\infty} e^{\kappa h} \frac{W(1)}{W(h)} = \lim_{h \rightarrow +\infty} e^{\kappa h} \frac{e^{\kappa} \mathcal{I}^{\sharp}([0,1])}{e^{\kappa h} \mathcal{I}^{\sharp}([0,h])} = e^{\kappa} \frac{\mathcal{I}^{\sharp}([0,1])}{\mathcal{I}^{\sharp}([0,+\infty[)}. \]
Putting together with \eqref{autrecte1} we get that $e^{\kappa} - \mathbb{E} [ e^{\kappa V_1(\tau(V_1, ]-\infty, 0]))} ] = e^{\kappa} \mathcal{I}^{\sharp}([0,1])/\mathcal{I}^{\sharp}([0,+\infty[)$, as we claimed. Also, according to (9.4.4) of \cite{Doney} we have that for any $x \geq 0$, $W(x) = e^{\kappa x} W^{\sharp}(x)$, where $W^{\sharp}$ is the scale function of $V^{\sharp}$. We thus get 
\[ \lim_{h \rightarrow +\infty} e^{\kappa h} p_h = \lim_{h \rightarrow +\infty} e^{\kappa h} \frac{e^{\kappa} W^{\sharp}(1)}{e^{\kappa h} W^{\sharp}(h)} = \lim_{h \rightarrow +\infty} e^{\kappa} \mathbb{P} \left ( \inf_{[0, \tau (V_1^{\sharp}, h )]} V_1^{\sharp} > 0 \right ) = e^{\kappa} \mathbb{P} \left (\inf_{[0, +\infty[} V_1^{\sharp} > 0 \right ). \]
Putting together with \eqref{autrecte1} we get that $e^{\kappa} - \mathbb{E} [ e^{\kappa V_1(\tau(V_1, ]-\infty, 0]))} ] = e^{\kappa} \mathbb{P} (\inf_{[0, +\infty[} V_1^{\sharp} > 0 )$, as we claimed. 

\end{remarque}

}

We can now get the asymptotic behavior of $m^*(h)$. 

\begin{prop} \label{asympmin}

\[ e^{-\kappa h} \ m^*(h) \overset{\mathcal{L}}{\underset{h \rightarrow +\infty}{\longrightarrow}} \mathcal{E}(q), \]
where $\mathcal{E}(q)$ is the exponential distribution with parameter 
\[ q := \mathcal{N}(\mathcal{F}_{1, +}) \times \left (e^{\kappa} - \mathbb{E} \left [ e^{\kappa V_1(\tau(V_1, ]-\infty, 0]))} \right ] \right ) \big / \int_0^{+\infty} \zeta(x) \mathcal{N}(dx). \]

\end{prop}

The parameter $q$ here is the one appearing in Theorem \ref{seqmincvverspoiss}. If $V = W_{\kappa}$, the $\kappa$-drifted Brownian motion, the calculations can be made more explicit. Using $(2.7)$ of \cite{Faggionato}, we can prove that the parameter $q$ in the above proposition (and therefore in Theorem \ref{seqmincvverspoiss}) equals $\kappa^2/2$. 

\begin{proof} of Proposition \ref{asympmin}

Thanks to Lemma \ref{asympminsubexp}, we are reduced to study $S^{h, -}(T_h)$. {Let us fix $M > 0$ arbitrary. For any $h \in [M, +\infty[$ and $t \geq 0$ we have $S^{M, -}(t) \leq S^{h, -}(t) \leq S(t)$, so in particular 
\[ \frac{S^{M, -}(T_h)}{T_h} \leq \frac{S^{h, -}(T_h)}{T_h} \leq \frac{S(T_h)}{T_h}. \]
According to Lemma \ref{asympminparam} the parameter of $T_h$, $\mathcal{N}(\mathcal{F}_{h, +})$, converges to $0$ when $h$ goes to infinity so $\mathbb{P} (T_h > t)$ converges to $1$ when $h$ goes to infinity, for any $t > 0$. According to Lemma \ref{asympminesp}, $S(1)$ and, therefore, $S^{M, -}(1)$ have finite expectation so we can use the law of large numbers for L\'evy processes (see for example Theorem 36.5 in \cite{Sato}). We deduce that $S(t)/t$ and $S^{M, -}(t)/t$ converge almost surely (and therefore in probability) to respectively $\mathbb{E} [S(1)]$ and $\mathbb{E} [S^{M, -}(1)]$. We deduce that 
\[ \mathbb{E} [S^{M, -}(1)] \overset{\mathbb{P}}{\underset{h \rightarrow +\infty}{\longleftarrow}} \frac{S^{M, -}(T_h)}{T_h} \leq \frac{S^{h, -}(T_h)}{T_h} \leq \frac{S(T_h)}{T_h} \overset{\mathbb{P}}{\underset{h \rightarrow +\infty}{\longrightarrow}} \mathbb{E} [S(1)]. \]
By monotone convergence $\mathbb{E} [S^{M, -}(1)]$ converges to $\mathbb{E} [S(1)]$ when $M$ goes to infinity. Combining with the above inequality we deduce that} 
\begin{eqnarray}
S^{h, -}(T_h) / T_h \overset{\mathbb{P}}{\underset{h \rightarrow +\infty}{\longrightarrow}} \mathbb{E} [S(1)]. \label{asympmin3}
\end{eqnarray}

Recall also that $\mathbb{E} [S(1)] = \mathbb{E} [L^{-1}(1)] = \int_0^{+\infty} x \ \zeta \mathcal{N}(dx)$. Then, 
\[ e^{-\kappa h} S^{h, -}(T_h) = (e^{-\kappa h}/\mathcal{N}(\mathcal{F}_{h, +})) \times (\mathcal{N}(\mathcal{F}_{h, +}) \times T_h) \times (S^{h, -}(T_h) / T_h). \]
Combining Lemma \ref{asympminparam}, the fact that $\mathcal{N}(\mathcal{F}_{h, +})$ is the parameter of $T_h$, \eqref{asympmin3}, and Slutsky's Lemma, we get that $e^{-\kappa h} S^{h, -}(T_h)$ converges to an exponential distribution with parameter 
\[ q := \mathcal{N} ( \mathcal{F}_{1, +}) \times \left (e^{\kappa} - \mathbb{E} \left [ e^{\kappa V_1(\tau(V_1, ]-\infty, 0]))} \right ] \right ) \big / \int_0^{+\infty} \zeta(x) \mathcal{N}(dx). \]
Thanks to Lemma \ref{asympminsubexp} we get the result. 

\end{proof}

We can now prove Theorem \ref{seqmincvverspoiss}. 

\begin{proof} of Theorem \ref{seqmincvverspoiss}

As we said in the proof of Proposition \ref{Fact_Williams}, Lemma 1 of \cite{Cheliotis2006715} applies here, so the random variables $(m_{i+1} - m_i)_{i \geq 1}$ are \textit{iid}, and $m_1$ is also independent from this sequence. To prove Theorem \ref{seqmincvverspoiss}, we thus only need to prove that the random variables $e^{-\kappa h} m_1$ and $e^{-\kappa h}(m_2 - m_1)$ converge in distribution to an exponential distribution with parameter $q$. Then, note that according to Lemma \ref{minimacoincide} applied with $n = 2$ we have $\mathbb{P}(m_1 = \tilde m_1, m_2 = \tilde m_2) \longrightarrow_{h \rightarrow +\infty} 1$, so we only need to prove the convergence of $e^{-\kappa h} \tilde m_1$ and $e^{-\kappa h}(\tilde m_2 - \tilde m_1)$. Now, according to Subsection \ref{coin}, 
\[ \tilde m_2 - \tilde m_1 = (\tilde \tau_1(h) - \tilde m_1) + (\tilde L_1 - \tilde \tau_1(h)) + (\tilde L_2^{\sharp} - \tilde L_1) + (\tilde m_2 - \tilde L_2^{\sharp}). \]
For the first term we use Proposition \ref{standardwilliams}, we get that it has the same law as $\tau(V^{\uparrow}, h)$. For the second and third terms we use the fact that, by definition, $\tilde L_1 = \tau(V^{\tilde \tau_1(h)}, ]-\infty, -h/2]) + \tilde \tau_1(h)$, $\tilde L_2^{\sharp} = \tau(V^{\tilde L_1}, ]-\infty, -e^{(1-\delta)\kappa h}]) + \tilde L_1$ and the Markov property at the stopping times $\tilde \tau_1(h)$, $\tilde L_1$, we get that these terms have the same law as respectively $\tau(V, ]-\infty, -h/2])$ and $\tau(V, ]-\infty, -e^{(1-\delta)\kappa h}])$. For the {fourth} term we use the definitions of $\tilde m_2$ and $m^*(h)$ and the Markov property at the stopping time $\tilde L_2^{\sharp}$, we get that this term has the same law as $m^*(h)$. 
%We get that the terms on the right hand side have respectively the same law as $\tau(V^{\uparrow}, h)$, $\tau(V, ]-\infty, -h/2])$, $\tau(V, ]-\infty, -e^{(1-\delta)\kappa h}])$, and $m^*(h)$. 
By Lemma \ref{vposlapltpsatt}, $e^{-\kappa h} \tau(V^{\uparrow}, h)$ converges to $0$ in probability when $h$ goes to infinity. By Lemma \ref{tpsatteinthatv}, $e^{-\kappa h} \tau(V, ]-\infty, -h/2])$ and $e^{-\kappa h} \tau(V, ]-\infty, -e^{(1-\delta)\kappa h}])$ converge to $0$ in probability when $h$ goes to infinity. As a consequence, the first, second and third terms, renormalized by $e^{-\kappa h}$, converge in probability to $0$ when $h$ goes to infinity. Proposition \ref{asympmin} gives the convergence in distribution of $e^{-\kappa h} \ m^*(h)$ and, therefore, of the last term renormalized by $e^{-\kappa h}$. Combining with Slutsky's Lemma we get that $e^{-\kappa h}(\tilde m_2 - \tilde m_1)$ converge in distribution to an exponential distribution with parameter $q$. 

For $\tilde m_1$ it is even simpler. {Recall from Subsection \ref{coin} that, by definition, $\tilde L_0 = 0$ so we have $\tilde m_1 = (\tilde L_1^{\sharp} - \tilde L_0) + (\tilde m_1 - \tilde L_1^{\sharp})$}, and we can conclude the same way. 

\end{proof}

\section{Supremum of the local time when $0 < \kappa < 1$} \label{genedesres}

%ATTENTION A $P$ ET $\mathbb{P}$, $h_t$, $N_t$, $n_t$, $A(r)$, $H(r)$. 
%Using the estimates of Section \ref{estimates}, 
We now generalize, in the context of the diffusion in $V$, the arguments of \cite{advech} to prove the convergence of the supremum of the local time when $0 < \kappa < 1$. 
%The biggest part of the work is actully to prove the required estimates in Section \ref{estimates}, so we refer to \cite{AndDev} and \cite{advech} for the proofs of some facts in the present section, and only precise what are the differences between their proofs and ours. 
First, let us recall some definitions from \cite{AndDev} and \cite{advech}. 
%for the study of the diffusion across the valleys. 

The diffusion moves across valleys. {When we look at the trajectory of the diffusion until a fixed time $t$, the height $h_t$ of the valleys that we consider has to be suitably chosen to ensure that with a large probability: 1) the diffusion has spent most of its time in the bottom of the $h_t$-valleys, 2) the important peaks of local time have been made only in the bottoms of the $h_t$-valleys, 3) the diffusion has never gone back to an $h_t$-valley after leaving it. As a consequence $h_t$ cannot be too small nor too large with respect to the time scale $t$, and we have to make $h_t$ grow with time $t$. We define }
\begin{align}
h_t := \log(t) - \phi(t), \ \text{where $\phi$ is any fixed function that satisfies} \ \log ( \log (t) ) << \phi(t) << \log(t). \label{defhtetphi}
\end{align}
{It will be proved by Facts \ref{lemtps}, \ref{analog3.3} of Subsection \ref{proofmainth} and Lemmas \ref{noreturn}, \ref{sortparladroite} (applied with $h=h_t$) of Subsection \ref{neghalfline} that our constraints are satisfied with such a definition of $h_t$.} We also define $N_t$, the index of the largest $h_t$-minima visited by $X$ until time $t$, 
\begin{eqnarray}
N_t:=\max \left\{ k\in\mathbb{N},\ \sup_{ 0 \leq s \leq t}X(s) \geq m_{k} \right \}. \label{defnt}
\end{eqnarray}

{The random number $N_t$ is the number of valleys visited until time $t$. Instead of working with the $N_t$ first $h_t$-valleys, it is more convenient to rather prove properties of a large deterministic number, $n_t$, of $h_t$-valleys ($n_t$ being chosen in such a way that it is greater than $N_t$ with a large probability). We define $n_t := \lfloor e^{\kappa (1+\delta) \phi(t)} \rfloor$. In the remainder, most of the estimates that we prove are true with a large probability and simultaneously for the $n_t$ first $h_t$-valleys. It will be proved by Lemma \ref{nbvalleesvisit} of Subsection \ref{reprise} that $N_t \leq n_t$ with a large probability. 
%prove that our estimates on the $h_t$-valleys are true simultaneously for 
%ou delta est dÈfini plus haut, mais c'est n'est pas oblig\'e...
%POUR LE LEMME \ref{lemtps} (BALANCER $0<\delta<2^{-3/2}$). C'EST SANS DOUTE A CAUSE DE LA DEF DU $h_h^+$ DANS ON PEUT LE VIRER. CONTRAINTES DANS TOUS LES LEMMES. POUR LE TL HORS DE $D_j$ IL SUPPOSE $\delta < 1/8$, CE N'EST PAS LIE AU $h_t^+$ NI AU $n_t$ DONC CA NE SERT A RIEN

%In fact, this negligibility that we seek is really the delicate point to prove in our study and it completely fails to hold when $\kappa \geq 1$ (even in the drifted Brownian case), which obliges us to adopt a completely different strategy then. 

%ASSEZ SIMILAIRE A A-D ET A-D-V

In all this section we assume that $\delta$ (defined in Subsection \ref{coin}) has been chosen small enough so that $(1+3\delta)\kappa<1$, and we assume that the hypothesis of Theorem \ref{cvdutl} are satisfied: $V$ has unbounded variation and there exists $p>1$ such that $V (1) \in L^p$. 
%Recall that this implies the Cramer's condition \eqref{cramercond}. As a consequence, 
As a consequence, all the results of Sections \ref{etudevallees} and \ref{estimates} apply here. }

\subsection{Proof of Theorem \ref{cvdutl}} \label{proofmainth}

{We now outline the proof of Theorem \ref{cvdutl} and show how it can be reduced to the proof of two Propositions. As in \cite{advech} the idea is first to make appear, for any fixed $t$, three independent sequences of \textit{iid} positive random variables $( e_i )_{i \geq 1}$, $( S_i^t )_{i \geq 1}$ and $(R_i^t)_{i \geq 1}$, so that $e_i S_i^t$ represents the peak of local time in the $i^{th}$ valley and $e_i S_i^t R_i^t$ represents the escaping time from the $i^{th}$ valley. The precise definitions of $e_i$, $S_i^t$ and $R_i^t$ will be given in Subsection \ref{maincontibcommeiid}. A key point in the proof of Theorem \ref{cvdutl} is to approximate $\mathcal{L}_X^*(t) /t$ by a functional of the sequence $( e_i S_i^t,  e_i S_i^t R_i^t)_{i \geq 1}$. First, for any $a \geq 0$ let us define
\[ \mathcal{N}_a := \min \left \{ j \geq 1, \ \sum_{i=1}^{j} e_i S_i^t R_i^t > a \right \}. \]
} 
Our approximation of $\mathcal{L}_X^*(t) /t$ is similar to Proposition 5.1 of \cite{advech} and can be stated as follows:

%We finally arrive to the main Proposition of this subsection which, just as Proposition 5.1 of \cite{advech}, approximates the supremum of the local time by a functional of $(Y_1, Y_2)^t$. 

\begin{prop} \label{analogue5.1}
For any $\eta \in ]0, 1/2[, t > 0$ and $\alpha \geq 0$, let us define the distribution functions
\[ \mathcal{P}_{\eta, t}^{\pm}(\alpha) := \mathbb{P} \left ( {\max_{ 1\leq i \leq \mathcal{N}_{t(1-\eta)} - 1} \frac{e_i S_i^t}{t} \leq \alpha_t^{\pm} }, \left (1- \frac1{t} \sum_{i=1}^{\mathcal{N}_{t(1-\eta)} - 1} e_i S_i^t R_i^t \right )\frac{e_{\mathcal{N}_{t(1-\eta)}} S_{\mathcal{N}_{t(1-\eta)}}^t} {e_{\mathcal{N}_{t(1-\eta)}} S_{\mathcal{N}_{t(1-\eta)}}^t R_{\mathcal{N}_{t(1-\eta)}}^t} \leq \alpha_t^{\pm} \right ), \]
where $\alpha^{\pm}_t:=\alpha (1\pm (\log \log t)^{-1/2})$. Then, for all $t$ large enough we then have
\begin{align*}
\mathcal{P}_{\eta, t}^-(\alpha) -v(\eta,t) \leq \mathbb{P} \left( \mathcal{L}_X^*(t)/t \leq \alpha  \right) \leq \mathcal{P}_{\eta, t}^{+}(\alpha) + v(\eta,t), 
\end{align*}
where $v$ is a positive function such that $\lim_{\eta \rightarrow 0} \limsup_{t \rightarrow + \infty} v(\eta,t) = 0$. 

\end{prop}

%ERREUR DANS ENONCE DE LA PROP ($\mathcal{N}_{t(1-\eta)}$ AU LIEU DE $\mathcal{N}_{t(1-\eta)} - 1$). ON A CORRIGE MAIS TOUT VERIFIER DU COUP. OK

{Recall the bivariate $\kappa$-stable subordinator $(\mathcal{Y}_1,\mathcal{Y}_2)$ defined just before Theorem \ref{cvdutl}. The next step in the proof of Theorem \ref{cvdutl} is to identify the objects in $\mathcal{P}_{\eta, t}^{\pm}(\alpha)$ as continuous functionals of a process $(Y_1, Y_2)^t$ that converges in distribution to $(\mathcal{Y}_1,\mathcal{Y}_2)$ when $t$ goes to infinity.} Let $(D([0, +\infty[, \mathbb{R}^2), J_1)$ be the space of c\`ad-l\`ag functions taking values in $\mathbb{R}^2$, equipped with the $J_1$-Skorokhod topology. If, as in Section 1.2 of \cite{advech}, we define $(Y_1, Y_2)^t \in D([0, +\infty[, \mathbb{R}^2)$ by
\begin{eqnarray}
\forall s \geq 0, \ (Y_1, Y_2)^t(s) := \frac1{t} \sum_{j=1}^{\lfloor s e^{\kappa \phi(t)} \rfloor} (e_j S_j^t , e_j S_j^t R_j^t), \label{defy1y2t}
\end{eqnarray}
then we have: 
\begin{prop} \label{propcvsub}
$(Y_1, Y_2)^t$ converges in distribution to $(\mathcal{Y}_1,\mathcal{Y}_2)$ in $(D([0, +\infty[, \mathbb{R}^2), J_1)$. 
\end{prop}

{Recall the notations defined just before Theorem \ref{cvdutl}. The objects in $\mathcal{P}_{\eta, t}^{\pm}(\alpha)$ can be easily written in term of functionals of $(Y_1, Y_2)^t$. 
{\begin{align*}
\mathcal{P}_{\eta, t}^{\pm}(\alpha) = \mathbb{P} & \left ( \max \left \{ Y_1^{t, \natural} ( Y_2^{t, -1}(1-\eta)- ), \right. \right. \\ 
& \left. \left. \left [ 1- Y_2^{t}(Y_2^{t, -1}(1-\eta)-) \right ] \times \frac{Y_1^{t}(Y_2^{t, -1}(1-\eta)) - Y_1^{t}(Y_2^{t, -1}(1-\eta)-)} {Y_2^{t}(Y_2^{t, -1}(1-\eta)) - Y_2^{t}(Y_2^{t, -1}(1-\eta)-)} \right \} \leq \alpha_t^{\pm} \right ) \\
= \mathbb{P} & \left ( \max \left \{ J_{I, 1-\eta}^- \left [ (Y_1, Y_2)^t \right ], \right. \right. \\ 
& \left. \left. \left [ 1 - \tilde K_{I, 1-\eta}^- \left [ (Y_1, Y_2)^t \right ] \right ] \times \frac{K_{I, 1-\eta} \left [ (Y_1, Y_2)^t \right ] - K_{I, 1-\eta}^- \left [ (Y_1, Y_2)^t \right ]}{\tilde K_{I, 1-\eta} \left [ (Y_1, Y_2)^t \right ] - \tilde K_{I, 1-\eta}^- \left [ (Y_1, Y_2)^t \right ]} \right \} \leq \alpha_t^{\pm} \right ), 
\end{align*}}
where the functionals $J_{I, a}^-$, $K_{I, a}$, $K_{I, a}^-$, $\tilde K_{I, a}$ and $\tilde K_{I, a}^-$ are defined in Subsection 4.3 of \cite{advech}. Thanks to Lemma 4.5 there, we see that a realization of the $\kappa$-stable subordinator $(\mathcal{Y}_1,\mathcal{Y}_2)$ is almost surely a point of continuity of these functionals. As a consequence, using Proposition \ref{propcvsub} and the \textit{continuous mapping theorem}, we get, when $t$ goes to infinity, the convergence of $\mathcal{P}_{\eta, t}^{\pm}(\alpha)$ to 
{\begin{align*}
& \mathbb{P} \left ( \max \left \{ J_{I, 1-\eta}^- \left [ (\mathcal{Y}_1,\mathcal{Y}_2) \right ], \ \left [ 1 - \tilde K_{I, 1-\eta}^- \left [ (\mathcal{Y}_1,\mathcal{Y}_2) \right ] \right ] \times \frac{K_{I, 1-\eta} \left [ (\mathcal{Y}_1,\mathcal{Y}_2) \right ] - K_{I, 1-\eta}^- \left [ (\mathcal{Y}_1,\mathcal{Y}_2) \right ]}{\tilde K_{I, 1-\eta} \left [ (\mathcal{Y}_1,\mathcal{Y}_2) \right ] - \tilde K_{I, 1-\eta}^- \left [ (\mathcal{Y}_1,\mathcal{Y}_2) \right ]} \right \} \leq \alpha \right ) \\
= & \mathbb{P} \left ( \max \left \{ \mathcal{Y}_1^{\natural}(\mathcal{Y}_2^{-1}(1-\eta)-), \ \left [ 1- \mathcal{Y}_2(\mathcal{Y}_2^{-1}(1-\eta)-) \right ] \times \frac{\mathcal{Y}_1(\mathcal{Y}_2^{-1}(1-\eta)) - \mathcal{Y}_1(\mathcal{Y}_2^{-1}(1-\eta)-)} {\mathcal{Y}_2(\mathcal{Y}_2^{-1}(1-\eta)) - \mathcal{Y}_2(\mathcal{Y}_2^{-1}(1-\eta)-)} \right \} \leq \alpha \right ). 
\end{align*}}
Then, we have almost surely that for all $\eta$ small enough, $\mathcal{Y}_2^{-1}(1-\eta) = \mathcal{Y}_2^{-1}(1)$ (since almost surely $\mathcal{Y}_2(\mathcal{Y}_2^{-1}(1)-) < 1$). As a consequence, when $\eta$ goes to $0$, the above expression converges to $\mathbb{P} (\max \{ \mathcal{I}_1, \mathcal{I}_2 \} \leq \alpha)$, where $\mathcal{I}_1$ and $\mathcal{I}_2$ are defined just before Theorem \ref{cvdutl}. We have obtained that, modulo Proposition \ref{propcvsub}, $\lim_{\eta \rightarrow 0} \lim_{t \rightarrow + \infty} \mathcal{P}_{\eta, t}^{\pm}(\alpha) = \mathbb{P} (\max \{ \mathcal{I}_1, \mathcal{I}_2 \} \leq \alpha)$.

As a consequence, Theorem \ref{cvdutl} will follow if we prove Propositions \ref{analogue5.1} and \ref{propcvsub}. For Proposition \ref{analogue5.1} we first need to show that, outside the bottoms of the $n_t$ first standard $h_t$-valleys, the amount of time spent by the diffusion and the local time are negligible compared to $t$. This is the object of the next two facts that are taken from \cite{AndDev} and \cite{advech} (however, the extension of these results to our context requires some precautions, this is why we give some details of proofs in Subsection \ref{justoffacts}). The reason to prove the negligibility compared to $t$ is because $t$ is the expected renormalization for $\mathcal{L}^*_X(t)$ and, obviously, the total amount of time spent by the diffusion until time $t$. Note that what we neglect compared to $t$ can however be very large when $t$ is large. 
%, this is why it makes sense to prove negligibility compared to $t$, even though these quantities can be very large when $t$ is. 
Recall from Subsection \ref{factsandnotations} that for any $r \in \mathbb{R}$, $H(r) := \tau(X, r)$, the hitting times of $r$ by the diffusion $X$. We have }

%We prove that, at instant $t$, the total time spent between the bottoms of the valleys is negligible compared with $t$. Our result is the analogue of Lemma 3.7 in \cite{AndDev}. 

\begin{fact} \label{lemtps}
There exists a positive constant $C > 0$ such that for $t$ large enough, 
\[ \mathbb{P} \left( \mathcal{A}^1_t := \bigcap_{j=1}^{ n_t}\left\{0\leq H(\tilde m_{j})-\sum_{i=1}^{j-1}\left ( H(\tilde L_i)-H(\tilde m_i) \right ) \leq \frac{2t}{\log h_t} \right\}\right) \geq 1 - C e^{-\phi(t)} n_t \log h_t. \]
Recall that by convention $\sum_{i=1}^{0}... = 0$. Note that the lower bound converges to $1$ since $n_t \sim e^{\kappa (1+\delta) \phi(t)}$, {$\log(h_t) << e^{2\kappa \delta \phi(t)}$, and $(1+3\delta)\kappa <1$ (because of the choice of $\delta$, made in the beginning of this section). }
%\geq 1- c n_t (\log h_t) e^{- \phi(t)}
%where $\sum_{i=1}^0\dots =0$ by convention. 
%Notice in particular that
%$
%    n_t (\log h_t) e^{- \phi(t)}
%%=
%%    o(\phi(t))e^{[(1+\delta)\k-1]\phi(t)}
%\leq
%    (\log \log t)e^{[(1+\delta)\kappa-1]\phi(t)}
%=
%    o(1)
%$ as $t \rightarrow+\infty$. 
\end{fact}

%We now prove that the local time is small (compared with $t$) outside the bottom of the valleys. Recall the definitions of $\tau^*(h)$ and $m^*(h)$ in Subsection \ref{firstmin}. We have

Before stating the next fact, which proves that the local time is negligible {(compared to $t$)} outside the "deep bottoms" of the standard valleys, we need to define what we mean exactly by "deep bottom". We define 
\begin{eqnarray}
\mathcal{D}_j := [\tilde \tau_j^-((\phi(t))^2), \tilde \tau_j^+((\phi(t))^2)], \label{defdj}
\end{eqnarray}
where $\tilde \tau_j^-(.)$ and $\tilde \tau_j^+(.)$ have the same meaning as in the beginning of Subsection \ref{coin}. Note that the definition we give for $\mathcal{D}_j$ is different from the one in Section 3.2.2 of \cite{advech}. The need for a different definition comes from the fact that, in our case, the potential $V$ is allowed to make negative jumps. Because of this, some arguments given in \cite{advech}, with their definition of $\mathcal{D}_j $, fail to hold in our case. 
%This is because, due to the negative jumps, some estimates used in \cite{advech} would fail to hold, in our context, if we kept their definition. We have :

\begin{fact} \label{analog3.3}
Recall the definitions of $\tau^*(h)$ and $m^*(h)$ in Subsection \ref{firstmin}. 
There are positive constants $C_1, C_2, C_3$ such that for $t$ large enough, 
\begin{align}
\mathbb{P} \left( \sup_{x \in [0, m^*(h_t)]} \mathcal{L}_X (H(\tau^*(h_t)),x) > t e^{(\kappa (1+3\delta )-1)\phi(t)} \right) & \leq \frac{C_1}{n_t e^{\kappa \delta \phi(t)}}, \\
\mathbb{P} \left( \mathcal{A}^2_t := \bigcap_{j=0}^{n_t - 1}\left\{ \sup_{x \in \mathbb{R}} \left ( \mathcal{L}_X (H(\tilde m_{j+1}),x) - \mathcal{L}_X (H(\tilde L_j),x) \right ) \leq t e^{(\kappa (1+3\delta )-1)\phi(t)} \right \} \right) & \geq 1 - \frac{C_2}{e^{\kappa \delta \phi(t)}}, \label{negltl2} \\
\mathbb{P} \left( \mathcal{A}^3_t := \bigcap_{j=1}^{ n_t}\left\{ \sup_{x \in [\tilde{L}_{j-1}, \tilde{L}_j] \cap \mathcal{D}_j^c} \left ( \mathcal{L}_X (H(\tilde L_j),x) - \mathcal{L}_X (H(\tilde m_{j}),x) \right ) \leq t e^{-2 \phi(t)}, \right \} \right) & \geq 1-\frac{C_3 n_t}{e^{2 \phi(t)}}. \label{negltl3}
\end{align}
{Recall that $\kappa (1+3\delta )-1<0$ because of the choice of $\delta$ made in the beginning of this section. }
\end{fact}

The next step in the proof of Proposition \ref{analogue5.1} is to show that the main contributions to the local time and to the time spent by the diffusion in the bottoms of the standard valleys can be approximated by the sequence $( e_i S_i^t,  e_i S_i^t R_i^t)_{i \geq 1}$ :

\begin{prop} \label{approxparliid}

The sequences $( e_j )_{j \geq 1}$, $( S_j^t )_{j \geq 1}$ and $(R_j^t)_{j \geq 1}$ are \textit{iid} and mutually independent. For any $\epsilon \in ]0, \max \{1/8, (1- (1+\delta) \kappa)/2 \} [$ there exists a positive constant $c$ (depending on $\delta$ and $\epsilon$) such that for $t$ large enough, 
\begin{align*}
\mathbb{P} \left ( \mathcal{A}^4_t := \cap_{j=1}^{n_t} \left \{ (1-e^{- \epsilon h_t/7}) e_j S_j^t \leq \mathcal{L}_X(H(\tilde{L}_j), \tilde m_j) \leq (1+e^{- \epsilon h_t/7}) e_j S_j^t \right \} \right ) & \geq 1 - e^{-c h_t}, \\
\mathbb{P} \left ( \mathcal{A}^5_t := \cap_{j=1}^{n_t} \left \{ (1-e^{- \epsilon h_t/7}) e_j S_j^t R_j^t \leq H(\tilde L_j)-H(\tilde m_j) \leq (1+e^{- \epsilon h_t/7}) e_j S_j^t R_j^t \right \} \right ) & \geq 1 - e^{-c h_t}. 
\end{align*}
%\begin{align*}
%\mathbb{P} \left ( \cap_{i=1}^{n_t} \left \{ (1-\epsilon_t) {\bf  e}_j S_j^t \leq \mathcal{L}_j (\tilde m_j, H_j(\tilde{L}_j)) \leq (1+\epsilon_t) {\bf  e}_j S_j^t \right \} \right ) & \geq 1 - e^{-c h_t} \\
%\mathbb{P} \left ( \cap_{i=1}^{n_t} \left \{ (1-\epsilon_t) {\bf  e}_j S_j^t R_j^t \leq H_j(\tilde{L}_j) \leq (1+\epsilon_t) {\bf  e}_j S_j^t R_j^t \right \} \right ) & \geq 1 - e^{-c h_t}, 
%\end{align*}

{Note that we have indeed $(1- (1+\delta) \kappa)/2 > 0$ because in the beginning of this section $\delta$ was chosen such that $\kappa (1+3\delta )<1$. }
\end{prop}

This proposition is proved in the following subsection. For the end of the proof of Proposition \ref{analogue5.1}, the idea is simply to use the previous steps to translate "$\mathcal{L}_X^*(t) /t \leq \alpha$" in term of events only involving the sequence $( e_i S_i^t,  e_i S_i^t R_i^t)_{i \geq 1}$. We do this in Subsection \ref{approxderep}, following the arguments of the proof of Proposition 5.1 of \cite{advech}.  

Finally, Proposition \ref{propcvsub} is roughly speaking a Donsker Theorem for heavy tailed random variables, its proof relies on the study of the right tail of the random variables $e_1 S_1^t$ and $e_1 S_1^t R_1^t$. This is done in Subsection \ref{reprise}. 

\subsection{Proof of Proposition \ref{approxparliid} and consequences} \label{maincontibcommeiid}

We now prove that $(\mathcal{L}_X(H(\tilde{L}_j), \tilde m_j), H(\tilde L_j)-H(\tilde m_j))_{j \geq 1}$ can be approximated by an \textit{iid} sequence. This generalizes Proposition 3.5 of \cite{advech} to our setting. First, let us adapt to our context the proof of the first point of Lemma 3.6 in \cite{advech}, {all the details are given for the sake of clarity. 

Let us define $X_{\tilde m_j} := X(. + H(\tilde m_j))$ which is, according to the Markov property at $H(\tilde m_j)$, a diffusion in the environment $V$ starting from $\tilde m_j$. We also define for any $r \in \mathbb{R}$, $H_{X_{\tilde m_j}}(r) := \tau(X_{\tilde m_j}, r)$, the hitting time of $r$ by $X_{\tilde m_j}$ and $\mathcal{L}_{X_{\tilde m_j}}(.,.)$, the local time of $X_{\tilde m_j}$. Also, still by the Markov property, we have that conditionally on the environment $V$, the processes $(X_{\tilde m_j}(t), 0 \leq t \leq H_{X_{\tilde m_j}}(\tilde L_j))$ are independent (because $\tilde m_j < \tilde L_j < \tilde m_{j+1}$). We now apply \eqref{expretl2} and \eqref{expretl2.1} with $r=\tilde L_j$ to the processes $(X_{\tilde m_j}(t), 0 \leq t \leq H_{X_{\tilde m_j}}(\tilde L_j))$. We obtain for each $j \geq 1$, 
\begin{align*}
\mathcal{L}_X(H(\tilde{L}_j), \tilde m_j) & = \mathcal{L}_{X_{\tilde m_j}}(H_{X_{\tilde m_j}}(\tilde L_j),\tilde m_j) = \mathcal{L}_{B^j} \left [ \tau \left ( B^j, A^j(\tilde{L}_j) \right ), 0 \right ] \\
H(\tilde L_j)-H(\tilde m_j) & = H_{X_{\tilde m_j}}(\tilde L_j) = \int_{-\infty}^{\tilde{L}_j} \mathcal{L}_{X_{\tilde m_j}}(H_{X_{\tilde m_j}}(\tilde L_j), u) d u \\
& = \int_{-\infty}^{\tilde{L}_j} e^{-\tilde V^{(j)}(u)} \mathcal{L}_{B^j} \left [ \tau \left ( B^j, A^j(\tilde{L}_j) \right ),A^j(u) \right ] d u, 
\end{align*}
where $A^j(u) := \int_{\tilde{m}_j}^{u} e^{\tilde V^{(j)}(x)}dx$ and $(B^j, j \geq 1)$ is a sequence of \textit{iid} standard Brownian motions starting at $0$, and independent from $V$. Note that, since the integrand in the above expression is $\mathcal{L}_{X_{\tilde m_j}}(H_{X_{\tilde m_j}}(\tilde L_j),u)$, this integrand is null on points that are not visited by $X_{\tilde m_j}$ before $H_{X_{\tilde m_j}}(\tilde L_j)$. According to Lemma \ref{sortparladroite} applied with $h=h_t$, we have $\mathbb{P} ( H_{X_{\tilde m_j}}(\tilde L_{j-1}) > H_{X_{\tilde m_j}}(\tilde L_j) ) \geq 1 - e^{- \kappa \delta h_t/6}$ for all $j \geq 1$ {(where $\delta$ has been defined in Subsection \ref{coin} and satisfies $\kappa (1+3\delta )<1$ as assumed in the beginning of this section)}. We thus get 
\begin{eqnarray}
\mathbb{P} \left ( \left ( \mathcal{L}_X(H(\tilde{L}_j), \tilde m_j), H(\tilde L_j)-H(\tilde m_j) \right )_{1 \leq j \leq n_t} = \left (\tilde{\mathcal{L}}_j, \tilde h_j \right )_{1 \leq j \leq n_t} \right ) \geq 1 - n_t e^{-\kappa \delta h_t /6}, \label{approxiid2}
\end{eqnarray}
where 
\begin{eqnarray}
(\tilde{\mathcal{L}}_j, \tilde h_j) := \left ( \mathcal{L}_{B^j} \left [ \tau \left ( B^j, A^j(\tilde{L}_j) \right ), 0 \right ], \int_{\tilde{L}_{j-1}}^{\tilde{L}_j} e^{-\tilde V^{(j)}(u)} \mathcal{L}_{B^j} \left [ \tau \left ( B^j, A^j(\tilde{L}_j) \right ),A^j(u)\right ] d u \right ). \label{defljhj}
\end{eqnarray}
Note that, thanks to Remark \ref{iid} (applied with $h=h_t$) and to the fact that the $B^j$'s are \textit{iid} Brownian motions, the sequence $(\tilde{\mathcal{L}}_j, \tilde h_j)_{j \geq 1}$ is \textit{iid}. Let us scale the Brownian motion $B^j$ in the following way: $\tilde B^j := B^j( (A^j(\tilde{L}_j))^2 . )/ A^j(\tilde{L}_j)$. Note that, conditionally on $V$, $(\tilde B^j, j \geq 1)$ is a sequence of \textit{iid} standard Brownian motions. This implies that the sequence $(\tilde B^j, j \geq 1)$ is independent from $V$, even though $V$ appears in the expression defining it. To avoid overuse of notations we now write $B^j$ instead of $\tilde B^j$. With this notation we have 
\begin{align}
(\tilde{\mathcal{L}}_j, \tilde h_j) = A^j(\tilde{L}_j) \times \left ( \mathcal{L}_{B^j} \left [ \tau(B^j, 1), 0 \right ], \int_{\tilde{L}_{j-1}}^{\tilde{L}_j} e^{-\tilde V^{(j)}(u)} \mathcal{L}_{B^j} \left [ \tau(B^j, 1),A^j(u)/A^j(\tilde{L}_j) \right ] d u \right ). \label{approxiid3}
\end{align}
}

Note that $\mathcal{L}_X(H(\tilde{L}_j), \tilde m_j) / A^j(\tilde{L}_j) = \mathcal{L}_{B^j} (\tau(B^j, 1), 0) =: e_j$ follows an exponential distribution with parameter $1/2$. Also, since $e_j$ only depends on $B^j$, the sequence $(e_j, j \geq 1)$ is \textit{iid} and independent from $V$. In order to prove Proposition \ref{approxparliid}, we are now reduced to give approximations for $\tilde h_j$ and $A^j(\tilde{L}_j)$. For this, we first prove lemmas to bound exponential functionals of the environment. 

Recall the definitions of $\tilde \tau_j^-(.)$ and $\tilde \tau_j^+(.)$ in Subsection \ref{coin}, we have: 

\begin{lemme} \label{boundj0}

Choose $\epsilon$ such that $0 < \epsilon < \max \{1/8, (1- (1+\delta) \kappa)/2 \}$. Then, 
%\[ \forall j \geq 1, \ \mathbb{P} \left( \int_{\tilde{L}_{j-1}}^{\tilde \tau_j^-(h_t / 2)} e^{-\tilde V^{(j)}(u)} du > e^{-\epsilon h_t} \right) \leq e^{-c h_t}, \]
\[ \forall j \geq 1, \ \mathbb{P} \left( \int_{\tilde{L}_{j-1}}^{\tilde \tau_j^-(h_t / 2)} e^{-\tilde V^{(j)}(u)} du > e^{-\epsilon h_t} \right) \leq e^{-c h_t}, \]
for $t$ large enough and some positive constant $c$ depending on $\delta$ and $\epsilon$. {Recall that we have indeed $(1- (1+\delta) \kappa)/2 > 0$ because in the beginning of this section $\delta$ was chosen such that $\kappa (1+3\delta )<1$. }
\end{lemme}

\begin{proof}

Fix $j \geq 1$. We have 
\begin{eqnarray}
\int_{\tilde{L}_{j-1}}^{\tilde \tau_j^-(h_t / 2)} e^{-\tilde V^{(j)}(u)} du = \int_{\tilde{L}_{j-1}}^{\tilde \tau_j^-(h_t)} e^{-\tilde V^{(j)}(u)} du + \int_{\tilde \tau_j^-(h_t)}^{\tilde \tau_j^-(h_t / 2)} e^{-\tilde V^{(j)}(u)} du. \label{boundj0j21}
\end{eqnarray}

The first term of the right hand side is less than 
\[ \left ( \tilde \tau_j^-(h_t) - \tilde{L}_{j-1} \right ) \times \sup_{u \in [\tilde{L}_{j-1}, \tilde \tau_j^-(h_t)]} e^{-\tilde V^{(j)}(u)}. \]
{According to Lemma \ref{majoprevalley} (applied with $h=h_t$), the first factor is less than $e^{(1+\delta)\kappa h_t}$ with a probability greater than $1-e^{-\delta \kappa h_t /2}$ when $t$ is large enough. From the definition of $\tilde{L}_{j}^{\sharp}$ we have $V(\tilde L_{j}^{\sharp})=\inf_{[\tilde L_{j-1}, \tilde L_{j}^{\sharp}]} V$ so the second factor is equal to $\sup_{u \in [\tilde L_{j}^{\sharp}, \tilde \tau_j^-(h_t)]} e^{-\tilde V^{(j)}(u)}$. We thus apply Lemma \ref{minoprevalley} with $\alpha = 1, \eta = \epsilon, h=h_t$ and we get that the second factor is less than $e^{-(1-\epsilon)h_t}$ with a probability greater than $1 - e^{-\kappa \epsilon h_t /3}$ when $t$ is large enough. In conclusion, there is a positive constant $c_1$ (depending on $\delta$ and $\epsilon$) such that for $t$ large enough, }
\begin{eqnarray}
\mathbb{P} \left( \int_{\tilde{L}_{j-1}}^{\tilde \tau_j^-(h_t)} e^{-\tilde V^{(j)}(u)} du > e^{-\epsilon h_t}/2 > e^{((1+\delta) \kappa +\epsilon - 1)h_t} \right) \leq e^{-c_1 h_t}. \label{boundj0j22}
\end{eqnarray}

{Then, note that $\int_{\tilde \tau_j^-(h_t)}^{\tilde \tau_j^-(h_t / 2)} e^{-\tilde V^{(j)}(u)} du$ is a function of $\tilde P_1^{(j)}$ (defined in Subsection \ref{coin}) and, according to Proposition \ref{standardwilliams} (applied with $h=h_t$), we have $d_{VT}( \tilde P_1^{(j)}, P_1^{(2)}) \leq 2 e^{- \delta \kappa h_t /3}$ (where $P_1^{(2)}$ is defined in Subsection \ref{loidesval}). Recall that, according to Proposition \ref{Fact_Williams} (applied with $h=h_t$), the law of $P_1^{(2)}$ is absolutely continuous with respect to the law of the process $(\hat{V}^{\uparrow}(x))_{0 \leq x \leq \tau(\hat{V}^{\uparrow}, h_t+)}$ and the density is bounded by $2$ when $h_t$ is large enough. As a consequence, we get that for $t$ large enough, }
\[ \mathbb{P} \left ( \int_{\tilde \tau_j^-(h_t)}^{\tilde \tau_j^-(h_t / 2)} e^{-\tilde V^{(j)}(u)} du > e^{-\epsilon h_t}/2 \right ) \leq 2 \mathbb{P} \left( \int_{\tau(\hat V^{\uparrow}, [h_t/2, +\infty[)}^{\tau(\hat V^{\uparrow}, [h_t, +\infty[)} e^{-\hat V^{\uparrow}(u)} du > e^{-\epsilon h_t}/2 \right) + 2 e^{- \delta \kappa h_t /3}. \]
The integral on the right hand side is less than
\[ \tau \left ( \hat V^{\uparrow}, [h_t, +\infty[ \right ) \times \sup_{u \in [\tau(\hat V^{\uparrow}, [h_t/2, +\infty[), +\infty[} e^{-\hat V^{\uparrow}(u)}. \]
Then, according to Lemma \ref{tpsatteintvdown} applied with $y=h_t$, $r=e^{h_t /8}$ {we have $\mathbb{P}(\tau ( \hat V^{\uparrow}, [h_t, +\infty[ ) \leq e^{h_t /8} ) \geq 1 - {K_0 [\exp(-\kappa h_t) + \exp (K_1 h_t -K_2 e^{h_t /8})]}$, where $K_0, K_1, K_2$ are positive constants. According to Lemma \ref{vdownrestegrand} applied with $a = h_t/4, b = h_t/2$, $z=0$ we have $\mathbb{P} ( \inf_{[\tau(\hat V^{\uparrow}, [h_t/2, +\infty[), +\infty[} \hat V^{\uparrow} \geq h_t/4 ) \geq 1 - e^{-\kappa h_t/4} / (1-e^{-\kappa h_t/2})$. In conclusion, there is a positive constant $c_2$ such that for $t$ large enough, }
\begin{eqnarray}
\mathbb{P} \left( \int_{\tilde \tau_j^-(h_t)}^{\tilde \tau_j^-(h_t / 2)} e^{-\tilde V^{(j)}(u)} du > e^{-\epsilon h_t}/2 > e^{(1/8 - 1/4)h_t} \right) \leq e^{-c_2 h_t}. \label{boundj0j23}
\end{eqnarray}

The combination of \eqref{boundj0j21}, \eqref{boundj0j22} and \eqref{boundj0j23} yields the result. 

\end{proof}

\begin{lemme} \label{boundj2}

Choose $\epsilon$ such that $0 < \epsilon < 1/8$. There is a positive constant $c$ such that for $t$ large enough, 
\[ \forall j \geq 1, \ \mathbb{P} \left( \int_{\tilde \tau_j^+(h_t / 2)}^{\tilde L_j} e^{-\tilde V^{(j)}(u)} du > e^{-\epsilon h_t} \right) \leq e^{-c h_t}. \]

\end{lemme}

\begin{proof}

Fix $j \geq 1$. We have 
\begin{eqnarray}
\int_{\tilde \tau_j^+(h_t / 2)}^{\tilde L_j} e^{-\tilde V^{(j)}(u)} du = \int_{\tilde \tau_j^+(h_t / 2)}^{\tilde \tau_j(h_t)} e^{-\tilde V^{(j)}(u)} du + \int_{\tilde \tau_j(h_t)}^{\tilde L_j} e^{-\tilde V^{(j)}(u)} du. \label{boundj21}
\end{eqnarray}

According to Proposition \ref{standardwilliams} (applied with $h=h_t$), the fist term on the right hand side is equal in law to $\int_{\tau(V^{\uparrow}, h_t/2)}^{\tau(V^{\uparrow}, h_t)} e^{-V^{\uparrow}(u)} du$ which is less than 
\[ \tau(V^{\uparrow}, h_t) \times \sup_{u \in [\tau(V^{\uparrow}, h_t/2), +\infty[} e^{-V^{\uparrow}(u)}. \]
{According to Lemma \ref{vposlapltpsatt} applied with $y=h_t, r=e^{h_t /8}$, the first factor is less than $e^{h_t /8}$ with a probability greater than $1-\exp(K_1 h_t - K_2 e^{h_t /8})$ (for some positive constants $K_1$ and $K_2$). By the Markov property applied to $V^{\uparrow}$ at time $\tau(V^{\uparrow}, h_t/2)$, the second factor is equal in law to $\sup_{u \in [0, +\infty[} e^{-V_{h_t/2}^{\uparrow}(u)}$. We thus apply Lemma \ref{vuprestegrand} with $a = h_t/4, b = h_t/2$ and we get that the second factor is less than $e^{-h_t/4}$ with a probability greater than $1- K_4 e^{-K_3 h_t/4}$ (for some positive constants $K_3$ and $K_4$). In conclusion, there is a positive constant $c_1$ such that for $t$ large enough, }
\begin{eqnarray}
\mathbb{P} \left( \int_{\tilde \tau_j^+(h_t / 2)}^{\tilde \tau_j(h_t)} e^{-\tilde V^{(j)}(u)} du > e^{-\epsilon h_t}/2 > e^{(1/8 - 1/4)h_t} \right) \leq e^{-c_1 h_t}. \label{boundj22}
\end{eqnarray}

{By definition we have $\tilde L_j = \tau(\tilde V^{(j)}(\tilde \tau_j(h_t) + .), ]-\infty, h_t /2]) + \tilde \tau_{j}(h_t)$ so the second term in the right hand side of \eqref{boundj21} is less than $e^{-h_t/2}(\tilde L_j - \tilde \tau_j(h_t))$. Because of Proposition \ref{standardwilliams} (applied with $h=h_t$), $\tilde L_j - \tilde \tau_j(h_t)$ is equal in law to $\tau(V, ]-\infty, -h_t/2])$. Applying Lemma \ref{tpsatteinthatv} with $y = h_t/2$ and $r = e^{h_t/4}$ we get that $\tau(V, ]-\infty, -h_t/2])$ is less than $e^{h_t/4}$ with a probability greater than $1-\exp(K_5 h_t/2 - K_6 e^{h_t/4})$, where $K_5$ and $K_6$ are some positive constants. Putting all this together we get that for $t$ large enough, }
\begin{eqnarray}
\mathbb{P} \left( \int_{\tilde \tau_j(h_t)}^{\tilde L_j} e^{-\tilde V^{(j)}(u)} du > e^{-\epsilon h_t}/2 > e^{(1/4 - 1/2)h_t} \right) \leq e^{-h_t}. \label{boundj23}
\end{eqnarray}

The combination of \eqref{boundj21}, \eqref{boundj22} and \eqref{boundj23} yields the result. 

\end{proof}

Recall the definition of $A^j(.)$ {a little before \eqref{approxiid2}}. We have 

\begin{lemme} \label{truccentral}

Choose $\epsilon$ such that $0 < \epsilon < 1/4$. There is a positive constant $c$ (depending on $\epsilon$) such that for $t$ large enough, 
%\[ \forall j \geq 1, \ \mathbb{P} \left( \int_{\tilde{L}_{j-1}}^{\tilde \tau_j^-(h_t / 2)} e^{-\tilde V^{(j)}(u)} du > e^{-\epsilon h_t} \right) \leq e^{-c h_t}, \]
\[ \forall j \geq 1, \ \mathbb{P} \left ( \sup_{u \in [\tilde \tau_j^-(h_t /2), \tilde \tau_j^+(h_t /2)]} \left | A^j(u) / A^j(\tilde{L}_j) \right | \leq e^{-(1-2\epsilon) h_t/2} \right ) \geq 1 - e^{-c h_t}. \]
\end{lemme}

\begin{proof}

For any $j \geq 1$ we have 
\[ A^j(\tilde{L}_j) \geq \int_{\tilde m_j}^{\tilde \tau_j(h_t)} e^{\tilde V^{(j)}(u)} du \overset{\mathcal{L}}{=} \int_0^{\tau(V^{\uparrow}, h_t)} e^{V^{\uparrow}(u)} du \ \ \ \text{and} \ \ \ A^j(\tilde \tau_j^+(h_t /2)) \overset{\mathcal{L}}{=} \int_0^{\tau(V^{\uparrow}, h_t /2)} e^{V^{\uparrow}(u)} du \]
where we have used Proposition \ref{standardwilliams} (with $h=h_t$) for the equalities in law. {\eqref{0MinorationAVallee} applied with $h = h_t$, $\eta = \epsilon/2$ yields $\mathbb{P}(A^j(\tilde{L}_j) > e^{(1-\epsilon/2) h_t}) \geq 1-e^{- K \epsilon h_t / 2}$ for some positive constant $K$ and when $t$ is large enough. \eqref{0MajorationAVallee} applied with $h = h_t/2$, $\eta = \epsilon$ yields $\mathbb{P}(A^j(\tilde \tau_j^+(h_t /2)) \leq e^{(1+\epsilon)h_t/2}) \geq 1-e^{-h_t}$ for $t$ large enough. In conclusion, we get the existence of a positive constant $c_1$ such that for $t$ large enough, }
\begin{eqnarray}
\mathbb{P} \left ( 0 \leq A^j(\tilde \tau_j^+(h_t /2)) / A^j(\tilde{L}_j) \leq e^{-(1-2\epsilon) h_t/2} \right ) \geq 1 - e^{-c_1 h_t}. \label{approxdeu1}
\end{eqnarray}
{Then, $A^j(\tilde \tau_j^-(h_t / 2)) = \int_{\tilde m_j}^{\tilde \tau_j^-(h_t / 2)} e^{\tilde V^{(j)}(u)} du$. That is, $A^j(\tilde \tau_j^-(h_t / 2))$ is a function of $\tilde P_1^{(j)}$ and, according to Proposition \ref{standardwilliams} (applied with $h=h_t$), we have $d_{VT}( \tilde P_1^{(j)}, P_1^{(2)}) \leq 2 e^{- \delta \kappa h_t /3}$. Recall that, according to Proposition \ref{Fact_Williams} (applied with $h=h_t$), the law of $P_1^{(2)}$ is absolutely continuous with respect to the law of the process $(\hat{V}^{\uparrow}(x))_{0 \leq x \leq \tau(\hat{V}^{\uparrow}, h_t+)}$ and the density is bounded by $2$ when $h_t$ is large enough. As a consequence, we get that for $t$ large enough, }
\begin{align*}
\mathbb{P} \left ( A^j(\tilde \tau_j^-(h_t / 2)) \leq - e^{(1 + \epsilon) h_t/2} \right ) & \leq 2 \mathbb{P} \left( \int_0^{\tau(\hat V^{\uparrow}, h_t /2+)} e^{\hat V^{\uparrow}(u)} du  \geq e^{(1 + \epsilon) h_t/2} \right) + 2 e^{- \delta \kappa h_t /3} \\
& \leq 2 e^{- c_4 h_t} + 2 e^{- \delta \kappa h_t /3}. 
\end{align*}
The last inequality comes from \eqref{d0MajorationAVallee} applied with $h = h_t/2$, $\eta = \epsilon$, it is true for some positive constant $c_4$ and $t$ large enough. Now combining with our lower bound for  $A^j(\tilde{L}_j)$, that is, $\mathbb{P}(A^j(\tilde{L}_j) > e^{(1-\epsilon/2) h_t}) \geq 1-e^{- K \epsilon h_t / 2}$, we get\begin{eqnarray}
\mathbb{P} \left ( - e^{-h_t (1 - 2\epsilon)/2} \leq A^j(\tilde \tau_j^-(h_t / 2)) / A^j(\tilde{L}_j) \leq 0 \right ) \geq 1 - e^{-c_5 h_t}, \label{approxdeu2}
\end{eqnarray}
for some positive constant $c_5$, and when $t$ is large enough. The combination of \eqref{approxdeu1} and \eqref{approxdeu2} with the increase of $A^j(.)$ yields the result. 

\end{proof}

We now approximate $\tilde h_j$, {where $(\tilde{\mathcal{L}}_j, \tilde h_j)$ is defined in \eqref{defljhj}. }

\begin{lemme} \label{approxdeu}
Choose $\epsilon$ such that $0 < \epsilon < \max \{1/8, (1- (1+\delta) \kappa)/2 \}$. There is a positive constant $c$ (depending on $\delta$ and $\epsilon$) such that for all $t$ large enough we have
\begin{equation}\label{eqPropIV}
\forall j \geq 1, \ \mathbb{P} \left( \left | \tilde h_j - A^j(\tilde{L}_j) R_j^t e_j \right |
            \leq
        e^{-\epsilon h_t/6} A^j(\tilde{L}_j) R_j^t e_j \right)
\geq
    1- e^{- c h_t}, 
%=\int_{\tilde L_1^-}^{\tilde L_1} e^{-(W_{\kappa}(u)-W_{\kappa}(\tilde m_1))} \lo_B(\tau^B(\tA2),\tilde{A}_2(u)) \textnormal{d}u
\end{equation}
where
\[ R_j^t := \int_{\tilde{\tau}_j^-(h_t/2)}^{ \tilde{\tau}_j^+(h_t/2)} e^{-\tilde V^{(j)}(u)}  \textnormal{d}u. \]
%Moreover, ${\bf  e}_1$ is independent of $V$, and exponentially distributed with mean $2$.
\end{lemme}

\begin{proof}

Our proof has the same spirit as the one of Lemma 4.7 of \cite{AndDev} but relies on our estimates. We thus give all the details. For any $j \geq 1$, we use the expression \eqref{approxiid3} for $\tilde h_j / A^j(\tilde{L}_j)$ that can be cut into three parts: 
\begin{align}
\tilde h_j / A^j(\tilde{L}_j) & = \int_{\tilde{L}_{j-1}}^{\tilde L_j} e^{-\tilde V^{(j)}(u)} \mathcal{L}_{B^j}\left ( \tau (B^j, 1),A^j(u)/A^j(\tilde{L}_j) \right ) \textnormal{d}u \nonumber \\
& = \int_{\tilde{L}_{j-1}}^{\tilde \tau_j^-(h_t / 2)} + \int_{\tilde \tau_j^-(h_t / 2)}^{\tilde \tau_j^+(h_t / 2)} + \int_{\tilde \tau_j^+(h_t / 2)}^{\tilde L_j} e^{-\tilde V^{(j)}(u)} \mathcal{L}_{B^j}\left ( \tau (B^j, 1),A^j(u)/A^j(\tilde{L}_j) \right ) \textnormal{d}u \nonumber \\
& =: \mathcal{J}^j_0+\mathcal{J}^j_1+\mathcal{J}^j_2. \label{approxdeuori}
\end{align}

We start by bounding $\mathcal{J}^j_2$ and $\mathcal{J}^j_0$: 
%\[ \mathcal{J}^j_2 \leq \left ( \sup_{[0, 1]} \mathcal{L}_{B}\big[\tau (B, 1), . \big] \right ) \times \left ( \sup_{[\tilde \tau_j(h_t / 2), \tilde \tau_j(h_t)]} e^{-\tilde V^{(j)}(.)} \right ) \times \left ( \tilde L_j - \tilde \tau_j(h_t / 2) \right ), \]
%so
%\begin{align*}
%\mathbb{P} \left ( \mathcal{J}^j_2 \geq e^{- (1/2 - 3 \epsilon)h_t} \right ) & \leq \mathbb{P} \left ( \sup_{[0, 1]} \mathcal{L}_{B}\big[\tau (B, 1), . \big] \geq e^{\epsilon h_t} \right ) \\
%& + \mathbb{P} \left ( \inf_{[\tilde \tau_j(h_t / 2), \tilde \tau_j(h_t)]} \tilde V^{(j)}(.) \leq \left (\frac1{2} - \epsilon \right ) h_t \right ) \\
%& + \mathbb{P} \left ( \tilde L_j - \tilde \tau_j(h_t / 2) \geq e^{\epsilon h_t} \right ). 
%\end{align*}
\[ \mathcal{J}^j_2 \leq \left ( \sup_{x \in [0, 1]} \mathcal{L}_{B^j} \left ( \tau (B^j, 1), x \right ) \right ) \times \int_{\tilde \tau_j^+(h_t / 2)}^{\tilde L_j} e^{-\tilde V^{(j)}(u)} du. \]
Thanks to estimate (7.12) of \cite{advech} applied with $x=e^{\epsilon h_t/2}$ {we have $\mathbb{P} (\sup_{x \in [0, 1]} \mathcal{L}_{B^j} ( \tau (B^j, 1), x ) \leq e^{\epsilon h_t/2}) \geq 1 - 4 \exp(-e^{\epsilon h_t/2}/2)$. Combining with Lemma \ref{boundj2}, we get the existence of a positive constant $c_1$ such that for $t$ is large enough, }
%We bound the first term of the right hand side thanks to $(4.8)$ of \cite{AndDev}. For the second we note that, by proposition \ref{Fact_Williams}, the strong Makov property and $(\ref{03.10})$ applied with $\gamma = (1/2 - \epsilon), \alpha = 1/2, \omega = 1$, we have
%\[ \mathbb{P} \left ( \inf_{[\tilde \tau_j(h_t / 2), \tilde \tau_j(h_t)]} \tilde V^{(j)}(.) \leq \left (\frac1{2} - \epsilon \right ) \right ) = \mathbb{P} \left ( \inf_{[0, \tau(V^{\uparrow}, h_t)]} V^{\uparrow}_{h_t /2} \leq \left (\frac1{2} - \epsilon \right ) \right ) \leq e^{- c_3 \epsilon h_t}, \]
%for $t$ large enough, and where $c_3$ here denotes the constant $c_1$ of lemma \ref{estpourvup}. For the third term, note that lemma \ref{lemMi} gives
%\[ ... \]
%Combining those three estimates we get that
\begin{eqnarray}
\mathbb{P} \left ( \mathcal{J}^j_2 < e^{- \epsilon h_t/2} \right ) \geq 1 - e^{-c_1 h_t}. \label{approxdeu3}
\end{eqnarray}

For $\mathcal{J}^j_0$: 
\[ \mathcal{J}^j_0 \leq \left ( \sup_{x \in ]-\infty, 0]} \mathcal{L}_{B^j} \left ( \tau (B^j, 1), x \right ) \right ) \times \int_{\tilde{L}_{j-1}}^{\tilde \tau_j^-(h_t / 2)} e^{-\tilde V^{(j)}(u)} du. \]
%\[ \mathcal{J}^j_0 \leq \left ( \sup_{]-\infty, 0]} \mathcal{L}_{B}\big[\tau (B, 1), . \big] \right ) \times \left ( \sup_{[\tilde{L}_{j}^{\sharp}, \tilde \tau_j^-(h / 2)]} e^{-\tilde V^{(j)}(.)} \right ) \times \left ( \tilde \tau_j^-(h_t / 2) - \tilde{L}_{0} \right ), \]
Thanks to estimate (7.13) of \cite{advech} applied with $x=e^{\epsilon h_t/2}$ {we have $\mathbb{P} (\sup_{x \in ]-\infty, 0]} \mathcal{L}_{B^j} ( \tau (B^j, 1), x ) \leq e^{\epsilon h_t/2}) \geq 1 - 4 e^{-\epsilon h_t/2}$. Combining with Lemma \ref{boundj0}, we get the existence of a positive constant $c_2$ such that for $t$ is large enough, }
%\begin{align*}
%\mathbb{P} \left ( \mathcal{J}^j_0 \geq e^{- (1/2 - 3 \epsilon)h_t} \right ) & \leq \mathbb{P} \left ( \sup_{]-\infty, 0]} \mathcal{L}_{B}\big[\tau (B, 1), . \big] \geq e^{\epsilon h_t} \right ) \\
%& + \mathbb{P} \left ( \inf_{[\tilde{L}_{j}^{\sharp}, \tilde \tau_j^-(h / 2)]} \tilde V^{(j)}(.) \leq \left (\frac1{2} - \epsilon \right ) h_t \right ) \\
%& + \mathbb{P} \left ( \tilde \tau_j^-(h_t / 2) - \tilde{L}_{0} \geq e^{\epsilon h_t} \right ). 
%\end{align*}
%We bound the first term of the right hand side thanks to $(4.9)$ of \cite{AndDev} and the second
%%\[ \mathbb{P} \left ( \sup_{[\tilde{L}_{j}^{\sharp}, \tilde \tau_j^-(h / 2)]} e^{-\tilde V^{(j)}(.)} \leq e^{- \epsilon h} \right ) \geq 1 - e^{- \epsilon h}/4 \]
%%\[ \mathbb{P} \left ( \sup_{]-\infty, 0]} \mathcal{L}_{B}\big[\tau (B, 1), . \big] \leq e^{h/8} \right ) \geq 1 - 2e^{- \kappa h/4} \]
%thanks to lemma \ref{minoprevalley}. For the third term, note that lemma \ref{lemMi} gives
%\[ ... \]
\begin{eqnarray}
\mathbb{P} \left ( \mathcal{J}^j_0 < e^{- \epsilon h_t/2} \right ) \geq 1 - e^{-c_2 h_t}. \label{approxdeu3.1}
\end{eqnarray}

{For $\mathcal{J}^j_1$, recall the definition of $e_j$ in the beginning of this subsection and apply estimate (7.11) of \cite{advech} with $\delta = e^{-(1 - 2\epsilon) h_t /2}$, $\epsilon = e^{-(1 - 2\epsilon) h_t /6}$. We get 
\[ \mathbb{P} \left ( \sup_{z \in [-e^{-(1 - 2\epsilon) h_t /2}, e^{-(1 - 2\epsilon) h_t /2}]} \left | \mathcal{L}_{B^j}\left ( \tau (B^j, 1), z \right ) - e_j \right | \leq e^{-(1 - 2\epsilon) h_t /6} e_j \right ) \geq 1 - C e^{-(1 - 2\epsilon) h_t /60}, \] 
where $C$ is some positive constant. Combining with Lemma \ref{truccentral} we get
\[ \mathbb{P} \left ( \sup_{u \in [\tilde \tau_j^-(h_t/2), \tilde \tau_j^+(h_t/2)]} \left | \mathcal{L}_{B^j}\left ( \tau (B^j, 1),A^j(u)/A^j(\tilde{L}_j) \right ) - e_j \right | \leq e^{-(1 - 2\epsilon) h_t /6} e_j \right ) \geq 1 - e^{-c_6 h_t}, \] 
for some positive constant $c_6$, when $t$ is large enough. Putting in the expression of $\mathcal{J}^j_1$ we deduce that for $t$ large enough 
\begin{eqnarray}
\mathbb{P} \left ( \left | \mathcal{J}^j_1 - e_j \int_{\tilde \tau_j^-(h_t/2)}^{\tilde \tau_j(h_t/2)} e^{- \tilde V^{(j)}(u)} du \right | \leq e^{-(1 - 2\epsilon) h_t /6} e_j \int_{\tilde \tau_j^-(h_t/2)}^{\tilde \tau_j^+(h_t/2)} e^{- \tilde V^{(j)}(u)} du \right ) \geq 1 - e^{-c_6 h_t}. \label{approxdeu4}
\end{eqnarray}
Recall that by definition $R_j^t := \int_{\tilde{\tau}_j^-(h_t/2)}^{ \tilde{\tau}_j^+(h_t/2)} e^{-\tilde V^{(j)}(u)}  \textnormal{d}u$. }
%\mathbb{P} \left ( \right ) \geq 1 - e^{-c_5 h_t}
Then, 
\[ \int_{\tilde \tau_j^-(h_t/2)}^{\tilde \tau_j^+(h_t/2)} e^{- \tilde V^{(j)}(u)} du \geq \int_{\tilde \tau_j^+(\epsilon h_t/16)}^{\tilde \tau_j^+(\epsilon h_t/8)} e^{- \tilde V^{(j)}(u)} du \overset{\mathcal{L}}{=} \int_{\tau(V^{\uparrow}, \epsilon h_t/16)}^{\tau(V^{\uparrow}, \epsilon h_t/8)} e^{- V^{\uparrow}(u)} du, \]
where we have used Proposition \ref{standardwilliams} (applied with $h=h_t$) for the equality in law. We thus get
\begin{eqnarray}
\mathbb{P} \left ( \int_{\tilde \tau_j^-(h_t/2)}^{\tilde \tau_j^+(h_t/2)} e^{- \tilde V^{(j)}(u)} du \leq e^{- \epsilon h_t/8} \right ) \leq \mathbb{P} \left ( \tau(V^{\uparrow}, \epsilon h_t/8) - \tau(V^{\uparrow}, \epsilon h_t/16) \leq 1 \right ) \leq e^{-c_7 h_t}, \label{approxdeu0.3}
\end{eqnarray}
when $t$ is large enough, according to \eqref{03.10b} applied with $h=h_t, \alpha = \epsilon /8$, $\omega = \epsilon /16$, and where $c_7$ is a positive constant. 

We have $\mathbb{P}(e_j \leq e^{- \epsilon h_t/8}) \underset{t \rightarrow +\infty}{\sim} e^{- \epsilon h_t/8}/2$, so combining with \eqref{approxdeu4} and \eqref{approxdeu0.3} we get, for some positive constant $c_8$ and $t$ large enough, 
\begin{eqnarray}
\mathbb{P} \left ( \mathcal{J}^j_1 > e^{ - \epsilon h_t/5} \right ) \geq 1 - e^{-c_8 h_t}. \label{approxdeu5}
\end{eqnarray}

%Using $(\ref{approxdeu4})$ and the definition of $\delta_t$ we have
%\[ \bf U / A^j(\tilde{L}_j) \geq \mathcal{J}^j_1 \geq (1-e^{-(1-3\epsilon) h_t/3}) {\bf  e}_1 \int_{\tilde \tau_2^-(h_t/2)}^{\tilde \tau_2(h_t/2)} e^{- \tilde V^{(j)}(u)} du \]

Combining \eqref{approxdeu5} with \eqref{approxdeu3} and \eqref{approxdeu3.1} we get {that "$\mathcal{J}^j_0 + \mathcal{J}^j_2 \leq 2 e^{- \epsilon h_t/2} = 2 e^{- 3 \epsilon h_t/10} \times e^{ - \epsilon h_t/5}$ and $\mathcal{J}^j_1 > e^{ - \epsilon h_t/5}$" with probability at least $1 - e^{-c_9 h_t}$ for some positive constant $c_9$, and when $t$ is large enough. As a consequence, for $t$ large enough, 
\begin{eqnarray}
\mathbb{P} \left ( \mathcal{J}^j_0 + \mathcal{J}^j_2 \leq 2 e^{ - 3 \epsilon h_t/10} \mathcal{J}^j_1 \right ) \geq 1 - e^{-c_9 h_t}. \label{approxdeu6}
\end{eqnarray}
Now putting together \eqref{approxdeuori}, \eqref{approxdeu6}, and \eqref{approxdeu4}, we get the existence of a positive constant $c$ such that for $t$ large enough the following inequalities hold with probability greater than $1 - e^{-c h_t}$: 
\begin{align*}
\tilde h_j / A^j(\tilde{L}_j) & \geq \mathcal{J}^j_1 \geq (1 - e^{-(1 - 2\epsilon) h_t /6}) R_j^t e_j \text{   and   } \\
\tilde h_j / A^j(\tilde{L}_j) & = \mathcal{J}^j_0+\mathcal{J}^j_1+\mathcal{J}^j_2 \leq ( 1+2 e^{ - 3 \epsilon h_t/10} ) \mathcal{J}^j_1 \leq ( 1+2 e^{ - 3 \epsilon h_t/10} ) \times (1 + e^{-(1 - 2\epsilon) h_t /6}) R_j^t e_j. 
\end{align*}
Since, for $t$ large enough, we have $(1 - e^{-(1 - 2\epsilon) h_t /6}) \geq (1-e^{-\epsilon h_t/6})$ and $( 1+2 e^{ - 3 \epsilon h_t/10} ) \times (1 + e^{-(1 - 2\epsilon) h_t /6}) \leq (1+e^{-\epsilon h_t/6})$, we have proved \eqref{eqPropIV}, as required. }

\end{proof}

We now approximate $A^j(\tilde{L}_j)$. 

%\begin{lemme} \label{approxa}
%
%For each $j \in \{ 1, ..., n_t \}$, $S_j^t := \int_{\tilde \tau_j(h_t /2)}^{\tilde L_j} e^{\tilde V^{(j)}(u)} du$ is independent of $({\bf e}_j, R_j^t)$ and such that
%\[ \mathbb{P} \left ( S_j^t \leq A_V^j(\tilde{L}_j) \leq (1 + e^{-(1/2 - \eta)h_t}) S_j^t \right ) \geq 1 - e^{-c h_t},  \]
%for some positive constant $c$, when $t$ is large enough. 
%
%\end{lemme}
\begin{lemme} \label{approxa}

For each $j \geq 1$, $S_j^t := \int_{\tilde \tau_j^+(h_t /2)}^{\tilde L_j} e^{\tilde V^{(j)}(u)} du$ is independent from $R_j^t$ and such that
\[ \mathbb{P} \left ( S_j^t \leq A^j(\tilde{L}_j) \leq (1 + e^{-h_t / 7}) S_j^t \right ) \geq 1 - e^{-c h_t},  \]
for some positive constant $c$, when $t$ is large enough. 

\end{lemme}

\begin{proof}

%$S_j^t = \int_{\tilde \tau_j^+(h_t /2)}^{\tilde L_j} e^{\tilde V^{(j)}(u)} du$ so $(S_j^t, R_j^t)$ only depends on $V$ and is thus independent from $e_j$ which only depends on $B^j$. 
Note that, for any $j \geq 1$, $R_j^t$ is a function of $\tilde P_1^{(j)}$ and $(\tilde P_2^{(j)}(x), 0 \leq x \leq \tau (\tilde P_2^{(j)}, h_t /2))$ while $S_j^t$ is a function of $(\tilde P_2^{(j)}(x), \tau (\tilde P_2^{(j)}, h_t /2) \leq x \leq \tau (\tilde P_2^{(j)}, h_t))$ and $\tilde P_3^{(j)}$. According to Proposition \ref{standardwilliams} (applied with $h=h_t$) the processes $\tilde P_1^{(j)}$, $\tilde P_2^{(j)}$, and $\tilde P_3^{(j)}$ are independent and $\tilde P_2^{(j)}$ is equal in law to $(V^{\uparrow}(x), 0 \leq x \leq \tau(V^{\uparrow}, h))$. In particular, we can apply the Markov property to $\tilde P_2^{(j)}$ at time $\tau (\tilde P_2^{(j)}, h_t /2)$ and we get that $(\tilde P_2^{(j)}(x), 0 \leq x \leq \tau (\tilde P_2^{(j)}, h_t /2))$ and $(\tilde P_2^{(j)}(x), \tau (\tilde P_2^{(j)}, h_t /2) \leq x \leq \tau (\tilde P_2^{(j)}, h_t))$ are independent. We conclude that $S_j^t$ and $R_j^t$ are independent. 
%The random variables $e_j, S_j^t$ and $R_j^t$ are therefore mutually independent. 
Then, from the definitions of $S_j^t$ and $A^j(\tilde{L}_j)$ we have
\begin{eqnarray}
S_j^t \leq A^j(\tilde{L}_j) = \int_{\tilde m_j}^{\tilde \tau_j^+(h_t / 2)} e^{\tilde V^{(j)}(u)} du + S_j^t. \label{approxa0}
\end{eqnarray}
{According to Proposition \ref{standardwilliams} (applied with $h=h_t$) the first term in the right hand side is equal in law to $\int_0^{\tau ( V^{\uparrow}, h_t /2 )} e^{V^{\uparrow}(u)} d u$.} Applying estimate \eqref{0MajorationAVallee} with $h = h_t/2$, $\eta = 1/3$, we get
\begin{eqnarray}
\mathbb{P} \left ( \int_{\tilde m_j}^{\tilde \tau_j^+(h_t / 2)} e^{\tilde V^{(j)}(u)} du \leq e^{4 h_t / 6} \right ) \geq 1-e^{-h_t}, \label{approxa1}
\end{eqnarray}
when $t$ is large enough. Then, 
\[ S_j^t \geq \int_{\tilde \tau_j^+(h_t / 2)}^{\tilde \tau_j(h_t)} e^{\tilde V^{(j)}(u)} du = \int_{\tilde m_j}^{\tilde \tau_j(h_t)} e^{\tilde V^{(j)}(u)} du - \int_{\tilde m_j}^{\tilde \tau_j^+(h_t/2)} e^{\tilde V^{(j)}(u)} du. \]

{According to Proposition \ref{standardwilliams} (applied with $h=h_t$) the first term in the right hand side is equal in law to $\int_0^{\tau ( V^{\uparrow}, h_t)} e^{V^{\uparrow}(u)} d u$ while the second is equal in law to $\int_0^{\tau ( V^{\uparrow}, h_t /2 )} e^{V^{\uparrow}(u)} d u$. For the first term we apply estimate \eqref{0MinorationAVallee} with $h = h_t$, $\eta = 1/6$, and for the second term we apply \eqref{0MajorationAVallee} with $h = h_t/2$, $\eta = 1/3$. We get the existence of a constant $c_1$ such that for $t$ large enough $\mathbb{P} (\int_{\tilde m_j}^{\tilde \tau_j(h_t)} e^{\tilde V^{(j)}(u)} du \geq e^{5 h_t/6 }, \int_{\tilde m_j}^{\tilde \tau_j^+(h_t/2)} e^{\tilde V^{(j)}(u)} du \leq e^{4 h_t/6 }) \geq 1-e^{- c_1 h_t}$. As a consequence, for $t$ large enough, }
\begin{eqnarray}
\mathbb{P} \left ( S_j^t \geq e^{5 h_t/6 } - e^{4 h_t/6 } \right ) \geq 1-e^{- c_1 h_t}. \label{approxa2}
\end{eqnarray}
Putting together $(\ref{approxa1})$ and $(\ref{approxa2})$ we get that 
\[ \mathbb{P} \left ( \int_{\tilde m_j}^{\tilde \tau_j^+(h_t / 2)} e^{\tilde V^{(j)}(u)} du / S_j^t \leq e^{-h_t/7} \right ) \geq 1-e^{- c_2 h_t}, \]
for some positive constant $c_2$, when $t$ is large enough. This, combined with $(\ref{approxa0})$, yields the result. 

\end{proof}

{\begin{proof} of Proposition \ref{approxparliid}

Recall from the beginning of this subsection that the sequence $(B^j)_{j \geq 1}$ is \textit{iid} and, according to Remark \ref{iid} (applied with $h=h_t$) the sequence $(\tilde V^{(j)} (x + \tilde m_j),\ \tilde L_{j-1} - \tilde m_j \leq x \leq \tilde L_{j} - \tilde m_j)_{j \geq 1}$ is \textit{iid}. Then, note that for each $j \geq 1$, $e_j$ is some function, not depending on $j$, of the Brownian motion $B^j$, and $(S_j^t, R_j^t)$ is some function, not depending on $j$, of the valley $(\tilde V^{(j)} (x + \tilde m_j),\ \tilde L_{j-1} - \tilde m_j \leq x \leq \tilde L_{j} - \tilde m_j)$. We deduce that the sequences $( e_j )_{j \geq 1}$, and $( S_j^t, R_j^t)_{j \geq 1}$ are \textit{iid}. Moreover, these two sequences are independent since $( S_j^t, R_j^t)_{j \geq 1}$ is a function of $V$ and, as we said in the beginning of this subsection, the sequence $(e_j)_{j \geq 1}$ is independent from $V$. According to Lemma \ref{approxa}, we have that for each $j \geq 1$ the random variables $S_j^t$ and $R_j^t$ are independent so we deduce that the sequences $( e_j )_{j \geq 1}$, $( S_j^t )_{j \geq 1}$ and $(R_j^t)_{j \geq 1}$ are \textit{iid} and mutually independent, as claimed in the proposition. 

Choose $\epsilon \in ]0, \max \{1/8, (1- (1+\delta) \kappa)/2 \}[$. Applying \eqref{approxiid2}, \eqref{approxiid3}, 
%the \textit{iid} character of the sequence $(\tilde{\mathcal{L}}_j,\tilde h_j)_{1 \leq j \leq n_t}$, 
Lemma \ref{approxdeu} and Lemma \ref{approxa} simultaneously for all $j \in \{1,...,n_t\}$, we obtain the existence of a positive constant $c_1$ (depending on $\delta$ and $\epsilon$) such that for $t$ large enough, the following relations hold with probability greater than $1-n_t e^{-c_1 h_t}$: 
\begin{eqnarray*}
\forall j \geq 1, \ \mathcal{L}_X(H(\tilde{L}_j), \tilde m_j) = A^j(\tilde{L}_j) e_j \left |
    \begin{array}{ll}
        \leq (1 + e^{-h_t / 7}) e_j S_j^t, \\ 
        \geq e_j S_j^t, 
    \end{array}
\right . 
\end{eqnarray*} 
\begin{eqnarray*} 
H(\tilde L_j)-H(\tilde m_j) = \tilde h_j \left |
    \begin{array}{ll}
        \leq (1 + e^{-\epsilon h_t/6}) A^j(\tilde{L}_j) R_j^t e_j \leq (1 + e^{-\epsilon h_t/6}) \times (1 + e^{-h_t / 7}) e_j S_j^t R_j^t, \\
        \geq (1 - e^{-\epsilon h_t/6}) A^j(\tilde{L}_j) R_j^t e_j \geq (1 - e^{-\epsilon h_t/6}) e_j S_j^t R_j^t. 
    \end{array}
\right . 
\end{eqnarray*} 
When $t$ is large enough, the factors $(1 + e^{-h_t / 7})$ and $(1 + e^{-\epsilon h_t/6}) \times (1 + e^{-h_t / 7})$ are smaller than $(1 + e^{- \epsilon h_t / 7})$ while the factors $1$ and $(1 - e^{-\epsilon h_t/6})$ are greater than $(1 - e^{- \epsilon h_t / 7})$. We thus obtain that the inequalities asserted in Proposition \ref{approxparliid} hold, for $t$ large enough, with a probability greater than $1-n_t e^{-c_1 h_t}$. Finally, note that $n_t \sim e^{\kappa (1+\delta) \phi(t)}$ {(see the definition of $n_t$ in the beginning of this section)} where $\phi(t) << \log(t) \sim h_t$. As a consequence, there is a positive constant $c$ such that $n_t e^{-c_1 h_t} \leq e^{-c h_t}$ when $t$ is large enough. The proposition follows. 

\end{proof}

In order to translate "$\mathcal{L}_X^*(t) /t \leq \alpha$" in term of events only involving the sequence $( e_i S_i^t,  e_i S_i^t R_i^t)_{i \geq 1}$, we need to compare $N_t$, the number of $h_t$-valleys visited until instant $t$ (see \eqref{defnt}), with the overshoots of $\sum_{i=1}^{.} e_i S_i^t R_i^t$. We now do this thanks to Proposition \ref{approxparliid}. Recall the definition of the overshoots $\mathcal{N}_a$ in the beginning of Subsection \ref{proofmainth}. We have :} 

\begin{lemme} \label{overshoot}

Fix $\epsilon$ as in Proposition \ref{approxparliid} and $\eta \in ]0, 1[$. Assume that $t$ is so large such that $(1-e^{- \epsilon h_t/7})^{-1} < (1+\eta)$ and $(1+e^{- \epsilon h_t/7})^{-1} (1 - 2/\log(h_t)) \geq (1-\eta)$. Recall the definitions of the events $\mathcal{V}_{n_t, h_t}$, $\mathcal{A}^1_t$ and $\mathcal{A}^5_t$ introduced in respectively Lemma \ref{minimacoincide}, Fact \ref{lemtps} and Proposition \ref{approxparliid}. Then
%\[ \bigcap_{j=1}^{ n_t}\left\{0\leq H(\tilde m_{j})-\sum_{i=1}^{j-1}\left ( H(\tilde L_i)-H(\tilde m_i) \right ) \leq \frac{2t}{\log h_t} \right\} \cap \cap_{j=1}^{n_t} \left \{ (1-e^{- \epsilon h_t/6}) e^t_j S_j^t R_j^t \leq H(\tilde L_j)-H(\tilde m_j) \leq (1+e^{- \epsilon h_t/6}) e^t_j S_j^t R_j^t \right \} \subset \left \{ mathcal{N}_{(1-\epsilon)t} \leq N_t \leq \mathcal{N}_{(1+\epsilon)t} \right \} \]
\[ \mathcal{V}_{n_t, h_t} \cap \left \{ N_t < n_t \right \} \cap \mathcal{A}^1_t \cap \mathcal{A}^5_t \subset \left \{ \mathcal{N}_{(1-\eta)t} \leq N_t \leq \mathcal{N}_{(1+\eta)t} \right \}. \]

\end{lemme}

Even though it is used in the following subsection, an other interest of this lemma is that we expect it to be useful in the future study of the almost sure behavior of $\mathcal{L}_X^*(t)$. Indeed, for the almost sure behavior, the contribution of the last valley can be sometimes omitted, sometimes totally included, so we are left to study some behavior of $\sum_{i=1}^{N_t - 1} e_i S_i^t R_i^t$ or $\sum_{i=1}^{N_t} e_i S_i^t R_i^t$. The above lemma allows to replace $N_t$ by some $\mathcal{N}_a$, which is more convenient since it only depends on the sequence $(e_i S_i^t R_i^t, \ i \geq 1)$. 

\begin{proof}

Recall that by convention $\sum_{i=1}^{0}... = 0$. Assume that the event $\mathcal{V}_{n_t, h_t} \cap \left \{ N_t < n_t \right \} \cap \mathcal{A}^1_t \cap \mathcal{A}^5_t$ is realized. Then, for any $k \in \{1,...,n_t\}$ we have
%\[ k \leq N_t \Rightarrow H(\tilde m_k) \leq t \Rightarrow \sum_{i=1}^{k-1}\left ( H(\tilde L_i)-H(\tilde m_i) \right ) \leq t \Rightarrow \sum_{i=1}^{k-1} e_i S_i^t R_i^t \leq t (1-e^{- \epsilon h_t/6})^{-1} \leq (1+\epsilon)t \Rightarrow k \leq \mathcal{N}_{(1+\epsilon)t}, \]
\begin{align*}
N_t \geq k & \Rightarrow H(m_k) \leq t \Rightarrow H(\tilde m_k) \leq t \Rightarrow \sum_{i=1}^{k-1}\left ( H(\tilde L_i)-H(\tilde m_i) \right ) \leq t \\
& \Rightarrow \sum_{i=1}^{k-1} e_i S_i^t R_i^t \leq t (1-e^{- \epsilon h_t/7})^{-1} \leq (1+\eta)t \Rightarrow \mathcal{N}_{(1+\eta)t} \geq k, 
\end{align*}
and
\begin{align*}
\mathcal{N}_{(1-\eta)t} \geq k & \Rightarrow \sum_{i=1}^{k-1} e_i S_i^t R_i^t \leq (1-\eta)t \Rightarrow \sum_{i=1}^{k-1}\left ( H(\tilde L_i)-H(\tilde m_i) \right ) \leq (1-\eta) (1+e^{- \epsilon h_t/7}) t \\
& \Rightarrow H(\tilde m_k) \leq t \left [ (1-\eta) (1+e^{- \epsilon h_t/7}) + 2/\log(h_t) \right ] \leq t \Rightarrow H(m_k) \leq t \Rightarrow N_t \geq k. \\
\end{align*}

We have thus proved that $\mathcal{N}_{(1-\eta)t} \leq N_t \leq \mathcal{N}_{(1+\eta)t}$ is satisfied on $\mathcal{V}_{n_t, h_t} \cap \left \{ N_t < n_t \right \} \cap \mathcal{A}^1_t \cap \mathcal{A}^5_t$. 

\end{proof}

\subsection{Proof of Proposition \ref{analogue5.1}} \label{approxderep}

We now use the preceding results of this section to approximate the distribution function of $\mathcal{L}_X^*(t)/t$ by a distribution function involving the sequence $(e_i S_i^t, e_i S_i^t R_i^t, \ i \geq 1)$. 
%Theorem \ref{cvdutl} will follow from the combination of this and from the results of Subsection \ref{reprise}. This finishes to extend the arguments of \cite{advech} to our context. For simplicity, we keep the same scheme as they have and precise, for all their results, thanks to which estimate they can be generalized. 
We first state three facts. 
%We recall some of their notations: 
%\[ \bar{\ell}_k:=Y_1(ke^{-\kappa \phi(t)})=\frac{1}{t}\sum_{i=1}^k e^t_j S_j^t, \ \ \ \ \ \bar{\mathcal{H}}_k:= Y_2(ke^{-\kappa \phi(t)})= \frac{1}{t}\sum_{i=1}^k e^t_j S_j^t R_j^t, \]
%and $\mathcal{N}_t^{\epsilon} :=\inf\{k \geq 1, \bar{\mathcal{H}}_k >1- \epsilon \}$. 

{Recall the definition of $\mathcal{D}_j$ in \eqref{defdj}. We also recall the notations $X_{\tilde m_j} := X(. + H(\tilde m_j))$ and $H_{X_{\tilde m_j}}(r) := \tau(X_{\tilde m_j}, r)$, the fact that, according to the Markov property at $H(\tilde m_j)$, $X_{\tilde m_j}$ is a diffusion starting from $\tilde m_j$, 
%We also define for any $r \in \mathbb{R}$, $H_{X_{\tilde m_j}}(r) := \tau(X_{\tilde m_j}, r)$, the hitting time of $r$ by $X_{\tilde m_j}$. 
and the fact that $\mathcal{L}_{X_{\tilde m_j}}(.,.)$ denotes the local time of $X_{\tilde m_j}$. The first fact says that the supremum of the local time on $\mathcal{D}_j$ can be approximated by the local time at $\tilde m_j$. 

\begin{fact} \label{analogue(5.22)}

There is a positive constant $c$ such that for $t$ large enough, 
\[ \mathbb{P} \left ( \mathcal{A}^7_t := \cap_{j=1}^{n_t} \left \{ \sup_{y \in \mathcal{D}_j} \mathcal{L}_{X_{\tilde m_j}}(H_{X_{\tilde m_j}}(\tilde L_j),y) \leq (1+e^{-h_t/9}) \mathcal{L}_X(H(\tilde L_j),\tilde m_j) \right \} \right ) \geq 1 - e^{-ch_t}. \]

\end{fact}
}

The next two facts come from Lemmas 5.2 and 5.3 of \cite{advech}. 

\begin{fact} \label{analogue5.2}

Fix $\epsilon$ as in Proposition \ref{approxparliid}. For any $k \geq 1$ let us define the distribution functions, depending on $t$, 
 \begin{align*}
F_{\gamma, k}(x) &:= \mathbb{P} \left(\max_{ 1\leq j \leq k-1} \mathcal{L}_X(H(\tilde L_j),\tilde m_j) \leq \gamma t, \ \sum_{j=1}^{k-1} \left ( H(\tilde L_j)-H(\tilde m_j) \right ) \leq xt \right), \\
F^{\pm}_{\gamma, k}(x) &:= \mathbb{P} \left(\max_{ 1\leq j \leq k-1} e_j S_j^t \leq \gamma t (1 \pm 2e^{- \epsilon h_t/7}), \ \sum_{j=1}^{k-1} e_j S_j^t R_j^t \leq x t (1 \pm 2 e^{- \epsilon h_t/7})  \right). 
\end{align*} 
%where $\epsilon_t$ is as in Proposition \ref{approxparliid}. 
Then, there is a positive constant $c$ such that for all $t$ large enough: 
 \begin{align*}
\forall \ 1 \leq k \leq n_t, \forall \ 0<x \leq 1, \forall \ \gamma>0, \ F^-_{\gamma, k}(x) -e^{-c h_t} \leq F_{\gamma, k}(x)  \leq F^+_{\gamma, k}(x) +e^{-c h_t}. 
\end{align*} 
{Note that the conventions $\max_{1 \leq i \leq 0}... = 0$ and $\sum_{i=1}^{0}... = 0$ imply that $F_{\gamma, 1}$ and $F^{\pm}_{\gamma, 1}$ are constant equal to $1$.} 

\end{fact}

\begin{proof}
This is a direct consequence of Proposition \ref{approxparliid}. 
\end{proof}

%Recall the notations $X_{\tilde m_j}$, $H_{X_{\tilde m_j}}(r)$, $\mathcal{L}_{X_{\tilde m_j}}(.,.)$. 

\begin{fact} \label{analogue5.3}

Fix $\epsilon$ as in Proposition \ref{approxparliid} {and recall the definition of $\mathcal{D}_j$ in \eqref{defdj}}. 
%and recall the definition of $X_{\tilde m_1}$ in Subsection \ref{neghalfline}. 
We define the distribution functions, depending on $t$, 
\begin{align*}
f_{\gamma}(x) & := \mathbb{P} \left ( \mathcal{L}_{X_{\tilde m_1}}(t(1-x), \tilde m_1) \leq \gamma t, H_{X_{\tilde m_1}}(\tilde L_1)>t(1-x), H_{X_{\tilde m_1}}(\tilde L_{1})< H_{X_{\tilde m_1}}(\tilde L_{0}) \right ), \\
\tilde f_{\gamma}(x) & := \mathbb{P} \left ( \sup_{y \in \mathcal{D}_1} \mathcal{L}_{X_{\tilde m_1}}(t(1-x), y) \leq \gamma t, H_{X_{\tilde m_1}}(\tilde L_1)>t(1-x), H_{X_{\tilde m_1}}(\tilde L_{1})< H_{X_{\tilde m_1}}(\tilde L_{0}) \right ), \\
f_{\gamma}^{\pm}(x) & := \mathbb{P} \left( 1/R_1^t \leq \gamma (1\pm 2 e^{- \epsilon h_t/7}) / (1-x), \ e_1 S_1^t R_1^t > t(1-x)(1 \mp 2 e^{- \epsilon h_t/7}) \right). 
\end{align*}
%where $\epsilon_t$ is as in Proposition \ref{approxparliid}. 
Fix $\eta \in ]0, 1/2[$. Then, there is a positive constant $c$ such that for all $t$ large enough: 
%Then, there is a positive constant $c$ such that for any $0<x \leq 1$ and $\gamma>0$ possibly depending on $t$, 
\[ \forall \ x \in ]\eta, 1-\eta[, \forall \ \gamma>0, \ f^-_{\gamma}(x) - e^{-c h_t} \leq \tilde f_{\gamma}(x) \leq f_{\gamma}(x) \leq f^+_{\gamma}(x) + e^{-c h_t}. \]
%\[ f^-_{\gamma}(x) - e^{-c h_t} \leq f_{\gamma}(x)  \leq f^+_{\gamma}(x) + e^{-c h_t} \text{ and } f^-_{\gamma}(x) - ... \leq \tilde f_{\gamma}(x)  \leq f^+_{\gamma}(x) + ...,, \]

%Let us define the distribution functions
%\begin{align*}
%f_{\gamma}(x) & := \mathbb{E} \left[ P_{\tilde m_1}^{V}(\mathcal{L}_{X'}(t(1-x),\tilde m_1) \leq \gamma t, H'(\tilde L_1)>t(1-x), H'(\tilde L_{1})< H'(\tilde L_{0}) )\right], \\
%\tilde f_{\gamma}(x) & := \mathbb{E} \left[ P_{\tilde m_1}^{V} \left ( \sup_{y \in \mathcal{D}_1} \mathcal{L}_{X'}(t(1-x), y) \leq \gamma t, H'(\tilde L_1)>t(1-x), H'(\tilde L_{1})< H'(\tilde L_{0}) \right ) \right], 
%\end{align*}
%and
%\[ f_{\gamma}^{\pm}(x) := \mathbb{P} \left(\frac{1}{R_1} \leq \frac{\gamma}{1-x}(1\pm 2 \epsilon_t), \mathcal{H}_1 > t(1-x)(1 \mp 2 \epsilon_t) \right) \]
%where $\epsilon_t$ is as in Proposition \ref{approxparliid}. Then, there is a positive constant $c$ such that for any $0<x \leq 1$ and $\gamma>0$ possibly depending on $t$, 
%\[ f^-_{\gamma}(x) - e^{-c h_t} \leq f_{\gamma}(x)  \leq f^+_{\gamma}(x) + e^{-c h_t} \text{ and } f^-_{\gamma}(x) - ... \leq \tilde f_{\gamma}(x)  \leq f^+_{\gamma}(x) + ...,, \]
%for all $t$ large enough. 

\end{fact}

For the justification of Facts \ref{analogue(5.22)} and \ref{analogue5.3} in our context, we give some details in Subsection \ref{justoffacts}. 

%In order to generalize Lemma 5.4 of \cite{advech} to our context, l
Let us recall the definitions of the functionals $\tilde K_{I, a}$ and $\tilde K_{I, a}^-$ defined in Section 4.3 of \cite{advech}: 
\[ \forall a > 0, \ {\forall (f_1, f_2) \in D([0, +\infty[, \mathbb{R}^2),} \ \tilde K_{I, a}(f_1, f_2) := f_2 (f_2^{-1}(a)), \ \tilde K_{I, a}^-(f_1, f_2) := f_2 (f_2^{-1}(a)-). \]
{In the above expression we have used the notations defined just before Theorem \ref{cvdutl}.} Note that the functionals $\tilde K_{I, a}$ and $\tilde K_{I, a}^-$ actually do not involve $f_1$. According to Lemma 4.5 of \cite{advech}, {a realization of the $\kappa$-stable subordinator $(\mathcal{Y}_1,\mathcal{Y}_2)$, defined just before Theorem \ref{cvdutl}}, is almost surely a point of continuity of these functionals. Thanks to this we can prove: 

{
\begin{lemme} \label{analogue5.4}
\begin{align}
%& \underset{\eta \rightarrow 0}{\lim} \ \underset{t \rightarrow +\infty}{\limsup} \left ( s(\eta,t) := \sum_{k \leq n_t} \mathbb{P} \left ( \frac1{t} \sum_{i=1}^k e_i S_i^t R_i^t > 1-\eta/2, \ 1-2\eta < \frac1{t} \sum_{i=1}^{k-1} e_i S_i^t R_i^t \leq 1-\eta \right ) \right )= 0. \label{rec1} \\
& \underset{\eta \rightarrow 0}{\lim} \ \underset{t \rightarrow +\infty}{\limsup} \left ( s(\eta,t) := \sum_{k=1}^{+\infty} \mathbb{P} \left ( 1-5\eta/4 < \frac1{t} \sum_{i=1}^{k-1} e_i S_i^t R_i^t \leq 1-3\eta/4, \ \frac1{t} e_k S_k^t R_k^t > 3\eta/4 \right ) \right )= 0. \label{rec1factsim} \\
& \underset{\eta \rightarrow 0}{\lim} \ \underset{t \rightarrow +\infty}{\limsup} \left [ \tilde s(\eta,t) := 1 - \mathbb{P} \left( \eta t \leq  H(m_{N_t}) \leq (1-\eta) t \right) \right ] = 0 \label{NNeps}. 
\end{align}
Recall that we use the convention $\sum_{i=1}^{0}... = 0$ (therefore the first term in the sum defining $s(\eta,t)$ is actually equal to $0$). 
\end{lemme}

%This result is less precise than Lemma 5.4 of \cite{advech} of which it is the analogous, but even in \cite{advech}, they only need that the limits in $t$ converge to $0$ when $\eta$ goes to $0$. We will thus be able to substitute our Lemma to theirs. 

%ON A SUPPRIME \eqref{rec1} (VOIR PREUVE DE LA PROP 4.1) ET ON A RAJOUTE LE FAIT SIMILAIRE \eqref{rec1factsim} DONT ON A BESOIN. IL FAUT FAIRE ATTENTION AU $\epsilon$ QUI VA PEUT ETRE CHANGER (LE CHANGEMENT EST FAIT MAIS VERIFIER) VERIFIER CETTE NOUVELLE PREUVE, VERIFIER LES REPERCURSIONS SUR LA PREUVE DU DEUXIEME POINT. VERIFIER LES CHANGEMENTS DANS LA PREUVE DE LA PROP. 

\begin{proof}
%We extend the \text{iid} sequence $(\mathcal{H}_{k}, k \leq n_t)$ to be defined define $s(a,t)$ to be the sum in \eqref{rec1} without the bound $n_t$: 
%\[ s(a,t) := \sum_{k \geq 1} \mathbb{P} \left ( \bar{\mathcal{H}}_{k} > 1-a/2, \ 1-2a < \bar{\mathcal{H}}_{k-1}\leq 1-a \right ) \]
%CA OU BIEN LA CV DE $Y_2^t$ SE FAISAIT DEJA AVEC LA BORNE $n_t$ ? DANS CE CAS ON DIT DIRECTEMENT $=s(a,t)$ ET L'ENONCE PORTE UNIQUEMENT SUR LA LIMITE
Recall the definition of $(Y_1, Y_2)^t$ in \eqref{defy1y2t} and the definition of $\mathcal{N}_a$ in the beginning of Subsection \ref{proofmainth}. For any $\eta \in ]0, 4/5[$, $t > 0$ and $k \geq 1$, we have 
\begin{align*}
& \left \{ 1-5\eta/4 < \frac1{t} \sum_{i=1}^{k-1} e_i S_i^t R_i^t \leq 1-3\eta/4, \ \frac1{t} e_k S_k^t R_k^t > 3\eta/4 \right \} \\
\subset & \left \{ 1-5\eta/4 < \frac1{t} \sum_{i=1}^{k-1} e_i S_i^t R_i^t \leq 1-3\eta/4, \ \frac1{t} \sum_{i=1}^{k} e_i S_i^t R_i^t > 1-\eta/2 \right \} \\
= & \left \{ \mathcal{N}_{(1-\eta/2)t} = k, \ 1-5\eta/4 < \frac1{t} \sum_{i=1}^{\mathcal{N}_{(1-\eta/2)t}-1} e_i S_i^t R_i^t \leq 1-3\eta/4 \right \} \\
\subset & \left \{ \mathcal{N}_{(1-\eta/2)t} = k, \ \frac1{t} \sum_{i=1}^{\mathcal{N}_{(1-\eta/2)t}-1} e_i S_i^t R_i^t > 1-5\eta/4 \right \} =: E_k. 
\end{align*}
The events $E_k, \ k \geq 1$ are clearly disjoint. We thus have 
\begin{align*}
s(\eta,t) & \leq \sum_{k=1}^{+\infty} \mathbb{P} \left ( E_k \right ) = \mathbb{P} \left ( \cup_{k=1}^{+\infty} E_k \right ) = \mathbb{P} \left ( \frac1{t} \sum_{i=1}^{\mathcal{N}_{(1-\eta/2)t}-1} e_i S_i^t R_i^t > 1-5\eta/4 \right ) \\
& = \mathbb{P} \left ( Y_2^t(Y_2^{t, -1}(1-\eta/2)-) > 1-5\eta/4 \right ) = \mathbb{P} \left ( \tilde K_{I, 1-\eta/2}^-[(Y_1, Y_2)^t] > 1-5\eta/4 \right ). 
\end{align*}
%\noindent $\left \{ \frac1{t} \sum_{i=1}^k e_i S_i^t R_i^t > 1-\eta/2, \ 1-2\eta < \frac1{t} \sum_{i=1}^{k-1} e_i S_i^t R_i^t \leq 1-\eta \right \}$ for $k \in \{1, ..., n_t \}$ are clearly disjoint so 
%%\[ s(a,t) = \mathbb{P} \left ( \frac1{t} \sum_{i=1}^{\mathcal{N}_{(1-a)t}} e_i S_i^t R_i^t > 1-a/2, \ 1-2a < \frac1{t} \sum_{i=1}^{\mathcal{N}_{(1-a)t} - 1} e_i S_i^t R_i^t \leq 1-a, \ \mathcal{N}_{(1-a)t} \leq n_t \right ), \]
%%and this probability is nothing but
%\begin{align*}
%s(\eta,t) = & \mathbb{P} \left ( \frac1{t} \sum_{i=1}^{\mathcal{N}_{(1-\eta)t}} e_i S_i^t R_i^t > 1-\eta/2, \ 1-2\eta < \frac1{t} \sum_{i=1}^{\mathcal{N}_{(1-\eta)t} - 1} e_i S_i^t R_i^t \leq 1-\eta, \ \mathcal{N}_{(1-\eta)t} \leq n_t \right ) \\
%= & \mathbb{P} \left ( Y_2^t(Y_2^{t, -1}(1-\eta)) > 1-\eta/2, \ Y_2^t(Y_2^{t, -1}(1-\eta)-) > 1-2\eta, \ \mathcal{N}_{(1-\eta)t} \leq n_t \right ) \\
%= & \mathbb{P} \left ( \tilde K_{I, 1-\eta}[(Y_1, Y_2)^t] > 1-\eta/2, \ \tilde K_{I, 1-\eta}^-[(Y_1, Y_2)^t] > 1-2\eta, \ \mathcal{N}_{(1-\eta)t} \leq n_t \right ). 
%\end{align*}
Now, according to Proposition \ref{propcvsub} (to be proved in the following subsection), $(Y_1, Y_2)^t$ converges to $(\mathcal{Y}_1, \mathcal{Y}_2)$, defined just before Theorem \ref{cvdutl}, for the convergence in distribution in $D(\mathbb{R}_+, \mathbb{R}^2)$ with the $J_1$ topology. Moreover we have that for any fixed $\eta \in ]0, 4/5[$, $(\mathcal{Y}_1, \mathcal{Y}_2)$ is almost surely a point of continuity for $\tilde K_{I, 1-\eta/2}^-$. Using the \textit{continuous mapping theorem} we thus get
\begin{eqnarray}
\underset{t \rightarrow +\infty}{\limsup} \ s(\eta,t) \leq \mathbb{P} \left ( \mathcal{Y}_2(\mathcal{Y}_2^{-1}(1-\eta/2)-) > 1-5\eta/4 \right ). \label{contdestrucs1}
\end{eqnarray}
We now study the limit when $\eta$ goes to $0$. Since $\mathcal{Y}_2$ is a $\kappa$-stable subordinator, it is known that almost surely $Z := 1 - \mathcal{Y}_2(\mathcal{Y}_2^{-1}(1)-) > 0$ so $\mathcal{Y}_2^{-1}(1-\eta/2) = \mathcal{Y}_2^{-1}(1)$ for all $0 < \eta < 2 Z$. Then, for all $0 < \eta < 4Z/5 < 2Z$, $\mathcal{Y}_2(\mathcal{Y}_2^{-1}(1-\eta/2)-) = \mathcal{Y}_2(\mathcal{Y}_2^{-1}(1)-) = 1 - Z \leq 1-5\eta/4$. This proves that for almost every realization of $\mathcal{Y}_2$, the event in the probability in \eqref{contdestrucs1} fails to happen for all $\eta$ small enough, so by dominated convergence, this probability converges to $0$ when $\eta$ goes to $0$. This proves \eqref{rec1factsim}. 

We now prove \eqref{NNeps}. Let us fix $\epsilon$ as in Proposition \ref{approxparliid}, $\eta \in ]0, 1/3[$, put $\epsilon_t := e^{- \epsilon h_t/7}$ and choose $t$ large enough so that: $2 /\log h_t \leq \eta, (1-\epsilon_t)^{-1} \leq 2, (1- 3\eta) \leq (1- 2\eta) (1+\epsilon_t)^{-1}$ and Lemma \ref{overshoot} applies for our choice of $\eta$. {Recall the definitions of the events $\mathcal{V}_{n_t, h_t}$, $\mathcal{A}^1_t$ and $\mathcal{A}^5_t$ introduced in respectively Lemma \ref{minimacoincide}, Fact \ref{lemtps} and Proposition \ref{approxparliid}.} Using successively the definitions of $\mathcal{V}_{n_t, h_t} \cap \{ N_t \leq n_t \}$, $\mathcal{A}^1_t \cap \{ N_t \leq n_t \}$, $\mathcal{A}^5_t \cap \{ N_t \leq n_t \}$, and then Lemma \ref{overshoot}, we get that $\mathbb{P} ( \eta t \leq  H(m_{N_t}) \leq (1-\eta) t )$ is greater than
\begin{align*}
& \mathbb{P} \left( \eta t \leq  H(\tilde m_{N_t}) \leq (1-\eta) t, \ \mathcal{V}_{n_t, h_t}, \ N_t \leq n_t, \ \mathcal{A}^1_t, \mathcal{A}^5_t \right) \\
\geq & \mathbb{P} \left( \eta t \leq \sum_{i=1}^{N_t-1}\left ( H(\tilde L_i)-H(\tilde m_i) \right ) \leq (1- 2\eta) t \leq \left (1-\eta - \frac{2}{\log h_t} \right ) t, \ \mathcal{V}_{n_t, h_t}, \ N_t \leq n_t, \ \mathcal{A}^1_t, \mathcal{A}^5_t \right) \\
\geq & \mathbb{P} \left( \eta (1-\epsilon_t)^{-1} t \leq 2 \eta t \leq \sum_{i=1}^{N_t-1} e_i S_i^t R_i^t \leq (1- 3\eta) t \leq (1- 2\eta) (1+\epsilon_t)^{-1} t, \ \mathcal{V}_{n_t, h_t}, \ N_t \leq n_t, \ \mathcal{A}^1_t, \mathcal{A}^5_t \right) \\
\geq & \mathbb{P} \left( 2 \eta t \leq \sum_{i=1}^{\mathcal{N}_{(1-\eta)t}-1} e_i S_i^t R_i^t \leq \sum_{i=1}^{\mathcal{N}_{(1+\eta)t}-1} e_i S_i^t R_i^t \leq (1- 3\eta) t, \ \mathcal{V}_{n_t, h_t}, \ N_t \leq n_t, \ \mathcal{A}^1_t, \mathcal{A}^5_t \right) \\
= & \mathbb{P} \left( 2\eta \leq Y_2^t(Y_2^{t, -1}(1-\eta)-) \leq Y_2^t(Y_2^{t, -1}(1+\eta)-) \leq (1-3\eta), \ \mathcal{V}_{n_t, h_t}, \ N_t \leq n_t, \ \mathcal{A}^1_t, \mathcal{A}^5_t \right) \\
= & \mathbb{P} \left( 2\eta \leq \tilde K_{I, 1-\eta}^-[(Y_1, Y_2)^t] \leq \tilde K_{I, 1+\eta}^-[(Y_1, Y_2)^t] \leq (1-3\eta), \ \mathcal{V}_{n_t, h_t}, \ N_t \leq n_t, \ \mathcal{A}^1_t, \mathcal{A}^5_t \right). 
\end{align*}

According to Lemma \ref{minimacoincide} (applied with $n = n_t, h=h_t$), Fact \ref{lemtps}, Proposition \ref{approxparliid} and Lemma \ref{nbvalleesvisit} of the following subsection, we have 
\[ \mathbb{P} \left (\mathcal{V}_{n_t, h_t} \cap \{ N_t \leq n_t \} \cap \mathcal{A}^1_t \cap \mathcal{A}^5_t \right ) \underset{t \rightarrow + \infty}{\longrightarrow} 1. \]
Then, combining with the convergence of $(Y_1, Y_2)^t$ to $(\mathcal{Y}_1, \mathcal{Y}_2)$, the continuity of $\tilde K_{I, 1-\eta}^-$ and $\tilde K_{I, 1+\eta}^-$ at $(\mathcal{Y}_1, \mathcal{Y}_2)$, and the \textit{continuous mapping theorem}, we get
\begin{eqnarray}
\underset{t \rightarrow +\infty}{\limsup} \ \tilde s(\eta,t) \leq 1 - \mathbb{P} \left ( 2 \eta \leq \mathcal{Y}_2(\mathcal{Y}_2^{-1}(1-\eta)-) \leq \mathcal{Y}_2(\mathcal{Y}_2^{-1}(1+\eta)-) \leq (1-3\eta) \right ). \label{contdestrucs2}
\end{eqnarray}
We have almost surely $E_2 := \min \{ \mathcal{Y}_2(\mathcal{Y}_2^{-1}(1)) - 1, \ 1 - \mathcal{Y}_2(\mathcal{Y}_2^{-1}(1)-), \ \mathcal{Y}_2(\mathcal{Y}_2^{-1}(1)-) \} > 0$ so $\mathcal{Y}_2^{-1}(1 + \eta) = \mathcal{Y}_2^{-1}(1)$ for all $0 \leq \eta \leq E_2$, and for $\eta \leq E_2 /3$: $\mathcal{Y}_2(\mathcal{Y}_2^{-1}(1+\eta)-) = \mathcal{Y}_2(\mathcal{Y}_2^{-1}(1)-) \leq 1 - E_2 \leq 1-3\eta$. Also, $\mathcal{Y}_2^{-1}(1 - \eta) = \mathcal{Y}_2^{-1}(1)$ for all $0 \leq \eta \leq E_2$, and for $0 \leq \eta \leq E_2/2$: $\mathcal{Y}_2(\mathcal{Y}_2^{-1}(1-\eta)-) = \mathcal{Y}_2(\mathcal{Y}_2^{-1}(1)-) \geq E_2 \geq 2\eta$. This proves that, almost surely, the event in the probability in \eqref{contdestrucs2} happens for all $\eta$ small enough, so by dominated convergence, this probability converges to $1$ when $\eta$ goes to $0$. This proves \eqref{NNeps}. 

\end{proof}
}

{
\begin{proof} of Proposition \ref{analogue5.1}

The proof follows the one of Proposition 5.1 in \cite{advech}. However, we have to use systematically our estimates and notations instead of the ones in \cite{advech} and we can make some simplifications of their arguments. For the sake of clarity we here give the details. 

\textbf{Upper bound: }

\textit{Main ideas:} The main stake of the proof is to get independence between what happens in the last valley and what happened before. The idea is to first apply the Markov property at $H(\tilde m_{N_t})$ in order to get independence in the quenched setting and then to approximate each of the two parts by objects depending on portions of the environment that are independent from one another. We then get true independence and use Facts \ref{analogue5.2} and \ref{analogue5.3} to approximate the distribution function of the quantity of interest by the distribution function of a functional of the sequence $( e_i S_i^t,  e_i S_i^t R_i^t)_{i \geq 1}$, leading to the result. 
%that we have to consider the last valley for a time $t - H(\tilde m_{N_t})$, which creates a lack of independence between that valley and the previous ones. We need to work in a quenched setting, which allows to use the Markov property at $H(\tilde m_{N_t})$. 

We have 
\[ \mathbb{P} \left( \mathcal{L}_X^*(t)/t \leq \alpha \right) = E \left [ P^V \left ( \mathcal{L}_X^*(t)/t \leq \alpha \right ) \right ] \leq E \left [ P^V \left ( \max_{1 \leq j \leq N_t} \mathcal{L}_X(t, \tilde m_j)/t \leq \alpha \right ) \right ]. \]
Recall the notations $X_{\tilde m_j} := X(. + H(\tilde m_j))$, $H_{X_{\tilde m_j}}(r) := \tau(X_{\tilde m_j}, r)$ and $\mathcal{L}_{X_{\tilde m_j}}(.,.)$ introduced in Subsection \ref{maincontibcommeiid}. Let us define $X_{\tilde L_j} := X(H(\tilde L_j) + .)$ 
%\ \ \text{and} \ \ X_i^* := X(H(\tilde \tau_i(h)) + .), \] LA DIFF PARTANT DE \tilde L_i^* N'EST UTILISEE QU'IMPLICITEMENT QUAND JE DIS QUE LE MEME ARGUEMENT QU'ADV MARCHE
which is, according to the Markov property at $H(\tilde L_j)$, a diffusion in the environment $V$ starting from $\tilde L_j$. We also define for any $r \in \mathbb{R}$, $H_{X_{\tilde L_j}}(r) := \tau(X_{\tilde L_j}, r)$, the hitting time of $r$ by $X_{\tilde L_j}$. Now let us define the event 
\begin{eqnarray}
\mathcal{A}^6_t := \bigcap_{j=1}^{n_t} \left \{ H_{X_{\tilde m_j}}(\tilde L_j) < H_{X_{\tilde m_j}}(\tilde L_{j-1}), \ H_{X_{\tilde L_j}}(+\infty) < H_{X_{\tilde L_j}}(\tilde \tau_j(h_t)) \right \}. \label{defA6}
\end{eqnarray}
According to the combination of Lemmas \ref{noreturn} and \ref{sortparladroite}, there is a positive constant $c$ such that for $t$ large enough we have $\mathbb{P} ( \overline{\mathcal{A}^6_t} ) \leq e^{-c h_t}$. On $\{ N_t < n_t \} \cap \mathcal{V}_{n_t, h_t}$ the sequences $(m_j)_{j \geq 1}$ and $(\tilde m_j)_{j \geq 1}$ coincide until the index $j=n_t > N_t$, so in particular $\tilde m_{N_t}$ is the last $\tilde m_j$ visited by the diffusion before instant $t$. On $\mathcal{A}^6_t \cap \{ N_t < n_t \} \cap \mathcal{V}_{n_t, h_t}$, for any $j < N_t$, $\tilde m_j$ is no longer visited after $H(\tilde L_j)$ which is strictly less than $t$ so we have $\mathcal{L}_X(t, \tilde m_j) = \mathcal{L}_X(H(\tilde L_j), \tilde m_j)$. For $\eta \in ]0, 1/2[$ we thus have 
%$\mathcal{V}_{n_t, h_t}$ EST ICI POUR DEUX RAISONS : 1) ON TRAVAIL AVEC NOS TRUCS JUSQU'A $N_t$, OR LES ESTIMES SONT VRAIS JUSQU'A $n_t$, DONC ON A BESOIN QUE LES MINIMA COINCIDENT EN PLUS DE $N_t < n_t$, 2) POUR LE TEMPS LOCAL AU TEMPS $t$ = LE TL EN LA SORTIE DE SA VALLEE. CES DEUX TRUCS SONT AUSSI PRIS EN COMPTE DANS LE LOWER BOUND. 
\begin{align}
\mathbb{P} \left( \mathcal{L}_X^*(t)/t \leq \alpha \right) \leq & E \left [ P^V \left ( \mathcal{L}_X(t, \tilde m_{N_t}) \vee \max_{1 \leq j \leq N_t-1} \mathcal{L}_X(H(\tilde L_j), \tilde m_j) \leq \alpha t, \right. \right. \nonumber \\
& \left. \left. \eta t \leq  H(\tilde m_{N_t}) \leq (1-\eta) t, \mathcal{V}_{n_t, h_t}, N_t < n_t, \mathcal{A}^6_t, \mathcal{A}^1_t \right ) \right ] + u_1(t, \eta) \nonumber \\ 
& =: E \left [ F^V \right ] + u_1(t, \eta), \label{reecriturede5.1-0}
\end{align}
where we have put $u_1(t, \eta) = 1-\mathbb{P}(\eta t \leq  H(m_{N_t}) \leq (1-\eta) t, \mathcal{V}_{n_t, h_t}, N_t < n_t, \mathcal{A}^6_t, \mathcal{A}^1_t)$. Note that $F^V$ also depends on $\eta$ and $t$. According to \eqref{NNeps}, Lemma \ref{minimacoincide} (applied with $n=n_t, h=h_t$), Lemma \ref{nbvalleesvisit} of the following subsection, the estimate $\mathbb{P} ( \overline{\mathcal{A}^6_t} ) \leq e^{-c h_t}$, and Fact \ref{lemtps}, 
%\begin{eqnarray}
%\mathbb{P} \left( \mathcal{L}_X^*(t)/t \leq \alpha \right) \leq E \left [ F^V \right ] + u_1(t, \eta), \label{reecriturede5.1-0}
%\end{eqnarray}
we have $\lim_{\eta \rightarrow 0} \limsup_{t \rightarrow +\infty} u_1(t,\eta) = 0$. For any fixed trajectory of $V$ and $k \in \{1,...,n_t\}$ let us define 
%A PRIORI C'EST OK POUR $\mathcal{V}_{n_t, h_t}$ MAIS VERIFIER, VOIR SI AUTRE CHOSE LE REND NECESSAIRE ET VOIR PREUVE ORIGINALE. 
%in $\mathcal{V}_{n_t, h_t}$ (C'EST NECESSAIRE ? SI OUI $\mathcal{V}_{n_t, h_t}$ DOIT ETRE RAJOUTE) 
\begin{align*}
h_1^{V,k}(y) := P^V & \left ( \max_{1 \leq j \leq k-1} \mathcal{L}_X(H(\tilde L_j), \tilde m_k) \leq \alpha t, \right. \\ 
& \left. 0\leq H(\tilde m_{k})-\sum_{i=1}^{k-1}\left ( H(\tilde L_i)-H(\tilde m_i) \right ) \leq \frac{2t}{\log h_t}, \ H(\tilde m_k) \leq yt \right ) \\
h_2^{V,k}(y) := P^V & \left ( \mathcal{L}_{X_{\tilde m_k}}(t(1-y), \tilde m_k) \leq \alpha t, \right. \\ 
& \left. t(1-y) < H_{X_{\tilde m_k}}(\tilde m_{k+1}) < H_{X_{\tilde m_k}}(\tilde L_{k-1}), \ H_{X_{\tilde m_k}}(\tilde m_{k+1}) - H_{X_{\tilde m_k}}(\tilde L_{k}) < \frac{2t}{\log h_t} \right ). 
\end{align*}
Recall that by convention quantities such as $\sum_{i=1}^{k-1}...$ and $\max_{1 \leq j \leq k-1} ...$ are set to equal $0$ in the case $k=1$. 
%VERIFIER QUE TOUT MARCHE BIEN MAINTENANT POUR LE CAS $k=1$, NOTAMMENT AVEC TOUTES LES APPLICATIONS DU FACT \ref{analogue5.2} (OK MAIS DU COUP LE PREMIER TERME DES SOMMES EST QUASIMENT TOUJOURS EGAL A 0, IL FAUDRA PEUT-ETRE PRECISER QU'ON A TENU COMPTE DU CAS $k=1$ ET QUE CA DONNE DES SOMMES DONT LE PREMIER TERME EST SOUVENT 0, NOTAMMENT DANS LE LEMME 4.14). ET AUSSI QU'IL N'Y A PLUS DE PROBLEME DE SOMMES AMBIGUES AILLEURS DANS L'ARTICLE. SI, DANS LE LEMME 4.14, OK

%NOTE : IL VA FALLOIR DELOCALISER ICI LA DEFINITION DES $H_{X_{\tilde m_k}}$ (ok) ET TOUT CA. QUOI D'AUTRE ? FIXER $\eta$ (ok). ATTENTION AUX GAMMA QUI DEVRAIENT ETRE DES ALPHA (ok), LES OVERSHHOT (ok), $\mathcal{R}$ (ok). A PART CA, OK POUR L'UPPER BOUND , RESTE LE LOWER BOUND (OK, TRUCS MANQUANTS NOTES) ET LE LEMME 4.14 (ok)

Now, in $F^V$, we partition on the possible values of $N_t$ in $\{1,...,n_t\}$ and $H(\tilde m_{N_t})$ in $[\eta, 1-\eta]$ (recall that $\tilde m_{N_t}$ is the last $\tilde m_j$ visited by the diffusion before instant $t$ since we are on $\{ N_t < n_t \} \cap \mathcal{V}_{n_t, h_t}$), and we apply the strong Markov property at $H(\tilde m_k)$ under the quenched probability measure $P^V$. We obtain 
%INDEP PAR STRONG MP at $H(\tilde m_k)$. + DEFO OF $F^V$ (+ MINIMA COINCIDENT). 
\begin{eqnarray}
F^V \leq \sum_{k=1}^{n_t} \int_{\eta}^{1-\eta} h_2^{V,k}(y) dh_1^{V,k}(y). \label{reecriturede5.1-1}
\end{eqnarray}
We can approximate the $H(\tilde m_k)$ in $h_1^{V,k}(y)$ by $\sum_{i=1}^{k-1}( H(\tilde L_i)-H(\tilde m_i) )$. Let us define 
\begin{align*}
\tilde h_1^{V,k}(y) := P^V & \left ( \max_{1 \leq j \leq k-1} \mathcal{L}_X(H(\tilde L_j), \tilde m_k) \leq \alpha t, \ 0\leq H(\tilde m_{k})-\sum_{i=1}^{k-1}\left ( H(\tilde L_i)-H(\tilde m_i) \right ) \leq \frac{2t}{\log h_t}, \right. \\
& \left. \sum_{i=1}^{k-1}\left ( H(\tilde L_i)-H(\tilde m_i) \right ) + \frac{2t}{\log h_t} \leq yt \right ). 
\end{align*}
Clearly we have $\tilde h_1^{V,k}(y) \leq h_1^{V,k}(y)$. In order to replace $h_1^{V,k}(y)$ by $\tilde h_1^{V,k}(y)$ in \eqref{reecriturede5.1-1} we need to proceed an integration by parts, which we can do since the functions $\tilde h_1^{V,k}$, $h_1^{V,k}$ and $h_2^{V,k}$ are positive increasing. 
\begin{align}
\int_{\eta}^{1-\eta} h_2^{V,k}(y) dh_1^{V,k}(y) & = \left [ h_1^{V,k}(y) h_2^{V,k}(y) \right ]_{\eta}^{1-\eta} - \int_{\eta}^{1-\eta} h_1^{V,k}(y) dh_2^{V,k}(y) \nonumber \\
& \leq \left [ h_1^{V,k}(y) h_2^{V,k}(y) \right ]_{\eta}^{1-\eta} - \int_{\eta}^{1-\eta} \tilde h_1^{V,k}(y) dh_2^{V,k}(y) \nonumber \\
& = \left [ \left ( h_1^{V,k}(y) - \tilde h_1^{V,k}(y) \right ) h_2^{V,k}(y) \right ]_{\eta}^{1-\eta} + \int_{\eta}^{1-\eta} h_2^{V,k}(y) d\tilde h_1^{V,k}(y) \nonumber \\
& \leq \left ( h_1^{V,k}(1-\eta) - \tilde h_1^{V,k}(1-\eta) \right ) h_2^{V,k}(1-\eta) + \int_{\eta}^{1-\eta} h_2^{V,k}(y) d\tilde h_1^{V,k}(y). \label{reecriturede5.1-1.1}
\end{align}
Then, note that we have $d\tilde h_1^{V,k}(y) \leq d h_3^{V,k}(y)$ where 
\[ h_3^{V,k}(y) := P^V \left ( \max_{1 \leq j \leq k-1} \mathcal{L}_X(H(\tilde L_j), \tilde m_k) \leq \alpha t, \ \sum_{i=1}^{k-1}\left ( H(\tilde L_i)-H(\tilde m_i) \right ) + \frac{2t}{\log h_t} \leq yt \right ). \]
We thus get that 
\[ \int_{\eta}^{1-\eta} h_2^{V,k}(y) dh_1^{V,k}(y) \leq \left ( h_1^{V,k}(1-\eta) - \tilde h_1^{V,k}(1-\eta) \right ) h_2^{V,k}(1-\eta) + \int_{\eta}^{1-\eta} h_2^{V,k}(y) d h_3^{V,k}(y). \]
The interest of the previous manipulations is that the two functions in the integral now depend on independent parts of the environment: $h_3^{V,k}(y)$ is measurable with respect to the $\sigma$-field $\sigma (V(x), x \leq \tilde L_{k-1})$ while $h_2^{V,k}(y)$ is measurable with respect to the $\sigma$-field $\sigma (V^{\tilde L_{k-1}}(x), x \geq 0)$. These two $\sigma$-fields are independent according to Remark \ref{iid}. Integrating with respect to the environment and summing over $k$ we thus obtain 
\begin{align}
\sum_{k=1}^{n_t} E \left [ \int_{\eta}^{1-\eta} h_2^{V,k}(y) dh_1^{V,k}(y) \right ] & \leq \sum_{k=1}^{n_t} E \left [ \left ( h_1^{V,k}(1-\eta) - \tilde h_1^{V,k}(1-\eta) \right ) h_2^{V,k}(1-\eta) \right ] \nonumber \\
& + \sum_{k=1}^{n_t} \int_{\eta}^{1-\eta} E[h_2^{V,k}(y)] dE[ h_3^{V,k}(y)]. \label{reecriturede5.1-2}
\end{align}

Let us now prove that the sum over $k$ of the terms $E [ ( h_1^{V,k}(1-\eta) - \tilde h_1^{V,k}(1-\eta) ) h_2^{V,k}(1-\eta) ] $ is negligible. By the definitions of $h_1^{V,k}(1-\eta)$ and $\tilde h_1^{V,k}(1-\eta)$ we have 
\[ 0 \leq h_1^{V,k}(1-\eta) - \tilde h_1^{V,k}(1-\eta) \leq P^V \left ( (1-\eta)t - \frac{2t}{\log h_t} < \sum_{i=1}^{k-1}\left ( H(\tilde L_i)-H(\tilde m_i) \right ) \leq (1-\eta)t \right ) =: h_4^{V,k}(1-\eta). \]
Note that $h_4^{V,k}(1-\eta)$ is measurable with respect to the $\sigma$-field $\sigma (V(x), x \leq \tilde L_{k-1})$ while $h_2^{V,k}(1-\eta)$ is measurable with respect to the $\sigma$-field $\sigma (V^{\tilde L_{k-1}}(x), x \geq 0)$. Integrating with respect to the environment and summing over $k$ we thus obtain 
\begin{eqnarray}
\sum_{k=1}^{n_t} E \left [ \left ( h_1^{V,k}(1-\eta) - \tilde h_1^{V,k}(1-\eta) \right ) h_2^{V,k}(1-\eta) \right ] \leq \sum_{k=1}^{n_t} E[h_4^{V,k}(1-\eta)] \times E[h_2^{V,k}(1-\eta)]. \label{reecriturede5.1-2.1}
\end{eqnarray}
%\begin{align*}
%E \left [ \left ( h_1^{V,k}(1-\eta) - \tilde h_1^{V,k}(1-\eta) \right ) h_2^{V,k}(1-\eta) \right ] & \leq E[h_4^{V,k}(1-\eta)] \times E[h_2^{V,k}(1-\eta)] \\
%& \leq E[h_4^{V,1}(1-\eta)] \times E[h_2^{V,k}(1-\eta)], 
%\end{align*}
%where $E[h_4^{V,k}(1-\eta)] = E[h_4^{V,1}(1-\eta)]$ comes from the fact that the diffusion conditionally to the environment is Markovian and from the fact that the valleys are \textit{iid}, by Remark \ref{iid}. 
Let us fix an $\epsilon$, chosen as in Proposition \ref{approxparliid}, for the rest of this proof. Using the definition of $h_4^{V,k}$ and Proposition \ref{approxparliid} we get, for some $c$ and when $t$ is large enough, 
\begin{align}
E[h_4^{V,k}(1-\eta)] & = \mathbb{P} \left ( (1-\eta)t - \frac{2t}{\log h_t} < \sum_{i=1}^{k-1}\left ( H(\tilde L_i)-H(\tilde m_i) \right ) \leq (1-\eta)t \right ) \nonumber \\
& \leq \mathbb{P} \left ( (1+e^{- \epsilon h_t/7})^{-1} [(1-\eta) - 2/\log h_t]t < \sum_{i=1}^{k-1} e_i S_i^t R_i^t \leq (1-e^{- \epsilon h_t/7})^{-1} (1-\eta)t \right ) + e^{-c h_t}. \label{reecriturede5.1-3}
\end{align}
Similarly, since $H_{X_{\tilde m_k}}(\tilde L_k) = H(\tilde L_k)-H(\tilde m_k)$, we have for $E[h_2^{V,k}(1-\eta)]$: 
\begin{align}
E[h_2^{V,k}(1-\eta)] & \leq \mathbb{P} \left ( H(\tilde L_k)-H(\tilde m_k) > (\eta - 2/\log h_t) t \right ) \nonumber \\
& \leq \mathbb{P} \left ( e_k S_k^t R_k^t > (1+e^{- \epsilon h_t/7})^{-1} (\eta - 2/\log h_t) t \right ) + e^{-c h_t}. \label{reecriturede5.1-4}
\end{align}

Now, if $t$ is large enough so that $(1+e^{- \epsilon h_t/7})^{-1} [(1-\eta) - 2/\log h_t] > 1-5\eta/4$, $(1-e^{- \epsilon h_t/7})^{-1} (1-\eta) < 1-3\eta/4$, and $(1+e^{- \epsilon h_t/7})^{-1} (\eta - 2/\log h_t) > 3\eta/4$, then we get from \eqref{reecriturede5.1-3} and \eqref{reecriturede5.1-4} that $E[h_4^{V,k}(1-\eta)] \times E[h_2^{V,k}(1-\eta)]$ is less than 
\begin{align*}
& \mathbb{P} \left ( 1-5\eta/4 < \frac1{t} \sum_{i=1}^{k-1} e_i S_i^t R_i^t \leq 1-3\eta/4 \right ) \times \mathbb{P} \left ( \frac1{t} e_k S_k^t R_k^t > 3\eta/4 \right ) + 3 e^{-c h_t} \\
= & \mathbb{P} \left ( 1-5\eta/4 < \frac1{t} \sum_{i=1}^{k-1} e_i S_i^t R_i^t \leq 1-3\eta/4, \ \frac1{t} e_k S_k^t R_k^t > 3\eta/4 \right ) + 3 e^{-c h_t}. 
%\leq & \mathbb{P} \left ( 1-5\eta/4 < \frac1{t} \sum_{i=1}^{k-1} e_i S_i^t R_i^t \leq 1-3\eta/4, \ \frac1{t} \sum_{i=1}^{k} e_i S_i^t R_i^t > 1-\eta/2 \right ) + 3 e^{-c h_t} \\
%=: & \mathbb{P} \left ( E_k \right ) + 3 e^{-c h_t} 
\end{align*}
In the second line, we have used the fact that the sequence $( e_j S_j^t R_j^t)_{j \geq 1}$ is \textit{iid}, according to Proposition \ref{approxparliid} (also, note that the convention $\sum_{i=1}^{0}... = 0$ implies that the above probability is equal to $0$ when $k=1$, in the rest of the proof we shall not distinguish the case $k=1$ even though this case is often trivial). 
%The events $E_k, \ k \in \{ 1,...,n_t \}$ are disjoint so we get 
%\begin{align*}
%\sum_{k=1}^{n_t} \mathbb{P} \left ( E_k \right ) & = \mathbb{P} \left ( \cup_{k=1}^{n_t} E_k \right ) = \mathbb{P} \left ( \frac1{t} \sum_{i=1}^{\mathcal{N}_{(1-\eta/2)t}-1} e_i S_i^t R_i^t > 1-5\eta/4 \right ) \\
%& = \mathbb{P} \left ( Y_2^t(Y_2^{t, -1}(1-\eta/2)-) > 1-5\eta/4 \right ) \\
%& = \mathbb{P} \left ( \tilde K_{I, 1-\eta/2}^-[(Y_1, Y_2)^t] > 1-5\eta/4 \right ). 
%\end{align*}
In conclusion, for $t$ large enough, 
\[ \sum_{k=1}^{n_t} E[h_4^{V,k}(1-\eta)] \times E[h_2^{V,k}(1-\eta)] \leq s(\eta,t) + 3 n_t e^{-c h_t}, \]
where $s(\eta,t)$ is defined in Lemma \ref{analogue5.4}. Since $n_t \sim e^{\kappa (1+\delta) \phi(t)}$ (see the definition of $n_t$ in the beginning of this section) where $\phi(t) << \log(t) \sim h_t$, $n_t$ is negligible compared to quantities of the type $e^{-c h_t}$. The term $3 n_t e^{-c h_t}$ thus converges to $0$ when $t$ goes to infinity. Using \eqref{rec1factsim} we thus get 
%Similarly as for proving \eqref{contdestrucs2}, we can use Proposition \ref{propcvsub} and the \textit{continuous mapping theorem} and we deduce that 
\begin{eqnarray}
\underset{\eta \rightarrow 0}{\lim} \ \underset{t \rightarrow +\infty}{\limsup} \ \sum_{k=1}^{n_t} E[h_4^{V,k}(1-\eta)] \times E[h_2^{V,k}(1-\eta)] = 0. \label{reecriturede5.1-5}
\end{eqnarray}
%As in ..., it is not difficult to see that the right hand side goes to $0$ when $\eta$ goes to $0$. 

%LE PROBLEME C'EST QU'IL FAUT UN TRUC A LA PLACE DU LEMME 6.2 DE ADV (QUI A REMPLACE LES BASIC COMPUTATIONS DES VERSIONS PRECEDENTES). AVEC UN PEU DE CHANCE ON PEUT LE METTRE EN FACT ET DIRE QU'IL SE PROUVE PAREIL QUE DANS ADV.  AUTRE IDEE : ON UTILISE L'INDEPENDENCE DES DEUX FACTEURS POUR EN FAIRE UNE SEULE PROBA ET VOIR SI ON A UNE SOMME D'EVENEMENTS DISJOINTS QUI PEUVENT S'EXPRIMER GRACE A NOS FONCTIONNELLES QUI CONVERGENT. ON DIRA "At this point, we simplify a little the original proof from \cite{advech}." ON PEUT AUSSI DIRE QUE LEUR PREUVE SE REPETE FACILEMENT. 

Putting together \eqref{reecriturede5.1-1}, \eqref{reecriturede5.1-2}, \eqref{reecriturede5.1-2.1} and \eqref{reecriturede5.1-5} we obtain 
\begin{eqnarray}
E[F_V] \leq \sum_{k=1}^{n_t} \int_{\eta}^{1-\eta} E[h_2^{V,k}(y)] dE[ h_3^{V,k}(y)] + u_2(t,\eta), \label{reecriturede5.1-5.0}
\end{eqnarray}
where $u_2(t,\eta)$ is such that $\lim_{\eta \rightarrow 0} \limsup_{t \rightarrow +\infty} u_2(t,\eta) = 0$. 

We now study the integral on $[\eta, 1-\eta]$. Note that by definitions of $h_3^{V,k}(y)$ and of the annealed probability measure we have $E[ h_3^{V,k}(y)] = F_{\alpha, k}(y - 2/\log h_t) \geq F^-_{\alpha, k}(y - 2/\log h_t) -e^{-c h_t}$ where $F_{\alpha, k}$ and $F^-_{\alpha, k}$ are defined in Fact \ref{analogue5.2} and $c$ is some positive constant. The last inequality comes from that Fact and is true simultaneously for all $k \in \{ 1,..., n_t \}$ and $y \in [\eta, 1-\eta]$ when $t$ is large enough. Integrating by part twice as in \eqref{reecriturede5.1-1.1} (which we can do since the functions $E[h_2^{V,k}(.)]$, $E[ h_3^{V,k}(.)]$ and $F^-_{\alpha, k}$ are positive increasing) and proceeding the change of variable $u = y - 2/\log h_t$ we get 
\begin{align}
& \sum_{k=1}^{n_t} \int_{\eta}^{1-\eta} E[h_2^{V,k}(y)] dE[ h_3^{V,k}(y)] \leq \sum_{k=1}^{n_t} \int_{\eta - 2/\log h_t}^{1-\eta - 2/\log h_t} E[h_2^{V,k}(u + 2/\log h_t)] dF^-_{\alpha, k}(u) \nonumber \\
& + n_t e^{-c h_t} + \sum_{k=1}^{n_t} \left ( F_{\alpha, k}(1-\eta - 2/\log h_t) - F^-_{\alpha, k}(1-\eta - 2/\log h_t) \right ) E[h_2^{V,k}(1-\eta)]. \label{reecriturede5.1-5.1}
\end{align}

Let us discuss the factors $E[h_2^{V,k}(.)]$ that appear in two terms of the right hand side. Since the diffusion conditionally to the environment is Markovian and the valleys are \textit{iid} (by Remark \ref{iid}), we see that $E[h_2^{V,k}(y)] = E[h_2^{V,1}(y)]$. Then, by definition of $h_2^{V,1}$ we have 
\begin{align}
E[h_2^{V,1}(y)] = & \mathbb{P} \left ( \mathcal{L}_{X_{\tilde m_1}}(t(1-y), \tilde m_1) \leq \alpha t, \right. \nonumber \\ 
& \left. t(1-y) < H_{X_{\tilde m_1}}(\tilde m_{2}) < H_{X_{\tilde m_1}}(\tilde L_{0}), \ H_{X_{\tilde m_1}}(\tilde m_{2}) - H_{X_{\tilde m_1}}(\tilde L_{1}) < \frac{2t}{\log h_t} \right ) \nonumber \\
\leq & \mathbb{P} \left ( \mathcal{L}_{X_{\tilde m_1}}(t(1-y), \tilde m_1) \leq \alpha t, \ t(1-y - 2/\log h_t) < H_{X_{\tilde m_1}}(\tilde L_{1}) < H_{X_{\tilde m_1}}(\tilde L_{0}) \right ) \nonumber \\
\leq & \mathbb{P} \left ( \mathcal{L}_{X_{\tilde m_1}}(t(1-y - 2/\log h_t), \tilde m_1) \leq \alpha t, \ t(1-y - 2/\log h_t) < H_{X_{\tilde m_1}}(\tilde L_{1}) < H_{X_{\tilde m_1}}(\tilde L_{0}) \right ) \nonumber \\
= & f_{\alpha}(y + 2/\log h_t) \leq f^+_{\alpha}(y + 2/\log h_t) + e^{-c h_t}, \label{reecriturede5.1-5.2}
\end{align}
where $f_{\alpha}$ and $f^+_{\alpha}$ are defined in Fact \ref{analogue5.3} and $c$ is some positive constant. The last inequality comes from that Fact and is true uniformly in $y \in [\eta, 1-\eta]$ for $t$ large enough. 

We now prove that the third term in the right hand side of \eqref{reecriturede5.1-5.1} is negligible. According to Fact \ref{analogue5.2}, there is a positive constant $c$ such that for any $z \in ]0,1[$ 
\begin{align*}
F_{\alpha, k}(z) - F^-_{\alpha, k}(z) \leq F^+_{\alpha, k}(z) - F^-_{\alpha, k}(z) + e^{-c h_t} & \leq \mathbb{P} \left(\max_{ 1\leq j \leq k-1} \frac{e_j S_j^t}{t} \in [\alpha (1 - 2e^{- \epsilon h_t/7}), \alpha (1 + 2e^{- \epsilon h_t/7})], \right ) \\
& + \mathbb{P} \left( \frac1{t} \sum_{j=1}^{k-1} e_j S_j^t R_j^t \in [z (1 - 2 e^{- \epsilon h_t/7}), z (1 + 2 e^{- \epsilon h_t/7})] \right) \\
& + e^{-c h_t}. 
\end{align*}
Combining this with \eqref{reecriturede5.1-5.2} we can bound the third term in the right hand side of \eqref{reecriturede5.1-5.1} as follows 
\begin{align}
& \sum_{k=1}^{n_t} \left ( F_{\alpha, k}(1-\eta - 2/\log h_t) - F^-_{\alpha, k}(1-\eta - 2/\log h_t) \right ) E[h_2^{V,k}(1-\eta)] \nonumber \\
\leq & \left ( f^+_{\alpha}(1-\eta + 2/\log h_t) + e^{-c h_t} \right ) \times \left [ \sum_{k=1}^{n_t} \mathbb{P} \left(\max_{ 1\leq j \leq k-1} \frac{e_j S_j^t}{t} \in [\alpha (1 - 2e^{- \epsilon h_t/7}), \alpha (1 + 2e^{- \epsilon h_t/7})] \right ) \right. \nonumber \\
+ & \left. \sum_{k=1}^{n_t} \mathbb{P} \left( \frac1{t} \sum_{j=1}^{k-1} e_j S_j^t R_j^t \in [(1-\eta - 2/\log h_t) (1 - 2 e^{- \epsilon h_t/7}), (1-\eta - 2/\log h_t) (1 + 2 e^{- \epsilon h_t/7})] \right) + n_t e^{-c h_t} \right ] \nonumber \\
\leq & \sum_{k=1}^{n_t} (k-1) \mathbb{P} \left( \frac{e_1 S_1^t}{t} \in [\alpha (1 - 2e^{- \epsilon h_t/7}), \alpha (1 + 2e^{- \epsilon h_t/7})] \right ) \nonumber \\
+ & f^+_{\alpha}(1-\eta + 2/\log h_t) \times \sum_{k=1}^{n_t} \mathbb{P}\left( \frac1{t} \sum_{j=1}^{k-1} e_j S_j^t R_j^t \in I(t,\eta) \right) + 4n_t e^{-c h_t} =: S_1 + S_2 + 4n_t e^{-c h_t}. \label{reecriturede5.1-5.3}
\end{align}
For the last inequality we have used the fact that $f^+_{\alpha}(.) \leq 1$ (since it's a probability), that the sequences $( e_j )_{j \geq 1}$ and $( S_j^t )_{j \geq 1}$ are \textit{iid}, and we have set $I(t,\eta) := [(1-\eta - 2/\log h_t) (1 - 2 e^{- \epsilon h_t/7}), (1-\eta - 2/\log h_t) (1 + 2 e^{- \epsilon h_t/7})]$. We have 
\begin{align*}
S_1 & \leq n_t^2 \mathbb{P} \left( e_1 \in \left [ \frac{\alpha (1 - 2e^{- \epsilon h_t/7})}{S_1^t/t}, \frac{\alpha (1 + 2e^{- \epsilon h_t/7})}{S_1^t/t} \right ], \ S_1^t/t > e^{- \epsilon h_t/14} \right ) + n_t^2 \mathbb{P} \left( S_1^t/t \leq e^{- \epsilon h_t/14} \right ) \\
& \leq n_t^2 \mathbb{E} \left [ \frac1{2} \int_{\frac{\alpha (1 - 2e^{- \epsilon h_t/7})}{S_1^t/t}}^{\frac{\alpha (1 + 2e^{- \epsilon h_t/7})}{S_1^t/t}} e^{-u/2} du \times \mathds{1}_{\{ S_1^t/t > e^{- \epsilon h_t/14} \}} \right ] + n_t^2 \mathbb{P} \left( S_1^t/t \leq e^{- \epsilon h_t/14} \right ) \\
& \leq n_t^2 \frac{4 \alpha e^{- \epsilon h_t/7}}{2 e^{- \epsilon h_t/14}} + n_t^2 o(e^{- \epsilon h_t/15}) = o(1). 
\end{align*}
For the first term we have used that $e_1$ follows an exponential distribution with parameter $1/2$. The upper bound for the second term can be deduced from the fact that $S_1^t/t$ is stochastically greater than $e^{-\phi(t)} \int_0^{\tau(V, ]-\infty, -h_t/2])} e^{V(y)} dy$ (see \eqref{cvmeasure4.1} and \eqref{cvmeasure1} in the next subsection), a simple manipulation on the truncated functional $\int_0^{\tau(V, ]-\infty, -h_t/2])} e^{V(y)} dy$, and the first assertion of Lemma \ref{foncexpov}. The fact that we eventually obtain $o(1)$ comes from the fact that, as we said before, $n_t$ is negligible compared to quantities of the type $e^{-c h_t}$. 

We now turn to $S_2$. Since the sequences $( e_j )_{j \geq 1}$, $( S_j^t )_{j \geq 1}$ and $(R_j^t)_{j \geq 1}$ are \textit{iid} (by Proposition \ref{approxparliid}), in the definition of $f^+_{\alpha}$ the index $1$ can be replaced by $k$ without changing the value of the function. We can thus write 
\begin{align*}
S_2 = \sum_{k=1}^{n_t} \mathbb{P} & \left( \frac1{t} \sum_{i=1}^{k-1} e_i S_i^t R_i^t \in I(t,\eta), \ 1/R_k^t \leq \alpha (1+ 2 e^{- \epsilon h_t/7}) / (\eta - 2/\log h_t), \right. \\
& \left. e_k S_k^t R_k^t > t(\eta - 2/\log h_t)(1- 2 e^{- \epsilon h_t/7}) \right) \\
\leq \sum_{k=1}^{n_t} \mathbb{P} & \left( \frac1{t} \sum_{i=1}^{k-1} e_i S_i^t R_i^t \in I(t,\eta), \ e_k S_k^t R_k^t > t(\eta - 2/\log h_t)(1- 2 e^{- \epsilon h_t/7}) \right). 
\end{align*}
If $t$ is large enough so that $I(t,\eta) \subset [1-5\eta/4,1-3\eta/4]$ and $(\eta - 2/\log h_t)(1- 2 e^{- \epsilon h_t/7}) > 3\eta/4$, we have clearly 
\[ S_2 \leq \sum_{k=1}^{n_t} \mathbb{P} \left ( 1-5\eta/4 < \frac1{t} \sum_{i=1}^{k-1} e_i S_i^t R_i^t \leq 1-3\eta/4, \ \frac1{t} e_k S_k^t R_k^t > 3\eta/4 \right ) \leq s(\eta,t), \]
where $s(\eta,t)$ is defined in Lemma \ref{analogue5.4}. Putting the previous bounds on $S_1$ and $S_2$ into \eqref{reecriturede5.1-5.3} we get that for $t$ large enough, the third term in the right hand side of \eqref{reecriturede5.1-5.1} is less than $o(1) + s(\eta,t) + 4n_t e^{-c h_t}$. We can thus rewrite \eqref{reecriturede5.1-5.1} as 
\begin{align*}
\sum_{k=1}^{n_t} \int_{\eta}^{1-\eta} E[h_2^{V,k}(y)] dE[ h_3^{V,k}(y)] & \leq \sum_{k=1}^{n_t} \int_{\eta - 2/\log h_t}^{1-\eta - 2/\log h_t} E[h_2^{V,k}(u + 2/\log h_t)] dF^-_{\alpha, k}(u) \\
& + o(1) + s(\eta,t) + 5n_t e^{-c h_t} \\
& \leq \sum_{k=1}^{n_t} \int_{\eta - 2/\log h_t}^{1-\eta - 2/\log h_t} f^+_{\alpha}(u + 4/\log h_t) dF^-_{\alpha, k}(u) \\
& + o(1) + s(\eta,t) + 6n_t e^{-c h_t}, 
\end{align*}
where we have used \eqref{reecriturede5.1-5.2} for the second inequality. Recall that $n_t e^{-c h_t}$ converges to $0$ when $t$ goes to infinity. Using \eqref{rec1factsim} we thus get 
\begin{align}
\sum_{k=1}^{n_t} \int_{\eta}^{1-\eta} E[h_2^{V,k}(y)] dE[ h_3^{V,k}(y)] \leq \sum_{k=1}^{n_t} \int_{\eta - 2/\log h_t}^{1-\eta - 2/\log h_t} f^+_{\alpha}(u + 4/\log h_t) dF^-_{\alpha, k}(u) + u_3(t,\eta), \label{reecriturede5.1-5.4}
\end{align}
where $u_3(t,\eta)$ is such that $\lim_{\eta \rightarrow 0} \limsup_{t \rightarrow +\infty} u_3(t,\eta) = 0$. 

Recall that, in the definition of $f^+_{\alpha}$, the index $1$ can be replaced by $k$ without changing the value of the function. We can thus write 
\begin{align}
& \sum_{k=1}^{n_t} \int_{\eta - 2/\log h_t}^{1-\eta - 2/\log h_t} f^+_{\alpha}(u + 4/\log h_t) dF^-_{\alpha, k}(u) \nonumber \\
= & \sum_{k=1}^{n_t} \mathbb{P} \left( \max_{ 1\leq j \leq k-1} e_j S_j^t \leq \alpha t (1 - 2e^{- \epsilon h_t/7}), \ \eta - 2/\log h_t \leq \frac1{(1 - 2e^{- \epsilon h_t/7})t} \sum_{i=1}^{k-1} e_i S_i^t R_i^t \leq 1-\eta - 2/\log h_t, \right. \label{reecriturede5.1-6} \\ 
& \left. \left (1 - \frac{4}{\log h_t} - \frac1{(1 - 2e^{- \epsilon h_t/7})t} \sum_{i=1}^{k-1} e_i S_i^t R_i^t \right )/R_k^t \leq \alpha (1+ 2 e^{- \epsilon h_t/7}), \right. \label{reecriturede5.1-7} \\ 
& \left. e_k S_k^t R_k^t/t > \left (1 - \frac{4}{\log h_t} - \frac1{(1 - 2e^{- \epsilon h_t/7})t} \sum_{i=1}^{k-1} e_i S_i^t R_i^t \right )(1 - 2 e^{- \epsilon h_t/7}) \right). \label{reecriturede5.1-8}
\end{align}

%A PRIORI LE $\pm$ DOIT ETRE REMPLACE PAR UN MOINS MAIS VOIR SI LA PREUVE EST BONNE. ON AURAIT EN EFFET PU REMPLACER PAR UN $+$  EN MARJORAT P(TL<ALPHA,TRUC = DY), LA DIFFERENCE SE RETROUVE DANS LE TERME DE RESTE. VERIFIER QU'ON A BIEN EGALITE DE L'INTEGRAL AVEC LA PROBA. 
Note that $(1+ 2 e^{- \epsilon h_t/7})(1- 2 e^{- \epsilon h_t/7})<1$ so multiplying both sides by $(1- 2 e^{- \epsilon h_t/7})$ in \eqref{reecriturede5.1-7} we get that the expression in \eqref{reecriturede5.1-7} implies 
\[ \left (1 - \frac1{t} \sum_{i=1}^{k-1} e_i S_i^t R_i^t \right )/R_k^t \leq \alpha (1 + \tilde \epsilon(k)), \]
where $\tilde \epsilon(k) := (2 e^{- \epsilon h_t/7} + 4(1- 2 e^{- \epsilon h_t/7})/\log h_t)/R_k^t$. Note that $(1-\eta - 2/\log h_t)(1 - 2e^{- \epsilon h_t/7}) \leq 1-\eta$ and that, for $t$ large enough, $(1-4/\log h_t)(1 - 2e^{- \epsilon h_t/7}) > 1-\eta$. Then the expressions in \eqref{reecriturede5.1-6} and \eqref{reecriturede5.1-8} imply respectively 
\[ \max_{ 1\leq j \leq k-1} e_j S_j^t \leq \alpha t, \ \frac1{t} \sum_{i=1}^{k-1} e_i S_i^t R_i^t \leq 1-\eta \ \ \ \text{and} \ \ \ \frac1{t} \sum_{i=1}^{k} e_i S_i^t R_i^t > 1-\eta. \]
We thus get that for large $t$, $\sum_{k=1}^{n_t} \int_{\eta - 2/\log h_t}^{1-\eta - 2/\log h_t} f^+_{\alpha}(u + 4/\log h_t) dF^-_{\alpha, k}(u)$ is less than 
\begin{align*}
\sum_{k=1}^{n_t} \mathbb{P} & \left( \max_{ 1\leq j \leq k-1} e_j S_j^t \leq \alpha t, \ \frac1{t} \sum_{i=1}^{k-1} e_i S_i^t R_i^t \leq 1-\eta, \right. \\
& \left. \left (1 - \frac1{t} \sum_{i=1}^{k-1} e_i S_i^t R_i^t \right )/R_k^t \leq \alpha (1 + \tilde \epsilon(k)), \ \frac1{t} \sum_{i=1}^{k} e_i S_i^t R_i^t > 1-\eta \right ). 
\end{align*}

%EN ETANT PLUS VIGILANT SUR LE CHOIX DE L'OVERSHOOT ON A PAS EU BESOIN DE NEGLIGER LE (5.15) DE ADV ET DONC DU PREMIER POINT DU LEMME 5.4 DE ADV, C'EST A DIRE DE \eqref{rec1} MAIS IL FAUDRA LE REMPLACER PAR CE QUI SERT A PROUVER \eqref{reecriturede5.1-5}. 

Recall the definition of the overshoots $\mathcal{N}_a$ in the beginning of Subsection \ref{proofmainth} and note that $\{ \frac1{t} \sum_{i=1}^{k-1} e_i S_i^t R_i^t \leq 1-\eta, \ \frac1{t} \sum_{i=1}^{k} e_i S_i^t R_i^t > 1-\eta \} = \{ \mathcal{N}_{(1-\eta)t} = k\}$ so in particular the events in the above sum of probabilities are disjoint and the above sum equals 
\[ \mathbb{P} \left( \max_{ 1\leq j \leq \mathcal{N}_{(1-\eta)t}-1} \frac{e_j S_j^t}{t} \leq \alpha, \ \left (1 - \frac1{t} \sum_{i=1}^{\mathcal{N}_{(1-\eta)t}-1} e_i S_i^t R_i^t \right )/R_{\mathcal{N}_{(1-\eta)t}}^t \leq \alpha \left (1 + \tilde \epsilon(\mathcal{N}_{(1-\eta)t}) \right ), \ \mathcal{N}_{(1-\eta)t} \leq n_t \right ). \]
We thus deduce that 
\begin{eqnarray}
\sum_{k=1}^{n_t} \int_{\eta - 2/\log h_t}^{1-\eta - 2/\log h_t} f^+_{\alpha}(u + 4/\log h_t) dF^-_{\alpha, k}(u) \leq \mathcal{P}_{\eta, t}^{+}(\alpha) + u_4(t), \label{reecriturede5.1-9}
\end{eqnarray}
where $\mathcal{P}_{\eta, t}^{+}(\alpha)$ is defined in the statement of the proposition and where $u_4(t) := \mathbb{P} (\tilde \epsilon(\mathcal{N}_{(1-\eta)t}) \geq 1/\sqrt{\log (\log t)})$. Using the definition of $\tilde \epsilon(k)$, partitioning on the values of $\mathcal{N}_{(1-\eta)t}$ and using the fact that $(R_j^t)_{j \geq 1}$ is \textit{iid} and that $\log h_t \sim_{t \rightarrow +\infty} \log (\log t)$, we get that for $t$ large enough $u_4(t) \leq C \mathbb{P} (R_1^t \leq 2/\sqrt{\log (\log t)})$, where $C$ is some positive constant. According to Proposition \ref{cvr}, $R_1^t$ converges in distribution to $\mathcal{R}$ (defined in the Introduction, just after Theorem \ref{kappa>1favsite}) which is almost surely positive. We deduce that $\lim_{t \rightarrow +\infty} u_4(t) = 0$. 

Finally, the combination of \eqref{reecriturede5.1-0}, \eqref{reecriturede5.1-5.0}, \eqref{reecriturede5.1-5.4}, and \eqref{reecriturede5.1-9} yields the upper bound. 

\textbf{Lower bound: } 

For the proof of the upper bound we have used the obvious inequality $\max_{1 \leq j \leq N_t} \mathcal{L}_X(t, \tilde m_j)/t \leq \mathcal{L}_X^*(t)/t$. For the lower bound we need a converse inequality (at least for the distribution functions) that we prove thanks to the localisation in the bottom of the valleys of the main contributions to the local time. 

%LEUR EVENEMENT $\mathcal{E}_2$ SERT SANS DOUTE A DIRE QUE C'EST AU FOND DES VALLEES QU'EST ATTEINT LE SUPREMUM (CONFIRATION ?), OR CE N'EST PAS LA PEINE DE DIRE CA, IL SUFFIT DE DIRE QUE LE TL EST $\leq \alpha$ AU FOND DES VALLES ET PETIT AILLEURS. C'EST EXACTEMENT COMME CA QU'ON PROCEDE DANS LA MINORATION POUR LA LIMINF DU PS. 

%OBTENIR L'INEGALITE QUI SUIT (ok), DEFINIR LES EVENEMENTS (ok), RAJOUTER LA DERNIERE VALLEE (ok). POUR $\mathcal{E}^4_t$ (= $\mathcal{A}^7_t$ d'ici), C'EST-A-DIRE L'ANALOGUE DE (5.22) from \cite{advech} IL FAUDRAIT DELOCALISER ICI LA PREUVE QUI EST FAITE DANS L'ARTICLE SUR LE PS ET DU COUP REFERER A ICI DANS L'ARTICLE SUR LE PS (ok). 

%ON A AUSSI SANS DOUTE BESOIN DE $\mathcal{V}_{n_t, h_t}$. (ok)

Recall the definition of $\mathcal{A}^6_t$ in \eqref{defA6}. On $\mathcal{A}^6_t \cap \{ N_t < n_t \} \cap \mathcal{V}_{n_t, h_t}$ we have 
%PRECISER COMME DANS L'UPPER BOUND SI ON DEMANDE... 
\begin{align*}
\forall j < N_t, \ \forall x \in [\tilde{L}_{j-1}, \tilde{m}_j], \ \mathcal{L}_X (t,x) = \mathcal{L}_X (H(\tilde L_j),x) & \leq \left ( \mathcal{L}_X (H(\tilde L_j),x) - \mathcal{L}_X (H(\tilde m_{j}),x) \right ) \\
& + \left ( \mathcal{L}_X (H(\tilde m_{j}),x) - \mathcal{L}_X (H(\tilde L_{j-1}),x) \right ), \\
\forall j < N_t, \ \forall x \in [\tilde{m}_j, \tilde{L}_{j}], \ \mathcal{L}_X (t,x) = \mathcal{L}_X (H(\tilde m_{j+1}),x) & \leq \left ( \mathcal{L}_X (H(\tilde m_{j+1}),x) - \mathcal{L}_X (H(\tilde L_{j}),x) \right ) \\
& + \left ( \mathcal{L}_X (H(\tilde L_{j}),x) - \mathcal{L}_X (H(\tilde m_{j}),x) \right ). 
%\forall x \geq \tilde{L}_{N_t -1}, \ \mathcal{L}_X (t,x) \leq \left ( \mathcal{L}_X (t,x) - \mathcal{L}_X (H(\tilde m_{N_t}),x) \right ) %& + \left ( \mathcal{L}_X (H(\tilde m_{N_t}),x) - \mathcal{L}_X (H(\tilde L_{N_t -1}),x) \right ). 
\end{align*}
%the points $\tilde m_1, \tilde L_1, \tilde m_2, \tilde L_2, ..., \tilde L_{n_t-1}, \tilde m_{n_t}, \tilde L_{n_t-1}$ are. 
Recall the event $\mathcal{A}^2_t$ defined in Fact \ref{analog3.3}. The above shows that on $\mathcal{A}^6_t \cap \mathcal{A}^2_t \cap \{ N_t < n_t \} \cap \mathcal{V}_{n_t, h_t}$ we have 
\[ \sup_{x \in [\tilde{L}_{0}, \tilde{L}_{N_t - 1}]} \mathcal{L}_X(t, x) \leq \sup_{1 \leq j \leq N_t-1} \sup_{x \in [\tilde{L}_{j-1}, \tilde{L}_j]} \left ( \mathcal{L}_X (H(\tilde L_j),x) - \mathcal{L}_X (H(\tilde m_{j}),x) \right ) + t e^{(\kappa (1+3\delta )-1)\phi(t)}. \]
%DEVELOPPER COMMENT ON OBTIENT CA SI C'ST DEMANDE
Similarly, on $\mathcal{A}^6_t \cap \mathcal{A}^2_t \cap \{ N_t < n_t \} \cap \mathcal{V}_{n_t, h_t}$ we also have $\sup_{x \leq \tilde{L}_{0}} \mathcal{L}_X(t, x) \leq t e^{(\kappa (1+3\delta )-1)\phi(t)}$ and 
\[ \sup_{x \geq \tilde{L}_{N_t - 1}} \mathcal{L}_X(t, x) \leq \sup_{x \in [\tilde{L}_{N_t-1}, \tilde{L}_{N_t}]} \left ( \mathcal{L}_X (t,x) - \mathcal{L}_X (H(\tilde m_{N_t}),x) \right ) + t e^{(\kappa (1+3\delta )-1)\phi(t)}. \] We deduce that, for $t$ large enough so that $e^{(\kappa (1+3\delta )-1)\phi(t)}) < \alpha$, $\mathbb{P} \left( \mathcal{L}_X^*(t)/t \leq \alpha \right)$ is greater than 
\[ \mathbb{P} \left ( \max_{1 \leq j \leq N_t} \sup_{x \in [\tilde{L}_{j-1}, \tilde{L}_j]} \left ( \mathcal{L}_X (H(\tilde L_j) \wedge t,x) - \mathcal{L}_X (H(\tilde m_{j}),x) \right ) \leq t \alpha_t^1, \mathcal{V}_{n_t, h_t}, \ N_t < n_t, \ \mathcal{A}^6_t, \mathcal{A}^2_t \right ), \]
where $\alpha_t^1 := \alpha - e^{(\kappa (1+3\delta )-1)\phi(t)})$. Note that $H(\tilde L_j) \wedge t = t$ only for $j = N_t$. We have clearly $\alpha_t^1 \geq e^{-2 \phi(t)}$ for $t$ large enough. Recall the definition of $\mathcal{D}_j$ in \eqref{defdj}. Using the definition of $\mathcal{A}^3_t$ from Fact \ref{analog3.3} (and the fact that we are on $\{N_t < n_t\}$), we get that for such large $t$ the above is more than 
\[ \mathbb{P} \left ( \max_{1 \leq j \leq N_t} \sup_{x \in \mathcal{D}_j} \left ( \mathcal{L}_X (H(\tilde L_j) \wedge t,x) - \mathcal{L}_X (H(\tilde m_{j}),x) \right ) \leq t \alpha_t^1, \mathcal{V}_{n_t, h_t}, N_t < n_t, \mathcal{A}^6_t, \mathcal{A}^2_t, \mathcal{A}^3_t \right ). \]
Note that $\mathcal{L}_X (H(\tilde L_j),x) - \mathcal{L}_X (H(\tilde m_{j}),x) = \mathcal{L}_{X_{\tilde m_j}}(H_{X_{\tilde m_j}}(\tilde L_j),x)$. From the definition of $\mathcal{A}^7_t$ in Fact \ref{analogue(5.22)} (and the fact that we are on $\{N_t < n_t\}$) the above is more than
\[ \mathbb{P} \left ( \sup_{x \in \mathcal{D}_{N_t}} \mathcal{L}_X(t, x) \vee \max_{1 \leq j \leq N_t-1} \mathcal{L}_X(H(\tilde L_j), \tilde m_j) \leq t \alpha_t^2, \mathcal{V}_{n_t, h_t}, N_t < n_t, \ \mathcal{A}^6_t, \mathcal{A}^2_t, \mathcal{A}^3_t, \mathcal{A}^7_t \right ), \]
where $\alpha_t^2 := \alpha_t^1 / (1+e^{-h_t/9}) \underset{t \rightarrow +\infty}{\longrightarrow} \alpha$. In conclusion we have 
\[ \mathbb{P} \left( \mathcal{L}_X^*(t)/t \leq \alpha \right) \geq \mathbb{P} \left ( \sup_{x \in \mathcal{D}_{N_t}} \mathcal{L}_X(t, x) \vee \max_{1 \leq j \leq N_t-1} \mathcal{L}_X(H(\tilde L_j), \tilde m_j) \leq t \alpha_t^2 \right ) - \tilde u_1(t), \]
where $\tilde u_1(t) := 1-\mathbb{P} \left ( \mathcal{V}_{n_t, h_t}, N_t < n_t, \mathcal{A}^6_t, \mathcal{A}^2_t, \mathcal{A}^3_t, \mathcal{A}^7_t \right )$. The combination of Lemma \ref{minimacoincide} (applied with $n=n_t, h=h_t$), Lemma \ref{nbvalleesvisit} from the next subsection, the estimate $\mathbb{P} ( \overline{\mathcal{A}^6_t} ) \leq e^{-c h_t}$ (see just after \eqref{defA6}), \eqref{negltl2}, \eqref{negltl3}, and Fact \ref{analogue(5.22)} yields $\tilde u_1(t) \underset{t \rightarrow +\infty}{\longrightarrow} 0$. 

We are left to study the distribution function of $[\sup_{x \in \mathcal{D}_{N_t}} \mathcal{L}_X(t, x) \vee \max_{1 \leq j \leq N_t-1} \mathcal{L}_X(H(\tilde L_j), \tilde m_j)] /t$. Note that the latter quantity is quite similar to the quantity $[\mathcal{L}_X(t, \tilde m_{N_t}) \vee \max_{1 \leq j \leq N_t-1} \mathcal{L}_X(H(\tilde L_j), \tilde m_j)]/t$ studied in the proof of the upper bound. Reasoning similarly as in the proof of the upper bound, which allows to use this time the upper bound in Fact \ref{analogue5.2} and the lower bound in Fact \ref{analogue5.3}, we get 
\[ \mathbb{P} \left ( \sup_{x \in \mathcal{D}_{N_t}} \mathcal{L}_X(t, x) \vee \max_{1 \leq j \leq N_t-1} \mathcal{L}_X(H(\tilde L_j), \tilde m_j) \leq t \alpha_t^2 \right ) \geq \sum_{k=1}^{n_t} \int_{\eta}^{1-\eta} f^-_{\alpha_t^2}(u) dF^+_{\alpha_t^2, k}(u) - \tilde u_2(t, \eta), \]
where $\tilde u_2(t,\eta)$ is such that $\lim_{\eta \rightarrow 0} \limsup_{t \rightarrow +\infty} \tilde u_2(t,\eta) = 0$. Reasoning again as in the proof of the upper bound we get 
\[ \sum_{k=1}^{n_t} \int_{\eta}^{1-\eta} f^-_{\alpha_t^2}(u) dF^+_{\alpha_t^2, k}(u) \geq \mathcal{P}_{\eta, t}^{-}(\alpha) - \tilde u_3(t, \eta), \]
where $\tilde u_3(t,\eta)$ is such that $\lim_{\eta \rightarrow 0} \limsup_{t \rightarrow +\infty} \tilde u_3(t,\eta) = 0$. This yields the result. 
%(VOIR SI CA SUFFIT COMME PREUVE, EN COMPARANT AVEC LE LOWER BOUND DANS ADV, CA SEMBLE A PRIORI OK POUR ICI). 

%MENTIONNER $F_{\gamma}$ et $\tilde f_{\gamma}$ (ok) POUR METTRE EN EVIDENCE QUE C'EST NORMAL S'IL RESTE LE $\mathcal{D}$ DE LA DERNIERE VALLEE ET PAS LES AUTRES (on l'a fait en mentionnant le lower bound de Fact \ref{analogue5.3}), 

%ET S'IL Y A ENCORE DES ESTIMES A UTILISER : POUR CEUX DE LA PREUVE ANCIENNE JUSTE LES 3 PREMIERS (ok) ET LE DERNIER POURRA ETRE SUPPRIME ET OK POUR LES AUTRES), OK POUR LES AUTRES ESTIMES DE L'ARTICLE (SOIT ILS SONT UTILISES SOIT ILS SONT INUTILES, SOIT IL SERVENT A PROUVER LES ESTIMES UTILISES). PEUT-ETRE PAS LA PEINE DE REFERER A LA PREUVE ORIGINALE POUR LE LOWER BOUND PARCE QU'ELLE EST UN PEU PLUS COMPLIQUEE (OK). 

%EN FAIT IL FAUT LAISSER LE SUP SUR $\mathcal{D}_j$ (AU MOINS POUR LA DERNIERE) ET DU COUP FAIRE APPARAITRE $\tilde f$. OK (ON L'A FAIT IMPLICITEMENT EN MENTIONNANT le lower bound de Fact \ref{analogue5.3}). 

%COMPARAISON AVEC LA PREUVE ORIGINALE POUR VERIFIER QU'IL Y A TOUT (OK), COMPARAISON AVEC L'ANCIENNE PREUVE (OK), VERIFICATION DE CETTE PREUVE (OK, RESTE RELECTURE RAPIDE : OK). 

%VOIR SI ON PEUT REMPLACER $\epsilon/2$ PAR $\eta$. IL FAUT PEUT-ETRE AUSSI RAPPELER LA DEFINITION DES OVERSHOOT. DU COUP CA SERAIT PLUTOT $\epsilon$ PAR $\eta$. ON L'A FAIT, RESTE A RAPPELER LA DEF DES OVERSHOOT. 

\end{proof}
}

\subsection{Proof of Proposition \ref{propcvsub} and consequences} \label{reprise}

This proposition relies on: 
%We begin this subsection by proving an important lemma which is analogue to Lemma 4.1 in \cite{advech}. The validity of this lemma allows us to apply their argument and therefore to state, on one hand, the tightness of the family of processes $(Y_1^t, Y_2^t)_{t > 0}$ in the Skorokhod space, and on the other hand, to identify the limit distribution when $t$ goes to infinity as the law of $(\mathcal{Y}_1, \mathcal{Y}_2)$. 
\begin{lemme} \label{lemproba}
{Recall the constant $\mathcal{C}'$ and the random variable $\mathcal{R}$ defined a little before Theorem \ref{cvdutl}, recall also the function $\phi$ defined in \eqref{defhtetphi}.} Fix $\eta \in ]0, 1/3[$, we have
\begin{align}
\underset{t \rightarrow +\infty}{\lim} & \sup_{x \in \left [ e^{-(1-2 \eta)\phi(t)}, + \infty \right [} \ \left | x^{\kappa} e^{\kappa \phi(t)} \mathbb{P} \left ( e_1 S_1^t /t > x \right ) - \mathcal{C}' \right | & = 0, \label{cvmeasure7.1} \\
\underset{t \rightarrow +\infty}{\lim} & \sup_{y \in \left [ e^{-(1-3\eta)\phi(t)}, + \infty \right [} \ \left | y^{\kappa} e^{\kappa \phi(t)} \mathbb{P} \left ( e_1 S_1^t R_1^t /t > y \right ) - \mathcal{C}' \mathbb{E}\left [ \mathcal{R}^{\kappa} \right ] \right | & = 0. \label{cvmeasure9.1}
\end{align}
For any positive $\alpha$, $e^{\kappa \phi(t)} \mathbb{P} (e_1 S_1^t /t \geq x, \ e_1 S_1^t R_1^t /t \geq y )$ converges when $t$ goes to infinity, uniformly in $(x,y) \in [\alpha,+ \infty [ \times [\alpha,+ \infty [$, to $\mathcal{C}' y^{-\kappa} \mathbb{E}[\mathcal{R}^{\kappa}\mathds{1}_{\mathcal{R}\leq y/x}] + \mathcal{C}' x^{-\kappa}  \mathbb{P} (\mathcal{R} >y/x)$.

\end{lemme}

\begin{proof}

We start by proving an intermediary result analogous to (4.3) of \cite{advech}: 
\begin{eqnarray}
\underset{t \rightarrow +\infty}{\lim} \ \sup_{x \in \left [ e^{-(1-\eta)\phi(t)}, + \infty \right [} \ \left | x^{\kappa} e^{\kappa \phi(t)} \mathbb{P} \left ( S_1^t /t > x \right ) - \mathcal{C} \right | = 0, \label{cvmeasure5}
\end{eqnarray}
%\begin{eqnarray}
%\underset{t \rightarrow +\infty}{\lim} \ e^{\kappa \phi(t)} \mathbb{P} \left ( S_1^t /t > x \right ) = c_1 x^{-\kappa}, \ \ \forall x > 0 \label{cvmeasure4}
%\end{eqnarray}
where $\mathcal{C} := K / \Psi_V'(\kappa)$ is the constant defined a little before Theorem \ref{cvdutl}. We will then deduce the Lemma. First, 
\begin{eqnarray}
S_1^t = \int_{\tilde \tau_1^+ (h_t / 2)}^{\tilde \tau_1 (h_t)} e^{\tilde V^{(1)}(u)} du + \int_{\tilde \tau_1 (h_t)}^{\tilde L_1} e^{\tilde V^{(1)}(u)} du. \label{cvmeasure4.1}
\end{eqnarray}
Recall from \eqref{defhtetphi} that $t = e^{h_t} \times e^{\phi(t)}$. From Proposition \ref{standardwilliams} (applied with $h=h_t$) we get
\begin{eqnarray}
\frac1{t} \int_{\tilde \tau_1 (h_t)}^{\tilde L_1} e^{\tilde V^{(1)}(u)} du \overset{\mathcal{L}}{=} \frac{e^{h_t}}{t} \int_0^{\tau(V, ]-\infty, -h_t/2])} e^{V(y)} dy = e^{-\phi(t)} \int_0^{\tau(V, ]-\infty, -h_t/2])} e^{V(y)} dy, \label{cvmeasure1}
\end{eqnarray}
and
\begin{eqnarray}
\frac1{t} \int_{\tilde \tau_1 (h_t / 2)}^{\tilde \tau_1 (h_t)} e^{\tilde V^{(1)}(u)} du \overset{\mathcal{L}}{=} \frac{1}{t} \int_{\tau(V^{\uparrow}, h_t/2)} ^{\tau(V^{\uparrow}, h_t)} e^{V^{\uparrow}(y)} dy \leq \frac{e^{h_t}}{t} \tau(V^{\uparrow}, h_t) = e^{-\phi(t)} \tau(V^{\uparrow}, h_t). \label{cvmeasure1.1}
\end{eqnarray}

By the second assertion of Lemma \ref{foncexpov}, 
\begin{eqnarray}
\mathbb{P} \left ( \int_0^{+\infty} e^{V(y)} dy > u \right ) \underset{u \rightarrow +\infty}{\sim} \mathcal{C} u^{-\kappa}. \label{cvmeasure2}
\end{eqnarray}

We define $\epsilon_t := e^{-\eta \phi(t)/2}$, and we only consider values of $x$ in $\left [ e^{- (1-\eta)\phi(t)}, + \infty \right [$ (this implies $x e^{\phi(t)} \geq e^{\eta \phi(t)}$ which is large when $t$ is). From \eqref{cvmeasure4.1}, \eqref{cvmeasure1}, and \eqref{cvmeasure1.1}, $x^{\kappa} e^{\kappa \phi(t)} \mathbb{P} \left ( S_1^t /t > x \right )$ is less than
\begin{eqnarray}
x^{\kappa} e^{\kappa \phi(t)} \mathbb{P} \left ( \tau(V^{\uparrow}, h_t) > x e^{\phi(t)} \epsilon_t \right ) + x^{\kappa} e^{\kappa \phi(t)} \mathbb{P} \left ( \int_0^{+\infty} e^{V(y)} dy > (1 - \epsilon_t ) x e^{\phi(t)} \right ). \label{cvmeasure2.1}
\end{eqnarray}
%\begin{align*}
%& x^{\kappa} e^{\kappa \phi(t)} \mathbb{P} \left ( \int_{\tilde \tau_1 (h_t / 2)}^{\tilde \tau_1 (h_t)} e^{\tilde V^{(1)}(u)} du > t x \epsilon_t \right ) + x^{\kappa} e^{\kappa \phi(t)} \mathbb{P} \left ( \frac1{t} \int_{\tilde \tau_1 (h_t)}^{\tilde L_1} e^{\tilde V^{(1)}(u)} du > (1 - \epsilon_t ) x \right ) \\
%\leq & x^{\kappa} e^{\kappa \phi(t)} \mathbb{P} \left ( \tilde \tau_1 (h_t) > e^{\phi(t)} x \epsilon_t \right ) + x^{\kappa} e^{\kappa \phi(t)} \mathbb{P} \left ( e^{-\phi(t)} \int_0^{\tau(V, ]-\infty, -h_t/2])} e^{V(y)} dy > (1 - \epsilon_t ) x \right ), 
%\end{align*}
%because of $(\ref{cvmeasure2})$ for the second term. Now, the combination of Proposition \ref{standardwilliams} and Lemma \ref{vposlapltpsatt} proves that the limit of the first term is $0$ and $\int_0^{\tau(V, ]-\infty, -h_t/2])} e^{V(y)} dy$ is less than $e^{-\phi(t)} \int_0^{+\infty} e^{V(y)} dy$ so, by $(\ref{cvmeasure2})$ we get :
From Lemma \ref{vposlapltpsatt} {applied with $y=h_t$, $r=x e^{\phi(t)} \epsilon_t$, the first term is less than $x^{\kappa} e^{\kappa \phi(t)} \exp(c_1 h_t - c_2 x e^{\phi(t)} \epsilon_t)$ where $c_1$ and $c_2$ are some positive constants. Since $x e^{\phi(t)} \epsilon_t \geq e^{\eta \phi(t)/2}$, $\phi(t) >> \log(\log(t))$, and $h_t \sim \log(t)$, we get that for all $t$ large enough and $x \in [ e^{- (1-\eta)\phi(t)}, + \infty [$, $c_1 h_t - c_2 x e^{\phi(t)} \epsilon_t \leq - c_2 x e^{\phi(t)} \epsilon_t/2$. We deduce that 
\[ \underset{t \rightarrow +\infty}{\limsup} \ \sup_{x \in \left [ e^{- (1-\eta)\phi(t)}, + \infty \right [} \ x^{\kappa} e^{\kappa \phi(t)} \mathbb{P} \left ( \tau(V^{\uparrow}, h_t) > x e^{\phi(t)} \epsilon_t \right ) = 0. \]
From \eqref{cvmeasure2}, the second term in \eqref{cvmeasure2.1} is close to $\mathcal{C}$ when $t$ is large and $x \in \left [ e^{- (1-\eta)\phi(t)}, + \infty \right [$.} We thus get
\begin{eqnarray}
\underset{t \rightarrow +\infty}{\limsup} \ \sup_{x \in \left [ e^{- (1-\eta)\phi(t)}, + \infty \right [} \ x^{\kappa} e^{\kappa \phi(t)} \mathbb{P} \left ( S_1^t /t > x \right ) \leq \mathcal{C}. \label{cvmeasure3}
\end{eqnarray}
%and this convergence is uniform for $x$ belonging to intervals of the type $[a, +\infty[, a > 0$. 
We now prove a lower bound. From the Markov property applied at time $\tau(V, ]-\infty, -h_t/2])$ we have
\begin{align*}
e^{-\phi(t)} \int_0^{+\infty} e^{V(y)} dy & \overset{\mathcal{L}}{=} e^{-\phi(t)} \int_0^{\tau(V, ]-\infty, -h_t/2])} e^{V(y)} dy \\
& + e^{V(\tau(V, ]-\infty, -h_t/2])) - \phi(t)} \int_0^{+\infty} e^{\tilde V(y)} dy \\
& \leq e^{-\phi(t)} \int_0^{\tau(V, ]-\infty, -h_t/2])} e^{V(y)} dy \\
& + e^{-h_t/2 - \phi(t)} \int_0^{+\infty} e^{\tilde V(y)} dy, 
\end{align*}
where $\tilde V$ is an independent copy of $V$ and where we have used that $V(\tau(V, ]-\infty, -h_t/2])) \leq -h_t/2$. We now put $\epsilon_t := e^{-h_t /4}$. Then, $\mathbb{P} ( e^{-\phi(t)} \int_0^{+\infty} e^{V(y)} dy > x(1+\epsilon_t) )$ is less than
%\begin{align*}
%\mathbb{P} \left ( e^{-\phi(t)} \int_0^{+\infty} e^{V(y)} dy > x(1+\epsilon_t) \right ) & \leq \mathbb{P} \left ( e^{-\phi(t)} \int_0^{\tau(V, ]-\infty, -h_t/2])} e^{V(y)} dy > x \right ) \\
%& + \mathbb{P} \left ( e^{V(\tau(V, ]-\infty, -h_t/2])) - \phi(t)} \int_0^{+\infty} e^{\tilde V(y)} dy > x \epsilon_t \right ) \\
%& \leq \mathbb{P} \left ( S_1^t /t > x \right ) + \mathbb{P} \left ( e^{-h_t/2 - \phi(t)} \int_0^{+\infty} e^{\tilde V(y)} dy > x \epsilon_t \right ), 
%\end{align*}
\begin{align*}
& \mathbb{P} \left ( e^{-\phi(t)} \int_0^{\tau(V, ]-\infty, -h_t/2])} e^{V(y)} dy > x \right ) + \mathbb{P} \left ( e^{-h_t/2 - \phi(t)} \int_0^{+\infty} e^{\tilde V(y)} dy > x \epsilon_t \right ) \\
\leq & \mathbb{P} \left ( S_1^t /t > x \right ) + \mathbb{P} \left ( \int_0^{+\infty} e^{\tilde V(y)} dy > e^{h_t/4} x e^{ \phi(t)} \right ). 
\end{align*}
{For the first term, we have used the fact that $S_1^t /t$ is stochastically greater than 

\noindent $e^{-\phi(t)} \int_0^{\tau(V, ]-\infty, -h_t/2])} e^{V(y)} dy$ (because of \eqref{cvmeasure1} and \eqref{cvmeasure4.1}). 
%and the fact that $V(\tau(V, ]-\infty, -h_t/2]) < -h_t/2$ for the second term. 
%Multiplying the above inequality by $x^{\kappa} e^{\kappa \phi(t)}$ and using two times \eqref{cvmeasure2} we get 
By \eqref{cvmeasure2} we get that when $t$ is large and $x \in \left [ e^{- (1-\eta)\phi(t)}, + \infty \right [$, $x^{\kappa} e^{\kappa \phi(t)} \mathbb{P} ( e^{-\phi(t)} \int_0^{+\infty} e^{V(y)} dy > x(1+\epsilon_t) )$ is close to $\mathcal{C}$ while $x^{\kappa} e^{\kappa \phi(t)} \mathbb{P} ( \int_0^{+\infty} e^{\tilde V(y)} dy > e^{h_t/4} x e^{ \phi(t)} )$ is less than $2 \mathcal{C} e^{-\kappa h_t/4}$. Combing with the above inequality we get 
}
\[ \mathcal{C} \leq \underset{t \rightarrow +\infty}{\liminf} \ \inf_{x \in \left [ e^{- (1-\eta)\phi(t)}, + \infty \right [} \ x^{\kappa} e^{\kappa \phi(t)} \mathbb{P} \left ( S_1^t /t > x \right ) + 0. \]
%and again, this convergence is uniform for $x$ belonging to intervals of the type $[a, +\infty[, a > 0$. 
Combining this with \eqref{cvmeasure3} we get \eqref{cvmeasure5}. 

Now that we have proved \eqref{cvmeasure5}, the rest of the proof is exactly the same as the proof of Lemma 4.1 in \cite{advech} once they have proved (4.3). The argument needs: 1) the fact that $e_1$, $S_1^t$ and $R_1^t$ are mutually independent, which is true according to Proposition \ref{approxparliid}, 2) the fact that $( R_1^t )_{t > 1}$ converges in distribution to $\mathcal{R}$ and is bounded in all $L^p$ spaces, which is true from Proposition \ref{cvr} applied with $h=h_t$. {In the probability in \eqref{cvmeasure5}, we can therefore add successively the factors $e_1$ and $R_1^t$ (and each time one of these factors is added we add an $\eta$ in $e^{- (1-\eta)\phi(t)}$) so that we obtain successively \eqref{cvmeasure7.1} and \eqref{cvmeasure9.1}, and then the convergence of $e^{\kappa \phi(t)} \mathbb{P} (e_1 S_1^t /t \geq x, \ e_1 S_1^t R_1^t /t \geq y )$.} Therefore the lemma is proved.

\end{proof}

In \cite{advech}, the proof of Proposition 1.4 (that is, the convergence of $(Y_1, Y_2)^t$ toward $(\mathcal{Y}_1,\mathcal{Y}_2)$) relies only on their Lemma 4.1 from which is proved the tightness of the family $((Y_1, Y_2)^t, t > 0)$ and also the unicity and the identification of the limit distribution. Using Lemma \ref{lemproba} instead of Lemma 4.1 of \cite{advech}, and using the fact that, here also, the sequence $( e_i S_i^t,  e_i S_i^t R_i^t)_{i \geq 1}$ is \textit{iid} (see Proposition \ref{approxparliid}), the proof of Proposition 1.4 of \cite{advech} can be repeated here and we get Proposition \ref{propcvsub}.

As an other consequence of Lemma \ref{lemproba} {we can prove that $N_t$, the number of $h_t$-valleys visited until instant $t$ (see \eqref{defnt}), is less than $n_t$ with a large probability.} This fact has already be used and is fundamental, since most of the estimates we have proved are true not for all but for the $n_t$ first $h_t$-valleys. 

\begin{lemme} \label{nbvalleesvisit}

There is a positive constant $c$ such that for all $t$ large enough, 
\[ \mathbb{P} \left( N_t \geq n_t \right) \leq e^{- c h_t}. \]

\end{lemme}

\begin{proof}

We have $ \{ N_t \geq n_t \} = \{ H(m_{n_t}) \leq t \}$ and $\sum_{i=1}^{n_t - 1} H(\tilde L_i)-H(\tilde m_i) \leq H(\tilde m_{n_t}) = H(m_{n_t})$ on $\mathcal{V}_{n_t, h_t}$ {(the event defined in Lemma \ref{minimacoincide})}. Let us fix $\epsilon$ as in Proposition \ref{approxparliid}. Using, from this proposition, the definition of $\mathcal{A}^5_t$ and the fact that the sequence $(e_i S_i^t R_i^t)_{i \geq 1}$ is \textit{iid}, we get that $\mathbb{P} ( \{ N_t \geq n_t \} \cap \mathcal{V}_{n_t, h_t} \cap \mathcal{A}^5_t )$ is less than
\begin{align*}
& \mathbb{P} \left( \sum_{i=1}^{n_t - 1} H(\tilde L_i)-H(\tilde m_i) \leq t, \ \mathcal{A}^5_t \right) \leq \mathbb{P} \left( \sum_{i=1}^{n_t - 1} e_i S_i^t R_i^t \leq t (1- e^{- \epsilon h_t/7})^{-1} \right) \\
\leq & \mathbb{P} \left( \sup_{1 \leq i \leq n_t-1} e_i S_i^t R_i^t \leq t (1- e^{- \epsilon h_t/7})^{-1} \right) \leq \left [ 1 - \mathbb{P} \left( e_1 S_1^t R_1^t /t > (1- e^{- \epsilon h_t/7})^{-1} \right) \right ]^{n_t-1} \\
\leq & \exp \left [ - (n_t-1) \mathbb{P} \left( e_1 S_1^t R_1^t /t > (1- e^{- \epsilon h_t/7})^{-1} \right) \right ],
\end{align*}
where the last inequality comes from $1-x \leq \exp(-x)$ for $x \in [0, 1[$. According to Lemma \ref{lemproba} and the fact that $n_t \sim e^{\kappa (1+\delta) \phi(t)}$ {(see the definition of $n_t$ in the beginning of this section)}, we have that 
\[ (n_t-1) \mathbb{P} ( e_1 S_1^t R_1^t /t > (1- e^{- \epsilon h_t/7})^{-1} ) \underset{t \rightarrow +\infty}{\sim} \mathcal{C}' \mathbb{E}\left [ \mathcal{R}^{\kappa} \right ] e^{\kappa \delta \phi(t)}. \] 
Since $\phi(t) >> \log(\log(t))$ and $h_t \sim \log(t)$ we deduce that $\mathbb{P} ( \{ N_t \geq n_t \} \cap \mathcal{V}_{n_t, h_t} \cap \mathcal{A}^5_t ) \leq e^{-h_t}$ for $t$ large enough. {Also, when $t$ is large enough, $\mathbb{P} ( \mathcal{V}_{n_t, h_t}^c ) \leq n_t e^{- \delta \kappa h_t /3}$ according to Lemma \ref{minimacoincide} (applied with $n=n_t, h=h_t$) and $\mathbb{P} ( \mathcal{A}^{3, c}_t ) \leq e^{-c_1 h_t}$, for some positive constant $c_1$, according to Proposition \ref{approxparliid}. We thus get, for $t$ large enough, 
\[ \mathbb{P} \left( N_t \geq n_t \right) \leq \mathbb{P} ( \{ N_t \geq n_t \} \cap \mathcal{V}_{n_t, h_t} \cap \mathcal{A}^5_t ) + \mathbb{P} ( \mathcal{V}_{n_t, h_t}^c ) + \mathbb{P} ( \mathcal{A}^{3, c}_t ) \leq e^{-h_t} + n_t e^{- \delta \kappa h_t /3} + e^{-c_1 h_t}. \]
Recall that, since $n_t \sim e^{\kappa (1+\delta) \phi(t)}$ where $\phi(t) << \log(t) \sim h_t$, the $n_t$ before the exponential is not significant. The result follows.} 
\end{proof}

\section{Some estimates on $V$, $V^{\uparrow}$, $\hat V^{\uparrow}$ and the diffusion in $V$} \label{estimates}

In this section, we prove some estimates for the processes $V$, $V^{\uparrow}$ and $\hat V^{\uparrow}$, especially about the hitting times and the exponential functionals of these processes. We also prove some facts used in Section \ref{genedesres}. Even though Section \ref{genedesres} gives the main ideas, the estimates we prove here actually represent the biggest part of the work for the proof of Theorem \ref{cvdutl}. 
%Those estimates, some of which have already been used in the last section, will allow us to study the diffusion in potential $V$, and in particular, to extend to this context the arguments given in \cite{advech} for the case of a drifted Brownian environment. 
Some of them are rather classical but some others, especially those of Subsection \ref{new technical}, are new and technical. 

\subsection{Estimates on $V$}

%We first recall a well-known estimate on the probability that $V$ leaves an interval from above. 

\begin{lemme} \label{leaveanint}

Let $a$, $b$ be positive numbers and define $T := \inf \{ x \geq 0, \ V(x) \notin [-a, b] \}$, then
\[ (1-e^{-\kappa a}) e^{-\kappa b} \leq \mathbb{P} \left ( V(T) = b \right ) \leq e^{-\kappa b}. \]

\end{lemme}

\begin{proof}

For the upper bound, $\mathbb{P} \left ( V(T) = b \right ) \leq \mathbb{P} ( \sup_{[0, +\infty[} V \geq b ) = e^{-\kappa b}$ {because, as mentioned in Subsection \ref{factsandnotations}, $\sup_{[0, +\infty[} V$ follows an exponential distribution with parameter $\kappa$.} 

For the lower bound, note that the process $e^{\kappa V(. \wedge T)}$ is a bounded martingale so, by the $L^1$-convergence theorem for martingales 
%and the dominated convergence theorem 
we get
\begin{align*}
1 = \mathbb{E} \left [ e^{\kappa V(0 \wedge T)} \right ] = \mathbb{E} \left [ e^{\kappa V(T)} \right ] & = \mathbb{P} \left ( V(T) \leq -a \right ) \mathbb{E} \left [ e^{\kappa V(T)} | V(T) \leq - a \right ] + \mathbb{P} \left ( V(T) = b \right ) e^{\kappa b} \\
& \leq e^{-\kappa a} + \mathbb{P} \left ( V(T) = b \right ) e^{\kappa b}. 
\end{align*}
This yields the result. 

\end{proof}

We now study how $V$ leaves an intervalle from below. More precisely, we control the moments of $V \left (\tau (V, ]-\infty, -1]) \right )$. This is through the next lemma that the assumption "$V (1) \in L^p$ for some $p > 1$" is useful for the proof of Theorem \ref{cvdutl}. 

\begin{lemme} \label{leaveanintbelow}

\[ \forall p \geq 1, \ V (1) \in L^p \Rightarrow V \left (\tau (V, ]-\infty, -1]) \right ) \in L^p. \]

\end{lemme}

\begin{proof}

We use $V^{< -r}$ as defined in Subsection \ref{factsandnotations}, where $r$ is chosen such that $V - V^{< -r}$ drifts to $-\infty$. Let $\kappa_r$ denote the non trivial zero of $\Psi_{V - V^{< -r}}$, {the Laplace exponent of $V - V^{< -r}$. Recall that $\Delta V(t) := V(t)-V(t-)$, that is, $\Delta V$ is the process of the jumps of $V$.} We fix $x \geq r$. We have
\begin{align*}
\mathbb{P} \left ( \left | 1 + V \left (\tau (V, ]-\infty, -1]) \right ) \right | > x \right ) & \leq \mathbb{P} \left ( \tau (\Delta V, ]-\infty, -x]) \leq \tau (V, ]-\infty, -1]) \right ) \\
& \leq \mathbb{P} \left ( V \left ( \tau (\Delta V, ]-\infty, -x]) - \right ) \geq -1 \right ) \\
& \leq \mathbb{P} \left ( [V - V^{< -r}] \left ( \tau (\Delta V, ]-\infty, -x]) - \right ) \geq -1 \right ). 
\end{align*}
Then, since $x > r$, we have $\tau (\Delta V, ]-\infty, -x]) = \tau (\Delta V^{< -r}, ]-\infty, -x])$ {and the latter is independent from $V - V^{< -r}$, because the processes $V - V^{< -r}$ and $V^{< -r}$ are independent.} We thus get that $\mathbb{P} ( | 1 + V (\tau (V, ]-\infty, -1]) ) | > x )$ is less than
\begin{align*}
\mathbb{P} \left ( [V - V^{< -r}] \left ( \tau (\Delta V^{< -r}, ]-\infty, -x]) \right ) \geq -1 \right ) & \leq e^{\kappa_r /2} \mathbb{E} \left [ e^{\kappa_r [V - V^{< -r}] \left ( \tau (\Delta V^{< -r}, ]-\infty, -x]) \right )/2} \right ] \\
& = e^{\kappa_r /2} \mathbb{E} \left [ e^{\Psi_{V - V^{< -r}} (\kappa_r /2) \tau (\Delta V^{< -r}, ]-\infty, -x])} \right ]. 
\end{align*}
We have used Markov's inequality {and the independence between $\tau (\Delta V^{< -r}, ]-\infty, -x])$ and the process $V - V^{< -r}$. Then, note that $\Psi_{V - V^{< -r}} (\kappa_r /2) < 0$ thanks to the definition of $\kappa_r$, and that $\tau (\Delta V^{< -r}, ]-\infty, -x])$ follows an exponential distribution with parameter $\nu (]-\infty, -x])$, where we recall that $\nu$ is the L\'evy measure of $V$}. We thus get
\begin{align}
\mathbb{P} \left ( \left | 1 + V \left (\tau (V, ]-\infty, -1]) \right ) \right | > x \right ) \leq \frac{e^{\kappa_r /2}}{1 - \Psi_{V - V^{< -r}} (\kappa_r /2)/\nu (]-\infty, -x])} \leq C \nu (]-\infty, -x]), \label{leaveanintbelow1}
\end{align}
where we put $C := - e^{\kappa_r /2}/\Psi_{V - V^{< -r}} (\kappa_r /2) > 0$. We now choose $p \geq 1$ and assume $V (1) \in L^p$. Theorem 25.3 in \cite{Sato} implies that $\int_{-\infty}^{-r} |x|^p \nu (dx) < + \infty$, or equivalently $\int_{r}^{+\infty} x^{p-1} \nu (]-\infty, -x]) dx < +\infty$. Using \eqref{leaveanintbelow1} we deduce that
\[ \int_{-\infty}^{-r} x^{p-1} \mathbb{P} \left ( \left | 1 + V \left (\tau (V, ]-\infty, -1]) \right ) \right | > x \right ) dx < +\infty, \]
so $V (\tau (V, ]-\infty, -1]) ) \in L^p$. 

\end{proof}

{The next lemma is a combination of known results on tail asymptotics for exponential functionals of L\'evy processes. Despite being classical it is fundamental.} In Section \ref{genedesres}, it allows us to compute precisely the right tails of the contributions to the local and to the time spent by the diffusion in the bottoms of the valleys. This allows to prove that the sum of these contributions converges to the $\kappa$-stable subordinator $(\mathcal{Y}_1, \mathcal{Y}_2)$ from which is constructed the limit distribution in Theorem \ref{cvdutl}. 
%A Cramer's condition \eqref{cramercond} that is required in this lemma is therefore important. However, if $V (1) \in L^p$, as assumed in Theorem \ref{cvdutl}, we use Holder's inequality and the fact that $V(1)$ has finite exponential moments of any positive order (since $V$ is spectrally negative), and we see that \eqref{cramercond} is satisfied. 

\begin{lemme} \label{foncexpov}

%\begin{itemize}
%\item There is a positive constant $c$ such that for $x$ small enough, 
%\[ \mathbb{P} \left ( \int_0^{+\infty} e^{V(u)} du \leq x \right ) \leq c \sqrt{x}. \]
%\item Fix $\gamma \in ]0, \kappa[$, then for $x$ large enough, 
%\[ \mathbb{P} \left ( \int_0^{+\infty} e^{V(u)} du \geq x \right ) \leq x^{- \gamma}. \]
%\item If $V$ satisfies the following Cramer's condition: 
%\begin{eqnarray}
%\mathbb{E} \left [ -V(1) e^{V(1)} \right ] < +\infty, \label{cramercond}
%\end{eqnarray}
%then
%\[ \mathbb{P} \left ( \int_0^{+\infty} e^{V(u)} du \geq x \right ) \underset{x \rightarrow +\infty}{\sim} \mathcal{C} x^{-\kappa}, \]
%\end{itemize}
\begin{itemize}
\item For small $x$ 
\[ {\mathbb{P} \left ( \int_0^{+\infty} e^{V(u)} du \leq x \right ) = o(x).} \]
\item If $0 < \kappa < 1$ we have 
\[ \mathbb{P} \left ( \int_0^{+\infty} e^{V(u)} du \geq x \right ) \underset{x \rightarrow +\infty}{\sim} \mathcal{C} x^{-\kappa}, \]
\end{itemize}
where {$\mathcal{C} := K / \Psi_V'(\kappa)$ is the constant defined a little before Theorem \ref{cvdutl}. }
\end{lemme}

{
\begin{remarque}
The problem of studying asymptotics related to the distributions of exponential functionals of L\'evy processes has a great interest for applications and has been the object of a lot of researches. We want to point out that an exact equivalent can be given not only for the asymptotic tail distribution of an exponential functional, as in the second point of Lemma \ref{foncexpov}, but also, under some conditions on the characteristic exponent of the underlying L\'evy process, for the density of the exponential functional and for the successive derivatives of that density. Such equivalents are established in Theorem 2.14 of Patie, Savov \cite{bersteingamma}. 
\end{remarque}
}

\begin{proof} of Lemma \ref{foncexpov}

{The first assertion is an application to $-V$ of Theorem 2.19 in \cite{bersteingamma}. That result asserts that, when $x$ goes to $0$, $\mathbb{P} ( \int_0^{+\infty} e^{V(u)} du \leq x )/x$ converges to $-\varphi(0)$, where $\varphi : i \mathbb{R} \rightarrow \mathbb{C}$ is the Laplace exponent of $-V$. Note that $-\varphi(0)$ is in fact the killing rate of the associated L\'evy process. In this paper $V$, and therefore also $-V$, is not a killed L\'evy process, which is equivalent to $\varphi(0) = 0$. This justifies the first point of the Lemma. }

{We now justify the second assertion which is only an application of Lemma 4 of \cite{rivero2005} (see also Corollary 5 of \cite{Bertoinyor} applied to $-V$). We thus only need to check that the hypothesis $(H2)$ in \cite{rivero2005} is satisfied by $V$. First, if $V$ was both spectrally negative and arithmetic then it would be the opposite of a subordinator. By assumption $V$ is spectrally negative and not the opposite of a subordinator. Condition $(H2a)$ is thus satisfied. Since $V$ is spectrally negative, we know that $V(1)$ has a Laplace transform defined on $[0, +\infty[$ and the $\theta$ in condition $(H2b)$ is the non-trivial zero of the Laplace exponent of $V$, that is, $\kappa$. Then, the condition $(H2c)$ is trivially satisfied since the restriction of the L\'evy measure of $V$ is null on $]0, +\infty[$. 
%To check the condition $(H2c)$, we use the decomposition $V = (V - V^{< -1}) + V^{< -1}$. We have
%\begin{align*}
%\mathbb{E} \left [ | V(x) | \exp (\kappa V(x)) \right ] & \leq \mathbb{E} \left [ |(V - V^{< -1})(x)| \exp \left ( \kappa (V - V^{< -1})(x) + \kappa V^{< -1}(x) \right ) \right ] \\ 
%& + \mathbb{E} \left [ |V^{< -1}(x)| \exp \left ( \kappa (V - V^{< -1})(x) + \kappa V^{< -1}(x) \right ) \right ] \\
%& \leq \mathbb{E} \left [ |(V - V^{< -1})(x)| \exp \left ( \kappa (V - V^{< -1})(x) \right ) \right ] \\ 
%& + M \mathbb{E} \left [ \exp \left ( \kappa (V - V^{< -1})(x) \right ) \right ]. 
%\end{align*}
%where we have used the fact that $V^{< -1}(x) \leq 0$ to bound the first term and, for the second term, the fact that, since $V^{< -1}(x) \leq 0$, $|V^{< -1}(x)| \exp ( \kappa V^{< -1}(x) )$ is deterministically bounded by $M$, the constant bounding $(y \mapsto y e^{-\kappa y})$ on $\mathbb{R}_+$. Then, since $V - V^{< -1}$ is a L\'evy process with bounded jumps, $(V - V^{< -1})(x)$ admits finite exponential moments of any positive and negative order (see Theorem 25.3 in \cite{Sato}). As a consequence the above expression is finite which yields the condition $(H2c)$. 
The condition $(H2)$ in \cite{rivero2005} is thus satisfied by $V$. As a consequence, Lemma 4 of \cite{rivero2005} can be applied in the case $0 < \kappa < 1$ with $\alpha = 1$. This yields $\mathbb{P} ( \int_0^{+\infty} e^{V(u)} du \geq x ) \underset{x \rightarrow +\infty}{\sim} (K / \Psi_V'(\kappa)) x^{-\kappa}$. The second assertion is thus proved.} 

\end{proof}

\begin{lemme} \label{tpsatteinthatv}

There are two positive constants $c_1, c_2$ such that
\[ \forall y, r > 0, \ \mathbb{P} \left ( \tau(V, ]-\infty, -y]) > r \right ) \leq e^{c_1 y -c_2 r}. \]

\end{lemme}

\begin{proof}

Let us choose $c_1 \in ]0, \kappa[$ and define $c_2 := - \Psi_V(c_1)$. $c_2$ is positive because of the definition of $\kappa$. We have
\[ \mathbb{P} \left ( \tau(V, ]-\infty, -y]) > r \right ) \leq \mathbb{P} \left ( V(r) > -y \right ) \leq e^{c_1 y} \mathbb{E} \left [ e^{c_1 V(r)} \right ] = e^{c_1 y -c_2 r}, \]
where we have used Markov's inequality. 

\end{proof}

{Recall that $V - \underline{V}$, defined in Subsection \ref{factsandnotations}, denotes the process $V$ reflected at its infimum. We have :} 

\begin{lemme} \label{monteavantdescente}

Choose $\eta \in ]0, 1[$, then
\[ \forall a, b > 0, \ \mathbb{P} \left ( \tau(V-\underline{V}, a) < \tau(V, ]-\infty, -b]) \right ) \leq (b/\eta a +1) e^{-\kappa (1-\eta) a}. \]

\end{lemme}

\begin{proof}

We first remark that 
\[ \left \{ \tau(V-\underline{V}, a) < \tau(V, ]-\infty, - \eta a]) \right \} \subset \left \{ \sup_{[0, +\infty[} V > (1-\eta) a \right \}, \] so
\begin{eqnarray}
\mathbb{P} \left ( \tau(V-\underline{V}, a) < \tau(V, ]-\infty, - \eta a]) \right ) \leq e^{-\kappa (1-\eta) a}. \label{minimacoincide2}
\end{eqnarray}
{We have used the fact that, as mentioned in Subsection \ref{factsandnotations}, $\sup_{[0, +\infty[} V$ follows an exponential distribution with parameter $\kappa$.} In order to establish a boundary with $b$ instead of $\eta a$, we define the sequence of stopping times $(T_i)_{i \geq 0}$ by $T_0 := 0$
and 
\[ T_{i+1} := T_{i} + \min \left \{ \tau(V^{T_i}-\underline{V}^{T_i}, a), \tau \left (V^{T_i}, \left ]-\infty, -\eta a \right ] \right ) \right \}. \] 
We have 
\[ \overset{ \lfloor b/\eta a \rfloor}{\underset{i=0}{\cap}} \left \{ \tau(V^{T_i}-\underline{V}^{T_i},a) > \tau(V^{T_i}, ]-\infty, - \eta a]) \right \} \subset \left \{ \tau(V-\underline{V}, a) > \tau(V, ]-\infty, -b]) \right \}. \]
By the Markov property applied to $V$ at the stopping times $T_i$, the $\lfloor b/\eta a + 1 \rfloor$ events in the intersection have all the same probability equal to $\mathbb{P} (\tau(V-\underline{V},a) > \tau(V, ]-\infty, - \eta a]))$. Taking the complementary in the above expression and combining with \eqref{minimacoincide2} we get 
\begin{align*}
\mathbb{P} \left ( \tau(V-\underline{V}, a) < \tau(V, ]-\infty, -b]) \right ) & \leq \lfloor b/ \eta a +1 \rfloor  \mathbb{P} \left ( \tau(V-\underline{V}, a) < \tau(V, ]-\infty, - \eta a]) \right ) \\
& \leq (b/\eta a +1) e^{-\kappa (1-\eta) a}. 
\end{align*}

\end{proof}

\subsection{Estimates on $V^{\uparrow}$} \label{estimvup}

%We begin with an estimate on the hitting times of $V^{\uparrow}$: 

{Recall from Subsection \ref{extrema} the definition of the spectrally negative L\'evy process $V^{\sharp}$, known as \textit{$V$ conditioned to drift to $+\infty$}. Recall from there that $V^{\sharp}$ drifts to $+\infty$ (equivalently $\Psi'_{V^{\sharp}}(0) > 0$) and that for any $x > 0$, $V^{\uparrow}_x$ is equal in law to $V^{\sharp}_x$ conditionally on $\{ \inf_{[0, +\infty[} V^{\sharp}_x > 0 \}$. 
%The Laplace exponent $\Psi_{V^{\sharp}}$ of $V^{\sharp}$ satisfies $\Psi_{V^{\sharp}} = \Psi_V(\kappa + .)$. As a consequence $\Psi'_{V^{\sharp}}(0) > 0$, so $V^{\sharp}$ drift to infinity (because of Corollary VII.2 in \cite{Bertoin}) and it is also proven that $V^{\uparrow} = (V^{\sharp})^{\uparrow}$. Therefore, for $x > 0$, $V^{\uparrow}_x$ is only $V^{\sharp}_x$ conditioned in the usual {way} to remain positive, which is a useful property for our proofs. 
Also, note that from \eqref{convsgsharp} we have }
%in Section VII.1 of \cite{Bertoin}, we have
\begin{eqnarray}
\forall \lambda > - \kappa, \ \mathbb{E} \left [ e^{\lambda V^{\sharp}(1)} \right ] = \mathbb{E} \left [ e^{(\lambda + \kappa) V(1)} \right ] < + \infty. \label{eqintro1}
\end{eqnarray}
In particular, $V^{\sharp}(1)$ has its Laplace transform, as well as its Laplace exponent $\Psi_{V^{\sharp}}$, defined on the half-plane $\{ \mathcal{R}(z) > - \kappa \}$. 

\begin{lemme} \label{vposlapltpsatt}

There are two positive constants $c_1, c_2$ such that 
\[ \forall y, r > 0, \ \mathbb{P} \left ( \tau(V^{\uparrow}, y) > r \right ) \leq e^{c_1 y - c_2 r}. \]

\end{lemme}

\begin{proof}

Fix $y$ and $r > 0$. From the first point of Lemma 2.6 in \cite{foncexpovech} we have
\begin{eqnarray}
\mathbb{P} \left ( \tau(V^{\uparrow}, y) > r \right ) \leq \mathbb{P} \left ( \tau(V^{\sharp}, y) > r \right ). \label{newproof}
\end{eqnarray}

Then, $V^{\sharp}$ is a spectrally negative L\'evy process so, according to Theorem VII.1 in \cite{Bertoin}, the process $\tau ( V^{\sharp}, . )$ is a subordinator. Its Laplace exponent $\Phi_{V^{\sharp}}$ is defined for $\lambda \geq 0$ by 
\[ \Phi_{V^{\sharp}}(\lambda) := - \log \left ( \mathbb{E} \left [ e^{-\lambda \tau \left ( V^{\sharp}, 1 \right )} \right ] \right ) \]
and, according to Theorem VII.1 in \cite{Bertoin}, we have $\Phi_{V^{\sharp}} = \Psi_{V^{\sharp}}^{-1}$. 

From the discussion before the lemma, we know that $V^{\sharp}(1)$ has its Laplace transform, as well as its Laplace exponent $\Psi_{V^{\sharp}}$, defined in a neighborhood of $0$. Then, since $\Psi'_{V^{\sharp}}(0) > 0$, the holomorphic local inversion theorem yields that $\Psi_{V^{\sharp}}^{-1}$, that is, $\Phi_{V^{\sharp}}$, extends on a neighborhood of $0$. 

Therefore, $\tau ( V^{\sharp}, 1 )$ has a Laplace transform defined on a neighborhood of $0$. From Markov's inequality we get, for a positive $c_2$ in this neighborhood, 
\[ \mathbb{P} \left ( \tau ( V^{\sharp}, y ) > r \right ) \leq e^{-c_2 r} \mathbb{E} \left [ e^{c_2 \tau \left ( V^{\sharp}, y \right )} \right ] = e^{- y \Phi_{V^{\sharp}}(- c_2)} e^{-c_2 r}. \]

Note that $c_1 := - \Phi_{V^{\sharp}}(- c_2)$ is positive. Combining with \eqref{newproof} we get the result. 

\end{proof}

\begin{lemme} \label{vuprestegrand}

There are two positive constants $c_1, c_2$ such that, for all $1 < a < b$, we have 
%\[ \mathbb{P} \left ( \inf_{[0, +\infty[} V_b^{\uparrow} < a \right ) \leq c_2 e^{-c_1 (b-a)} / \mathbb{P} \left ( \inf_{[0, +\infty[} V^{\sharp} > -b \right ). \]
\[ \mathbb{P} \left ( \inf_{[0, +\infty[} V_b^{\uparrow} < a \right ) \leq c_2 e^{-c_1 (b-a)}. \]

\end{lemme}

\begin{proof}

Recall that $V^{\uparrow}_{b}$ is equal in law to $V^{\sharp}_{b}$ conditioned in the usual {way} to stay positive, so
\begin{align*}
\mathbb{P} \left( \inf_{[0, +\infty[} V_b^{\uparrow} < a \right) & = \mathbb{P} \left( 0 < \inf_{[0, +\infty[} V_b^{\sharp} < a \right) / \mathbb{P} \left( \inf_{[0, +\infty[} V^{\sharp}_{b} > 0 \right) \\
& \leq \mathbb{P} \left( \inf_{[0, +\infty[} V^{\sharp} < a-b \right) / \mathbb{P} \left( \inf_{[0, +\infty[} V^{\sharp} > -b \right) \\
& \leq \mathbb{P} \left( \inf_{[0, +\infty[} V^{\sharp} < a-b \right) / \mathbb{P} \left( \inf_{[0, +\infty[} V^{\sharp} > -1 \right). 
\end{align*}
The formula page 192 of \cite{Bertoin} gives the Laplace transform of $\inf_{[0, +\infty[} V^{\sharp}$: 
\begin{eqnarray}
\forall \lambda \geq 0, \ \mathbb{E} \left [ e^{\lambda \inf_{[0, +\infty[} V^{\sharp}} \right ] = \Psi_{V^{\sharp}}'(0) \frac{\lambda}{\Psi_{V^{\sharp}}(\lambda)}. \label{vuprestegrand1}
\end{eqnarray}
As we said in the beginning of this subsection, $\Psi_{V^{\sharp}}$ extends analytically on a neighborhood of $0$ and has a non null derivative at $0$. As a consequence, \eqref{vuprestegrand1} extends analytically too and $\inf_{[0, +\infty[} V^{\sharp}$ admits a Laplace transform on a neighborhood of $0$. We choose $c_1 > 0$ such that $-c_1$ is in this neighborhood and $c_2 := \mathbb{E} [ e^{-c_1 \inf_{[0, +\infty[} V^{\sharp}} ] / \mathbb{P} \left ( \inf_{[0, +\infty[} V^{\sharp} > -1 \right )$. {We then have 
\begin{align*}
\mathbb{P} \left( \inf_{[0, +\infty[} V_b^{\uparrow} < a \right) & \leq \mathbb{P} \left( e^{-c_1 \inf_{[0, +\infty[} V^{\sharp}} > e^{-c_1 (a-b)} \right) / \mathbb{P} \left( \inf_{[0, +\infty[} V^{\sharp} > -1 \right) \\
& \leq e^{-c_1 (b-a)} \mathbb{E} \left [ e^{-c_1 \inf_{[0, +\infty[} V^{\sharp}} \right ] / \mathbb{P} \left( \inf_{[0, +\infty[} V^{\sharp} > -1 \right) \\
& = c_2 e^{-c_1 (b-a)}, 
\end{align*}
where we have used Markov's inequality.} 

\end{proof}

\begin{lemme} \label{estpourvup}
There are two positive constants $c_1$ and $c_2$ such that for any $0< \alpha  <\omega $, $\eta \in ]0, 1[$ and all $h$ large enough, we have
\begin{eqnarray}
%\mathbb{P} \left(\tau ( V^{\uparrow}_{\alpha h}, \gamma  h ) < \tau ( V^{\uparrow}_{\alpha h}, \omega  h ) \right)
%& \leq  &
%    C e^{- c_1 (\alpha-\gamma) h }, 
%    \label{03.10}
%\\
    \mathbb{P} \left( \tau ( V^{\uparrow}, \omega  h ) - \tau ( V^{\uparrow}, \alpha h ) \leq 1 \right)
& \leq &
    e^{- c_1 (\omega - \alpha) h }, 
    \label{03.10b}
\\
%***    P\left(\tau^{ V_{h}}( \gamma  h ) \leq 1\right)
%& \leq  &
%%    1-2\exp[-[(1-\gamma)h-\kappa/2]^2/2],\qquad \gamma\neq 1,
%    2\exp[-[(1-\gamma)h]^2/3], \gamma\neq 1,
%    \label{03.14}
%\\
\mathbb{P} \left ( \int_0^{\tau ( V^{\uparrow}, h )} e^{V^{\uparrow}(u)} d u \geq e^{(1-\eta) h} \right )
& \geq  &
    1-e^{- c_2 \eta h},
\label{0MinorationAVallee}
\\
\mathbb{P} \left ( \int_0^{\tau ( V^{\uparrow}, h )} e^{V^{\uparrow}(u)} d u \leq e^{(1+\eta) h} \right )
& \geq  &
    1-e^{- h}.
\label{0MajorationAVallee}
%\\
%    P[\tau^R(h)>8h/\kappa ]
%& \leq &
%    C_+ \exp[- \kappa h /(2\sqrt{2})].
%    \label{0bessel4}
%%  P\left(\tau^{ W_{\kappa}}(- \alpha  h ) \geq 2 \omega h/ \kappa  \right) & \leq  & e^{-\kappa  (\omega-\alpha)^2 h/4 \omega}. \label{3.14b}
\end{eqnarray}
\end{lemme}

\begin{proof}

%From subsection \ref{conddrift}, $V^{\uparrow}_{\alpha h}$ is just $V^{\sharp}_{\alpha h}$ conditioned to stay positive in the usual {way}, so
%\begin{align*}
%\mathbb{P} \left(\tau ( V^{\uparrow}_{\alpha h}, \gamma  h ) < \tau ( V^{\uparrow}_{\alpha h}, \omega  h ) \right) & = \mathbb{P} \left(\tau ( V^{\sharp}_{\alpha h}, \gamma  h ) < \tau ( V^{\sharp}_{\alpha h}, \omega  h ) \right) / \mathbb{P} \left( \inf_{[0, +\infty[} V^{\sharp}_{\alpha h} > 0 \right) \\
%& \leq 2 \mathbb{P} \left(\tau ( V^{\sharp}_{\alpha h}, \gamma  h ) < \tau ( V^{\sharp}_{\alpha h}, \omega  h ) \right), 
%\end{align*}
%for $h$ large enough, 
%\[ \leq 2 \mathbb{P} \left( \inf_{[0, +\infty[} V^{\sharp}_{\alpha h} \leq \gamma  h \right) = 2 \mathbb{P} \left( \inf_{[0, +\infty[} V^{\sharp} \leq (\gamma - \alpha)  h \right). \]
%Then, according to \cite{Bertoin}, the Laplace transform of $\inf_{[0, +\infty[} V^{\sharp}$ is given by
%\[ \forall \lambda \geq 0, \ \mathbb{E} \left [ e^{\lambda \inf_{[0, +\infty[} V^{\sharp}} \right ] = \Phi_{V^{\sharp}}'(0) \frac{\lambda}{\Phi_{V^{\sharp}}(\lambda)}. \]
%Since, as we saw in the proof of lemma \ref{vposlapltpsatt}, $\Phi_{V^{\sharp}}$ extends analytically on a neighborhood of $0$, this expression extends too and $\inf_{[0, +\infty[} V^{\sharp}$ admits a Laplace transform on a neighborhood of $0$ which gives an exponential inequality just as $(\ref{03.10})$. 
%
From the Markov property at time $\tau ( V^{\uparrow}, \alpha h )$ and the fact that $V^{\uparrow}_{\alpha h}$ is equal in law to $V^{\sharp}_{\alpha h}$ conditioned in the usual {way} to stay positive, we have that $\mathbb{P} ( \tau ( V^{\uparrow}, \omega  h ) - \tau ( V^{\uparrow}, \alpha h ) \leq 1 )$ equals
\begin{align*}
\mathbb{P} \left( \tau ( V^{\uparrow}_{\alpha h}, \omega  h ) \leq 1 \right) & = \mathbb{P} \left( \tau ( V^{\sharp}_{\alpha h}, \omega  h ) \leq 1, \ \inf_{[0, +\infty[} V^{\sharp}_{\alpha h} > 0  \right) / \mathbb{P} \left( \inf_{[0, +\infty[} V^{\sharp}_{\alpha h} > 0 \right) \\
& = \mathbb{P} \left( \tau ( V^{\sharp}, (\omega - \alpha) h ) \leq 1, \ \inf_{[0, +\infty[} V^{\sharp} > -\alpha h  \right) / \mathbb{P} \left( \inf_{[0, +\infty[} V^{\sharp} > - \alpha h \right) \\
& \leq 2 \mathbb{P} \left( \tau ( V^{\sharp}, (\omega - \alpha) h ) \leq 1 \right). 
\end{align*}
This last inequality is true for $h$ large enough. Then, the latter equals 
\[ 2 \mathbb{P} \left( e^{-\tau ( V^{\sharp}, (\omega - \alpha) h )} \geq e^{-1} \right) \leq 2 e \times \mathbb{E} \left [ e^{- \tau ( V^{\sharp}, (\omega - \alpha) h )} \right ] = 2 e \times e^{- (\omega - \alpha) h \Phi_{V^{\sharp}}(1) }. \]
%\[ = 2 e \mathbb{E} \left [ e^{- \psi(1) (\omega - \alpha) h } \right ] , \]
We have used Markov's inequality and the fact that $\tau ( V^{\sharp}, . )$ is a subordinator with Laplace exponent $\Phi_{V^{\sharp}} = \Psi_{V^{\sharp}}^{-1}$ (see the proof of Lemma \ref{vposlapltpsatt}). Note that $\Phi_{V^{\sharp}}(1) > 0$ (because $\tau ( V^{\sharp}, . )$ is a subordinator). This proves \eqref{03.10b}. 

{Because of the Markov property at time $\tau ( V^{\uparrow}, (1 - \eta/2)h )$, $V^{\uparrow}(.+\tau ( V^{\uparrow}, (1 - \eta/2)h ))$ is equal in law to $V^{\uparrow}_{(1 - \eta/2)h}$. We can thus apply Lemma \ref{vuprestegrand} with $a = (1 - \eta)h$, $b = (1 - \eta/2)h$ and we get $\mathbb{P} ( \inf_{[\tau ( V^{\uparrow}, (1 - \eta/2)h ), +\infty[} V^{\uparrow} \geq (1 - \eta)h ) \geq 1 - K_2 e^{-K_1 \eta h/2}$ for some positive constants $K_1$ and $K_2$. \eqref{03.10b} applied with $\alpha = 1-\eta / 2$, $\omega = 1$ yields $\mathbb{P} \left( \tau ( V^{\uparrow}, h ) - \tau ( V^{\uparrow}, (1 - \eta/2)h ) > 1 \right) \geq 1 - e^{- c_1 \eta h/2 }$ for large $h$. Combining all this we get \eqref{0MinorationAVallee} when $h$ is large. }

Then, we have obviously 
\[ \int_{0}^{\tau(V^{\uparrow}, h)} e^{V^{\uparrow}(u)} du \leq e^h \tau(V^{\uparrow}, h), \]
so
\[ \mathbb{P} \left ( \int_{0}^{\tau(V^{\uparrow}, h)} e^{V^{\uparrow}(u)} du > e^{(1+\eta)h} \right ) \leq \mathbb{P} \left ( \tau(V^{\uparrow}, h) > e^{\eta h} \right ) \leq \exp \left( K_3 h - K_4 e^{\eta h} \right ). \]
{In the last inequality we have used Lemma \ref{vposlapltpsatt} with $y=h$, $r=e^{\eta h}$, and $K_3$, $K_4$ are some positive constants. For $h$ large enough we have $\exp(K_3 h - K_4 e^{\eta h}) \leq e^{-h}$. This yields \eqref{0MajorationAVallee}. }

\end{proof}

\subsection{Estimates on $\hat V^{\uparrow}$}

%One of the aims of this section is to get an estimate similar to Lemma \ref{vposlapltpsatt}, but for $\hat V^{\uparrow}$ instead of $V^{\uparrow}$. We start by proving that $\hat V^{\uparrow}$ does not get too small after its hitting times. This is fundamental to generalize Lemma \ref{vposlapltpsatt} to $\hat V^{\uparrow}$, and also to show that the valleys of $V$ have a "nice" shape, so that the supremum of the local time of the diffusion in $V$ is reached only in the bottoms of the valleys. 

The aim of this subsection is to get, for $\hat V^{\uparrow}$, estimates similar to those that we have proved for $V^{\uparrow}$ in Subsection \ref{estimvup}. We start with a generalization of Lemma \ref{vuprestegrand}. 

\begin{lemme} \label{vdownrestegrand}

For all $z \geq 0$ and $0 < a < b$ with $b>z$, we have, 
\[ \mathbb{P} \left ( \inf_{[\tau(\hat V_z^{\uparrow}, [b, +\infty[), +\infty[} \hat V_z^{\uparrow} < a \right ) \leq e^{-\kappa (b-a)} / (1-e^{-\kappa b}). \]

\end{lemme}

\begin{proof}

Let $T := \tau(\hat V_z^{\uparrow}, [b, +\infty[)$. $T$ is a stopping time for $\hat V_z^{\uparrow}$ so, if $U$ is a random variable having the same law as $\hat V^{\uparrow}_z(T)$, then $\hat V^{\uparrow}_z(T + .)$ is equal in law to $\hat V^{\uparrow}_U$, that is, the Markov process that conditionally on $\{ U=u \}$ has law $\hat V^{\uparrow}_u$. 
%$\hat V$ conditioned to stay positive and starting from the random variable $U$. 
We have
\begin{eqnarray}
\mathbb{P} \left ( \inf_{[T, +\infty[} \hat V^{\uparrow}_z < a \right ) = \mathbb{P} \left ( \inf_{[0, +\infty[} \hat V^{\uparrow}_U < a \right ) = \int_{b}^{+\infty} \mathbb{P} \left ( \inf_{[0, +\infty[} \hat V^{\uparrow}_u < a \right ) \times \mathbb{P} \left ( U \in du \right ). \label{tlf3}
\end{eqnarray}
Note that we have used that almost surely, $U \geq b$. Now, since $\hat V^{\uparrow}_u$ is equal in law to $\hat V_u$ conditioned in the usual {way} to remain positive, we have for any $u \geq b$, 
\begin{align*}
\mathbb{P} \left ( \inf_{[0, +\infty[} \hat V^{\uparrow}_u < a \right ) & = \mathbb{P} \left ( 0 < \inf_{[0, +\infty[} \hat V_u < a \right ) / \mathbb{P} \left ( \inf_{[0, +\infty[} \hat V_u > 0 \right ) \nonumber \\
& \leq \mathbb{P} \left ( \inf_{[0, +\infty[} \hat V < a-u \right ) / \mathbb{P} \left ( \inf_{[0, +\infty[} \hat V > -u \right ) \nonumber \\
& \leq \mathbb{P} \left ( \inf_{[0, +\infty[} \hat V < a-b \right ) / \mathbb{P} \left ( \inf_{[0, +\infty[} \hat V > -b \right ). 
\end{align*}
For the last inequality we have used the fact that $a -u \leq a-b$ for the numerator and the fact that $u \geq b$ for the denominator. Since $\hat V$ is equal in law to $-V$, the above estimate can be re-written 
\[ \mathbb{P} \left ( \inf_{[0, +\infty[} \hat V^{\uparrow}_u < a \right ) \leq \mathbb{P} \left ( \sup_{[0, +\infty[} V > b-a \right ) / \mathbb{P} \left ( \sup_{[0, +\infty[} V < b \right ) = e^{-\kappa (b-a)} / (1-e^{-\kappa b}). \]
{The last equality comes from the fact that, as mentioned in Subsection \ref{factsandnotations}, $\sup_{[0, +\infty[} V$ follows an exponential distribution with parameter $\kappa$.} Putting the above expression into \eqref{tlf3} we get the result. 
%\begin{eqnarray}
%\mathbb{P} \left ( \inf_{[T, +\infty[} \hat V^{\uparrow}_x < x \right ) \leq c_0 e^{-\kappa y}. \label{tlf4}
%\end{eqnarray}

\end{proof}

The proof of Lemma \ref{vposlapltpsatt} relies on the fact that the hitting times of $V^{\uparrow}$ are stochastically smaller than the hitting times of $V^{\sharp}$, for which we have precise estimates. Here, the procedure to link the hitting times of $\hat V^{\uparrow}$ and $\hat V$ is different. This is what we do now. 
%Here, even if Lemma \ref{tpsatteinthatv} gives us a precise estimate for the hitting times of $\hat V$, it is a little harder to link the hitting times for $\hat V^{\uparrow}$ and $\hat V$. 

Let us fix $x > 0$. Let $m_x$ be the point where the process $\hat V^{\uparrow}_x$ reaches its infimum, $m_x := \sup \{ s \geq 0, \ \hat V^{\uparrow}_x(s-) \wedge \hat V^{\uparrow}_x(s) = \inf_{[0, +\infty[} \hat V^{\uparrow}_x \}$. Note that from the absence of negative jumps, the infimum is always reached at least at $m_x-$ so $\hat V^{\uparrow}_x(m_x-) = \inf_{[0, +\infty[} \hat V^{\uparrow}_x$. The next lemma is contained in Theorem 25 of \cite{Doney}. 

\begin{lemme} \label{chaumont}

%\begin{itemize}
%\item The two processes $\left ( \hat V^{\uparrow}_y(m_y + s) - \hat V^{\uparrow}_y(m_y-), \ s \geq 0 \right )$ and $\left (\hat V^{\uparrow}_y(s), \ 0 \leq s < m_y \right )$ are independent, 
%
Assume that $V$ has unbounded variation, then 
\[ \left ( \hat V^{\uparrow}_x(m_x + s) - \hat V^{\uparrow}_x(m_x-), \ s \geq 0 \right ) \egloi \hat V^{\uparrow}. \]
%\item For any $z \in ]0, y[$, conditionally on $\{ \hat V^{\uparrow}_y (m_y) = z \}$ the process $\left (\hat V^{\uparrow}_y(s) \right )_{0 \leq s \leq m_y}$ is equal in law to $\hat V_y$ conditioned to die at $z$, which is defined in Subsection \ref{condstaypos}. \underset{[0, +\infty[}{\inf}
%\end{itemize}

\end{lemme}

We can thus obtain $\hat V^{\uparrow}$ from $\hat V^{\uparrow}_x$, and $\hat V^{\uparrow}_x$ is equal in law to $\hat V_x$ conditioned in the usual {way} to remain positive. This allows us to link $\hat V^{\uparrow}$ and $\hat V$, so we are now able to prove our estimate :

\begin{lemme} \label{tpsatteintvdown}

Assume that $V$ has unbounded variation. There are three positive constants $c_0, c_1$ and $c_2$ such that
\[ \forall y > 1, \ r > 0, \ \mathbb{P} \left ( \tau(\hat V^{\uparrow}, [y, +\infty[) > r \right ) \leq c_0 \left ( e^{-\kappa y} \ + \ e^{c_1 y -c_2 r} \right ). \]

\end{lemme}

\begin{proof}

Let us fix $r > 0$, $y >1$ and choose $x \in ]0, 1[$ (for example $x :=1/2$). According to Lemma \ref{chaumont} we have
\begin{align*}
\mathbb{P} \left ( \tau(\hat V^{\uparrow}, [y, +\infty[) > r \right ) & = \mathbb{P} \left ( \tau(\hat V^{\uparrow}_x(m_x + .) - \hat V^{\uparrow}_x(m_x-), [y, +\infty[) > r \right ) \\
& \leq \mathbb{P} \left ( \tau(\hat V^{\uparrow}_x(m_x + .), [y+x, +\infty[) > r \right ), 
\end{align*}
because $\hat V^{\uparrow}_x(m_x-) \leq x$. Now, if $\hat V^{\uparrow}_x$ never reaches $[0, x]$ after the instant $T := \tau(\hat V^{\uparrow}_x, [y+x, +\infty[)$, then the minimum $m_x$ is reached before $T$. In that case 
\[ \tau \left ( \hat V^{\uparrow}_x(m_x + .), [y+x, +\infty[ \right ) = \tau \left ( \hat V^{\uparrow}_x, [y+x, +\infty[ \right ) - m_x = T - m_x \leq T. \]
We deduce that
\begin{eqnarray} \mathbb{P} \left ( \tau(\hat V^{\uparrow}, [y, +\infty[) > r \right ) \leq \mathbb{P} \left (T > r \right ) + \mathbb{P} \left ( \inf_{[T, +\infty[} \hat V^{\uparrow}_x < x \right ). \label{tlf1}
\end{eqnarray}

We now bound the two terms of the right hand side. For the first term, since $\hat V^{\uparrow}_x$ is equal in law to $\hat V_x$ conditioned in the usual {way} to remain positive, {we have
\begin{align*}
\mathbb{P} \left (T > r \right ) & = \mathbb{P} \left ( \tau(\hat V_x, [y+x, +\infty[) > r, \ \inf_{[0, +\infty[} \hat V_x > 0 \right ) / \mathbb{P} \left ( \inf_{[0, +\infty[} \hat V_x > 0 \right ) \\
& \leq \mathbb{P} \left ( \tau(\hat V_x, [y+x, +\infty[) > r \right ) / \mathbb{P} \left ( \inf_{[0, +\infty[} \hat V_x > 0 \right ) \\
& = \mathbb{P} \left ( \tau(\hat V, [y, +\infty[) > r \right ) / \mathbb{P} \left ( \inf_{[0, +\infty[} \hat V > -x \right ) \\
& = \mathbb{P} \left ( \tau(V, ]-\infty, -y]) > r \right ) / \mathbb{P} \left ( \sup_{[0, +\infty[} V < x \right ) \\
& = \mathbb{P} \left ( \tau(V, ]-\infty, -y]) > r \right ) / (1 - e^{-\kappa x}).
\end{align*}
We have used that $\hat V$ is equal in law to $-V$ and that, as mentioned in Subsection \ref{factsandnotations}, $\sup_{[0, +\infty[} V$ follows an exponential distribution with parameter $\kappa$. }
Combining with Lemma \ref{tpsatteinthatv} we get
\begin{eqnarray}
\mathbb{P} \left (T > r \right ) \leq c_0 e^{c_1 y -c_2 r}, \label{tlf2}
\end{eqnarray}
where $c_1$ and $c_2$ are the constants in the lemma and $c_0 := 1/ (1 - e^{-\kappa x})$. We now turn to the second term of \eqref{tlf1}. According to Lemma \ref{vdownrestegrand} applied with $z = x$, $a = x$ and $b = x + y$ we get 
\begin{eqnarray}
\mathbb{P} \left ( \inf_{[T, +\infty[} \hat V^{\uparrow}_x < x \right ) \leq e^{-\kappa y} /(1-e^{-\kappa (x+y)}) \leq e^{-\kappa y} /(1-e^{-\kappa x}) = c_0 e^{-\kappa y}. \label{tlf4}
\end{eqnarray}

Now, combining \eqref{tlf2} and \eqref{tlf4} with \eqref{tlf1} we get the result. 

\end{proof}

%Just as we did for $V^{\uparrow}$, w
We can now use the previous estimate on the hitting times of $\hat V^{\uparrow}$ to bound an exponential functional of $\hat V^{\uparrow}$. 

\begin{lemme} \label{estpourvdown}
Assume that $V$ has unbounded variation and fix $\eta >0$. There is $h_{\eta} > 0$ such that 
\begin{eqnarray}
\forall h \geq h_{\eta}, \ \mathbb{P} \left ( \int_0^{\tau ( \hat V^{\uparrow}, [h, + \infty[ )} e^{\hat V^{\uparrow}(u)} d u \leq e^{(1+\eta) h} \right ) \geq 1-e^{- \kappa h /2}. \label{d0MajorationAVallee}
%\\
%    P[\tau^R(h)>8h/\kappa ]
%& \leq &
%    C_+ \exp[- \kappa h /(2\sqrt{2})].
%    \label{0bessel4}
%%  P\left(\tau^{ W_{\kappa}}(- \alpha  h ) \geq 2 \omega h/ \kappa  \right) & \leq  & e^{-\kappa  (\omega-\alpha)^2 h/4 \omega}. \label{3.14b}
\end{eqnarray}
\end{lemme}

\begin{proof}

%We choose $c \in ]0, \kappa[$. 
We have 
\[ \int_{0}^{\tau(\hat V^{\uparrow}, [h, + \infty[)} e^{\hat V^{\uparrow}(u)} du \leq e^h \tau(\hat V^{\uparrow}, [h, + \infty[), \]
so
\[ \mathbb{P} \left ( \int_{0}^{\tau(\hat V^{\uparrow}, [h, + \infty[)} e^{\hat V^{\uparrow}(u)} du > e^{(1+\eta)h} \right ) \leq \mathbb{P} \left ( \tau(\hat V^{\uparrow}, [h, + \infty[) > e^{\eta h} \right ). \]
{According to Lemma \ref{tpsatteintvdown} applied with $y=h$, $r=e^{\eta h}$, the latter is less than $c_0 [ \exp(-\kappa h) + \exp(c_1 h -c_2 e^{\eta h}) ]$, where $c_0, c_1, c_2$ are the constants in the lemma. We have $c_0 [ \exp(-\kappa h) + \exp(c_1 h -c_2 e^{\eta h}) ] \leq e^{- \kappa h /2}$ whenever $h$ is large enough (how large depends on $\eta$). This yields \eqref{d0MajorationAVallee}. }

\end{proof}

Note that we do not prove an analogous of \eqref{0MinorationAVallee} for $\hat V^{\uparrow}$, even if it would have been needed to repeat readily the arguments of \cite{advech} in our context. This is because the existence of possibly large negative jumps for $V$ do not allow such an estimate to hold in general. Because of this, we have to take some precautions, and in particular, to prove some extra technical estimates in Subsection \ref{new technical}. 

\subsection{Estimates on the first ascend of $h$ from the minimum}

In order to bound the local time and the amount of time spent by the diffusion between two valleys, we have to study the expectation of some functionals of $V$ involving the first ascend of $h$ from the minimum. This subsection uses the notations and estimates of Subsection \ref{firstmin}. 

\begin{lemme} \label{esptpsmonte}

There is a positive constant $C$ such that for $h$ large enough, 
\[ \mathbb{E} \left [ \tau^*(h) \right ] \leq C e^{\kappa h}. \]

\end{lemme}

\begin{proof}

We have $\tau^*(h) = m^*(h) + [\tau^*(h) - m^*(h)]$, and using Lemma \ref{asympminsubexp} we get 
%$m^*(h) \overset{\mathcal{L}}{=} S^{h, -}(T_h)$ and $\tau^*(h) - m^*(h) \overset{\mathcal{L}}{=} \tau(V^{\uparrow}, h))$. Therefore, 
\begin{eqnarray}
\mathbb{E} \left [ \tau^*(h) \right ] = \mathbb{E} \left [ S^{h, -}(T_h) \right ] + \mathbb{E} \left [ \tau(V^{\uparrow}, h) \right ]. \label{esptpsmonte1}
\end{eqnarray}

For the first term, note that $\mathbb{E} [ S^{h, -}(T_h) ] = \mathbb{E} [ S^{h, -}(1)] \times \mathbb{E} [ T_h ]$ since $S^{h, -}$ is a subordinator and is independent from $T_h$. Also, recall from Subsection \ref{firstmin} that $S^{h, -}(1) \leq S(1)$ where $S(1)$ has finite expectation according to Lemma \ref{asympminesp}. Then, $T_h$ follows an exponential distribution with parameter $\mathcal{N}(\mathcal{F}_{h, +}) \sim c e^{-\kappa h}$, according to Lemma \ref{asympminparam}, where $c$ is the positive constant obtained in the lemma. We thus get
\begin{eqnarray}
\mathbb{E} \left [ S^{h, -}(T_h) \right ] \leq \mathbb{E} [ S(1)] / \mathcal{N}(\mathcal{F}_{h, +}) \leq (2 \mathbb{E} [ S(1)] / c)e^{\kappa h}, \label{esptpsmonte2}
\end{eqnarray}
where the last inequality holds for $h$ large enough. 

For the second term in \eqref{esptpsmonte1} we use the first point of Lemma 2.6 in \cite{foncexpovech} which yields $\mathbb{E} [ \tau(V^{\uparrow}, h) ] \leq \mathbb{E} [ \tau(V^{\sharp}, h) ]$, {where $V^{\sharp}$ is \textit{$V$ conditioned to drift to $+\infty$}, as in Subsection \ref{extrema}.} Then, recall from the proof of Lemma \ref{vposlapltpsatt} that $\tau(V^{\sharp}, .)$ is a subordinator, in particular we have $\mathbb{E} [ \tau(V^{\sharp}, h) ] = h \mathbb{E} [ \tau(V^{\sharp}, 1) ]$. Then, since $V^{\sharp}$ drifts to $+\infty$, Proposition 17(ii) of \cite{Bertoin} ensures that $\mathbb{E} [ \tau(V^{\sharp}, 1) ] < +\infty$. Putting all this together we obtain 
\begin{eqnarray}
\mathbb{E} \left [ \tau(V^{\uparrow}, h) \right ] \leq \mathbb{E} \left [ \tau(V^{\sharp}, h) \right ] = h \mathbb{E} \left [ \tau(V^{\sharp}, 1) \right ] < +\infty. \label{esptpsmonte3}
\end{eqnarray}
Combining \eqref{esptpsmonte1}, \eqref{esptpsmonte2} and \eqref{esptpsmonte3} we get the result. 

\end{proof}

\begin{lemme} \label{espintmonte}

There is a positive constant $C$ such that
\[ \forall h > 0, \ \mathbb{E} \left [ \int_0^{\tau^*(h)} e^{V(u)} du \right ] \leq C e^{(1-\kappa) h}. \]

\end{lemme}

\begin{proof}

Since $V$ is spectrally negative, single points are not essentially polar for $V$. As a consequence, according to Theorem V.1 of \cite{Bertoin}, there is $\mathcal{L}_V$, a local time that satisfies the density of occupations formula for $V$: almost surely, for all mesurable function $f$ and $t > 0$ we have $\int_0^t f(V(s)) ds = \int_{\mathbb{R}} f(x) \mathcal{L}_V(t, x) dx$. Therefore, as in the proof of the majoration of $\beta_0(h)$ in Lemma 3.6 of \cite{AndDev} we get: 
\[  \mathbb{E} \left [ \int_0^{\tau^*(h)} e^{V(u)} du \right ] \leq \int_{-\infty}^h e^x \mathbb{E} \left [ \mathcal{L}_V(+\infty, x) \right ] dx. \]
Then, according to the Markov property, 
\[ \forall x \in \mathbb{R}, \ \mathbb{E} \left [ \mathcal{L}_V(+\infty, x) \right ] {= \mathbb{E} \left [ \mathds{1}_{\tau(V,x) < +\infty} \mathcal{L}_{V^{\tau(V,x)}}(+\infty, 0) \right ]} = \mathbb{P} \left ( \tau(V,x) < +\infty \right ) \times \mathbb{E} \left [ \mathcal{L}_V(+\infty, 0) \right ]. \]
Note that 
\[ \forall x \in \mathbb{R}, \ \mathbb{P} \left ( \tau(V,x) < +\infty \right ) \leq \mathbb{P} \left ( \sup_{[0, +\infty[} V \geq x \right ) = \mathds{1}_{x \leq 0} + e^{-\kappa x}\mathds{1}_{x > 0}, \]
{where we have used that, as mentioned in Subsection \ref{factsandnotations}, $\sup_{[0, +\infty[} V$ follows an exponential distribution with parameter $\kappa$.} We thus get
\begin{align*}
\mathbb{E} \left [ \int_0^{\tau^*(h)} e^{V(u)} du \right ] & \leq {\mathbb{E} \left [ \mathcal{L}_V(+\infty, 0) \right ]} \times \int_{-\infty}^h e^x (\mathds{1}_{x \leq 0} + e^{-\kappa x}\mathds{1}_{x > 0}) dx \\
& = \frac{1}{1-\kappa} \mathbb{E} \left [ \mathcal{L}_V(+\infty, 0) \right ] \times \left ( e^{(1-\kappa)h} - \kappa \right ). 
\end{align*}
Then, considering the Poisson point process of excursions of $V$ away from $0$ associated with the local time $\mathcal{L}_V(., 0)$ (we denote $\eta_V(.)$ the associated excursion measure), we have that $\mathcal{L}_V(+\infty, 0)$ is only the time when occurs the infinite excursion of $V$, and this follows an exponential distribution with parameter $\eta_V(\xi, \ \zeta(\xi) = +\infty) > 0$. As a consequence, $\mathbb{E} [ \mathcal{L}_V(+\infty, 0) ] < +\infty$ and the result follows. 

\end{proof}

\subsection{Estimates on the valleys} \label{new technical}

We now use the previous results to prove some estimates on the standard valleys. 
%All the results of this subsection are stated for the $i^{th}$ standard valley, for any $i \geq 1$, but thanks to Remark \ref{iid}, it is enough to prove them for the first standard valley. 

\begin{lemme} \label{trucavecnouvelledef}

Let $\delta$ be as defined in Subsection \ref{coin} and Section \ref{genedesres}. 
%(but this is actually not important here). SI DANS LA MESURE OU IL FAUDRA APPLIQUER CE LEMME APRES POUR PROUVER NOTAMMENT 5.22. APRES IL EST POSSIBLE QUE LE DELTA ICI (QUI DOIT COINCIDER AVEC TOUTES LES FOIS OU ON APPLIQUE CE LEMME LA) N'A PAS BESOIN D'ETRE LE MEME QUE CELUI DANS Subsection \ref{coin} and Section \ref{genedesres}. 
For all $h$ large enough we have
\[ \forall j \geq 1, \ \mathbb{P} \left ( \int_{\tilde L_{j-1}}^{\tilde m_{j}} e^{\tilde V^{(j)}(u)} du \geq e^{e^{(1-2\delta)\kappa h}} \right ) \geq 1 -e^{-\kappa \delta h /4}. \]

\end{lemme}

\begin{proof}

Thanks to Remark \ref{iid} we only need to prove the result for $j=1$. Recall the sequence of stopping times $(T_i)_{i \geq 0}$ from the proof of Lemma \ref{monteavantdescente} in which we set $a=h$, $b=e^{(1-\delta)\kappa h}$ and $\eta = \delta /2$. For $h \geq 1$ let us define: 
\[ p(h) := \mathbb{P} \left ( T_1 \geq 1 \right ) = \mathbb{P} \left ( \min \left \{ \tau(V-\underline{V}, h), \tau \left (V, \left ]-\infty, -\delta h/2 \right ] \right ) \right \} \geq 1 \right ) \geq p(1) > 0. \]
Thanks to the Markov property, the sequence $(T_{i} - T_{i-1})_{i \geq 1}$ is \textit{iid}. The probability that $T_{i} - T_{i-1} < 1$ for all $i \leq \lfloor e^{(1-2 \delta)h} \rfloor$ is then 
\[ (1-p(h))^{\lfloor e^{(1-2 \delta) \kappa h} \rfloor} \leq (1-p(1))^{\lfloor e^{(1-2 \delta) \kappa h} \rfloor} \leq e^{-h}, \]
where the last inequality holds for $h$ large enough. {As a consequence, when $h$ is large enough, the probability that there exists an index $1 \leq i \leq \lfloor e^{(1-2 \delta)h} \rfloor$ such that $T_{i} - T_{i-1} \geq 1$ is greater than $1-e^{-h}$. We have thus proved that for all $h$ large enough 
\begin{eqnarray}
\mathbb{P} \left ( T_{\lfloor e^{(1-2 \delta)h} \rfloor} > 1 \right ) \geq 1-e^{-h}. \label{trucavecnouvelledef1}
\end{eqnarray}

We now explain why the event $\{ \inf_{[0, T_{\lfloor e^{(1-2 \delta)h} \rfloor}[} \tilde V^{(1)} \geq e^{(1-2 \delta) \kappa h}, T_{\lfloor e^{(1-2 \delta)h} \rfloor} \leq \tilde m_1 \}$ has a large probability. Let $I_0 := \min \{ i \geq 1, \ T_i = \tau(V^{T_{i-1}} - \underline{V}^{T_{i-1}}, h) \}$. It is not difficult to see that $T_{I_0} = \tau(V - \underline{V}, h)$. We know from the proof of Lemma \ref{monteavantdescente} that 
\begin{align}
\mathbb{P} \left ( I_0 > 2 e^{(1-\delta) \kappa h}/ \delta h \right ) & \geq 1- (b/\eta a +1) e^{-\kappa (1-\eta) a} \nonumber \\
& = 1- (1+ 2e^{(1-\delta)\kappa h}/\delta h) e^{-(1-\delta/2)\kappa h} \nonumber \\
& \geq 1 -e^{-\kappa \delta h /3}, \label{trucavecnouvelledef2}
\end{align}
where the last inequality holds for $h$ is large enough. 

%with probability greater than $1- (b/\eta a +1) e^{-\kappa (1-\eta) a} = 1- (1+ 2e^{(1-\delta)\kappa h}/\delta h) e^{-(1-\delta/2)\kappa h}$ (which is more than $1 -e^{-\kappa \delta h /3}$, at least when $h$ is large enough), the first index $i$ such that $T_i = \tau(V - \underline{V}, h)$ is greater than $2 e^{(1-\delta) \kappa h}/ \delta h$. 

Let us assume that $h$ is large enough so that $2 e^{(1-\delta) \kappa h}/ \delta h - \lfloor e^{(1-2 \delta) \kappa h} \rfloor \geq e^{(1-2 \delta) \kappa h}$ and $\delta h/2 \geq 1$. Note that 
\[ T_{I_0 - 1} \leq \inf \{ x \geq 0, \  V(x) = \underline{V} (\tau(V - \underline{V}, h)) \} \leq \tilde m_1, \]
where the last inequality comes from the definition of $\tilde m_1$ {(in Subsection \ref{coin}) and from the fact that, by definition, $\tilde L_0 = 0$}. As a consequence $\{I_0 > 2 e^{(1-\delta) \kappa h}/ \delta h\} \subset \{ T_{\lfloor e^{(1-2 \delta)h} \rfloor} \leq \tilde m_1 \}$. Then, from the definition of $(T_i)_{i \geq 0}$ and $I_0$ we have 
\[ \forall i < I_0, \ \inf_{[0, T_{i}[} \tilde V^{(1)} \geq (I_0 - 1 - i) \delta h/2 \geq I_0 - 1 - i. \]
As a consequence $\{I_0 > 2 e^{(1-\delta) \kappa h}/ \delta h\} \subset \{ \inf_{[0, T_{\lfloor e^{(1-2 \delta)h} \rfloor}[} \tilde V^{(1)} \geq e^{(1-2 \delta) \kappa h} \}$. Combining all this with \eqref{trucavecnouvelledef2} we get 
\begin{align}
\mathbb{P} \left ( \inf_{[0, T_{\lfloor e^{(1-2 \delta)h} \rfloor}[} \tilde V^{(1)} \geq e^{(1-2 \delta) \kappa h}, T_{\lfloor e^{(1-2 \delta)h} \rfloor} \leq \tilde m_1 \right ) \geq \mathbb{P} \left ( I_0 > 2 e^{(1-\delta) \kappa h}/ \delta h \right ) \geq 1 -e^{-\kappa \delta h /3}. \label{trucavecnouvelledef3}
\end{align}

%after the $\lfloor e^{(1-2 \delta) \kappa h} \rfloor^{th}$ one, $V$ will still makes at least $2 e^{(1-\delta) \kappa h}/ \delta h - \lfloor e^{(1-2 \delta) \kappa h} \rfloor \geq e^{(1-2 \delta) \kappa h}$ descents larger than $\delta h/2$ before 
%In $[0, \tilde m_1]$, there will therefore be an interval larger than $1$ on which $\tilde V^{(1)}$ is greater than $e^{(1-2 \delta)\kappa h}$. We thus have

{Using that $\tilde L_0 = 0$ and} putting together \eqref{trucavecnouvelledef1} and \eqref{trucavecnouvelledef3} we obtain that the following is true with probability greater than $1-e^{-h}-e^{-\kappa \delta h /3}$ (for $h$ large enough): 
\[ \int_{\tilde L_{0}}^{\tilde m_1} e^{\tilde V^{(1)}(u)} du = \int_{0}^{\tilde m_1} e^{\tilde V^{(1)}(u)} du \geq \int_{0}^{T_{\lfloor e^{(1-2 \delta)h} \rfloor}} e^{\tilde V^{(1)}(u)} du \geq e^{e^{(1-2 \delta)\kappa h}} \times T_{\lfloor e^{(1-2 \delta)h} \rfloor} \geq e^{e^{(1-2 \delta)\kappa h}}. \]
The result follows.} 
%We get the result using the \textit{iid} character of the valleys $(\tilde V^{(i), h} (x),\ \tilde L_{i-1} \leq x \leq \tilde L_{i} )$. 
%, we get for $h$ large enough
%\[ \mathbb{P} \left ( \forall i \in \{ 1,..., n_h \}, \ \int_{\tilde L_{i}^+}^{\tilde m_{i+1}} e^{V^{(i+1), h}(u)} du \geq e^{e^{(1-2\delta)\kappa h}} \right ) \geq 1 - n_h e^{-\kappa \delta h /3}, \]
%which yields the result. 

\end{proof}

\begin{lemme} \label{minoprevalley}

Fix $0 < \eta < \alpha \leq 1$. For $h$ large enough we have
\[ \forall j \geq 1, \ \mathbb{P} \left ( \tilde{L}_{j}^{\sharp} < \tilde \tau_j^-(\alpha h), \ \inf_{[\tilde{L}_{j}^{\sharp}, \tilde \tau_j^-(\alpha h)]} \tilde V^{(j)} > \left ( \alpha - \eta \right )h\right ) \geq 1 - e^{-\kappa \eta h /3}. \]
%\[ \forall i \geq 1, \ \mathbb{P} \left ( \tilde{L}_{i}^{\sharp} < \tilde \tau_i^-(\alpha h), \ \inf_{[\tilde{L}_{i}^{\sharp}, \tilde \tau_i^-(\alpha h)]} V^{(i)} > \left ( \alpha - \eta \right )h\right ) \geq 1 - 4e^{- \eta \kappa h/2}/\eta. \]
\end{lemme}

It would seem convenient to use the time-reversal property to prove this lemma. However this is not possible here so we have to show that $V$ cannot get too close to its future minimum before time $\tilde \tau_j^-(\alpha h)$. 

\begin{proof}
%CA PEUT ARRIVER SI LA VALLE ARRIVE TOUT DE SUITE APRES $\tilde{L}_{i}^{\sharp}$. (PARCE QUE CE QU'ON VEUT C'EST LE RESULTAT AVEC $\tilde{L}_{i}$ AU LIEU DE $\tilde{L}_{i}^{\sharp}$. CETTE PREUVE NE TIENT EFFECTIVEMENT PAS COMPTE DE CE QUI SE PASSE SI LA MONTEE SE FAIT TOUT DE SUITE APRES $\tilde{L}_{i}^{\sharp}$, ON PEUT FACILEMENT MAJORER LA PROBA QUE CA ARRIVE (UNE AUTRE ALTERNATIVE EST DE DIRE QUE LA DESCENTE FORCEE NE REMONTE PAS AVANT SA DERNIERE DESCENTE), MAIS DANS CE CAS EST-IL UTILE DE GARDER LA DESCENTE IMPOSEE ? ELLE NE SERT QU'A AVOIR LA MONTE A GAUCHE DE $h$ PAR RAPPORT AU MIN
%J'AI CORRIGE MAIS IL RESTE UNE AUTREV ERREUR (VOIR PLUS BAS)
Thanks to Remark \ref{iid} we only need to prove the result for $j=1$. Recall from Subsection \ref{coin} that, by definition, $\tilde \tau_1(h)$ is the first time after $\tilde{L}_{1}^{\sharp}$ when $V - \underline V$ reaches $h$ and $\tilde m_1$ is the associated minimum ($V(\tilde m_1) = \underline V(\tilde \tau_1(h))$). {Recall also that by definition of $\tau_1^-(\alpha h)$, $\sup_{[\tilde \tau_1^-(\alpha h), \tilde m_1]} V - V(\tilde m_1) \leq \alpha h \leq h$. As a consequence 
\begin{align*}
\mathbb{P} \left ( \tilde{L}_{1}^{\sharp} \geq \tilde \tau_1^-(\alpha h) \right ) = \mathbb{P} \left ( \tilde \tau_1^-(\alpha h) \leq \tilde{L}_{1}^{\sharp} \leq \tilde m_1 \right ) & \leq \mathbb{P} \left ( V(\tilde{L}_{1}^{\sharp}) - V(\tilde m_1) \leq h \right ) \\
& \leq \mathbb{P} \left ( \tau(V^{\tilde{L}_{1}^{\sharp}}-\underline{V}^{\tilde{L}_{1}^{\sharp}}, h) < \tau(V^{\tilde{L}_{1}^{\sharp}}, ]-\infty, -h[) \right ). 
\end{align*}
Then, from the Markov property, $V^{\tilde{L}_{1}^{\sharp}}$ is equal in law to $V$. Combining with Lemma \ref{monteavantdescente} applied with $a = b = h$, $\eta = 1/2$ we get 
\begin{eqnarray}
\mathbb{P} \left ( \tilde{L}_{1}^{\sharp} \geq \tilde \tau_1^-(\alpha h) \right ) \leq \mathbb{P} \left ( \tau(V-\underline{V}, h) < \tau(V, ]-\infty, -h[) \right ) \leq 3 e^{-\kappa h /2}. \label{minoprevalley0}
\end{eqnarray}}
When $\tilde{L}_{1}^{\sharp} < \tilde \tau_1^-(\alpha h)$, let $u_1$ be the unique point where $\tilde V^{(1)}$ reaches its minimum on $[\tilde{L}_{1}^{\sharp}, \tilde \tau_1^-(\alpha h)]$. On the event $\{ \tilde{L}_{1}^{\sharp} < \tilde \tau_1^-(\alpha h), \ \inf_{[\tilde{L}_{1}^{\sharp}, \tilde \tau_1^-(\alpha h)]} \tilde V^{(1)} \leq \left ( \alpha - \eta \right )h \}$ we have 
\begin{eqnarray}
V(\tilde m_1) \leq V(u_1) \leq V(\tilde m_1) + \left ( \alpha - \eta \right )h \leq V(\tilde \tau_1^-(\alpha h)-) - \eta h. \label{minoprevalley0.1}
\end{eqnarray}
By the definition of $\tilde{L}_{1}^{\sharp}$ in Subsection \ref{coin} we have $V(\tilde{L}_{1}^{\sharp}) = \inf_{[\tilde L_0, \tilde{L}_{1}^{\sharp}]} V = \inf_{[0, \tilde{L}_{1}^{\sharp}]} V$ {(because $\tilde L_0 = 0$ by definition)}, so $V(u_1) = \underline V(u_1) = \underline V(\tilde \tau_1^-(\alpha h))$. This and \eqref{minoprevalley0.1} imply that $(V - \underline V)(u_1) = 0$ and $(V - \underline V)(\tilde \tau_1^-(\alpha h)) \geq \eta h$. As a consequence, $V - \underline V$ reaches $\eta h$ between $u_1$ and $\tilde \tau_1^-(\alpha h)$, then $V$ descends lower than its past minimum level $V(u_1)$ (because $V(\tilde m_1) < V(u_1)$) and then, $V - \underline V$ reaches $h$ before $V$ reaches $]-\infty, V(u_1) - ( \alpha - \eta )h[$ (because otherwise we would not have $V(\tilde m_1) \geq V(u_1) - ( \alpha - \eta )h$). We thus consider the times when $V - \underline V$ reaches $\eta h$ and separate them by the times when $V$ gets lower than its previous minimum. We introduce the sequences of stopping times $(S_j)_{j \geq 0}$ and $(T_j)_{j \geq 1}$ where $S_0 := \tilde{L}_{1}^{\sharp}$ and 
\[ T_{j} := \inf \left \{ t \geq S_{j-1}, \ (V - \underline V)(t) = \eta h \right \}, \ \ \ S_{j}  := \inf \left \{ t \geq T_j, \ V(t) < \underline V(T_j) \right \}. \]
%S_{j} & := \inf \left \{ t \geq T_j, \ (V - \underline V)(t) = h \right \} \wedge \inf \left \{ t \geq T_j, \ V(t) < \underline V(T_j) \right \}. 
%NE PAS METTRE $t \geq T_j, \ (V - \underline V)(t) = h$ DANS LA DEF DE $S_{j}$ (CA RUINE LE CARACTERE TEMPS D'ARRET, QU'IL FAUDRAIT D'AILLEURS JUSTIFIER). 
%S_{j} & := \inf \left \{ t \geq T_j, \ V(t) \notin [\underline V(T_j,) - h/4, \underline V(T_j,) + 3h/4] \right \}. 
Note that for all $j \geq 1$, $\underline V(T_j) = V(T_j) - \eta h$ and $V(S_j) = \underline V(S_j)$. Combining with the Markov property, this implies that the sequences of truncated processes $(V(t + S_{j-1}) - V(S_{j-1}), \ 0 \leq t \leq T_j - S_{j-1})_{j \geq 1}$ and $(V(t + T_j) - V(T_j), \ 0 \leq t \leq S_j - T_j)_{j \geq 1}$ are both \textit{iid}, and the two sequences are independent. 

We see that $\tilde \tau_1^+(\eta h) = T_{J_0}$, where we define $J_0$ to be the first index $j \geq 1$ for which $\inf  \{ t \geq T_j, \ (V - \underline V)(t) = h \} < S_j$, and that $\tilde m_1$ is the minimum of $V$ before $T_{J_0}$ ($V(\tilde m_1) = \underline V(T_{J_0}) = V(T_{J_0}) - \eta h$). Moreover, we saw before that, on the event $\{ \tilde{L}_{1}^{\sharp} < \tilde \tau_1^-(\alpha h), \ \inf_{[\tilde{L}_{1}^{\sharp}, \tilde \tau_1^-(\alpha h)]} \tilde V^{(1)} \leq ( \alpha - \eta )h \}$, we have $J_0 \geq 2$ and $u_1$ is the minimum of $V$ before $T_{K}$ for some random $K < J_0$ ($V(u_1) = \underline V(T_{K}) = V(T_{K}) - \eta h$). Using $J_0 -1 \geq K$ and \eqref{minoprevalley0.1} we get 
\[ \underline V (T_{J_0 - 1}) \leq \underline V(T_{K}) = V(u_1) \leq V(\tilde m_1) + ( \alpha - \eta )h =\underline V(T_{J_0}) + ( \alpha - \eta )h = V(T_{J_0}) + ( \alpha - 2 \eta )h. \]
%that is, $\underline V(T_{J_0 - 1}) \leq \underline V(T_{J_0}) + ( \alpha - \eta )h = V(T_{J_0}) + ( \alpha - 2 \eta )h$. 
%$\underline V(T_{J_0}) \leq V(S_{J_0 - 1}) - \alpha h$
From the definition of $S_j$ we have $V(S_{J_0 - 1}) \leq \underline V(T_{J_0 - 1})$ so we get
\begin{eqnarray}
\left \{ \tilde{L}_{1}^{\sharp} < \tilde \tau_1^-(\alpha h), \ \inf_{[\tilde{L}_{1}^{\sharp}, \ \tilde \tau_1^-(\alpha h)]} \tilde V^{(1)} \leq \left ( \alpha - \eta \right )h \right \} \subset \left \{ V(S_{J_0 - 1}) \leq V(T_{J_0}) + ( \alpha - 2 \eta )h \right \}. \label{minoprevalley1new}
\end{eqnarray}
For any $k \geq 1$, the event $\{ V(S_{k - 1}) \leq V(T_{k}) + ( \alpha - 2 \eta )h \}$ only depends on $(V(t + S_{k-1}) - V(S_{k-1}), \ 0 \leq t \leq T_k - S_{k-1})$ whereas the event $\{ J_0 = k \}$ only depends on the sequence $(V(t + T_j) - V(T_j), \ 0 \leq t \leq S_j - T_j)_{j \geq 1}$. Partitioning on the possible values for $J_0$ and using the fact that the sequences $(V(t + S_{j-1}) - V(S_{j-1}), \ 0 \leq t \leq T_j - S_{j-1})_{j \geq 1}$ and $(V(t + T_j) - V(T_j), \ 0 \leq t \leq S_j - T_j)_{j \geq 1}$ are independent and both \textit{iid}, we get
%\begin{align*}
%\mathbb{P} \left ( V(S_{J_0 - 1}) \leq V(T_{J_0}) + ( \alpha - 2 \eta )h \right ) & = \mathbb{P} \left ( V(S_{0}) \leq V(T_{1}) + ( \alpha - 2 \eta )h \right ) \\
%& = \mathbb{P} \left ( V(\tau \left ( V - \underline V, \eta h \right )) \geq - ( \alpha - 2 \eta )h \right ) \\
%& = \mathbb{P} \left ( \tau \left ( V - \underline V, \eta h \right ) < \tau \left ( V, ]- \infty, -( \alpha - \eta )h[ \right ) \right ) \\
%& \leq ((2 \alpha / \eta) - 1) e^{-\kappa \eta h /2}. 
%\end{align*} 
\begin{align*}
\mathbb{P} \left ( V(S_{J_0 - 1}) \leq V(T_{J_0}) + ( \alpha - 2 \eta )h \right ) & = \mathbb{P} \left ( V(S_{0}) \leq V(T_{1}) + ( \alpha - 2 \eta )h \right ) \\
& = \mathbb{P} \left ( V^{S_0} \left (\tau ( V^{S_0} - \underline V^{S_0}, \eta h ) \right ) \geq - ( \alpha - 2 \eta )h \right ). 
\end{align*} 
Thanks to the Markov property at $S_{0} = \tilde{L}_{1}^{\sharp}$, the above equals 
\begin{align*}
\mathbb{P} \left ( V \left (\tau ( V - \underline V, \eta h ) \right ) \geq - ( \alpha - 2 \eta )h \right ) & = \mathbb{P} \left ( \tau \left ( V - \underline V, \eta h \right ) < \tau \left ( V, ]- \infty, -( \alpha - \eta )h[ \right ) \right ) \\
& \leq [(2 \alpha / \eta) - 1] e^{-\kappa \eta h /2}. 
\end{align*} 
The last inequality comes from Lemma \ref{monteavantdescente} applied with $a = \eta h$, $b =( \alpha - \eta )h$, $\eta = 1/2$. Combining with \eqref{minoprevalley1new} we obtain 
\[ \mathbb{P} \left ( \tilde{L}_{1}^{\sharp} < \tilde \tau_1^-(\alpha h), \ \inf_{[\tilde{L}_{1}^{\sharp}, \ \tilde \tau_1^-(\alpha h)]} \tilde V^{(1)} \leq \left ( \alpha - \eta \right )h \right ) \leq [(2 \alpha / \eta) - 1] e^{-\kappa \eta h /2}. \]
Then, putting together with \eqref{minoprevalley0} we get the result for $h$ large enough. 

\end{proof}

\begin{lemme} \label{majoprevalley}

Let $\delta$ be as defined in Subsection \ref{coin} and Section \ref{genedesres}. For all $h$ large enough we have
\[ \forall j \geq 1, \ \mathbb{P} \left ( \tilde \tau_j^-(h) - \tilde{L}_{j-1} > e^{(1+\delta)\kappa h} \right ) \leq e^{-\delta \kappa h /2}. \]

\end{lemme}

\begin{proof}

Here again, thanks to Remark \ref{iid}, we only need to prove the result for $j=1$. From the definitions of $\tilde \tau_1^-(h)$ and $\tilde \tau_1(h)$ in Subsection \ref{coin}, we have 
\begin{eqnarray}
\tilde \tau_1^-(h) - \tilde{L}_{0} \leq \tilde \tau_1(h) - \tilde{L}_{0} = \left ( \tilde \tau_1(h) - \tilde{L}_{1}^{\sharp} \right ) + \left ( \tilde{L}_{1}^{\sharp} - \tilde{L}_{0} \right ). \label{majoprevalley0}
\end{eqnarray}
For the first term {recall that, by definition, $\tilde \tau_1(h) = \tau ( V^{\tilde{L}_{1}^{\sharp}}-\underline{V}^{\tilde{L}_{1}^{\sharp}}, h) + \tilde{L}_{1}^{\sharp}$. According to the Markov property at the stopping time $\tilde{L}_{1}^{\sharp}$ and to the definition of $\tau^*(h)$ in the beginning of Subsection \ref{firstmin}}, we see that $\tilde \tau_1(h) - \tilde{L}_{1}^{\sharp}$ is equal in law to $\tau^*(h)$. We thus have $\mathbb{P} ( \tilde \tau_1(h) - \tilde{L}_{1}^{\sharp} > e^{(1+\delta)\kappa h}/2 ) = \mathbb{P} ( \tau^*(h) > e^{(1+\delta)\kappa h}/2 )$. Combining with Markov's inequality and Lemma \ref{esptpsmonte}, we get 
\begin{eqnarray}
\mathbb{P} \left ( \tilde \tau_1(h) - \tilde{L}_{1}^{\sharp} > e^{(1+\delta)\kappa h}/2 \right ) \leq 2 e^{-(1+\delta)\kappa h} \mathbb{E} \left [ \tau^*(h) \right ] \leq 2 e^{-(1+\delta)\kappa h} \times C e^{\kappa h} = 2C e^{-\delta\kappa h}. \label{majoprevalley1}
\end{eqnarray}
The last inequality in the above holds for $h$ large enough and $C$ is the constant in Lemma \ref{esptpsmonte}. For the second term in the right hand side of \eqref{majoprevalley0} recall {from the beginning of Subsection \ref{coin} that, by definition, $\tilde L_0 = 0$ and that} $\tilde{L}_{1}^{\sharp} = \tau(V^{\tilde L_0}, ]-\infty, -e^{(1-\delta)\kappa h}]) + \tilde L_0 = \tau(V, ]-\infty, -e^{(1-\delta)\kappa h}])$. We can thus apply Lemma \ref{tpsatteinthatv} with $y=e^{(1-\delta)\kappa h}$, $r=e^{(1+\delta)\kappa h}/2$ and we get, for $h$ large enough, 
\begin{align}
\mathbb{P} \left ( \tilde{L}_{1}^{\sharp} - \tilde{L}_{0} > \frac{e^{(1+\delta)\kappa h}}{2}  \right ) & = \mathbb{P} \left ( \tau \left ( V, ]-\infty, -e^{(1-\delta)\kappa h}] \right ) > \frac{e^{(1+\delta)\kappa h}}{2} \right ) \nonumber \\
& \leq \exp \left ( c_1 e^{(1-\delta)\kappa h} -c_2 e^{(1+\delta)\kappa h}/2 \right ) \leq e^{-h}. \label{majoprevalley2}
\end{align}
In the above, $c_1, c_2$ are the constants in Lemma \ref{tpsatteinthatv} and the last inequality is true for $h$ large enough. 

The combination of \eqref{majoprevalley0}, \eqref{majoprevalley1} and \eqref{majoprevalley2} yields the result. 

\end{proof}

\subsection{Estimates on the diffusion in potential $V$} \label{neghalfline}

We now give an upper bound for the amount of time spent by the diffusion and the local time in the negative half-line. Restricted to the drifted Brownian case, our result is a little stronger than Lemma 3.5 of \cite{AndDev}. It is the only estimate that we need for both cases $0 < \kappa < 1$ and $\kappa > 1$. 

\begin{lemme} \label{tpsdanslesnegl}

%Recall the definition of $H_{-}$ in Subsection \ref{factsandnotations}. Fix $\gamma \in ]0, \kappa[$. For $r$ large enough, 
%\[ \mathbb{P} \left ( H_-(+\infty) > r \right ) \leq r^{-\gamma/(2+\gamma)} \ \ \ \text{and} \ \ \ \mathbb{P} \left ( \inf_{]-\infty, 0]} \mathcal{L}_X(+\infty, .) > r \right ) \leq 3 r^{-\kappa/(2+\kappa)}. \]
%If moreover $V$ satisfies the Cramer's condition \eqref{cramercond}, then there is a positive constant $C$ such that for $r$ large enough, 
%\[ \mathbb{P} \left ( H_-(+\infty) > r \right ) \leq C r^{-\kappa/(2+\kappa)}. \]
Recall the definition of $H_{-}$ in Subsection \ref{factsandnotations}. There is a positive constant $C$ such that for $r$ large enough, 
\[ \mathbb{P} \left ( H_-(+\infty) > r \right ) \leq C r^{-\kappa/(2+\kappa)} \ \ \ \text{and} \ \ \ \mathbb{P} \left ( {\sup_{x \in ]-\infty, 0]} \mathcal{L}_X(+\infty, x)} > r \right ) \leq C r^{-\kappa/(2+\kappa)}. \]

\end{lemme}

\begin{proof}

From the definition of the local time and formula \eqref{expretl2}, we have
\begin{align*}
H_-(+\infty) & = \int_{-\infty}^{0} \mathcal{L}_X(+\infty, x) dx = \int_{-\infty}^{0} e^{-V(x)} \mathcal{L}_{B}{[\tau (B, A_V(+\infty)),A_V(x)]} dx \\
& = A_V(+\infty) \int_{-\infty}^{0} e^{-V(x)} \mathcal{L}_{B'}{[\tau (B', 1), A_V(x)/A_V(+\infty)]} dx,
\end{align*}
where $B' := B((A_V(+\infty))^2 .) /A_V(+\infty)$. By scale invariance, we see that conditionally to $V$, $B'$ is a Brownian motion so
\begin{align}
H_-(+\infty) & \egloi A_V(+\infty) \int_{-\infty}^{0} e^{-V(x)} \mathcal{L}_{B}{[\tau (B, 1), A_V(x)/A_V(+\infty)]} dx \\
& \leq \left ( \sup_{y \leq 0} \mathcal{L}_{B}(\tau (B, 1), y) \right ) \times A_V(+\infty) \times \int_{-\infty}^{0} e^{-V(x)} dx \nonumber \\
& = \left ( \sup_{y \leq 0} \mathcal{L}_{B}(\tau (B, 1), y) \right ) \times A_V(+\infty) \times \int_{0}^{+\infty} e^{\tilde V(x)} dx, \label{tpsdanslesnegl1}
\end{align}
where $\tilde V(x) := -V(-(x-))$. Recall that, by the time-reversal property, the processes $(\tilde V(x), \ x \geq 0)$ and $(V(x), \ x \geq 0)$ have the same law. As a consequence the right hand side of \eqref{tpsdanslesnegl1} features three factors, the last two of them being equal in law to $\int_0^{+\infty} e^{V(u)} du$. We thus have
\[ \mathbb{P} \left ( H_-(+\infty) > r \right ) \leq \mathbb{P} \left ( \sup_{y \leq 0} \mathcal{L}_{B}(\tau (B, 1), y) > r^{\kappa/(2+\kappa)} \right ) + 2 \mathbb{P} \left ( \int_0^{+\infty} e^{V(u)} du > r^{1/(2+\kappa)} \right ). \]
{Thanks to the inequality $(7.13)$ of \cite{advech} applied with $x = r^{\kappa/(2+\kappa)}$, the first term is less than $4r^{-\kappa/(2+\kappa)}$. By the second assertion of Lemma \ref{foncexpov}, the second term is less than $4 \mathcal{C} r^{-\kappa/(2+\kappa)}$ when $r$ is large enough, and where $\mathcal{C}$ is the constant in the lemma. We thus get that $\mathbb{P} ( H_-(+\infty) > r )$ is less than $4 (1+\mathcal{C}) r^{-\kappa/(2+\kappa)}$ when $r$ is large enough, as asserted.} The proof can be repeated for the assertion about the local time, we only replace the integrals on $]-\infty, 0]$ by supremums on $]-\infty, 0]$ and the integrals of $e^{\tilde V(.)}$ and $e^{V(.)}$ on $[0, +\infty[$ by supremums on $[0, +\infty[$. Since, as mentioned in Subsection \ref{factsandnotations}, $\sup_{[0, +\infty[} V$ (and $\sup_{[0, +\infty[} \tilde V$) follows an exponential distribution with parameter $\kappa$, the result follows similarly. 

%If $V$ satisfies the Cramer's condition \eqref{cramercond}, we can replace $\gamma$ by $\kappa$ and use the third assertion in Lemma \ref{foncexpov}. 

\end{proof}

The next lemma provides a useful estimate to bound the amount of time spent by the diffusion between two standard valleys. It generalizes a part of Lemma 3.6 of \cite{AndDev}. 

\begin{lemme}\label{LemTpsVal}
Recall the definition of $H_{+}$ in Subsection \ref{factsandnotations} and the definition of $\tau^*(h)$ in the beginning of Subsection \ref{firstmin} (in particular, $\tau^*(h)$ coincide with $\tilde \tau_{1}^*(h)$ defined in \eqref{ascend}). There exists a constant $C>0$ such that for $h$ large enough,
\begin{equation} \label{eqLemmaInegH1}
\mathbb{E}[H_+(\tau^*(h))]\leq C e^h. 
%P(H_-(r) \geq z ) \leq C [(\log z)/z]^{\k/(\k+2)}.c
\end{equation}
%where
%$\tau^*_1(h)=\inf \left \{u\geq 0,\ V(u)-\inf_{[0,u]}V\geq h \right \}$
%as in... Moreover,
%\begin{equation}\label{eqLemmaInegH2}
%\mathbb{P} [H_-(\tilde m_1) \geq t/\log h_t ] \leq C_+ [(\log t)^2/t]^{\k/(\k+2)}.
%\end{equation}
%%\begin{equation*}
%%\P[H(\tilde m_1) \geq t \log h_t ] \leq C_+ e^{-\phi(t)}/ \log h_t.
%%\end{equation*}
\end{lemme}

\begin{proof}

The beginning of the proof is similar to the beginning of the proof of Lemma 3.6 in \cite{AndDev}. In fact $(3.37)$ of \cite{AndDev} is still true in our setting, with $V$ instead of the drifted Brownian motion: 
%\begin{eqnarray}
\[ \mathbb{E}[H_+(\tau^*(h))]\leq 2 \mathbb{E} \left [ \tau^*(h) \right ] \times \mathbb{E} \left [ \int_0^{\tau^*(h)} e^{V(u)} du \right ]. \]
Combining this with Lemmas \ref{esptpsmonte} and \ref{espintmonte}, we get the result. 

\end{proof}

A fundamental point to make appear the renewal structure is the fact that the diffusion never goes back to a previous valley. Recall the notations $X_{\tilde L_i} := X(H(\tilde L_i) + .)$ and $H_{X_{\tilde L_i}}(r) := \tau(X_{\tilde L_i}, r)$ introduced in the proof of Proposition \ref{analogue5.1}, and recall from Subsection \ref{factsandnotations} the definition of the probability measure $P^V(.)$. The following lemma proves that, with high probability, the diffusion does not go back to a valley from which it has already escaped: 

\begin{lemme} \label{noreturn}

There is a positive constant $c$ such that for $h$ large enough, 
\begin{align*}
\forall i \geq 1, \ P \left ( P^V \left ( H_{X_{\tilde L_i}}(\tilde \tau_i(h)) > H_{X_{\tilde L_i}}(+\infty) \right ) \geq 1 - e^{-h/4} \right ) & \geq 1 - e^{-ch}, \\
\forall i \geq 1, \ \mathbb{P} \left ( H_{X_{\tilde L_i}}(\tilde \tau_i(h)) > H_{X_{\tilde L_i}}(+\infty) \right ) & \geq 1 - e^{-c h}.
\end{align*}
{Note that we have almost surely $H_{X_{\tilde L_i}}(+\infty) = +\infty$ so the above proves that $H_{X_{\tilde L_i}}(\tilde \tau_i(h))$ is actually infinite with large probability. }
\end{lemme}

\begin{proof}

Let us fix $i \geq 1$. {Recall from Subsection \ref{factsandnotations} the expression of the scale function of the diffusion. Then,} at fixed environment $V$, we have that $P^V ( H_{X_{\tilde L_i}}(\tilde \tau_i(h)) < H_{X_{\tilde L_i}}(+\infty) )$ equals
\begin{align}
%& \frac{\int_{\tilde L_i}^{\tilde m_{i+1}} e^{V(u)} du}{\int_{\tilde \tau_i(h)}^{\tilde m_{i+1}} e^{V(u)} du} = \frac1{1 + \int_{\tilde \tau_i(h)}^{\tilde L_i} e^{V(u)} du / \int_{\tilde L_i}^{\tilde m_{i+1}} e^{V(u)} du}  \leq \max \left ( 1, \frac1{1 + \int_{\tilde \tau_i(h)}^{\tilde L_i} e^{V(u)} du / \int_{\tilde L_i}^{+ \infty} e^{V(u)} du} \right ) \nonumber \\
%= & \max \left ( 1, \frac{\int_{\tilde L_i}^{+ \infty} e^{V(u)} du}{\int_{\tilde \tau_i(h)}^{+ \infty} e^{V(u)} du} \right ) = \max \left ( 1, e^{V(\tilde L_i) - V(\tilde \tau_i(h))} \frac{\int_0^{+ \infty} e^{V(u + \tilde L_i) - V(\tilde L_i)} du}{\int_0^{+ \infty} e^{V(u + \tilde \tau_i(h)) - V(\tilde \tau_i(h))} du} \right ) \nonumber \\
& \frac{\int_{\tilde L_i}^{+ \infty} e^{V(u)} du}{\int_{\tilde \tau_i(h)}^{+ \infty} e^{V(u)} du} = e^{V(\tilde L_i) - V(\tilde \tau_i(h))} \frac{\int_0^{+ \infty} e^{V(u + \tilde L_i) - V(\tilde L_i)} du}{\int_0^{+ \infty} e^{V(u + \tilde \tau_i(h)) - V(\tilde \tau_i(h))} du} := e^{V(\tilde L_i) - V(\tilde \tau_i(h))} I_1/I_2. \label{lemtps1}
\end{align}
Since $\tilde L_i$ and $\tilde \tau_i(h)$ are stopping times for $V$, both $I_1$ and $I_2$ have the same law as $\int_0^{+ \infty} e^{V(u)} du$ (note that $I_1$ and $I_2$ are not independent). Then, 
\[ \mathbb{P} \left ( I_1 / I_2 > e^{h/4} \right ) \leq \mathbb{P} \left ( I_1 > e^{h/8} \right ) + \mathbb{P} \left ( I_2 < e^{- h/8} \right ) \leq 2 \mathcal{C} e^{-\kappa h/8} + e^{- h/8}, \]
where the last inequality and the constant $\mathcal{C}$ come from the two assertions of Lemma \ref{foncexpov}. This last inequality holds for $h$ large enough. From the definition of $\tilde L_i$ {in the beginning of Subsection \ref{coin}} we have $V(\tilde L_i) - V(\tilde \tau_i(h)) \leq -h / 2$ so, thanks to \eqref{lemtps1}, we have that $P^V ( H_{X_{\tilde L_i}}(\tilde \tau_i(h)) < H_{X_{\tilde L_i}}(+\infty) )$ is less than $e^{-h / 2} \times e^{h / 4} = e^{-h / 4}$ {on $\{ I_1 / I_2 \leq e^{h/4} \}$ and the latter event has probability greater than $1 - 2 \mathcal{C} e^{-\kappa h/8} - e^{- h/8}$, for large $h$. 
%As a consequence, there is a positive constant $c_1$ such that for any $i \geq 1$ and $h$ large enough we have 
%\[ \mathbb{P} \left ( P^V \left ( H_{X_{\tilde L_i}}(\tilde \tau_i(h)) < H_{X_{\tilde L_i}}(+\infty) \right ) > e^{-h/4} \right ) \leq e^{-c_1 h}. \]
%If we consider that $h$, in addition, $h$ is large enough so that e thus get 
%\begin{align*}
%& \mathbb{P} \left ( P^V \left ( \cup_{i=1}^{n_t} \left \{ H_{X_{\tilde L_i}}(\tilde \tau_i(h)) < H_{X_{\tilde L_i}}(+\infty) \right \} \right ) > e^{-h/5} \right ) \\
%\leq & \mathbb{P} \left ( \sum_{i=1}^{n_t} P^V \left ( H_{X_{\tilde L_i}}(\tilde \tau_i(h)) < H_{X_{\tilde L_i}}(+\infty) \right ) > e^{-h/5} \right ). 
%\end{align*}
This proves the first assertion. 

On the complementary event $\{ I_1 / I_2 > e^{h/4} \}$, that has probability less than $2 \mathcal{C} e^{-\kappa h/8} + e^{- h/8}$ when $h$ is large, $P^V ( H_{X_{\tilde L_i}}(\tilde \tau_i(h)) < H_{X_{\tilde L_i}}(+\infty) )$ is bounded by $1$.} Integrating $P^V ( H_{X_{\tilde L_i}}(\tilde \tau_i(h)) < H_{X_{\tilde L_i}}(+\infty) )$ with respect to $V$ we thus get 
\[ \mathbb{P} \left ( H_{X_{\tilde L_i}}(\tilde \tau_i(h)) < H_{X_{\tilde L_i}}(+\infty) \right ) \leq e^{-h / 4} + 2 \mathcal{C} e^{-\kappa h/8} + e^{- h/8}, \]
and the second assertion follows. 

\end{proof}

Recall the notations $X_{\tilde m_i} := X(. + H(\tilde m_i))$ and $H_{X_{\tilde m_i}}(r) := \tau(X_{\tilde m_i}, r)$ introduced in Subsection \ref{maincontibcommeiid}. 
%just before Fact \ref{analogue5.3}. 
The next lemma proves that, with high probability, the standard valleys are left from the right. 

\begin{lemme} \label{sortparladroite}

Let $\delta$ be as defined in Subsection \ref{coin} and Section \ref{genedesres}. For $h$ large enough, 
\[ \forall i \geq 1, \ \mathbb{P} \left ( H_{X_{\tilde m_i}}(\tilde L_{i-1}) < H_{X_{\tilde m_i}}(\tilde L_i) \right ) \leq e^{- \kappa \delta h/6}. \]

\end{lemme}

This is where appears a technical difference with the Brownian case. In Lemma 3.2 of \cite{AndDev}, they have a similar result with, instead of $\tilde L_{i-1}$, an other random time that they denote by $\tilde L_{i}^-$. In the context of a L\'evy environment, the existence of jumps may allow one of the $\tilde m_i - \tilde L_{i}^-$ to be quite small with a non negligible probability, which would allow some standard valleys to be left from $\tilde L_{i}^-$. In fact, it would require an analogous of \eqref{0MinorationAVallee} for $\hat V^{\uparrow}$ to have the same result as in Lemma 3.2 of \cite{AndDev}, this is thus hopeless here. In order to still have the standard valleys left from the right, we have given here 
%construct the standard $h$-valleys differently (in particular, we have a different definition of $\tilde L_i^{\sharp}$) and 
a different definition of "leave from the right" (replacing $\tilde L_{i}^-$ by $\tilde L_{i-1}$). A consequence of this is that the study of the descending parts of the standard $h$-valleys is more technical (since we have to consider a bigger part than in \cite{AndDev}, and we do not know its law precisely) and this study requires Subsection \ref{new technical}. In particular, our proof that the standard valleys are left from the right requires the technical Lemma \ref{trucavecnouvelledef}, but the idea of the result has no qualitative difference with the drifted Brownian case. 
%Moreover, even though our construction of the standard $h$-valleys is different, Lemma \ref{minimacoincide} says that they still coincide with the classical valleys with an overwhelming probability. 

\begin{proof} of Lemma \ref{sortparladroite}

Let us fix $i \geq 1$. {We use again the probability measure $P^V(.)$ and the expression of the scale function of the diffusion, both defined in Subsection \ref{factsandnotations}.} At fixed environment $V$, $P^V ( H_{X_{\tilde m_i}}(\tilde L_{i-1}) < H_{X_{\tilde m_i}}(\tilde L_i) )$ equals
%\begin{align}
%& \frac{\int_{\tilde m_i}^{\tilde L_i} e^{\tilde V^{(i)}(u)} du}{\int_{\tilde L_{i-1}}^{\tilde L_i} e^{\tilde V^{(i)}(u)} du} = \frac1{1 + \int_{\tilde L_{i-1}}^{\tilde m_i} e^{\tilde V^{(i)}(u)} du / \int_{\tilde m_i}^{\tilde L_i} e^{\tilde V^{(i)}(u)} du} \nonumber \\
%\leq & \max \left ( 1, \int_{\tilde m_i}^{\tilde L_i} e^{\tilde V^{(i)}(u)} du / \int_{\tilde L_{i-1}}^{\tilde m_i} e^{\tilde V^{(i)}(u)} du \right ). \label{sortparladroite1}
%\end{align}
\begin{align}
\frac{\int_{\tilde m_i}^{\tilde L_i} e^{\tilde V^{(i)}(u)} du}{\int_{\tilde L_{i-1}}^{\tilde L_i} e^{\tilde V^{(i)}(u)} du} = \frac1{1 + \int_{\tilde L_{i-1}}^{\tilde m_i} e^{\tilde V^{(i)}(u)} du / \int_{\tilde m_i}^{\tilde L_i} e^{\tilde V^{(i)}(u)} du} \leq \max \left \{ 1, \int_{\tilde m_i}^{\tilde L_i} e^{\tilde V^{(i)}(u)} du / \int_{\tilde L_{i-1}}^{\tilde m_i} e^{\tilde V^{(i)}(u)} du \right \}. \label{sortparladroite1}
\end{align}

We first provide an upper bound for
\[ \int_{\tilde m_i}^{\tilde L_i} e^{\tilde V^{(i)}(u)} du = \int_{\tilde m_i}^{\tilde \tau_i(h)} e^{\tilde V^{(i)}(u)} du + \int_{\tilde \tau_i(h)}^{\tilde L_i} e^{\tilde V^{(i)}(u)} du. \]
According to proposition \ref{standardwilliams}, the terms of the right hand side have respectively the same law as $\int_{0}^{\tau(V^{\uparrow}, h)} e^{V^{\uparrow}(u)} du$ and $e^{h} \int_{0}^{\tau(V, ] - \infty, -h/2])} e^{V(u)} du \leq e^{h} \int_{0}^{+\infty} e^{V(u)} du$. We thus have
\begin{align}
\mathbb{P} \left ( \int_{\tilde m_i}^{\tilde L_i} e^{\tilde V^{(i)}(u)} du > e^{(1+\delta)h} \right ) & \leq \mathbb{P} \left ( \int_{0}^{\tau(V^{\uparrow}, h)} e^{V^{\uparrow}(u)} du > e^{(1+\delta)h}/2 \right ) \nonumber \\
& + \mathbb{P} \left ( \int_{0}^{+\infty} e^{V(u)} du > e^{\delta h}/2 \right ). \label{retourarriere1}
\end{align}
For $h$ large enough, the first term of the right hand side is bounded by $e^{-h}$ because of \eqref{0MajorationAVallee} and the second is bounded by $2 \mathcal{C} \times 2^{\kappa} e^{-\kappa \delta h}$ because of the second assertion of Lemma \ref{foncexpov}, where $\mathcal{C}$ is the constant in the Lemma. As a consequence, for $h$ large enough, 
%\[ \int_{0}^{\tau(V_h, ] - \infty, h/2])} e^{V_h(u)} du \leq \tau(V_h, ] - \infty, h/2]) e^{h + \sup_{[0, +\infty[} V}, \]
%so
%\[ \mathbb{P} \left ( \int_{0}^{\tau(V_h, ] - \infty, h/2])} e^{V_h(u)} du > e^{(1+\delta)h}/2 \right ) \leq \mathbb{P} \left ( \sup_{[0, +\infty[} V > \delta h/2 \right ) + \mathbb{P} \left ( \tau(V_h, ] - \infty, h/2]) > e^{\delta h/2}/2 \right ). \]
%Now, $\mathbb{P} ( \sup_{[0, +\infty[} V > \delta h/2 ) = e^{-\kappa \delta h/2}$ and, since $e^{\delta h/2}/2$ becomes greater than any linear function of $h$, lemma \ref{redescente} applies to control $\mathbb{P} ( \tau(V_h, ] - \infty, h/2]) > e^{\delta h/2} )$. We obtain that for $h$ large enough: 
%\begin{eqnarray}
%\mathbb{P} \left ( \int_{0}^{\tau(V_h, ] - \infty, h/2])} e^{V_h(u)} du > e^{(1+\delta)h}/2 \right ) \leq 2e^{-h}. \label{retourarriere3}
%\end{eqnarray}
%
%Putting $(\ref{retourarriere1})$, $(\ref{0MajorationAVallee})$ and $(\ref{retourarriere3})$ together we get 
\begin{eqnarray}
\mathbb{P} \left ( \int_{\tilde m_i}^{\tilde L_i} e^{\tilde V^{(i)}(u)} du \leq e^{(1+\delta)h} \right ) > 1 - e^{-\kappa \delta h/3}. \label{retourarriere4}
\end{eqnarray}

{According to Lemma \ref{trucavecnouvelledef} we have that, for large $h$, $\int_{\tilde L_{i-1}}^{\tilde m_i} e^{\tilde V^{(i)}(u)} du \geq \exp [ \exp [(1-2\delta)\kappa h] ]$ with a probability greater than $1 -e^{-\kappa \delta h /4}$. Combining with \eqref{retourarriere4}, we get that for $h$ large enough, $\int_{\tilde m_i}^{\tilde L_i} e^{\tilde V^{(i)}(u)} du / \int_{\tilde L_{i-1}}^{\tilde m_i} e^{\tilde V^{(i)}(u)} du$ is smaller than $e^{-h}$ with a probability greater than $1 - e^{-\kappa \delta h/3} - e^{- \kappa \delta h/4}$. In conclusion, for $h$ large enough, }
\[ \mathbb{P} \left ( \int_{\tilde m_i}^{\tilde L_i} e^{\tilde V^{(i)}(u)} du / \int_{\tilde L_{i-1}}^{\tilde m_i} e^{\tilde V^{(i)}(u)} du \geq e^{-h} \right ) \leq e^{- \kappa \delta h/5}. \]
{Combining with \eqref{sortparladroite1}, we get that $P^V ( H_{X_{\tilde m_i}}(\tilde L_{i-1}) < H_{X_{\tilde m_i}}(\tilde L_i) )$ can be bounded by $e^{-h}$, except on an event of probability smaller than $e^{- \kappa \delta h/5}$ where it is bounded by $1$. Then,} integrating $P^V ( H_{X_{\tilde m_i}}(\tilde L_{i-1}) < H_{X_{\tilde m_i}}(\tilde L_i) )$ with respect to $V$ we get the result for $h$ large enough. 
%that for $h$ large enough, 
%\[ \mathbb{P} \left ( H_{X_{\tilde m_i}}(\tilde L_{i-1}) < H_{X_{\tilde m_i}}(\tilde L_i) \right ) \leq e^{- \kappa \delta h/6}, \]
%and taking the union of these events for all $i \in \left \{ 1, ..., n_t \right \}$ we get the result since $n_t e^{-c' h} \leq e^{- c' h/2}$ for large $h$. 

\end{proof}

\subsection{Justification of some facts} \label{justoffacts}

We now justify Facts \ref{lemtps}, \ref{analog3.3}, \ref{analogue(5.22)} and \ref{analogue5.3}. They are taken from \cite{AndDev} and \cite{advech}, and their proofs use estimates that are true only for the drifted Brownian potential. This is why we have to adapt the proofs to our context and to use our estimates instead of the ones from \cite{AndDev} and \cite{advech}. We shall refer to \cite{AndDev} and \cite{advech} for the proofs and only precise what are the differences between their proofs and ours.

\begin{proof} of Fact \ref{lemtps}

This is Lemma 3.7 of \cite{AndDev}. Here are the modifications that we make on the original proof: 

{$W_{\kappa}$ is, off course, replaced here by $V$. The standard valleys and the random points $\tilde m_i$, $\tilde L_i$, $\tilde{L}_{i}^{\sharp}$, $\tilde\tau_{i}(h_t)$, $\tilde \tau_{i}^*(h_t)$ have to be considered as the ones defined in this paper (in Subsection \ref{coin}). 
%, in particular the definitions of $\tilde L_i$ is the one given in Subsection \ref{coin}. 
Also, $\tilde L_i^*$ of \cite{AndDev} has to be replaced here by $\tilde\tau_{i}(h_t)$ so Lemma  3.3 of \cite{AndDev} can be replaced here by the first point of Lemma \ref{noreturn} (applied with $h=h_t$ for the first $n_t$ indices). 
%and $\tilde \tau_{i}^*(h_t)$ has to be considered has defined in Subsection \ref{coin}. 

For the "Step 1", the event $\mathcal{E}_2^{3.7}$ is replaced here by $\cap_{i=1}^{n_t} \{ \tilde \tau_{i}^*(h_t) = \tilde\tau_{i}(h_t) \}$ (note that this includes $\{ \tilde \tau_1^*(h_t) = \tilde\tau_1 (h_t) \}$) so we have $1-\mathbb{P}(\mathcal{E}_2^{3.7}) \leq n_t e^{-c_1 h_t}$ for some positive constant $c_1$ and large $t$, according to Lemma \ref{tpscoinc} (applied with $h=h_t$ for the first $n_t$ indices). 
%The constant $c_1$ in the exponential might be different from the one in \cite{AndDev} but it does not matter. 

For the "Step 2", the event $\mathcal{E}_1^{3.3}$ is replaced here by $\cap_{i=1}^{n_t} \{ H_{X_{\tilde L_i}}(\tilde\tau_{i}(h_t)) > H_{X_{\tilde L_i}}(\tilde m_{i+1}) \}$ so we have $1-\mathbb{P}(\mathcal{E}_1^{3.3}) \leq n_t e^{-c_2 h_t}$ for some positive constant $c_2$ and large $t$, according to the second point of Lemma \ref{noreturn} (applied with $h=h_t$ for the first $n_t$ indices). The notations $H_{+}$ and $H_{-}$, defined in Subsection \ref{factsandnotations}, have the same meaning here and in \cite{AndDev}. Then, estimate (3.34) of \cite{AndDev} (which is, in the proof of Lemma 3.7 there, referred to as "Lemma 3.6") has to be replaced here by Lemma \ref{LemTpsVal} (applied with $h=h_t$). For $H_+(\tilde m_1)$, it is obviously less than $H_+(\tilde \tau_1(h_t))$ which equals $H_+(\tilde \tau_1^*(h_t))$ on $\{ \tilde \tau_{1}^*(h_t) = \tilde\tau_{1}(h_t) \}$ so, here, the expectation $\mathbb{E}[H_+(\tilde m_1) \mathds{1}_{\mathcal{E}_1^{3.3}}]$ is bounded by $\mathbb{E}[H_+(\tau^*(h_t))]$, just as the other terms. In particular, there is no need, here, to bother with the event $\mathcal{E}_3^{3.7}$ nor to specify we are on the event $\mathcal{V}_t$ where the classical and standard valleys coincide. In place of (3.43) of \cite{AndDev} we thus get 
\begin{align}
& \mathbb{P} \left( H_+(\tilde m_1) + \sum_{i=1}^{n_t-1}\left ( H(\tilde m_{i+1}) - H(\tilde L_i) \right ) \geq \frac{t}{\log h_t}, \mathcal{E}_1^{3.3}, \mathcal{E}_2^{3.7} \right) \nonumber \\
\leq & \frac{\log h_t}{t} n_t \mathbb{E}[H_+(\tau^*(h_t))] \leq \frac{\log h_t}{t} n_t c_3 e^{h_t}, \label{analogue3.43}
\end{align}
for $t$ large enough and where $c_3$ is a positive constant. 

Then, to bound $H_-(\tilde m_1)$, we see from $t/\log(h_t) = e^{h_t} e^{\phi(t)}/\log(h_t)> e^{h_t}$ (which is true at least for large $t$) and from Lemma \ref{tpsdanslesnegl} applied with $r=e^{h_t}$ that, for $t$ large enough, 
\[ \mathbb{P} \left ( H_-(\tilde m_1) > t/\log(h_t) \right ) \leq \mathbb{P} \left ( H_-(\tilde m_1) > e^{h_t} \right ) \leq \mathbb{P} \left ( H_-(+\infty) > e^{h_t} \right ) \leq c_4 e^{-\kappa h_t /(2+\kappa)}, \]
where $c_4$ is a positive constant. 
%\begin{eqnarray}
%\mathbb{P} \left ( H_-(\tilde m_1) > e^{h_t} \right ) \leq \mathbb{P} \left ( H_-(+\infty) > e^{h_t} \right ) \leq C e^{-\kappa h_t /(2+\kappa)}. \label{tpsdanslesnegl10}
%\end{eqnarray}
%Then, estimate (3.34) (which is, in the proof of Lemma 3.7 in \cite{AndDev}, referred to as "Lemma 3.6"), estimate (3.35) and Lemma 3.3 of \cite{AndDev} have to be replaced here by respectively Lemma \ref{LemTpsVal}, \eqref{tpsdanslesnegl10}, and Lemma \ref{noreturn}. 

Finally, in order to get the analogous of (3.44) from \cite{AndDev} we use our previous bound for $\mathbb{P} ( H_-(\tilde m_1) > t/\log(h_t) )$, \eqref{analogue3.43}, $1-\mathbb{P}(\mathcal{E}_1^{3.3}) \leq n_t e^{-c_2 h_t}$, and $1-\mathbb{P}(\mathcal{E}_2^{3.7}) \leq n_t e^{-c_1 h_t}$ (we do not need $\mathcal{E}_3^{3.7}$ and $\mathcal{V}_t$ here). }

\end{proof}

\begin{proof} of Fact \ref{analog3.3}

The first point is Lemma 3.3 of \cite{advech}. Here are the modifications that we make on the original proof: 

$W_{\kappa}$ is replaced here by $V$, $A(x)$ and $A_{\infty}$ from \cite{advech} are replaced here by respectively $A_V(x) = \int_0^x e^{V(u)} du$ and $A_V(+\infty) = \int_0^{+\infty} e^{V(u)} du$, and we use our notations $m^*(h_t)$ and $\tau^*(h_t)$, defined in Subsection \ref{firstmin}, instead of the notations $m_1^*(h_t)$ and $\tau_1^*(h_t)$ from \cite{advech}. We change a little the definition of $b(t)$, that is, here $b(t) := 8R \phi(t) e^{\kappa h_t} / (1-e^{-\kappa})$ where $R := - \mathbb{E} [V (\tau (V, ]-\infty, -1]) )]$. Note that $R$ is finite according to Lemma \ref{leaveanintbelow} and the hypothesis assumed on $V$ (we have assumed that $V (1) \in L^p$ for some $p>1$). {The functions $k(t)$, $a(t)$ and the event $\mathcal{A}_0$ are defined similarly as in the original proof of \cite{advech} with, of course, $V$ instead of $W_{\kappa}$ in the definition of $\mathcal{A}_0$: 
\[ k(t) := e^{2 \kappa^{-1} \phi(t)}, \ \ \ a(t) := 4 \phi(t), \ \ \ \mathcal{A}_0 := \left \{ \int_0^{+\infty} e^{V(u)} du \leq k(t) \right \}. \] 
We have the upper bound for $\mathbb{P} (\mathcal{A}_0^c)$, when $t$ is large enough, thanks to the second assertion of Lemma \ref{foncexpov}: $\mathbb{P} (\mathcal{A}_0^c) \leq 2 \mathcal{C} e^{-2 \phi(t)}$, where $\mathcal{C}$ is the constant in the lemma. $\mathcal{A}_1$ has the same meaning as in the original proof of \cite{advech}, it is an inequality satisfied (with high probability) by the Brownian local time identified as a squared Bessel process $Q_2^2$: 
\[ \mathcal{A}_1 := \left \{ \forall u \in ]0, k(t)], \ Q_2^2(u) \leq 2eu \left [ a(t) + 4 \log \log[ek(t)/u] \right ] \right \}. \] 
In particular we still have $\mathbb{P} (\mathcal{A}_1^c) \leq C e^{-2 \phi(t)}$ as in there, for some positive constant $C$. $\mathcal{A}_2$ is also defined similarly as in the original proof of \cite{advech} with, of course, $V$ instead of $W_{\kappa}$: 
\[ \mathcal{A}_2 := \left \{ \inf_{[0, \tau^*(h_t)]} V \geq - b(t) \right \}. \]} 
Because of the negative jumps we have, before bounding $\mathbb{P} (\mathcal{A}_2^c)$, to give a slightly different definition for the stopping times $f_i$ used in the original proof of \cite{advech}, and to define a new event $\mathcal{A'}_2$. First, $f_0 := 0$ and
\[ \forall i \geq 1, \ f_i := \inf \left \{ x \geq f_{i-1}, \ V(x) \leq V(f_{i-1}) - 1 \right \}. \]
Let $I(t) := \max \{ i \in \mathbb{N}, \ f_{i-1} \leq m^*(h_t) \}$, we now define $\mathcal{A'}_2$ by
\[ \mathcal{A'}_2 := \left \{ I(t) \leq b(t)/2R \right \}. \]
We define, for $i \geq 1$, $E_i := \{ \sup_{[f_{i-1}, f_i]} V - \underline{V} \geq h_t \}$. Since $\inf_{[0, f_{i-1}]} V = V(f_{i-1})$  we have
\[ E_i = \left \{ \sup_{[0, f_i-f_{i-1}]} V^{f_{i-1}} - \underline{V^{f_{i-1}}} \geq h_t \right \}. \]
Because of the Markov property applied at $f_{i-1}$, the events $E_i$ are independent and have the same probability that we denote by $p_t$. Also, $I(t)$ is the smallest $i \geq 1$ for which $E_i$ is realized, so $I(t)$ follows a geometric distribution with parameter $p_t$. Then, 
\[ p_t := \mathbb{P} \left ( E_1 \right ) \geq \mathbb{P} \left ( V \text{ leaves } [-1, h_t] \text{ from above} \right ) \geq (1-e^{-\kappa}) e^{-\kappa h_t}, \]
where, for the last inequality, we have applied Lemma \ref{leaveanint} with $a=1$ and $b=h_t$. We deduce 
\begin{eqnarray}
\mathbb{P} (\mathcal{A'}_2^c) = (1-p_t)^{\lfloor b(t)/2R \rfloor} = e^{\lfloor b(t)/2R \rfloor \log(1-p_t)} \leq e^{\lfloor b(t)/2R \rfloor \log(1-(1-e^{-\kappa}) e^{-\kappa h_t})} \leq e^{-3 \phi(t)}, \label{mvbe0}
\end{eqnarray}
where we have used that $b(t) = 8R \phi(t) e^{\kappa h_t} / (1-e^{-\kappa})$ in the last inequality that holds for $t$ large enough. Then, 
\[ \left \{ V \left ( f_{\lfloor b(t)/2R \rfloor} \right ) < - b(t) \right \} \subset \left \{ \sum_{i=1}^{\lfloor b(t)/2R \rfloor} \left ( V(f_i) - V(f_{i-1}) + R \right ) < - b(t)/2 \right \}. \]
The random variables $V(f_i) - V(f_{i-1}) + R$ are \textit{iid}, having the same law as $V (\tau (V, ]-\infty, -1]) ) + R$. In particular they have null expectation and they belong to $L^p$ because of Lemma \ref{leaveanintbelow}, where $p$ is such that $V (1) \in L^p$ (by our assumptions on $V$ such a $p > 1$ exists {and by decreasing it if necessary, we assume that $p \in ]1,2[$}). We can thus apply successively Markov's inequality and Von Barh-Esseen's inequality (with $p$): 
\begin{align}
\mathbb{P} \left ( V \left ( f_{\lfloor b(t)/2R \rfloor} \right ) < - b(t) \right ) & \leq \left ( \frac{2}{b(t)} \right )^p \mathbb{E} \left [ \left | \sum_{i=1}^{\lfloor b(t)/2R \rfloor} \left ( V(f_i) - V(f_{i-1}) + R \right ) \right |^p \right ] \nonumber \\
& \leq 2 \left ( \frac{2}{b(t)} \right )^p \times \left ( \frac{b(t)}{2R} \right ) \times \mathbb{E} \left [ \left | V (\tau (V, ]-\infty, -1]) ) + R \right |^p \right ] \nonumber \\
& \leq e^{-(p-1) \kappa h_t}. \label{mvbe}
\end{align}
The last inequality comes from the definition of $b(t)$ and holds for $t$ large enough. Now, recall that for all $i \geq 0$ we have $\inf_{[0, f_{i}]} V = V(f_{i})$ and that, by the definitions, we have $f_{I(t)-1} < m^*(h_t) < \tau^*(h_t) < f_{I(t)}$. We thus have 
\begin{align*}
\mathcal{A'}_2 \cap \left \{ V \left ( f_{\lfloor b(t)/2R \rfloor} \right ) \geq - b(t) \right \} & = \left \{ f_{I(t)} \leq f_{\lfloor b(t)/2R \rfloor} \right \} \cap \left \{ \inf_{[0, f_{\lfloor b(t)/2R \rfloor}]} V \geq - b(t) \right \} \\
& \subset \left \{ \inf_{[0, \tau^*(h_t)]} V \geq - b(t) \right \} = \mathcal{A}_2. 
\end{align*}
Taking the complementary and combining with \eqref{mvbe0} and \eqref{mvbe} we get 
\[ \mathbb{P} (\mathcal{A}_2^c) \leq \mathbb{P} (\mathcal{A'}_2^c) + \mathbb{P} \left ( V \left ( f_{\lfloor b(t)/2R \rfloor} \right ) < - b(t) \right ) \leq e^{-3 \phi(t)} + e^{-(p-1) \kappa h_t} \leq e^{-2 \phi(t)}, \]
where the last inequality holds for $t$ large enough. We thus have $\mathbb{P} (\mathcal{A}_2^c) \leq e^{-2 \phi(t)}$ for large $t$. 

We then have to use our own definition of $(f_i)_{i \geq 0}$ given above, rather than the one given in the original proof of \cite{advech}. For any $x \in [0, \tau^*(h_t)]$ we still have $f_i \leq x < f_{i+1}$ for some $i$, and we still have that this $i$ is less than $b(t)$ because of $\mathcal{A}_2$. According to our definition of the sequence $(f_i)_{i \geq 0}$, we still have $e^{-V(x)} \leq e^{-V(f_i) +1}$ (for the index $i$ such that $f_i \leq x < f_{i+1}$) so we have the analogous of $(3.8)$ from \cite{advech} in our setting: 
{ \[ e^{-V(x)} [A_V(\tau^*(h_t)) - A_V(x)] \leq e^{-V(x)} \int_{x}^{\tau^*(h_t)} e^{V(u)} du \leq e \int_{f_i}^{\tau^*(h_t)} e^{V(u) - V(f_i)} du. \]}

In our setting $\mathcal{A}_3$ is defined similarly as in the original proof of \cite{advech} with, of course, $V$ instead of $W_{\kappa}$: 
{ \[ \mathcal{A}_3 := \overset{\lfloor b(t) \rfloor}{\underset{i=0}{\cap}} \left \{ \int_{f_i}^{\tau^*(h_t)} e^{V(u) - V(f_i)} du \leq e^{(1-\kappa) h_t} b(t) n_t e^{\kappa \delta \phi(t)} \right \}, \]}
where $\delta$ is as defined in Subsection \ref{coin} and Section \ref{genedesres}. Our upper bound for $\mathbb{E} [ \int_0^{\tau^*(h)} e^{V(u)} du ]$ (denoted by $\beta_0(h)$ in the proof of Lemma 3.3 of \cite{advech}) is given by Lemma \ref{espintmonte} and we thus still have an upper bound of the type $\mathbb{P} (\mathcal{A}_3^c) \leq C/n_t e^{\kappa \delta \phi(t)}$ in our setting.  
%
%The biggest difference with the proof in \cite{advech} is that we can not ensure that $A(\tau_1^*(h)) - A(x) \geq e^{-b(t)}$. We thus define $Z(x) := e^{-V(x)}(A(\tau_1^*(h)) - A(x))$ and the event $\mathcal{A}_5 := \{ \sup V \leq 2 \kappa^{-1} \phi(t) \}$. Since $\sup V$ follows an exponential law with parameter $\kappa$ it is easy to see that $\mathbb{P} (\mathcal{A}_5^c) \leq e^{-2 \phi(t)}$ and 

In our setting $\mathcal{A}_4$ is defined similarly as in the original proof of \cite{advech}: 
{\[ \mathcal{A}_4 := \left \{ \tau^*(h_t) - m^*(h_t) \geq 1 \right \}, \]}
and to bound $\mathbb{P} (\mathcal{A}_4^c) = \mathbb{P} (\tau^*(h_t) - m^*(h_t) < 1)$, we use Lemma \ref{asympminsubexp} (applied with $h=h_t$). We thus have 
%note that $(V(m^*(h) + x), \ 0 \leq x \leq \tau^*(h) - m^*(h))$ represents the first excursion greater than $h$ of $V - \underline V$ so according to BERTOIN it is equal in law to $(V^{\uparrow}(x), \ 0 \leq x \leq \tau(V^{\uparrow}, h))$ (FAIRE UN LEMME AVAEC CA PARCE QUE CA SERT AUSSI POUR LES $\beta(h)$). Therefore
\[ \mathbb{P} (\mathcal{A}_4^c) = \mathbb{P} \left ( \tau(V^{\uparrow}, h_t) < 1 \right ) \leq \mathbb{P} \left ( \tau(V^{\uparrow}, h_t) - \tau(V^{\uparrow}, h_t/2) < 1 \right ) \leq e^{-c h_t/2}. \]
The last inequality holds for some positive constant $c$ and $t$ large enough, according to \eqref{03.10b} applied with $h=h_t, \omega = 1$ and $\alpha = 1/2$. 

%\end{proof}

%\begin{proof}
%
%We choose $R > 0$ such that $V - V^{< -R}$, defined as in the proof of Lemma \ref{foncexpoV}, converges to $-\infty$. We then apply Lemmas \ref{analog3.3sautscoupes} to the diffusion in $V - V^{< -R}$, denoted by $X_R$: 
%\[ \int P^v \left( \sup_{x \in [0, m^*(h_t)]} \mathcal{L}_X (H(\tau^*(h_t)),x) > t e^{(\kappa (1+3\delta )-1)\phi(t)} \right) \mathcal{L}_{V}(dv) \leq \frac{C}{n_t e^{\kappa \delta \phi(t)}}. \]
%
%\end{proof}

%We can now prove

%\begin{lemme} \label{analog3.2}
%
%There is a positive constant $C$ such that for $t$ large enough, 
%\[ \mathbb{P} \left( \bigcap_{j=0}^{n_t - 1}\left\{ \sup_{\mathbb{R}} \left ( \mathcal{L}_X (H(\tilde m_{j+1}),x) - \mathcal{L}_X (H(\tilde L_j),x) \right ) \leq t e^{(\kappa (1+3\delta )-1)\phi(t)} \right \} \right) \geq 1 - \frac{C}{e^{\kappa \delta \phi(t)}}. \]
%
%\end{lemme}
%
%\begin{proof} 

We now justify the second point of Fact \ref{analog3.3}. {Let us define $r_t := t e^{(\kappa (1+3\delta )-1)\phi(t)} = e^{h_t + \kappa (1+3\delta ) \phi(t)}$, where  $\delta$ is as defined in Subsection \ref{coin} and Section \ref{genedesres}. We first treat separately the case $j = 0$. Since $H(\tilde L_0) = H(0) = 0$ {(because by definition $\tilde L_0 = 0$)}, the term $\mathcal{L}_X (H(\tilde L_0),x)$ can be omitted. Let us recall that $m^*(.)$ and $\tau^*(.)$ are defined in the beginning of Subsection \ref{firstmin} and that almost surely we have $m^*(h_t) = \tilde m_1^*$ (where $\tilde m_i^*$ is defined in \eqref{min}). $\mathbb{P} ( \sup_{x \in \mathbb{R}} \mathcal{L}_X (H(\tilde m_{1}),x) > r_t )$ is less than
\[ {\mathbb{P} \left ( \sup_{x \in ]-\infty, 0]} \mathcal{L}_X(+\infty, x) > r_t \right )} + \mathbb{P} \left( \sup_{x \in [0, m^*(h_t)]} \mathcal{L}_X [H(\tau^*(h_t)), x] > r_t \right) + \mathbb{P} \big ( \tilde m_1 \neq m^*(h_t) \big ). \]
By Lemma \ref{tpsdanslesnegl} applied with $r=r_t$ the first term is, for large $t$, less than $C e^{-\kappa(h_t + \kappa (1+3\delta ) \phi(t))/(2+\kappa)}$ which is less than $e^{-\kappa h_t /(2+\kappa)}$ when $t$ is large. By the first point of Fact \ref{analog3.3}, the second term is less than $C_1/n_t e^{\kappa \delta \phi(t)}$ when $t$ is large. By Lemma \ref{tpscoinc} (applied with $h=h_t$) the third term is less than $e^{-ch_t}$ for $t$ large enough, and where $c$ is some positive constant. Since $e^{-\kappa h_t /(2+\kappa)}$ and $e^{-ch_t}$ are ultimately smaller than $1/n_t e^{\kappa \delta \phi(t)}$, we can put all this together and we get that for some constant $C_2'$ and $t$ large enough,} \begin{eqnarray}
\mathbb{P} \left( \sup_{x \in \mathbb{R}} \mathcal{L}_X (H(\tilde m_{1}),x) > t e^{(\kappa (1+3\delta )-1)\phi(t)} \right) \leq \frac{C_2'}{n_t e^{\kappa \delta \phi(t)}}. \label{analog3.2eq1}
\end{eqnarray}

For $j \geq 1$, the proof has the same idea as the one of Lemma 3.2 in \cite{advech}. Recall the notations $X_{\tilde L_j}$ and $H_{X_{\tilde L_j}}$ introduced in the proof of Proposition \ref{analogue5.1} (and recalled before Lemma \ref{noreturn}). Thanks to the second point of Lemma \ref{noreturn} and to Lemma \ref{tpscoinc} (both applied with $h=h_t$ for $n_t - 1$ indices) we have, for some constant $c$ and $t$ large enough, 
\begin{align}
\mathbb{P} \left( \bigcap_{j=1}^{n_t - 1}\left\{ H_{X_{\tilde L_j}}(\tilde m_{j+1}) < H_{X_{\tilde L_j}}(\tilde \tau_j(h_t)), \ \tilde \tau_{j+1}^*(h_t) = \tilde\tau_{j+1}(h_t), \ \tilde m_{j+1}^* = \tilde m_{j+1} \right \} \right) \geq 1 - n_t e^{-c h_t}. \label{analog3.2eq2}
\end{align}
On this event we have that, for $i \in \{ 1, ..., n_t -1 \}$, the chain of inequalities (3.11) of \cite{advech} is still true when $\tilde m_{i+1}$, $\tilde L_i$, $\tilde m_{i+1}^*$ and $\tilde \tau_{i+1}^*(h_t)$ are defined as in this paper (in Subsection \ref{coin}) and when $\tilde L_i^*$ and $X_i^*$ of \cite{advech} are replaced by respectively $\tilde \tau_i(h_t)$ and $X(H(\tilde \tau_i(h)) + .)$: 
{
\begin{align*}
& \sup_{x \in \mathbb{R}} \left ( \mathcal{L}_X (H(\tilde m_{i+1}),x) - \mathcal{L}_X (H(\tilde L_i),x) \right ) \\
= & \sup_{\tilde \tau_i(h_t) \leq x \leq \tilde m_{i+1}} \left ( \mathcal{L}_X (H(\tilde m_{i+1}),x) - \mathcal{L}_X (H(\tilde L_i),x) \right ) \\
\leq & \sup_{\tilde \tau_i(h_t) \leq x \leq \tilde m_{i+1}} \mathcal{L}_{X(H(\tilde \tau_i(h)) + .)} \left ( H_{X(H(\tilde \tau_i(h)) + .)} (\tilde m_{i+1}),x \right ) \\
\leq & \sup_{\tilde \tau_i(h_t) \leq x \leq \tilde m_{i+1}^*} \mathcal{L}_{X(H(\tilde \tau_i(h)) + .)} \left ( H_{X(H(\tilde \tau_i(h)) + .)} (\tilde \tau_{i+1}^*(h_t)),x \right ). 
\end{align*}
Moreover, thanks to the Markov property applied to $X$ at $H(\tilde \tau_i(h))$ and to $V$ at $\tilde \tau_{i}(h_t)$, and to the definitions of $\tilde m_{i+1}^*$ and $\tilde \tau_{i+1}^*(h_t)$, the last term of this chain of inequality has the same law as $\sup_{x \in [0, m^*(h_t)]} \mathcal{L}_X [H(\tau^*(h_t)),x]$.} Therefore, combining \eqref{analog3.2eq2} and the first point of Fact \ref{analog3.3} we get
\begin{align}
\mathbb{P} \left( \bigcap_{j=1}^{n_t - 1}\left\{ \sup_{x \in \mathbb{R}} \left ( \mathcal{L}_X (H(\tilde m_{j+1}),x) - \mathcal{L}_X (H(\tilde L_j),x) \right ) \leq t e^{(\kappa (1+3\delta )-1)\phi(t)} \right \} \right) \geq 1 - \frac{C_2''(n_t - 1)}{n_t e^{\kappa \delta \phi(t)}}, \label{analog3.2eq3}
\end{align}
for some constant $C_2''$ and $t$ large enough. We have used the fact that, since $n_t \sim e^{\kappa (1+\delta) \phi(t)}$ where $\phi(t) << \log(t) \sim h_t$, we have $n_t e^{-c h_t} << e^{-\kappa \delta \phi(t)}$. The combination of \eqref{analog3.2eq1} and \eqref{analog3.2eq3} is the sought result. 

%$\mathbb{P} (\mathcal{A}_5^c)$ and $\mathbb{P} (\mathcal{A}_6^c)$ are controlled thanks to Lemmas \ref{noreturn} and \ref{tpscoinc}. Lemma 3.3 of \cite{advech} has to be replaced here by Lemma \ref{analog3.3}. Finally, $\{ \tilde m_1 = m_1^*(h_t) \} \subset \{ \tilde \tau_{1}^*(h_t) = \tilde\tau_{1}(h_t) \}$, so $\mathbb{P} (\tilde m_1 \neq m_1^*(h_t)) \leq e^{-c h_t}$ according to Lemma \ref{tpscoinc}. 

%\end{proof}
%
%
%\begin{lemme} \label{analog3.4}
%
%There is a positive constant $C$ such that, 
%\[ \mathbb{P} \left( \bigcap_{j=1}^{ n_t}\left\{ \sup_{x \in [\tilde{L}_{j-1}, \tilde{L}_j] \cap \mathcal{D}_j^c} \left ( \mathcal{L}_X (H(\tilde L_j),x) - \mathcal{L}_X (H(\tilde m_{j}),x) \right ) \leq t e^{-2 \phi(t)}, \right \} \right) \geq 1-\frac{C n_t}{e^{2 \phi(t)}}. \]
%for $t$ large enough. 
%
%\end{lemme}
%
%\begin{proof} of third point

The third point of Fact \ref{analog3.3} is similar to Lemma 3.4 of \cite{advech}. Here are the modifications that we make on the original proof of \cite{advech}: 

$W_{\kappa}$ is replaced here by $V$. The standard valleys ($\tilde V^{(j)}(.)$ and the random points $\tilde m_j$, $\tilde L_j$), $A^j(.)$, $\mathcal{D}_j$ have to be considered as the ones defined in this paper (in Subsection \ref{coin} for the standard valleys, {a little before \eqref{approxiid2}} for $A^j(.)$, and in \eqref{defdj} for $\mathcal{D}_j$). Also, $\tilde \tau_j^-(h_t^+)$ of \cite{advech} has to be replaced here by $\tilde{L}_{j-1}$. 
%and the random times $\tau_j^-(h_t^+)$ of \cite{advech} have to be replaced here by $\tilde{L}_{j-1}$. 
We work with $r_t = (\phi(t))^2$ (instead of $r_t = C_0 \phi(t)$ in \cite{advech}). 

{$\mathcal{A}_1$ and $\mathcal{A}_2$ have the same meaning as in the original proof of \cite{advech}: 
\[ \mathcal{A}_1 := \left \{ \sup_{u \in \mathbb{R}} \mathcal{L}_B (\tau(B,1), u) \leq e^{2 \phi(t)} \right \}, \ \ \ \mathcal{A}_2 := \left \{ A^j(\tilde{L}_j) \leq 2 e^{h_t + 2 \phi(t)/\kappa} \right \}. \]
$\mathcal{A}_1$ is an inequality satisfied (with high probability) by the Brownian local time and in particular we still have $\mathbb{P} (\mathcal{A}_1^c) \leq 5 e^{-2 \phi(t)}$ as in the original proof of \cite{advech}.} For $\mathcal{A}_2$ we have that, in our setting, $\mathbb{P} ( \mathcal{A}_2^c )$ is less than
\begin{align*}
& \mathbb{P} \left( \int_{\tilde m_{j}}^{\tilde \tau_j(h_t)} e^{\tilde V^{(j)}(y)} dy > e^{h_t + 2\phi(t)/\kappa} \right ) + \mathbb{P} \left( \int_{\tilde \tau_j(h_t)}^{\tilde L_{j}} e^{\tilde V^{(j)}(y)} dy > e^{h_t + 2\phi(t)/\kappa} \right) \\
= & \mathbb{P} \left( \int_0^{\tau(V^{\uparrow}, h_t)} e^{V^{\uparrow}(y)} dy > e^{h_t + 2\phi(t)/\kappa} \right ) + \mathbb{P} \left( e^{h_t} \int_0^{\tau \left (V, \left ]-\infty, -\frac{h_t}{2} \right ] \right )} e^{V(y)} dy > e^{h_t + 2\phi(t)/\kappa} \right), 
\end{align*}
where we have used Proposition \ref{standardwilliams} (applied with $h=h_t$) for the laws of $\tilde P_2^{(j)}$ and $\tilde P_3^{(j)}$, 
\begin{align*}
& \leq \mathbb{P} \left( \tau(V^{\uparrow}, h_t) > e^{2\phi(t)/\kappa} \right ) + \mathbb{P} \left( \int_0^{+\infty} e^{V(y)} dy > e^{2\phi(t)/\kappa} \right), \\
& \leq e^{c_1 h_t - c_2 e^{2\phi(t)/\kappa}} + 2 \mathcal{C} e^{-2\phi(t)}. 
\end{align*}
The last inequality is true for $t$ large enough. It follows from Lemma \ref{vposlapltpsatt} applied with $y=h_t, r=e^{2\phi(t)/\kappa}$ for the first term, and from the second assertion of Lemma \ref{foncexpov} for the second term (the constants are the ones from these lemmas). Since $\phi(t) >> \log(\log(t))$ and $h_t \sim \log(t)$, we get that $\mathbb{P} ( \mathcal{A}_2^c ) \leq c e^{-2\phi(t)}$ for $t$ large enough, and where $c$ is some positive constant. 

The event $\mathcal{A}_3$ from the original proof of \cite{advech} is not needed here thanks to our definition of $\mathcal{D}_j$. {Recall the definitions of $\tilde \tau_j^-(.)$ and $\tilde \tau_j^+(.)$ in the beginning of Subsection \ref{coin}.} We redefine 
\[ \mathcal{A}_4 := \left \{ \inf_{[\tilde \tau_j^+((\phi(t))^2), \tilde \tau_j(h_t)]} \tilde V^{(j)} \geq (\phi(t))^2/2 \right \}. \]
{According to Proposition \ref{standardwilliams} (applied with $h=h_t$) $\inf_{[\tilde \tau_j^+((\phi(t))^2), \tilde \tau_j(h_t)]} \tilde V^{(j)}$ is equal in law to $\inf_{[\tau(V^{\uparrow}, (\phi(t))^2), \tau(V^{\uparrow}, h_t)]} V^{\uparrow}$ which, by the Markov property at time $\tau(V^{\uparrow}, (\phi(t))^2)$, is equal in law to $\inf_{[0, \tau(V_{(\phi(t))^2}^{\uparrow}, h_t)]} V_{(\phi(t))^2}^{\uparrow}$. We then apply Lemma \ref{vuprestegrand} with $a = (\phi(t))^2/2, b = (\phi(t))^2$ and we get $\mathbb{P} ( \mathcal{A}_4^c ) \leq c_2 e^{-c_1 (\phi(t))^2 /2}$, where $c_1, c_2$ are the constants in the lemma. For $t$ large enough we thus have $\mathbb{P} ( \mathcal{A}_4^c ) \leq e^{-2\phi(t)}$. } We also redefine
\[ \mathcal{A}_5 := \left \{ \inf_{[\tilde \tau^-_j(h_t), \tilde \tau^-_j((\phi(t))^2)]} \tilde V^{(j)} \geq (\phi(t))^2/2 \right \} \text{ and } \mathcal{A}_6 := \left \{ \inf_{[\tilde{L}_{j-1}, \tilde \tau_j^-(h_t / 2)]} \tilde V^{(j)} > h_t/4 \right \}. \]
{Note that $\inf_{[\tilde \tau^-_j(h_t), \tilde \tau^-_j((\phi(t))^2)]} \tilde V^{(j)}$ is a function of $\tilde P_1^{(j)}$ and, according to Proposition \ref{standardwilliams} (applied with $h=h_t$), we have $d_{VT}( \tilde P_1^{(j)}, P_1^{(2)}) \leq 2 e^{- \delta \kappa h_t /3}$ (where $\delta$ is as defined in Subsection \ref{coin} and Section \ref{genedesres}). Recall that, according to Proposition \ref{Fact_Williams} (applied with $h=h_t$), the law of $P_1^{(2)}$ is absolutely continuous with respect to the law of the process $(\hat{V}^{\uparrow}(x))_{0 \leq x \leq \tau(\hat{V}^{\uparrow}, h_t+)}$ and the density is bounded by $2$ when $h_t$ is large enough. As a consequence, we get that for $t$ large enough, 
\begin{align*}
\mathbb{P} ( \mathcal{A}_5^c ) & \leq 2 e^{- \delta \kappa h_t /3} + 2 \mathbb{P} \left ( \inf_{[\tau(\hat V^{\uparrow}, (\phi(t))^2+), \tau(\hat V^{\uparrow}, h_t+)]} \hat V^{\uparrow} < (\phi(t))^2/2 \right ) \\
& \leq 2 e^{- \delta \kappa h_t /3} + 2 \mathbb{P} \left ( \inf_{[\tau(\hat V^{\uparrow}, (\phi(t))^2+), +\infty[} \hat V^{\uparrow} < (\phi(t))^2/2 \right ) \\
& \leq 2 e^{- \delta \kappa h_t /3} + 2 e^{-\kappa (\phi(t))^2/2} / (1-e^{-\kappa (\phi(t))^2}) \leq e^{-2 \phi(t)}. 
\end{align*}
In the last line we have used Lemma \ref{vdownrestegrand} with $z=0, a=(\phi(t))^2/2, b=(\phi(t))^2$ and the last inequality holds for $t$ large enough. From the definition of $\tilde{L}_{j}^{\sharp}$ we have $V(\tilde L_{j}^{\sharp})=\inf_{[\tilde L_{j-1}, \tilde L_{j}^{\sharp}]} V$ so $\inf_{[\tilde{L}_{j-1}, \tilde \tau_j^-(h_t / 2)]} \tilde V^{(j)} = \inf_{[\tilde L_{j}^{\sharp}, \tilde \tau_j^-(h_t/2)]} \tilde V^{(j)}$. We can thus apply Lemma \ref{minoprevalley} with $\alpha = 1/2, \eta = 1/4, h=h_t$ and we get $\mathbb{P} ( \mathcal{A}_6^c ) \leq e^{-\kappa h_t/12}$ for large $t$. }

Finally, note that on $\cap_{i=4}^6 \mathcal{A}_i$ we have $e^{-\tilde V^{(j)}(x)} \leq e^{-(\phi(t))^2/2}, \ \forall x \in [\tilde{L}_{j-1}, \tilde{L}_{j}] \cap \mathcal{D}_j^c$, so the conclusion follows as in the proof of Lemma 3.4 in \cite{advech} (but here we do not need to intersect with the event $\mathcal{V}_t$ of \cite{advech}). 
%A CAUSE DES SAUTS LE POTENTIEL PEUT REDEVENIR PETIT AVANT $\tilde{L}_{j}$ -> NON, PLUS MAINTENANT QUE J'AI REMPLACE TILDE L_I^* PAR TILDE TAU_I(H)
\end{proof}

{
\begin{proof} of Fact \ref{analogue(5.22)}

This is the analogous of (5.23) from \cite{advech}. Let us fix $j \geq 1$. Expressing $\sup_{y \in \mathcal{D}_j} \mathcal{L}_{X_{\tilde m_j}}(H_{X_{\tilde m_j}}(\tilde L_j),y)$ in term of the Brownian motion $B^j$ appearing in \eqref{approxiid3} we get 
\begin{align}
\sup_{y \in \mathcal{D}_j} \mathcal{L}_{X_{\tilde m_j}}(H_{X_{\tilde m_j}}(\tilde L_j),y) & = \sup_{y \in \mathcal{D}_j} e^{-\tilde V^{(j)}(y)} A^j(\tilde{L}_j) \mathcal{L}_{B^j} \left [ \tau(B^j, 1),A^j(y)/A^j(\tilde{L}_j) \right ] \nonumber \\
& \leq A^j(\tilde{L}_j) \sup_{y \in \mathcal{D}_j} \mathcal{L}_{B^j} \left [ \tau(B^j, 1),A^j(y)/A^j(\tilde{L}_j) \right ]. \label{analogue(5.22)1}
\end{align}
It follows from the definition of $\mathcal{D}_j$ in \eqref{defdj} that $\mathcal{D}_j \subset [\tilde \tau_1^-(h_t/2), \tilde \tau_1(h_t/2)]$. Therefore, according to Lemma \ref{truccentral} applied with $\epsilon = 1/6$, there is a positive constant $c$ such that for $t$ large enough $|A^j(u) / A^j(\tilde{L}_j)|$ is bounded on $\mathcal{D}_j$ by $e^{-h_t/3}$, with probability at least $1 - e^{-c h_t}$. Combining with (7.11) from \cite{advech} applied to $\mathcal{L}_{B^j}$ with $\delta = e^{-h_t/3}$, $\epsilon = e^{-h_t/9}$, we get for $c$ possibly decreased and $t$ large enough : 
\[ \mathbb{P} \left ( \sup_{y \in \mathcal{D}_j} \mathcal{L}_{B^j} \left [ \tau(B^j, 1),A^j(y)/A^j(\tilde{L}_j) \right ] \leq (1 + e^{-h_t/9}) \mathcal{L}_{B^j} \left ( \tau(B^j, 1),0 \right ) \right ) \geq 1- e^{-c h_t}. \]
Putting this into \eqref{analogue(5.22)1} and using that, thanks to \eqref{approxiid2} and \eqref{approxiid3}, we have the expression $\mathcal{L}_X(H(\tilde{L}_j), \tilde m_j) = A^j(\tilde{L}_j) \mathcal{L}_{B^j} (\tau(B^j, 1),0 )$, we get 

\noindent $\mathbb{P} ( \sup_{y \in \mathcal{D}_j} \mathcal{L}_{X_{\tilde m_j}}(H_{X_{\tilde m_j}}(\tilde L_j),y) \leq (1+e^{-h_t/9}) \mathcal{L}_X(H(\tilde L_j),\tilde m_j) ) \geq 1 - e^{-ch_t}$ for some constant $c$ (not depending on $j$) and $t$ large enough. 

We thus obtain that $\mathbb{P} ( \mathcal{A}^7_t) \geq 1-n_t e^{-c h_t}$ for $t$ large enough. Finally, recall that $n_t \sim e^{\kappa (1+\delta) \phi(t)}$ (by the definition of $n_t$ in the beginning of Section \ref{genedesres}) where $\phi(t) << \log(t) \sim h_t$. By decreasing $c$ a little we thus obtain the asserted result. 

\end{proof}
}

\begin{proof} of Fact \ref{analogue5.3}

This is Lemma 5.3 of \cite{advech}. As in there, let $\sigma_{X}(a,b) := \inf \{ s \geq 0, \ \mathcal{L}_X(s, b) > a \}$ be the inverse of the local time of $X$ at $b$. Here are the modifications that we make on the proof of Lemma 5.3 of \cite{advech}: 

$W_{\kappa}$ is replaced here by $V$. {The standard valleys ($\tilde V^{(j)}(.)$ and the random points $\tilde m_j$, $\tilde L_j$, $\tilde \tau_j^-(h_t/2)$, $\tilde \tau_j^+(h_t/2)$), $A^j(.)$, $\mathcal{D}_j$, 
%$e_j$ 
and $f_{\gamma}(.)$, $\tilde f_{\gamma}(.)$, $f_{\gamma}^{\pm}(.)$ have to be considered as the ones defined in this paper (in Subsection \ref{coin} for the standard valleys, {a little before \eqref{approxiid2}} for $A^j(.)$, in \eqref{defdj} for $\mathcal{D}_j$, 
%just before Lemma \ref{boundj0} for $e_j$, 
and in the statement of Fact \ref{analogue5.3} for $f_{\gamma}(.)$, $\tilde f_{\gamma}(.)$, $f_{\gamma}^{\pm}(.)$). In this proof, we systematically replace $\tilde L_1^-$, $R_1$ and $\mathcal{H}_1$ of \cite{advech} by respectively $\tilde{L}_{0}$ {(which equals $0$ by definition)}, $R_1^t$ and $e_1 S_1^t R_1^t$.} In particular, the domain of integration in the integral $I$, which represents the inverse of the local time, is $[\tilde L_{0}, \tilde L_{1}]$: 
{ \[ I := \sigma_{X}(\gamma t, \tilde m_1) = \int_{\tilde L_{0}}^{\tilde L_{1}} e^{-\tilde V^{(1)}(z)} \mathcal{L}_B(\sigma_{B}(\gamma t,0), A^1(z)) dz = \gamma t \int_{\tilde L_{0}}^{\tilde L_{1}} e^{-\tilde V^{(1)}(z)} \mathcal{L}_{\tilde B}(\sigma_{\tilde B}(1,0), \tilde a(z)) dz. \] 
As in the original proof of \cite{advech}, $\tilde a (z) := (\gamma t)^{-1} A^1(z)$ and $\tilde B := B((\gamma t)^2 . )/(\gamma t)$ is a Brownian motion that we still denote $B$ in the sequel. } 
%We already made the remark (after Lemma \ref{sortparladroite}) that $\tilde L_j^-$ of \cite{advech} (also denoted $\tilde \tau_j^-(h_t^+)$) had to be replaced here by $\tilde L_{j-1}$ in order to still have the valleys left from the right. 

{Using Proposition \ref{approxparliid} we get the analogous of $(5.28)$ of \cite{advech}: 
\begin{eqnarray}
\mathbb{P} \left ( (1-e^{- \epsilon h_t/7}) e_1 S_1^t R_1^t \leq H_{X_{\tilde m_1}}(\tilde L_{1}) \leq (1+e^{- \epsilon h_t/7}) e_1 S_1^t R_1^t \right ) \geq 1 - e^{-c h_t}, \label{analogue5.28}
\end{eqnarray}
for $t$ large enough, and where $\epsilon$ and $c$ are as in Proposition \ref{approxparliid}. We recall the notations $X_{\tilde m_i} := X(. + H(\tilde m_i))$ and $H_{X_{\tilde m_i}}(r) := \tau(X_{\tilde m_i}, r)$ that we have used in the above expression. As in the original proof of \cite{advech} we put 
\[ I_1 := \int_{\tilde \tau_1^-(h_t/2)}^{\tilde \tau_1^+(h_t/2)} e^{-\tilde V^{(1)}(z)} \mathcal{L}_B(\sigma_{B}(1,0), \tilde a(z)) dz \ \ \ \text{and} \ \ \ I_2 := (\gamma t)^{-1} I - I_1. \]} 
A little after, our lower bound for $\mathbb{P} ( |A^1(\tilde \tau_1^-(h_t/2))| \leq e^{h_t(1 + \epsilon)/2}, |A^1(\tilde \tau_1^+(h_t /2))| \leq e^{h_t(1 + \epsilon)/2} )$ comes from the proof of Lemma \ref{truccentral}. {Indeed, it can be seen from that proof that there is a positive constant $c_1$ such that for $t$ large enough we have $\mathbb{P} ( |A^1(\tilde \tau_1^-(h_t/2))| \leq e^{h_t(1 + \epsilon)/2}, |A^1(\tilde \tau_1^+(h_t /2))| \leq e^{h_t(1 + \epsilon)/2} ) \geq 1-e^{-c_1 h_t}$. {As a consequence, in the exponential in the analogous of $(5.32)$ of \cite{advech}, we may have a constant $c_2$ that is different from the one from the original proof of \cite{advech} but it does not matter (note also that, from our definition of $R_1^t$ in Lemma \ref{approxdeu}, we can replace $\tilde R_1$ of \cite{advech} by our $R_1^t$). The analogous of $(5.32)$ of \cite{advech} is therefore 
\[ \mathbb{P} \left( |I_1 - R_1^t| \leq \hat \epsilon_t R_1^t \right ) \geq 1 - e^{-c_2 h_t}, \]
for $t$ large enough, and where $\hat \epsilon_t := t^{-(1-5\epsilon)/4}$ as in the original proof of \cite{advech}. }
%To prove that $\mathbb{P} (\forall {x \in [ \tilde \tau_1^-(h_t/2),\tilde \tau_1(h_t/2)]},  |\tilde a(x)| \leq e^{-(\log t) (1-3\epsilon)/2}) \geq 1-e^{-c_1 h_t}$ for some constant $c_1$, we only need to give an upper bound for and, and this is done as in the proof of Lemma \ref{approxdeu}. 

We now give the details about how, here, we obtain the analogous of $(5.33)$ of \cite{advech}. To bound $I_2$ we first use $(7.16)$ of \cite{advech} together with the second Ray-Knight theorem and we get 
\[ \mathbb{P} \left( \sup_{y \in \mathbb{R}} \mathcal{L}_B(\sigma_{B}(1,0), y) \geq e^{\epsilon h_t /2} \right ) \leq 2 e^{-\epsilon h_t /2}. \]
Note that, compared with the original proof of \cite{advech}, we took $\epsilon h_t/2$ instead of $\epsilon \log t$.} Then, since the domain of integration in $I$ is $[\tilde L_{0}, \tilde L_{1}]$, we have in our case
\[ I_2 \leq e^{\epsilon h_t /2} \left(\int_{\tilde L_{0}}^{\tilde \tau_1^-(h_t/2)}e^{-\tilde V^{(1)}(x)}dx+ \int_{\tilde \tau_1^+(h_t/2) }^{\tilde L_{1}} e^{-\tilde V^{(1)}(x)}dx \right), \]
with probability larger than $1-2e^{-\epsilon h_t /2}$. Assume that $\epsilon$ has been chosen small enough so that Lemmas \ref{boundj0} and \ref{boundj2} apply. By these Lemmas we have that for $t$ large enough, each of these two integrals is less than $e^{-\epsilon h_t}$ with a probability greater than $1- e^{-c h_t}$, where $c$ is some positive constant. As a consequence there is a positive constant $c_3$ such that for $t$ large enough :
\begin{eqnarray}
\mathbb{P} \left( I_2 \leq 2 e^{-\epsilon h_t /2} \right) \geq 1- e^{-c_3 h_t}. \label{genelemmefctrep1}
\end{eqnarray}
Our lower bound for $R_1^t$ comes from \eqref{approxdeu0.3}. Combing with \eqref{genelemmefctrep1} we get that for some positive constant $c_4$ and $t$ large enough, 
\[ \mathbb{P} \left( I_2 \leq e^{-\epsilon h_t /4} R_1^t \right) \geq 1- e^{-c_4 h_t}, \]
which is the analogous of $(5.33)$ of \cite{advech}. As a consequence, in place of $(5.34)$ of \cite{advech} we have 
%\begin{align}
%\mathbb{P} \left ( |\sigma_{X_{\tilde m_1}}(\gamma t,\tilde m_1) - \gamma t R_1^t | \leq e^{-\epsilon h_t /5} \gamma t R_1^t, H_{X_{\tilde m_1}}(\tilde L_1)>t(1-x), H_{X_{\tilde m_1}}(\tilde L_{1})< H_{X_{\tilde m_1}}(\tilde L_{0}) \right ) \label{genelemmefctrep2}
%\end{align}
\begin{align}
\mathbb{P} \left ( | I - \gamma t R_1^t | \geq e^{-\epsilon h_t /5} \gamma t R_1^t, H_{X_{\tilde m_1}}(\tilde L_1)> \sigma_X (\gamma t, \tilde m_1), H_{X_{\tilde m_1}}(\tilde L_{1})< H_{X_{\tilde m_1}}(\tilde L_{0}) \right ) \leq e^{-c_5 h_t} \label{genelemmefctrep2}
\end{align}
for all $t$ large enough, and where $c_5$ is a positive constant. {Using \eqref{analogue5.28} and \eqref{genelemmefctrep2} instead of respectively $(5.28)$ and $(5.34)$ of \cite{advech} we obtain the analogous of $(5.35)$ of \cite{advech}: 
\begin{align}
& \left \{ H_{X_{\tilde m_1}}(\tilde L_1)> \sigma_X (\gamma t, \tilde m_1) \geq t(1-x), H_{X_{\tilde m_1}}(\tilde L_{1})< H_{X_{\tilde m_1}}(\tilde L_{0}) \right \} \nonumber \\
\subset & \left \{ \frac1{R_1^t} \leq \frac{\gamma}{1-x}(1+e^{-\epsilon h_t /5}), e_1 S_1^t R_1^t > t(1-x)(1+e^{- \epsilon h_t/7})^{-1}, H_{X_{\tilde m_1}}(\tilde L_{1})> \sigma_X (\gamma t, \tilde m_1) \right \} \cup \mathcal{E}_{\epsilon}^1, \label{analogue5.35}
\end{align}
where $\mathbb{P}(\mathcal{E}_{\epsilon}^1) \leq e^{-c h_t} + e^{-c_5 h_t}$ for all $t$ large enough (and where $c$ and $c_5$ are as in \eqref{analogue5.28} and \eqref{genelemmefctrep2}). }
%We recall that $\epsilon_t$ has been defined in Proposition \ref{approxparliid} an equals $e^{-\epsilon h_t /6}$. 

%$\mathbb{E} \left[ P_{\tilde m_1}^{V}(H'(\tilde L_{0}) \leq H'(\tilde L_{1}))\right]$

%INUTILE: Finally, an upper bound for $\mathbb{P} (H_{X_{\tilde m_1}}(\tilde L_{0}) < H_{X_{\tilde m_1}}(\tilde L_{1}) )$ is given by Lemma \ref{sortparladroite} and we get our bounds for $f_{\gamma}(x)$ concluding as in the original proof. 

%Note that $\tilde \epsilon_t$ of \cite{advech} is replace by our $\epsilon_t$. 

%use respectively \eqref{0MajorationAVallee} and \eqref{d0MajorationAVallee} to bound $|A^1(\tilde \tau_j(h_t/2))|$ and $|A^1(\tilde \tau_j^-(h_t/2))|$: 
%\[ \mathbb{P} \left ( |A^1(\tilde \tau_j(h_t/2))| \leq e^{h_t (1+\epsilon)/2} \right ) \geq 1 - e^{-h_t/2} \ \text{and} \ \mathbb{P} \left ( |A^1(\tilde \tau_j^-(h_t/2))| \leq e^{h_t (1+\epsilon)/2} \right ) \geq 1 - e^{- c_7 h_t/2}. \]
%Combining with \eqref{approxdeu0}, we get $\mathbb{P} (\forall {x \in [ \tilde \tau_j(h_t/2),\tilde \tau_j^-(h_t/2)]},  |\tilde a(x)| \leq e^{-(\log t) (1-3\epsilon)/2}) \geq 1-e^{-c h_t}$ for some constant $c$, when $t$ is large enough. 

Repeating the reasoning of the original proof of \cite{advech} and using our definition of $e_1$ from Subsection \ref{maincontibcommeiid} we still have
 \[ \sigma_X(\gamma t,\tilde m_1) > H_{X_{\tilde m_1}}(\tilde L_{1}) \Leftrightarrow  \gamma t > A^1(\tilde L_1) e_1 \Leftrightarrow \gamma t R_1^t > A^1(\tilde L_1) e_1 R_1^t. \]
 $(3.18)$ of \cite{advech} has to be replaced here by Lemma \ref{approxa} and, as we already mentioned, $(5.28)$ by \eqref{analogue5.28}. In place of $(5.37)$ of \cite{advech} we thus have
\[ \sigma_X(\gamma t,\tilde m_1) > H_{X_{\tilde m_1}}(\tilde L_{1}) \Rightarrow \gamma t R_1^t > (1+e^{- \epsilon h_t/7})^{-1} H_{X_{\tilde m_1}}(\tilde L_{1}), \]
except on an event whose probability is less than $e^{- c_6 h_t}$ (for some positive constant $c_6$). The fact that the constants are modified (with respect to the ones in the original proof of \cite{advech}) has no importance so we omit to mention it for the rest of the proof. {Combining the above with \eqref{analogue5.28} we obtain the analogous of $(5.38)$ of \cite{advech}: 
\begin{align}
& \left \{ \sigma_X (\gamma t, \tilde m_1) > H_{X_{\tilde m_1}}(\tilde L_1)> t(1-x), H_{X_{\tilde m_1}}(\tilde L_{1})< H_{X_{\tilde m_1}}(\tilde L_{0}) \right \} \nonumber \\
\subset & \left \{ \frac1{R_1^t} \leq \frac{\gamma}{1-x}(1+e^{- \epsilon h_t/7}), e_1 S_1^t R_1^t > t(1-x)(1+e^{- \epsilon h_t/7})^{-1}, \sigma_X (\gamma t, \tilde m_1) > H_{X_{\tilde m_1}}(\tilde L_{1}) \right \} \cup \mathcal{E}_{\epsilon}^2, \label{analogue5.38}
\end{align}
where $\mathbb{P}(\mathcal{E}_{\epsilon}^2) \leq e^{-c h_t} + e^{-c_6 h_t}$ for all $t$ large enough (and where $c$ is as in \eqref{analogue5.28} and $c_6$ as above). Then, the asserted upper bound for $f_{\gamma}(x)$ follows as in the original proof of \cite{advech}, with \eqref{analogue5.35} and \eqref{analogue5.38} instead of respectively $(5.35)$ and $(5.38)$ of \cite{advech}.}

For the proof of the lower bound for $\tilde f_{\gamma}(x)$, recall that $\mathcal{D}_1$ has to be considered as the one defined in this paper (in \eqref{defdj}). 

{We now explain how, in our case, we bound {$\hat a(z) := A^1(z) / t y = (1-4\hat \epsilon_t) R_1^t A^1(z) / t(1-x)$ on $\mathcal{D}_1$ (where, similarly as in the original proof of \cite{advech}, $y := (1-x)/(1-4\hat \epsilon_t) R_1^t$, $\hat \epsilon_t := t^{-(1-5\epsilon)/4}$)}. Chernoff's inequality together with Proposition \ref{cvr} yields $\mathbb{P} (R_1^t \geq h_t) \leq e^{-c_7 h_t}$ for some positive constant $c_7$ and $t$ large enough. According to our definition of $\mathcal{D}_1$ we have $\mathcal{D}_1 \subset [\tilde \tau_1^-(h_t/2), \tilde \tau_1^+(h_t/2)]$ for $t$ large enough, so we can use our previous estimate $\mathbb{P} ( |A^1(\tilde \tau_1^-(h_t/2))| \leq e^{h_t(1 + \epsilon)/2}, |A^1(\tilde \tau_1^+(h_t /2))| \leq e^{h_t(1 + \epsilon)/2} ) \geq 1-e^{-c_1 h_t}$. We thus get the existence of a positive constant $c_8$ such that for $t$ large enough the following inequalities hold with probability greater than $1-e^{-c_8 h_t}$: 
\[ \forall z \in \mathcal{D}_1, \ | \hat a(z)| = (1-4\hat \epsilon_t) R_1^t |A^1(z)| / t(1-x) \leq h_t e^{h_t(1 + \epsilon)/2} / t \epsilon \leq e^{-(\log t)(1-3\epsilon)/2}. \]
%the lower bounds for $\mathbb{P} ( A^1(\tilde \tau_1^-(h_t/2)) \leq e^{h_t(1 + \epsilon)/2} )$ and $\mathbb{P} ( A^1(\tilde \tau_1(h_t/2)) \leq e^{h_t(1 + \epsilon)/2} )$ (from the proof of Lemma \ref{truccentral}) 
That is, we have $\mathbb{P} ( \forall {z \in \mathcal{D}_1}, \ |\hat a(z)| \leq e^{-(\log t) (1-3\epsilon)/2} ) \geq 1-e^{-c_8 h_t}$ for large $t$.} 
%for some constant $c$, when $t$ is large enough. This is how we bound $\tilde a$ in our case. 

$(5.34)$ of \cite{advech} has, of course, to be replaced here by \eqref{genelemmefctrep2}. {$(5.23)$ from \cite{advech} has to be replaced here by Fact \ref{analogue(5.22)} (which imposes in particular that, here, we set $\tilde \gamma := \gamma(1+e^{-h_t/9})^{-1}$, instead of $\gamma(1+e^{-h_t/12})^{-1}$).} Finally, $(3.2)$ of \cite{advech} has to be replaced here by Lemma \ref{sortparladroite} (applied with $h=h_t$). $(5.28)$ of \cite{advech} and Proposition 3.5 of \cite{advech} have to be replaced here by \eqref{analogue5.28} and Proposition \ref{approxparliid}. 

\end{proof}

\bibliographystyle{plain}
\bibliography{thbiblio}

\end{document}